\documentclass{article}
\usepackage[english]{babel}
\usepackage[utf8]{inputenc}
\usepackage{bm}
\usepackage{amsmath}
\usepackage{amsfonts}
\usepackage{graphicx}
\usepackage[colorinlistoftodos]{todonotes}
\usepackage{blindtext}
\usepackage{amssymb}
\usepackage{bbm}
\usepackage{authblk}
\usepackage{mathrsfs}
\usepackage{amsthm}
\numberwithin{equation}{section}
\newtheorem{theorem}{Theorem}[section]
\newtheorem{corollary}{Corollary}[section]
\newtheorem{lemma}[theorem]{Lemma}
\newtheorem{proposition}{Proposition}[section]

\theoremstyle{definition}
\newtheorem{remark}{Remark}[section]
\newtheorem{definition}{Definition}[section]

\DeclareMathOperator\supp{supp}

\title{Large deviation bounds for the Airy point process}

\author{Chenyang Zhong}
\affil{Department of Statistics, Columbia University}

\date{}

\begin{document}
\maketitle
\begin{abstract}
In this paper, we establish the first large deviation bounds for the Airy point process, thus providing a partial answer to a conjecture raised in \cite{Cor} and \cite{CGKLT}. The proof is based on a novel approach which relies upon the approximation of the Airy point process using the Gaussian unitary ensemble (GUE) up to an exponentially small probability, together with precise estimates for the stochastic Airy operator and edge rigidity for beta ensembles. 

As a by-product of our estimates for the Airy point process, we significantly improve upon previous results (\cite{Cor,Kim}) on the lower tail probability of the one-point distribution of the KPZ equation with narrow-wedge
initial data and the half-space KPZ equation with Neumann boundary
parameter $A=-1\slash 2$ and narrow-wedge initial data in a unified and much shorter manner. Our bounds
hold for all sufficiently large time $T$, and for the first time establish sharp super-exponential decay with exponent $3$ for tail depth less than $T^{2\slash 3}$ (with
sharp leading prefactors $1\slash 12$ and $1\slash 24$ for tail depth less than $T^{
1\slash 6}$).
\end{abstract}

\section{Introduction}\label{Sect.1}

The Airy point process, introduced by Tracy and Widom in \cite{TW}, is a simple determinantal point process on $\mathbb{R}$  which arises as the limit of rescaled eigenvalue configuration of the Gaussian unitary ensemble (GUE) near its spectral edge. The correlation kernel of the Airy point process is given by
\begin{equation*}
    K^{\text{Ai}}(x,y)=\frac{\text{Ai}(x)\text{Ai}'(y)-\text{Ai}(y)\text{Ai}'(x)}{x-y},
\end{equation*}
where $\text{Ai}(x)$ is the Airy function. The Airy point process plays a central role in random matrix theory and integrable probability. For example, it is the universal limit of rescaled eigenvalues near the soft edge of a wide class of random matrix ensembles \cite{BEY, DG, For, Sos2}, and describes the largest parts of random partitions picked from the Plancherel measure \cite{BOO} as well as the liquid-solid transition in random tiling models \cite{Joh}. The Airy point process is also part of the Airy line ensemble \cite{CH, PS}, a universal object in Kardar-Parisi-Zhang universality. For an introduction to determinantal point processes including the Airy point process, we refer to \cite[Section 4.2]{AGZ}. We also refer the reader to, for example, \cite{Buf, CC2, Cla, Sos} for recent developments on the Airy point process. 

The Airy point process is also connected to the stochastic Airy operator $H_{\beta}$ (where $\beta>0$ is a parameter) introduced in \cite{RRV}. The spectrum of the stochastic Airy operator describes the limiting rescaled eigenvalue configuration near the spectral edge of $\beta$-ensembles (see e.g. \cite{BEY,DE,KRV,RRV})---an important family of random tridiagonal matrix ensembles that are connected to Calogero-Sutherland quantum systems and Jack polynomials. When $\beta=2$, the negative of the spectrum of $H_{\beta}$ has the same distribution as the point configuration of the Airy point process. Formal definitions of the Airy point process and the stochastic Airy operator, together with relevant background materials, will be presented in Section \ref{Sect.n} of this article.

In this article, we focus on \emph{large deviations} of the Airy point process, which describe rare events involving the configuration of random points from this process. Besides its intrinsic mathematical interest, large deviations of the Airy point process are also connected to the lower tail of the Kardar-Parisi-Zhang (KPZ) equation and positive temperature free fermions (see e.g. \cite{Cor,DLMS} for details). In particular, a conjecture asking to rigorously establish large deviations of the Airy point process was proposed in \cite[Section 2.3]{Cor} and \cite{CGKLT}. Our main result, presented in Theorem \ref{Main1} below, establishes the first rigorous large deviation bounds for the Airy point process, offering a partial answer to this conjecture. Our result also extends to the Airy$_{\beta}$ point process---the point process formed by the space reversal of the eigenvalues of the stochastic Airy operator $H_{\beta}$---for general $\beta\in\mathbb{N}_{+}$. When $\beta=1$, this process is connected to the half-space KPZ equation (see the discussion before Theorem \ref{KPZ2}); for general $\beta$, this process is also connected to interacting trapped fermions \cite{SLMS}. Additionally, through Theorem \ref{Main3} below, our method yields large deviation bounds for the collection of extreme eigenvalues of Gaussian $\beta$-ensembles near the spectral edge; such large deviations are of interest in random matrix theory (see e.g. \cite{FS, For2}) and connect large deviations for random matrices and large deviations for the KPZ equation (see \cite[Open Problems 2.9]{Guio}).

Below we provide further background on the KPZ equation and the above-mentioned connection between large deviations of the Airy point process and the KPZ equation. The KPZ equation was originally introduced in \cite{KPZ} as a model for random surface growth and is related to various physical phenomena such as directed polymers in random environments \cite{hhf} and interacting particle systems. We refer to \cite{Cor2,Qua,QS} for further background on the KPZ equation. Formally, the KPZ equation is given by
\begin{equation}
    \partial_T H(T,X)=\frac{1}{2}\partial^2_{X}H(T,X)+\frac{1}{2}(\partial_X H(T,X))^2+\xi(T,X),\quad (T,X)\in [0,\infty)\times \mathbb{R},
\end{equation}
where $\xi(T,X)$ is the space-time white noise. A physically relevant solution to the KPZ equation is given by the Cole-Hopf solution with narrow-wedge initial data, which can be described as follows. Let $Z(T,X)$ be the solution to the $(1+1)d$ stochastic heat equation (SHE) with multiplicative white noise $\xi(T,X)$:
\begin{equation}\label{stochastic_heat_eq}
    \partial_T Z(T,X)=\frac{1}{2}\partial^2_X Z(T,X)+Z(T,X)\xi(T,X),\quad (T,X)\in [0,\infty)\times\mathbb{R}.
\end{equation}
The reader is referred to \cite{Cor2,Qua,Wal} for further background on the SHE. The Cole-Hopf solution to the KPZ equation with narrow-wedge initial data is then given by
\begin{equation}\label{CHsolution}
    H(T,X):=\log{Z(T,X)}, \text{ with }Z(0,X)=\delta_{X=0}.
\end{equation}
Besides the full-space case, there has also been much mathematical interest in the half-space KPZ equation with Neumann boundary conditions \cite{BBCW, CS, Par}. Referring to \cite[Definition 7.1]{BBCW}, we denote by $H^{hf}(T,X)$ the Cole-Hopf solution to the half-space KPZ equation with Neumann boundary parameter $A=-1\slash 2$\\ and narrow-wedge initial data. Recently, it was proposed in \cite{CGKLT} that large deviations of the Airy point process can be used to derive the lower tail of the Cole-Hopf solution (\ref{CHsolution}) to the KPZ equation. In \cite{CGKLT}, based on Coulomb gas heuristics, a non-rigorous derivation of a large deviation principle (LDP) for the Airy point process (and hence the lower tail of the KPZ equation) was proposed; in \cite{Cor}, Corwin and Ghosal used estimates for the Airy point process to obtain upper and lower bounds for the lower tail probability, and raised the question of large deviations for the Airy point process; in \cite{Tsa}, for the regime where the tail depth is of order $T^{2\slash 3}$, Tsai used the stochastic Airy operator to obtain exact lower tail large deviations of the KPZ equation; in \cite{CC}, Cafasso and Claeys obtained exact lower tail large deviations of the KPZ equation using a Riemann-Hilbert approach. For the half-space KPZ equation with Neumann boundary parameter $A=-1\slash 2$ and narrow-wedge initial data, Kim \cite{Kim} obtained upper and lower bounds for the lower tail probability. However, the problem of a rigorous derivation of large deviations for the Airy point process remains open.  

Now we introduce our main results on large deviations of the Airy point process. Throughout the article, we denote the ordered points of the Airy point process by $a_1>a_2>\cdots$. We define the scaled and space-reversed Airy point process empirical measure as 
\begin{equation}\label{mains}
\mu_k:=\frac{1}{k}\sum_{i=1}^{\infty}\delta_{-k^{-2\slash 3}a_i},
\end{equation}
where $k\in\mathbb{N}_{+}$ is a scaling parameter which will be sent to infinity. We note that for any fixed $R>0$, the restriction of the measure $\mu_k$ to the interval $[-R,R]$ captures the behavior of the random points $\{a_i\}_{i=1}^{\infty}$ from the Airy point process that fall into the large interval $[-R k^{2\slash 3}, R k^{2\slash 3}]$. From \cite{Sos}, we have 
\begin{equation}\label{EQd}
    \lim_{T\rightarrow\infty} T^{-3\slash 2} \mathbb{E}[\#\{i\in\mathbb{N}_{+}:a_i\geq -T\}]=\frac{2}{3\pi}.
\end{equation}
Using (\ref{EQd}), we can deduce that as $k\rightarrow\infty$, the number of points $\{a_i\}_{i=1}^{\infty}$ that fall into $[-R k^{2\slash 3}, R k^{2\slash 3}]$ is of order $k$ with high probability. This justifies the $k^{-1}$ normalization in the definition of $\mu_k$.

Now we define $\nu_0$ to be the Borel measure on $\mathbb{R}$ with density $\pi^{-1}\sqrt{x}\mathbbm{1}_{[0,\infty)}(x)$ with respect to the Lebesgue measure. Note that from (\ref{EQd}), one can deduce that $\mu_k$ approaches $\nu_0$ in vague topology with high probability as $k\rightarrow\infty$. Thus for each $k\in\mathbb{N}_{+}$, we center $\mu_k$ by $\nu_0$ and obtain
\begin{equation}
    \nu_k := \mu_k-\nu_0.
\end{equation}
For any $R>0$, we denote by $\nu_{0;R}$ the restriction of $\nu_0$ to $[-R,R]$ (so $\nu_{0;R}$ is the Borel measure on $[-R,R]$ with density $\pi^{-1}\sqrt{x} \mathbbm{1}_{[0,R]}(x)$ with respect to the Lebesgue measure). Thus for any $R>0$, the restriction of $\nu_k$ to $[-R,R]$ is given by
\begin{equation}\label{nukr}
    \nu_{k;R}:=\frac{1}{k}\sum_{i\in\mathbb{N}_{+}:-k^{-2\slash 3}a_i\in [-R,R]}\delta_{-k^{-2\slash 3}a_i}-\nu_{0;R}. 
\end{equation}

Throughout this article, we use the following distance called the Kantorovich-Rubinshtein distance \cite{Hanin}.

\begin{definition}\label{Defn1.1}
For any $R>0$ and finite signed Borel measures $\mu,\nu$ on $[-R,R]$, we define the Kantorovich-Rubinshtein distance between $\mu$ and $\nu$ to be
\begin{equation*}
d_R(\mu,\nu):=\sup_{\substack{f:[-R,R]\rightarrow \mathbb{R},\\
    \|f\|_{BL}\leq 1}} \bigg|  \int_{[-R,R]} fd\mu-\int_{[-R,R]} fd\nu\bigg|,
\end{equation*}
where $\|f\|_{BL}:=\max\{\|f\|_{\infty},\|f\|_{Lip}\}$ is the bounded Lipschitz norm.
\end{definition}

Now we set up the topology. Throughout the rest of the article, we fix an arbitrary $R_0\geq 1$. We take the topological space $\mathcal{X}$ to be
\begin{eqnarray}\label{Def:X}
    \mathcal{X}&:=&\text{ the set of finite signed Borel measures }\mu \text{ on }[-R_0,R_0] \text{ such that }\nonumber\\ && \mu+\nu_{0;R_0} \text{ is a positive measure on }  [-R_0,R_0].
\end{eqnarray}
The set $\mathcal{X}$ is equipped with the metric $d_{R_0}$. Note that $d_{R_0}$ metrizes the weak topology on $\mathcal{X}$. We also define
\begin{eqnarray}\label{DefiZ}
    \mathcal{Z}&:=& \text{ the set of compactly supported finite signed Borel measures } \mu  \text{ on } \mathbb{R}\nonumber \\
    &&\text{ such that } \mu+\nu_0 \text{ is a positive measure }.
\end{eqnarray}

The following definition concerning the rate function will be involved in our main theorem.
\begin{definition}
For any $s\in [0,\infty],t\in (0,\infty]$ and any finite signed Borel measure $\mu$ on $\mathbb{R}$ such that $\mu+\nu_0$ is a positive measure, we define
\begin{equation*}
    I_{s,t}(\mu):=-\int_{[-s,s]^2}\log(\max\{|x-y|,t^{-3}\})d\mu(x)d\mu(y)+\frac{4}{3}\int_{(-\infty,0]}|x|^{3\slash 2}d\mu(x).
\end{equation*}
For any $\mu\in\mathcal{X}$, we further define
\begin{equation*}
    I_1(\mu):=\inf_{\substack{\tilde{\mu}\in\mathcal{Z},\tilde{\mu}(\mathbb{R})=0,\\ \tilde{\mu}|_{[-R_0,R_0]}=\mu} }\{I_{\infty,\infty}(\tilde{\mu})\}.
\end{equation*}
\end{definition}

Now we state the main result of this paper, which establishes the first rigorous large deviation bounds for the Airy point process. 

\begin{theorem}\label{Main1}
Let $I_1$ be given as in the preceding and recall (\ref{nukr}). We have the following:
\begin{enumerate}
 \item[(a)](LDP lower bound) For any open set $O\subseteq \mathcal{X}$,
 \begin{equation}\label{LDP_LBD}
 \liminf_{k\rightarrow\infty}\frac{1}{k^2}\log{\mathbb{P}(\nu_{k;R_0}\in O)}\geq -\inf_{\mu\in O}I_1(\mu).
 \end{equation}
 \item[(b)](LDP upper bound) For any closed set $F\subseteq \mathcal{X}$,
 \begin{equation}\label{LDP_UBD}
 \limsup_{k\rightarrow\infty}\frac{1}{k^2}\log{\mathbb{P}(\nu_{k;R_0}\in F)}\leq -\inf_{\mu\in F}I_2(\mu),
 \end{equation}
 where $I_2$ is an explicit non-negative function (which will be defined in Definition \ref{I2} in Section \ref{Sect.3.4}), such that $\inf\limits_{\mu\in F}I_2(\mu)$ is strictly positive for any closed set $F$ that does not contain $0$.
 \item[(c)](Exponential tightness) Take
 $K_{\eta}:=\big\{\mu\in\mathcal{X}:|\mu|([-R_0,R_0])\leq \eta R_0^{3\slash 2}\big\}$ for any $\eta>0$. Then $K_\eta$ is a compact subset of $\mathcal{X}$, and
 \begin{equation}\label{exponential.tightness}
    \lim_{\eta\rightarrow\infty}\limsup_{k\rightarrow\infty}\frac{1}{k^2}\log{\mathbb{P}(\nu_{k;R_0}\notin K_{\eta})}=-\infty.
 \end{equation}
\end{enumerate}
\end{theorem}
\begin{remark}
Note that in Theorem \ref{Main1}, we obtain large deviation bounds for the restricted measure $\nu_{k;R_0}$. We do \emph{not} require $\nu_k$ to be compactly supported, and thus unbounded measures can be covered by the theorem. The result can be directly transferred to large deviation bounds for $\mu_k$ defined in (\ref{mains}) (see Remark \ref{R} for details).
\end{remark}
\begin{remark}
A similar result holds for the Airy$_{\beta}$ point process---the point process formed by the space reversal of the eigenvalues of the stochastic Airy operator $H_{\beta}$---with general $\beta\in\mathbb{N}_{+}$ using the same argument.
\end{remark}
\begin{remark}
Theorem \ref{Main1} allows us to recover the lower bound part of the exact lower tail of the KPZ equation with narrow-wedge initial data (\cite[Theorem 1.1 and Corollary 1.3]{Tsa}) by rigorously justifying the heuristic argument given in \cite{CGKLT}: For fixed $\zeta>0$,
\begin{equation*}
    \liminf_{T\rightarrow \infty} \frac{1}{T^2}\log{\mathbb{P}\Big(H(2T,0)+\frac{T}{12}<-\zeta T\Big)}\geq -\Phi_{-}(-\zeta),
\end{equation*}
\begin{equation*}
    \liminf_{T\rightarrow \infty} \frac{1}{T^2}\log{\mathbb{P}\Big(H^{hf}(2T,0)+\frac{T}{12} <-\zeta T\Big)}\geq -\frac{1}{2} \Phi_{-}(-\zeta),
\end{equation*}
where for any $z\leq 0$,
\begin{equation*}
    \Phi_{-}(z):=\frac{4}{15\pi^6}(1-\pi^2 z)^{5\slash 2}-\frac{4}{15\pi^6}+\frac{2}{3\pi^4}z-\frac{1}{2\pi^2}z^2.
\end{equation*}
A non-trivial upper bound (for $\limsup$) can also be obtained from Theorem \ref{Main1}. If a full LDP for $\nu_{k;R_0}$ with rate $I_1$ could be established (it suffices to strengthen part (b) with $I_2$ replaced by $I_1$), then a matching upper bound (for $\limsup$) can be obtained, and the heuristic argument in \cite{CGKLT} would be fully justified.
\end{remark}
\begin{remark}
The LDP lower bound matches the (non-rigorous) physical prediction given by \cite{CGKLT}. We conjecture that the lower bound is tight. A sharper LDP upper bound estimate for tubes around measures that decay sufficiently fast at infinity is obtained in Theorem \ref{Main2} below.
\end{remark}

For measures that decay sufficiently fast at infinity, we obtain sharper LDP upper bound estimates for tubes around them. Specifically, we introduce the following definition.

\begin{definition}
We define $\mathcal{W}$ to be the set of finite signed Borel measures $\mu$ on $\mathbb{R}$ such that $\mu+\nu_0$ is a positive measure and $\lim\limits_{R\rightarrow\infty}\log(R)|\mu|((-R,R)^c)=0$.
\end{definition}

We have the following sharper LDP upper bound estimates for tubes around measures in $\mathcal{W}$. Hereafter, for a finite signed Borel measure $\mu$ on $\mathbb{R}$ and an interval $I\subseteq \mathbb{R}$, we denote by $\mu|_I$ the restriction of $\mu$ to $I$. 

\begin{theorem}\label{Main2}
For any $\mu\in\mathcal{W}$, if $\mu(\mathbb{R})=0$, then
\begin{equation}
 \limsup_{R\rightarrow\infty}\limsup_{\delta\rightarrow 0^{+}}\limsup_{k\rightarrow\infty}\frac{1}{k^2}\log{\mathbb{P}(d_{R}(\nu_{k;R},\mu|_{[-R,R]})\leq \delta)}\leq -\limsup_{R\rightarrow \infty} I_{R,R}(\mu).
\end{equation}
For any $\mu\in\mathcal{W}$, if $\mu(\mathbb{R})\neq 0$, then for sufficiently small $\delta>0$ (depending only on $\mu$),
\begin{equation}
    \limsup_{R\rightarrow\infty}\limsup_{k\rightarrow\infty}\frac{1}{k^2}\log{\mathbb{P}(d_{R}(\nu_{k;R},\mu|_{[-R,R]})\leq \delta)}=-\infty.
\end{equation}
\end{theorem}
\begin{remark}\label{R}
Let $\tilde{\mathcal{X}}$ be the space of locally finite positive Borel measures on $\mathbb{R}$ endowed with the topology of vague convergence (see e.g. \cite{Kal}). Note that $\mu_k\in\tilde{\mathcal{X}}$ for all $k\in\mathbb{N}_{+}$. For any $\eta>0$, take
\begin{equation*}
\tilde{K}_{\eta}:=\{\mu\in\tilde{\mathcal{X}}:\mu([-R,R])\leq \eta R^{3\slash 2}\text{ for all }R\geq 1\}.
\end{equation*}
It can be shown that $\tilde{K}_{\eta}$ is a compact subset of $\tilde{\mathcal{X}}$. Moreover, by Proposition \ref{PropAiry1} in Section \ref{Sect.2}, it can be shown that
\begin{equation*}
\lim_{\eta\rightarrow\infty}\limsup_{k\rightarrow\infty}\frac{1}{k^2}\log{\mathbb{P}(\mu_k\notin \tilde{K}_{\eta})}=-\infty.
\end{equation*}
This establishes the exponential tightness of $\{\mu_k\}_{k=1}^{\infty}$. 

We note that Theorem \ref{Main2} provides a local LDP upper bound for measures in $\tilde{W}:=\{\mu+\nu_0:\mu\in\mathcal{W}\}$ in this setting. A local LDP upper bound for \emph{all} measures $\mu$ in $\tilde{\mathcal{X}}$ (with a rate that is strictly positive when $\mu\neq 0$) can be obtained from Theorem \ref{Main1}, part (b).
We also note that a local LDP lower bound (with a rate matching the local LDP upper bound obtained in Theorem \ref{Main2} for compactly supported measures) in this setting can be obtained from the proof of Theorem \ref{Main1}, part (a).

Therefore, if similar estimates as in Theorem \ref{Main2} can be worked out for $\mu\in\tilde{\mathcal{X}}\backslash\tilde{\mathcal{W}}$, then a full LDP for $\{\mu_k\}_{k=1}^{\infty}$ (and also $\{\nu_{k;R_0}\}_{k=1}^{\infty}$) can be established.
\end{remark}

Besides being useful in establishing large deviation bounds for the Airy point process, our estimates for the Airy point process established in Section \ref{Sect.2} (Propositions \ref{PropAiry1} and \ref{PropAiry2}) significantly improve upon previously best-known estimates on the Airy point process (see Theorems 1.4-1.5 of \cite{Cor}, which are also the technical heart of that paper) with a shorter proof. Combining our estimates for the Airy point process with the arguments in \cite{Cor} that lead to \cite[Theorem 1.1]{Cor} allows us to obtain sharp quantitative estimates for the lower tail probability of the KPZ equation. Specifically, our result (Theorem \ref{KPZ1}) significantly improves upon the main result (Theorem 1.1) of \cite{Cor}. Another advantage of our method is that the result holds for the stochastic Airy operator with general $\beta\in\mathbb{N}_{+}$, which allows us to simultaneously obtain similar estimates for the lower tail probability of the half-space KPZ equation with Neumann boundary parameter $A=-1\slash 2$ and narrow-wedge initial data. In comparison, the main estimate (Theorem 1.4) of \cite{Cor} relies on determinantal structures (in particular, Ablowitz-Segur solution of Painlev\'e II) and does not generalize to general $\beta$. Our result for this case (Theorem \ref{KPZ2}) also significantly improves upon previous results (\cite[Theorem 1.4]{Kim}).

In the following, we introduce our results on the lower tail of the KPZ equation. When time increases, the KPZ equation exhibits an overall linear decay at rate $-T\slash 24$ with fluctuations growing like $T^{1\slash 3}$. We set 
\begin{equation*}
    \gamma_T:=T^{-1\slash 3}(H(2T,0)+T\slash 12).
\end{equation*}
In \cite{ACQ}, it was shown that
\begin{equation}\label{Dis}
    \lim_{T\rightarrow \infty}\mathbb{P}(\gamma_T\leq s)= F_{GUE}(s), 
\end{equation}
where $F_{GUE}$ is the GUE Tracy-Widom distribution. For $s$ large, we have $F_{GUE}(-s)\approx \exp(-s^3\slash 12)$ (see e.g. \cite{BBD,RRV,TW} for details). 

The distributional limit in (\ref{Dis}) does not provide control over the tails of $\gamma_T$ for simultaneously growing $s$ and $T$. As a corollary of the estimates for the Airy point process established in Section \ref{Sect.2} of this article and Theorem 1.1 of \cite{Cor}, we establish the following result on the lower tail probability of the KPZ equation when $s$ and $T$ simultaneously grow. This result significantly improves upon previous results (see Remark \ref{rmk1.6} for details). In particular, it shows that the lower tail behavior of $\gamma_T$ as $s,T\rightarrow\infty$ for the regime $s\ll T^{1\slash 6}$ is \emph{exactly the same} as $F_{GUE}$ (here and throughout the rest of the article, for $A,B>0$ simultaneously growing to infinity, we write ``$A\ll B$'' if $\lim\limits_{A,B\rightarrow\infty} A\slash B=0$). 

\begin{theorem}\label{KPZ1}
Let $\gamma_T:=T^{-1\slash 3}(H(2T,0)+T\slash 12)$ be the centered and scaled solution to the KPZ equation with narrow-wedge initial data. Fix $\epsilon \in (0,1\slash 3)$ and $T_0>0$. Then there exist $S=S(\epsilon,T_0)>0$ (that only depends on $\epsilon,T_0$), $C=C(T_0)>0$ (that only depends on $T_0$), and $K=K(\epsilon, T_0)>0$ (that only depends on $\epsilon,T_0$), such that for any $s\geq S$ and $T\geq T_0$,
\begin{equation}
    \mathbb{P}(\gamma_T\leq -s)\leq e^{-\frac{4(1-C\epsilon)}{15\pi}T^{1\slash 3}s^{5\slash 2}}+e^{-K s^3-\epsilon T^{1\slash 3}s}+e^{-\frac{(1-C\epsilon)}{12}s^3},
\end{equation}
\begin{equation}
    \mathbb{P}(\gamma_T\leq -s)\geq e^{-\frac{4(1+C\epsilon)}{15\pi}T^{1\slash 3}s^{5\slash 2}}+e^{-\frac{(1+C\epsilon)}{12}s^3}.
\end{equation}
In particular, when $s,T\rightarrow\infty$ simultaneously such that $s\ll T^{1\slash 6}$, we have the following sharp tail behavior, which is consistent with the GUE Tracy-Widom distribution: 
\begin{equation}
    \lim_{\substack{s,T\rightarrow\infty:\\s\ll T^{1\slash 6}}} \frac{\log{\mathbb{P}(\gamma_{T}\leq -s)}}{s^3}=-\frac{1}{12}.
\end{equation}
\end{theorem}
\begin{remark}\label{rmk1.6}
This result significantly improves upon the main result (Theorem 1.1) of \cite{Cor} for the regime where $s\ll T^{2\slash 3}$. The upper bound in \cite[Theorem 1.1]{Cor} is
\begin{equation*}
    e^{-\frac{4(1-C\epsilon)}{15\pi}T^{1\slash 3}s^{5\slash 2}}+e^{-K_1 s^{3-\delta}-\epsilon T^{1\slash 3}s}+e^{-\frac{(1-C\epsilon)}{12}s^3},
\end{equation*}
and the lower bound is 
\begin{equation*}
    e^{-\frac{4(1+C\epsilon)}{15\pi}T^{1\slash 3}s^{5\slash 2}}+e^{-K_2 s^3}.
\end{equation*}
In particular, for the first time we obtain sharp super-exponential decay with exponent $3$ for the regime where $1\ll  s\ll T^{2\slash 3}$ (with sharp leading 
prefactor $1\slash 12$ when $1\ll s\ll T^{1\slash 6}$). We also note that the main results (Theorem 1.1 and Corollary 1.2) of \cite{CC} are focused on the regime where $s$ is at least of order $T^{2\slash 3}$ and do not cover the regime where $s\ll T^{2\slash 3}$.
\end{remark}

We also have a parallel result for the half-space KPZ equation with Neumann boundary parameter $A=-1\slash 2$ and narrow-wedge initial data, which significantly improves upon previous results (see Remark \ref{rmk.2} for details). In \cite{Par}, the following identity was established: For any $u>0$, \begin{equation}\label{Des}
    \mathbb{E}_{\text{KPZ}}\Big[\exp\Big(-u\exp(H^{hf}(2T,0)+T\slash 12)\Big)\Big]=\mathbb{E}_{\text{GOE}}\left[\prod_{i=1}^{\infty}\frac{1}{\sqrt{1+4u\exp(T^{1\slash 3}a_i)}}\right],
\end{equation}
where $\mathbb{E}_{\text{KPZ}}$ means taking expectation with respect to the half-space KPZ equation with Neumann boundary parameter $A=-1\slash 2$ and narrow-wedge initial data, and $\mathbb{E}_{\text{GOE}}$ means taking expectation with respect to the GOE point process (which is the space reversal of the eigenvalues of the stochastic Airy operator with $\beta=1$).

\begin{theorem}\label{KPZ2}
Let $\gamma^{hf}_T := T^{-1\slash 3}(H^{hf}(2T,0)+T\slash 12)$ be the centered and scaled solution to the half-space KPZ equation with narrow-wedge initial data and Neumann boundary parameter $A=-1\slash 2$. Fix $\epsilon \in (0,1\slash 3)$ and $T_0>0$. Then there exist $S=S(\epsilon,T_0)>0$ (that only depends on $\epsilon,T_0$), $C=C(T_0)>0$ (that only depends on $T_0$), and $K=K(\epsilon, T_0)>0$ (that only depends on $\epsilon,T_0$), such that for any $s\geq S$ and $T\geq T_0$,
\begin{equation}
    \mathbb{P}(\gamma^{hf}_T\leq -s)\leq e^{-\frac{2(1-C\epsilon)}{15\pi}T^{1\slash 3}s^{5\slash 2}}+e^{-K s^3-\epsilon T^{1\slash 3}s}+e^{-\frac{(1-C\epsilon)}{24}s^3},
\end{equation}
\begin{equation}
    \mathbb{P}(\gamma^{hf}_T\leq -s)\geq e^{-\frac{2(1+C\epsilon)}{15\pi}T^{1\slash 3}s^{5\slash 2}}+e^{-\frac{(1+C\epsilon)}{24}s^3}.
\end{equation}
In particular, when $s,T\rightarrow\infty$ simultaneously such that $s\ll T^{1\slash 6}$, we have the following sharp tail behavior, which is consistent with the GOE Tracy-Widom distribution: 
\begin{equation}
    \lim_{\substack{s,T\rightarrow\infty:\\s\ll T^{1\slash 6}}} \frac{\log{\mathbb{P}(\gamma^{hf}_{T}\leq -s)}}{s^3}=-\frac{1}{24}.
\end{equation}
\end{theorem}
\begin{remark}\label{rmk.2}
This result significantly improves upon the main result (Theorem 1.4) of \cite{Kim}. The upper bound in \cite[Theorem 1.4]{Kim} is  
\begin{equation}\label{Up1}
    e^{-\frac{2(1-C\epsilon)}{15\pi} T^{1\slash 3}s^{5\slash 2}}+e^{-\eta s^{3\slash 2}-\epsilon T^{1\slash 3}s}+e^{-\frac{(1-C\epsilon)}{24}s^3},
\end{equation}
and the lower bound is 
\begin{equation*}
    e^{-\frac{2(1+C\epsilon)}{15\pi}T^{1\slash 3}s^{5\slash 2}}+e^{-K_2 s^3},
\end{equation*}
where $\eta>0$ is fixed and $s\geq S(\epsilon,T_0,\eta)$ (depending on $\epsilon,T_0,\eta$). Note that for $s\gg T^{1\slash 6}$, the upper bound (\ref{Up1}) gives super-exponential decay with exponent $3\slash 2$, which does not capture the actual tail decay (super-exponential decay with exponents $5\slash 2$ and $3$ when $s\gg T^{2\slash 3}$ and $1\ll s \ll T^{2\slash 3}$, respectively). In comparison, our result allows us to obtain for the first time the crossover between regimes with exponents $5\slash 2$ and $3$: Our result demonstrates sharp super-exponential decay with exponent $5\slash 2$ when $s\gg T^{2\slash 3}$ (with sharp leading prefactor $2T^{1\slash 3}\slash (15\pi)$) and sharp super-exponential decay with exponent $3$ when $1\ll s\ll T^{2\slash 3}$ (with sharp leading prefactor $1\slash 24$ when $1\ll s\ll T^{1\slash 6}$).
\end{remark}
\begin{remark}
The proof of Theorem \ref{KPZ2} follows from a similar argument as in the proof of Theorem \ref{KPZ1} by replacing $\beta=2$ by $\beta=1$ and using (\ref{Des}).
\end{remark}

Below we outline the main approach for the proofs of Theorems \ref{Main1}-\ref{KPZ2}. This approach is novel in that it allows us to study the Airy point process \textbf{from two perspectives}: the first is the stochastic Airy operator together with its diffusion characterization, and the second is random matrix theory (in particular, Gaussian $\beta$-ensembles). We approach the problem from both sides, obtaining the best estimates from each perspective, and combine them in an intricate way to obtain the desired results.

There are several components in the proofs of Theorems \ref{Main1}-\ref{KPZ2}, as listed in the following. The reader is referred to Section \ref{Sect.n} for relevant background on the Airy point process, the stochastic Airy operator, and the Gaussian $\beta$-ensemble.

\paragraph{Step 1}
\textbf{Precise estimates for the Airy point process.} 
Propositions \ref{PropAiry1}-\ref{PropAiry2} in Section \ref{Sect.2} significantly improve upon previously best-known tail probability estimates for the Airy point process (see Theorems 1.4-1.5 of \cite{Cor}, which are also the technical heart of that paper) and are sharp in many cases. Moreover, it applies to the stochastic Airy operator $H_{\beta}$ with general $\beta>0$, which allows for a broader range of applications (in particular, we obtain Theorems \ref{KPZ1}-\ref{KPZ2} simultaneously). In comparison, the estimates in \cite{Cor} rely on determinantal structures (in particular, Ablowitz-Segur solution of Painlev\'e II), which restricts the application to $\beta=2$. Our estimates for $\beta=1$ also improve upon previously best-known tail bounds (see Theorems 1.10-1.12 of \cite{Kim}).

More specifically, up to exponentially small tail probabilities, Proposition \ref{PropAiry1} upper bounds the number of eigenvalues of $H_{\beta}$ that are less than or equal to $\lambda$, and Proposition \ref{PropAiry2} upper bounds the absolute difference between the number of eigenvalues of $H_{\beta}$ that are less than or equal to $\lambda$ and that of the Airy operator (see Section \ref{Sect.2} for details). The main tool used in the proofs of Propositions \ref{PropAiry1}-\ref{PropAiry2} is the diffusion characterization of the stochastic Airy operator developed in \cite{RRV}. Based on this tool, we develop precise estimates for the tail behavior of those diffusions, and combine such estimates with Tracy-Widom tails (see Proposition \ref{TailBound}) to obtain the proof of Proposition \ref{PropAiry1}. For the proof of Proposition \ref{PropAiry2}, we further develop a dyadic decomposition technique to localize the deviation event, and rely upon an intricate comparison argument with Airy operators (which is the most challenging part for this step). More precisely, we compare the stochastic Airy operator with a ``shifted version'' of the Airy operator (see the proof of Proposition \ref{PropAiry2} for details), and further compare it to the Airy operator.

\paragraph{Step 2}
\textbf{Approximation of the Airy point process (more generally, stochastic Airy spectrum) using the Gaussian unitary ensemble (more generally, the Gaussian $\beta$-ensemble) up to an \emph{exponentially small probability}.} The main result of this step is given in Theorem \ref{Main3} below. This step is challenging in that we need to estimate probabilities of unlikely events with great precision in order to transfer large deviation bounds from GUE to the Airy point process. In \cite{Tsa}, it is remarked that justifying the passage from the Gaussian $\beta$-ensemble to the Airy point process up to \emph{exponentially small probability} remains an open problem. In this part of the proof we resolve this problem.

The proof of Theorem \ref{Main3} relies upon quite a few new ideas. The explanations below are all up to events with exponentially small probabilities. The key is to establish \emph{exponential decay} of the top eigenfunctions of the stochastic Airy operator $H_{\beta}$ and the top eigenvectors of the Gaussian $\beta$-ensemble $H_{\beta,n}$ (see Section \ref{Sect.n2} for the definition of $H_{\beta,n}$) up to an \emph{exponentially small probability}. The decay bounds for $H_{\beta}$ are obtained in Proposition \ref{AiryDecay}. The proof of this proposition is based on the Riccati transform of the top eigenfunctions of $H_{\beta}$ (see Section \ref{Sect.n1} for definition): Suppose that the Riccati transform deviates a bit upward from the curve $y=-\sqrt{x}\slash 2$, then we can actually show (through an inductive argument over various regions) that it will quickly move above the curve $y=4\sqrt{x}\slash 5$ and will be ``trapped'' above the curve $y=\sqrt{x}\slash 2$ forever. This leads to a contradiction, as the eigenfunctions are $L^2$-bounded. From this, we conclude that the Riccati transform can never go above $y=-\sqrt{x}\slash 2$, and the conclusion 
of Proposition \ref{AiryDecay} can be obtained. The decay bounds for $H_{\beta,n}$---presented in Proposition \ref{DiscreteDecay}---can be obtained based on a discrete analogue of the previous argument. The discrete setting causes additional difficulties, which we rely on a more complicated argument to resolve. Having obtained the aforementioned decay bounds, we can couple $H_{\beta}$ and $H_{\beta,n}$ \emph{along their eigenfunctions\slash eigenvectors} to prove Theorem \ref{Main3}. New a priori estimates for the eigenfunctions (eigenvectors) based on the eigenvalue-eigenfunction (eigenvector) equations as well as the Koml\'os-Major-Tusn\'ady theorem are utilized in the proof.

\paragraph{Step 3}
\textbf{Finally, we obtain large deviation bounds for GUE, and apply the approximation result in Step 2 to transfer these bounds to the Airy point process.} The key for this part is an edge rigidity estimate (Proposition \ref{Edge}) and a separate lemma (Lemma \ref{field}) showing that the eigenvalues of GUE that are far from the edge only impose a negligible effect on the result. For the former, we improve upon edge rigidity estimates from \cite{BEY} so that the exponent in the upper bound on the tail probability satisfies our need (previous results give no control over the exponent). For the latter, we combine the edge rigidity estimate with our previous estimates for the Airy point process in Step 1 (transferred to GUE by Step 2), making full use of the two perspectives.

\bigskip

During the process of proving the large deviation bounds for the Airy point process, we also obtain a result (Theorem \ref{Main3} below) on the approximation of the Airy point process (and more generally, the eigenvalues of the stochastic Airy operator $H_{\beta}$ for $\beta\in\mathbb{N}_{+}$) using the Gaussian $\beta$-ensemble (defined in Section \ref{Sect.n2} below) up to an event with \emph{exponentially small probability}, which might be of independent interest. Throughout the article, we denote by $H_{\beta,n}$ the Gaussian $\beta$-ensemble of size $n$. Background materials on the stochastic Airy operator and the Gaussian $\beta$-ensemble will be presented in Section \ref{Sect.n}.

\begin{theorem}\label{Main3}
Assume that $\beta,n,k\in\mathbb{N}_{+}$ and $n^{e_0}\leq k \leq n^{1\slash 10000}$, where $e_0\in (0,1\slash 10000)$ is a fixed constant. Let $\lambda_1<\lambda_2<\cdots$ be the sorted eigenvalues of the stochastic Airy operator $H_{\beta}$ and $\lambda_1^{(n)}> \lambda_2^{(n)}> \cdots > \lambda_n^{(n)}$ be the sorted eigenvalues of $H_{\beta,n}$ (the Gaussian $\beta$-ensemble of size $n$). Let $a_i:=-\lambda_i$ for any $i\in\mathbb{N}_{+}$ and $\tilde{\lambda}_i^{(n)}:=n^{1\slash 6}(\lambda_i^{(n)}-2\sqrt{n})$ for any $i\in \{1,2,\cdots,n\}$. Then there exist positive constants $C_1,C_2,c$ that only depend on $\beta$ and $e_0$, a coupling between $H_{\beta}$ and $H_{\beta,n}$ on the same probability space $(\Omega,\mathcal{F},\mathbb{P})$, and an event $\mathscr{A}\in \mathcal{F}$, such that $\mathbb{P}(\mathscr{A}^c)\leq C_1\exp(-ck^3)$, and when the event $\mathscr{A}$ holds, $|\tilde{\lambda}_i^{(n)}-a_i|\leq C_2n^{-1\slash 24}$ for all $i\in \{1,2,\cdots,k\}$.
\end{theorem}
\begin{remark}
We have not optimized the exponents $1\slash 10000$ and $-1\slash 24$ here, and better bounds could be worked out using the arguments in Section \ref{Sect.3}.
\end{remark}

In the following, we set up some notations and conventions that are used throughout this article. We fix a probability space $(\Omega,\mathcal{F},\mathbb{P})$, on which a standard Brownian motion $(B_t, t\geq 0)$ is defined. We let $\mathcal{F}_t:=\sigma(B_s,s\in [0,t])$ for any $t\geq 0$. Unless otherwise specified, we use $C$ and $c$ to denote positive constants that only depend on $\beta$. The values of these constants may change from line to line. We denote the set of positive integers by $\mathbb{N}_{+}$ and the set of positive real numbers by $\mathbb{R}_{+}$, and let $\mathbb{N}:=\mathbb{N}_{+}  \cup  \{0\}$. We let $[0]:=\emptyset$ and $[n]:=\{1,2,\cdots,n\}$ for any $n\in\mathbb{N}_{+}$. For any vector $\alpha\in\mathbb{R}^m$ and $i\in [m]$, we denote by $\alpha(i)$ the $i$th entry of $\alpha$; for any $m\times m'$ matrix $A$ and $i\in [m],j\in [m']$, we denote by $A(i,j)$ the $(i,j)$ entry of $A$. For any finite set $A$, we denote by $\#A$ the cardinality of $A$. We denote by $\mathcal{B}_{\mathbb{R}}$ the Borel $\sigma$-algebra on $\mathbb{R}$. 

The rest of this article is organized as follows. In Section \ref{Sect.n}, we present background materials on the Airy point process, the stochastic Airy operator, and the Gaussian $\beta$-ensemble. Then in Section \ref{Sect.2}, we obtain some estimates for the Airy point process and establish the proof of parts (b) and (c) of Theorem \ref{Main1}. Based on the estimates in Section \ref{Sect.2}, we give the proof of Theorem \ref{KPZ1} in Section \ref{Sect.3.5}. In Section \ref{Sect.3}, we study the approximation of the Airy point process via the Gaussian $\beta$-ensemble and give the proof of Theorem \ref{Main3}. Then we refine a result from \cite{BEY} to obtain a more precise estimate on edge rigidity in Section \ref{Sect.4}. Finally, in Section \ref{Sect.5}, we finish the proofs of Theorems \ref{Main1} and \ref{Main2} based on the results from previous sections. 

\section{Airy point process, stochastic Airy operator, and Gaussian $\beta$-ensemble}\label{Sect.n}

We present background materials on the Airy point process, the stochastic Airy operator, and the Gaussian $\beta$-ensemble in Sections \ref{Sect.n0}-\ref{Sect.n2}, respectively.

\subsection{Airy point process}\label{Sect.n0}

The Airy point process is a simple determinantal point process on $\mathbb{R}$ that arises as the limit of rescaled eigenvalue configuration of the GUE near the spectral edge. In this subsection, we give the formal definition of the Airy point process following \cite[Section 4.2]{AGZ}.

We start by introducing the notion of \emph{point process}. Recall that $\mathcal{B}_{\mathbb{R}}$ is the Borel $\sigma$-algebra on $\mathbb{R}$. Let $\mu$ be a positive Radon measure on $(\mathbb{R},\mathcal{B}_{\mathbb{R}})$. A \emph{point process} on $\mathbb{R}$ is a random integer-valued positive Radon measure on $(\mathbb{R},\mathcal{B}_{\mathbb{R}})$; equivalently, it is a probability measure on the space of locally finite point configurations on $\mathbb{R}$. A point process $\chi$ on $\mathbb{R}$ is called \emph{simple} if $\mathbb{P}(\exists x\in\mathbb{R}: \chi(\{x\})>1)=0$.

Suppose that $\chi$ is a simple point process on $\mathbb{R}$. Assume the existence of locally integrable functions $\rho_k:\mathbb{R}^k\rightarrow [0,\infty)$ for every $k\in \mathbb{N}_{+}$, such that for any mutually disjoint Borel sets $D_1,\cdots,D_k\subseteq\mathbb{R}$,  
\begin{equation*}
    \mathbb{E}\Big[\prod_{i=1}^k \chi(D_i)\Big]=\int_{\prod_{i=1}^k D_i} \rho_k(x_1,\cdots,x_k)d\mu(x_1)\cdots d\mu(x_k).
\end{equation*}
The function $\rho_k$ is called the \emph{$k$-point correlation function} of the point process $\chi$ with respect to $\mu$. A simple point process on $\mathbb{R}$ is called a \emph{determinantal point process} with correlation kernel $K:\mathbb{R}^2\rightarrow \mathbb{C}$, if for every $k\in \mathbb{N}_{+}$, its $k$-point correlation function $\rho_k$ exists and is given by $\rho_k(x_1,\cdots,x_k)=\det[(K(x_i,x_j)]_{i,j\in [k]}$. The \emph{Airy point process} is a simple determinantal point process on $\mathbb{R}$ (with $\mu$ being the Lebesgue measure) with correlation kernel 
\begin{equation*}
    K^{\text{Ai}}(x,y)=\frac{\text{Ai}(x)\text{Ai}'(y)-\text{Ai}(y)\text{Ai}'(x)}{x-y}.
\end{equation*} 

\subsection{Stochastic Airy operator and associated diffusions}\label{Sect.n1}

The Airy point process is closely related to the \emph{stochastic Airy operator} introduced in \cite{RRV}. In this subsection, we review background materials on the stochastic Airy operator following \cite[Section 4.5]{AGZ}, \cite{Min}, and \cite{RRV}. We refer the reader to these references for further details.

Let $C_0^{\infty}=C_0^{\infty}(\mathbb{R}_{+})$ be the space of compactly supported smooth functions under the topology of uniform convergence of all derivatives on compact sets. Let $D=D(\mathbb{R}_{+})$ be the space of distributions, that is, the continuous dual of the space $C_0^{\infty}$. For any locally integrable function $f$ and any $k\in\mathbb{N}_{+}$, the $k$th formal derivative of $f$, denoted by $f^{(k)}$, is the element of $D$ that acts on $C_0^{\infty}$ via integration by parts:
\begin{equation*}
    \prec \varphi, f^{(k)} \succ:=(-1)^k \int f(x) \varphi^{(k)}(x)dx, \quad \forall \varphi\in C_0^{\infty}. 
\end{equation*}

In the following, we fix an arbitrary $\beta>0$. Define $H^1_{\mathrm{loc}}$ to be the set of absolutely continuous functions $f: [0,\infty)  \rightarrow  \mathbb{R}$ such that for any compact set $I\subseteq \mathbb{R}$, $f^{(1)}\mathbbm{1}_{I}\in L^2$. The stochastic Airy operator, denoted by $H_{\beta}$, is the random linear map from $H^1_{\mathrm{loc}}$ to $D$ that sends $f\in H^1_{\mathrm{loc}}$ to the distribution
\begin{equation*}
    H_{\beta} f = -f^{(2)}+xf+\frac{2}{\sqrt{\beta}} f B',
\end{equation*}
where $B=(B_x,x\geq 0)$ is the standard Brownian motion and $fB'$ is understood as the formal derivative of the continuous function $-\int_{0}^x B_t f^{(1)}(t)dt+f(x)B_x$ (where $x\in [0,\infty)$). 

Let $L^{*}$ be the set of functions $f\in H^1_{\mathrm{loc}}$ such that $f(0)=0$ and
\begin{equation*}
    \int_0^{\infty}(f^{(1)}(x)^2+(1+x)f^2(x))dx<\infty.
\end{equation*}
The eigenvalues and eigenfunctions of $H_{\beta}$ are defined as the pairs $(\lambda,f)\in \mathbb{R}\times L^{*}$ that satisfy $\|f\|_2=1$ and $H_{\beta}f=\lambda f$ in the sense of distributions. Note that the latter condition is equivalent to saying that for any $\phi\in C_0^{\infty}$,
\begin{eqnarray}\label{Eq2.1}
   && \lambda\int_0^{\infty} \phi(x)f(x)dx =
     \int_0^{\infty} (-\phi''(x)f(x)+x\phi(x)f(x))dx\nonumber\\
     &&\quad\quad\quad\quad\quad-\frac{2}{\sqrt{\beta}}\Big(\int_0^{\infty}\phi'(x)f(x)B_xdx+\int_0^{\infty}\phi(x)f^{(1)}(x)B_x dx\Big).
\end{eqnarray}
In \cite{RRV}, it was shown that for any $i\in\mathbb{N}_{+}$, the collection of eigenvalues of $H_{\beta}$ has a well-defined $i$th smallest element, which we denote by $\lambda_i$. We also denote by $f_i$ the eigenfunction that corresponds to $\lambda_i$. By \cite[Proposition 3.5]{RRV}, almost surely, we have $\lambda_1<\lambda_2<\cdots$ and $\lambda_i\rightarrow\infty$ as $i\rightarrow\infty$. From (\ref{Eq2.1}), it can be deduced that (see \cite{RRV} and \cite[Section 4.5]{AGZ}) for any $i\in\mathbb{N}_{+}$, we can choose $f_i$ so that $f_i\in\text{H\"older}(3\slash 2)^{-}$ and 
\begin{equation}\label{eigenv1}
    f_i'(x)-f_i'(0)=\int_{0}^x\Big((t-\lambda_i)f_i(t)-\frac{2}{\sqrt{\beta}} f_i'(t)B_t \Big) dt+\frac{2}{\sqrt{\beta}}f_i(x) B_x, \quad \forall x\geq 0.
\end{equation}

In the following, we introduce a bilinear form $\langle \cdot,\cdot \rangle_{H_{\beta}}$ on $L^{*}$ that is associated with $H_{\beta}$. For any $x\geq 0$, we define $\bar{B}_x:=\int_x^{x+1}B_y dy$ and 
\begin{equation}
  Q_x:=B_{x+1}-B_x, \quad R_x:=B_x-\bar{B}_x. 
\end{equation}
For any $f,g\in L^{*}$, we define
\begin{eqnarray}\label{Eqnn6}
   && \langle f, g \rangle_{H_{\beta}}:= \int_0^{\infty} f^{(1)}(x)g^{(1)}(x)dx+\int_0^{\infty} xf(x)g(x)dx\nonumber\\
   &&+\frac{2}{\sqrt{\beta}}\Big(\int_0^{\infty} Q_x f(x) g(x)dx-\int_0^{\infty} R_x (f(x)g^{(1)}(x)+f^{(1)}(x)g(x))dx \Big).\nonumber\\
   && 
\end{eqnarray}
For any $i\in\mathbb{N}_{+}$, the eigenvalue $\lambda_i$ satisfies the following minimax principle (see equation (2.7) of \cite{Tsa}):
\begin{equation}\label{mini}
    \lambda_i=\min\Big\{\max_{f\in \mathscr{E}:\|f\|_2=1} \{\langle f,f\rangle_{H_{\beta}}\}:\mathscr{E}\text{ is an }i\text{-dimensional subspace of }L^{*}\Big\}.
\end{equation}
Below we list additional properties of the eigenvalues and eigenfunctions of $H_{\beta}$ (see \cite{RRV} and \cite[Section 4.5]{AGZ} for the derivations). For any $i\in \mathbb{N}_{+}$,
\begin{eqnarray}\label{eigenv2}
     \lambda_i &=& \langle f_i,f_i\rangle_{H_{\beta}}\nonumber\\
     &=& \int_0^{\infty}f_i'(x)^2dx+\int_0^{\infty} xf_i^2(x)dx+\frac{2}{\sqrt{\beta}}\int_0^{\infty} f_i^2(x) Q_x dx \nonumber\\
     && -\frac{4}{\sqrt{\beta}}\int_0^{\infty}f_i'(x)f_i(x) R_x dx,
\end{eqnarray}
\begin{equation}\label{eigenv3}
    \lambda_i=\int_0^{\infty}f_i'(x)^2dx+\int_0^{\infty}xf_i^2(x)dx-\frac{4}{\sqrt{\beta}}\int_0^{\infty} f_i'(x)f_i(x)B_x dx.
\end{equation}

The following result from \cite{RRV} establishes the relation between the Airy point process and the eigenvalues of the stochastic Airy operator.

\begin{proposition}[\cite{RRV}]\label{Eigens}
Suppose that $\mathbf{a}=(a_1>a_2>\cdots)$ are the ordered points of the Airy point process, and $\mathbf{\Lambda}=(\lambda_1<\lambda_2<\cdots)$ are the ordered eigenvalues of $H_2$ (the stochastic Airy operator with $\beta=2$). Then the distributions of $\mathbf{a}$ and $-\mathbf{\Lambda}$ are equal.
\end{proposition}

The following result from \cite{DV} provides a tail bound on $\lambda_1$.

\begin{proposition}[\cite{DV}]\label{TailBound}
When $t\rightarrow\infty$, we have
\begin{equation*}
    \mathbb{P}(\lambda_1<-t)=t^{-3\beta\slash 4}\exp\Big( -\frac{2}{3}\beta t^{3\slash 2} +O(\sqrt{\log{t}})\Big).
\end{equation*}
In particular, there exists a positive constant $C$ that only depends on $\beta$, such that for any $t>0$,
\begin{equation*}
    \mathbb{P}(\lambda_1<-t)\leq C\exp\Big(-\frac{2}{3}\beta t^{3\slash 2}\Big).
\end{equation*}
\end{proposition}

As discussed in \cite{RRV}, the eigenvalues of $H_{\beta}$ can be studied via the Riccati transform of the eigenfunctions and certain diffusions. For any $i\in\mathbb{N}_{+}$, the Riccati transform of the $i$th eigenfunction $f_i$ of $H_{\beta}$, denoted by $\bar{p}_i$, is given by $\bar{p}_i(x)=f_i'(x)\slash f_i(x)$ for any $x\in [0,\infty)$. From (\ref{eigenv1}), it can be deduced that $\bar{p}_i(0)=+\infty$ and that for any interval $[a,b]\subseteq [0,\infty)$ (where $a\leq b$) on which $f_i(x)$ does not vanish, 
\begin{equation}\label{pbar}
   \bar{p}_i(b)-\bar{p}_i(a)=\int_a^b (x-\lambda_i-\bar{p}_i(x)^2)dx+\frac{2}{\sqrt{\beta}}(B_b-B_a).
\end{equation}
Note that $\bar{p}_i(x)$ may blow up to $-\infty$ at a finite time $x\in (0,\infty)$ (at which $f_i(x)=0$), and we restart $\bar{p}_i(x)$ at $+\infty$ immediately after the blow-up. We can view $\bar{p}_i(x)$ as taking values in a disjoint union of countable copies of the real line, denoted by $\mathbb{R}_0,\mathbb{R}_{-1},\cdots$. Points $(n,x)\in \{0,-1,\cdots\}\times\mathbb{R}$ in this disjoint union are ordered lexicographically (where $n$ marks which copy of the real line the point lies in). 

For any fixed non-random $\lambda\in\mathbb{R}$, we define the diffusion $p_{\lambda}(x), x\geq 0$ by
\begin{equation}\label{Eqn16}
    dp_{\lambda}(x)=(x-\lambda-p_{\lambda}^2(x))dx+\frac{2}{\sqrt{\beta}}dB_x, \quad p_{\lambda}(0)=\infty.
\end{equation}
The discussion regarding the blow-ups of $\bar{p}_i(x)$ in the previous paragraph also applies to $p_{\lambda}(x)$. We cite the following result from \cite{RRV}.

\begin{proposition}[\cite{RRV}]\label{Diffu}
For any fixed non-random $\lambda\in\mathbb{R}$, the number of blow-ups of $p_{\lambda}(x)$ on $[0,\infty)$ is equal to the number of eigenvalues of $H_{\beta}$ that are less than or equal to $\lambda$.
\end{proposition}

Combining Propositions \ref{TailBound}-\ref{Diffu}, we obtain the following corollary.

\begin{corollary}\label{Tails}
There exists a positive constant $C$ that only depends on $\beta$, such that for any $t>0$, we have
\begin{equation*}
    \mathbb{P}(p_{-t}(x) \text{ blows up in a finite time})\leq C\exp\Big(-\frac{2}{3}\beta t^{3\slash 2}\Big).
\end{equation*}
\end{corollary}

The non-random part of $H_{\beta}$ is called the \emph{Airy opera
ator}, denoted by $\mathcal{A}$ hereafter: $\mathcal{A}f=-f^{(2)}+xf$ for any $f\in H_{\mathrm{loc}}^1$. Throughout the article, we denote by $\gamma_1<\gamma_2<\cdots$ the ordered eigenvalues of the Airy operator $\mathcal{A}$. For any $\lambda\in\mathbb{R}$, we define
\begin{equation}\label{Eq2.14}
    N(\lambda):=\#\{i\in\mathbb{N}_{+}: \lambda_i\leq \lambda\}, \quad N_0(\lambda):=\#\{i\in \mathbb{N}_{+}: \gamma_i\leq \lambda\}. 
\end{equation}

The following result on the eigenvalues of the Airy operator $\mathcal{A}$ follows from classical works (for example, \cite{MT} and \cite{Titch}; see also \cite[Theorem 3.3]{Hua}).

\begin{proposition}\label{Airyop}
For any $i\in\mathbb{N}_{+}$, we have
\begin{equation*}
    \gamma_i=\Big(\frac{3\pi}{2}\Big(i-\frac{1}{4}+R(i)  \Big)\Big)^{2\slash 3},
\end{equation*}
where $|R(i)|\leq C\slash i$ for some absolute constant $C>0$.
\end{proposition}

\subsection{Gaussian $\beta$-ensemble}\label{Sect.n2}

For any $\beta>0$, the \emph{Gaussian $\beta$-ensemble} of size $n$ is the probability distribution on ordered points $\lambda_1^{(n)}> \lambda_2^{(n)} > \cdots > \lambda_n^{(n)}$ with joint density
\begin{equation}\label{Eqn10}
 Z_{\beta,n}^{-1} \exp\Big(-\frac{\beta}{4}\sum\limits_{i=1}^n (\lambda_i^{(n)})^2\Big)\prod_{1\leq i<j\leq n} \big|\lambda_i^{(n)}-\lambda_j^{(n)}\big|^{\beta},
\end{equation}
where $Z_{\beta,n}$ is a normalizing constant. For general $\beta>0$, (\ref{Eqn10}) can be viewed as a Gibbs measure of particles with logarithmic interactions. When $\beta\in\{1,2,4\}$, (\ref{Eqn10}) corresponds to the joint density of the eigenvalues of the Gaussian orthogonal\slash unitary\slash symplectic ensemble. We refer the reader to \cite[Section 4.5]{AGZ} for further background on the Gaussian $\beta$-ensemble.

For general $\beta>0$, Dumitriu and Edelman \cite{DE} introduced a class of tridiagonal random matrix models (which we denote by $H_{\beta,n}$ in this article) and showed that the joint density of the eigenvalues of such random matrices is given by (\ref{Eqn10}). We describe this random matrix $H_{\beta,n}$ in the following. We first recall the definition of $\chi$ random variables.

\begin{definition}
For any $t>0$, the $\chi$ distribution with $t$ degrees of freedom, denoted by $\chi_t$, is defined as the probability distribution on $\mathbb{R}_{+}$ with density 
\begin{equation*}
    f_t(x)= \frac{2^{1-t\slash 2}x^{t-1}e^{-x^2\slash 2}}{\Gamma(t\slash 2)}, \quad \forall x\in\mathbb{R}_{+}.
\end{equation*}
If $t\in\mathbb{N}_{+}$ and $X$ has a $\chi_t$ distribution, then the law of $X$ is the same as that of $\sqrt{\sum_{i=1}^t \xi_i^2 }$, where $\xi_1,\cdots,\xi_t$ are i.i.d. $N(0,1)$ random variables.
\end{definition} 
Let $\xi_1,\cdots,\xi_n$ be i.i.d. $N(0,1)$ random variables. Let $Y_1,\cdots,Y_{n-1}$ be independent random variables that are independent of $\{\xi_i\}_{i=1}^n$, such that $Y_i\sim \chi_{i\beta}$ for every $i\in [n-1]$. Then $H_{\beta,n}$ is an $n\times n$ random matrix whose entries are as follows: For any $i\in [n]$, $H_{\beta,n}(i,i)=\sqrt{2\slash\beta} \xi_i$; for any $i\in [n-1]$, $H_{\beta,n}(i,i+1)=H_{\beta,n}(i+1,i)=Y_{n-i}\slash\sqrt{\beta}$; for any $(i,j)\in [n]^2$ with $|i-j|>1$, $H_{\beta,n}(i,j)=0$.

For each $i\in [n]$, we denote by $g_i\in \mathbb{R}^n$ the normalized eigenvector of $H_{\beta,n}$ that corresponds to the eigenvalue $\lambda_i^{(n)}$, i.e., $H_{\beta,n} g_i =\lambda_i^{(n)} g_i$. We define
\begin{equation*}
    \gamma_0^{(n)}:=2, \quad \gamma_n^{(n)}:=-2,
\end{equation*}
and let $\{\gamma_j^{(n)}\}_{j=1}^{n-1}$ be such that
\begin{equation}\label{classical}
   \frac{1}{2\pi} \int_{\gamma_j^{(n)}}^2 \sqrt{(4-x^2)_{+}} dx=\frac{j}{n}, \quad \forall j\in [n-1].
\end{equation}
For each $i\in [n]$, we define 
\begin{equation}\label{tilde_l}
    \tilde{\lambda}_i^{(n)}:=n^{1\slash 6}(\lambda_i^{(n)}-2\sqrt{n}).
\end{equation}
Note that $\tilde{\lambda}_1^{(n)}> \tilde{\lambda}_2^{(n)}> \cdots> \tilde{\lambda}_n^{(n)}$ are the ordered eigenvalues of the matrix $\hat{H}_{\beta,n}:=n^{1\slash 6}(H_{\beta,n}-2\sqrt{n}I_n)$ (where $I_n$ is the $n\times n$ identity matrix).

\section{Estimates for the Airy point process}\label{Sect.2}

In this section, we present some estimates for the Airy point process (and more generally, the eigenvalues of the stochastic Airy operator $H_{\beta}$). The main tool used in the proofs of these estimates is the diffusion $p_{\lambda}(x)$ defined in (\ref{Eqn16}). Throughout this section, we fix an arbitrary $\beta>0$. We also recall the definitions of $N(\lambda)$ and $N_0(\lambda)$ in (\ref{Eq2.14}).

The main results of this section are the following two propositions.

\begin{proposition}\label{PropAiry1}
There exist positive constants $K,C,c$ that only depend on $\beta$, such that for any $\eta\geq 15$ and $\lambda\geq K$, we have
\begin{equation}
    \mathbb{P}(N(\lambda)\geq\eta \lambda^{3\slash 2})\leq C\exp(-c\eta \lambda^3).
\end{equation}
\end{proposition}

\begin{proposition}\label{PropAiry2}
There exist positive constants $K,C,c$ that only depend on $\beta$, such that for any $k,k'\in \mathbb{N}_{+}$ and $\lambda\in\mathbb{R}$ satisfying $k\geq K$, $(\log{k})^3\leq k'\leq k$, and $\gamma_k\leq \lambda<\gamma_{k+1}$, we have
\begin{equation}
    \mathbb{P}(|N(\lambda)-N_0(\lambda)|\geq k')\leq \exp(Ck\log{k})\exp\bigg(-
    \frac{c(k')^2}{\log(2k\slash k')}\bigg).
\end{equation}
\end{proposition}

The rest of this section is organized as follows. In Section \ref{Sect.2.2}, we state and prove some preliminary lemmas. Based on these lemmas, we give the proofs of Propositions \ref{PropAiry1} and \ref{PropAiry2} in Sections \ref{Sect.2.3} and \ref{Sect.2.4}, respectively. In Section \ref{Sect.3.4}, we give the proof of parts (b) and (c) of Theorem \ref{Main1} based on Propositions \ref{PropAiry1} and \ref{PropAiry2}.

\subsection{Some preliminary lemmas}\label{Sect.2.2}

In this subsection, we state and prove some preliminary lemmas that will be used in the proofs of Propositions \ref{PropAiry1} and \ref{PropAiry2}.

\begin{lemma}\label{LemmaAiry2}
Assume that $a>0$. Let $q(x),x\geq 0$ be the solution to the following ordinary differential equation:
\begin{equation*}
    q'(x)=x-a-q^2(x), \quad q(0)=\infty.
\end{equation*}
Let $\Delta>0$ be the first time when $q(x)$ blows up to $-\infty$. We have $\Delta\geq\pi\slash\sqrt{a}$. Moreover, if $a\geq (4\pi)^{2\slash 3}$, then $\Delta\leq \pi\slash\sqrt{a-2\pi \slash\sqrt{a}}$.
\end{lemma}
\begin{proof}

Let $r(x),x\geq 0$ be the solution to 
\begin{equation}\label{Eqn4}
    r'(x)=-a-r^2(x), \quad r(0)=\infty.
\end{equation}
Solving (\ref{Eqn4}), we obtain that $r(x)=\sqrt{a}\tan(\pi\slash 2-\sqrt{a}x)$. Hence $r(x)$ has first blow-up time $\pi\slash\sqrt{a}$. For $x\in [0,\pi\slash \sqrt{a}]$, $q(x)$ is lower bounded by $r(x)$, hence $\Delta\geq \pi\slash\sqrt{a}$.

Now suppose $a\geq (4\pi)^{2\slash 3}$. Let $w(x),x\geq 0$ be the solution to
\begin{equation}\label{Eqn9}
    w'(x)=2\pi\slash\sqrt{a}-a-w^2(x),\quad w(0)=\infty.
\end{equation}
Note that our assumption $a\geq (4\pi)^{2\slash 3}$ implies $a-2\pi \slash \sqrt{a} \geq a\slash 4>0$. Solving (\ref{Eqn9}), we obtain that
\begin{equation*}
    w(x)=\sqrt{a-2\pi \slash \sqrt{a}}\tan\Big(\frac{\pi}{2}-\sqrt{a-2\pi \slash\sqrt{a}}x\Big).
\end{equation*}
Hence $w(x)$ has first blow-up time $\pi\slash\sqrt{a-2\pi 
\slash \sqrt{a}}\leq 2\pi\slash\sqrt{a}$. Thus for any $x\in [0,\min\{\Delta,\pi\slash\sqrt{a-2\pi 
\slash \sqrt{a}}\}]$, $q(x)$ is upper bounded by $w(x)$. Hence we have $\Delta\leq \pi\slash\sqrt{a-2\pi \slash \sqrt{a}}$.

\end{proof}

\begin{lemma}\label{LemmaAiry1}
For $a>(12\pi)^{2\slash 3}$, consider the diffusion $p(x),x\geq 0$ given by
\begin{equation*}
   dp(x)=(x-a-p^2(x))dx+\frac{2}{\sqrt{\beta}}dB_x, \quad
p(0)=\infty.
\end{equation*}
Assume that $\epsilon,\delta\in(0,1\slash 3)$. Let $\Delta>0$ be the first time when $p(x)$ blows up to $-\infty$. When the event
\begin{equation*}
    \mathcal{C}_{\epsilon,\delta}:=\bigg\{\sup\limits_{x\in [0, 2\pi\slash\sqrt{a}]}|B_x|\leq\sqrt{\frac{\beta a\epsilon\delta}{8}}\bigg\}
\end{equation*}
holds, we have
\begin{equation*}
    \frac{\pi}{\sqrt{(1+\epsilon)(1+\delta)a}}\leq \Delta\leq\frac{\pi}{\sqrt{(1-\epsilon)((1-\delta)a-2\pi\slash\sqrt{a})}}.
\end{equation*}
Moreover,
\begin{equation*}
    \mathbb{P}(\mathcal{C}_{\epsilon,\delta}^c)\leq 4\exp\Big(-\frac{\beta\delta\epsilon a^{3\slash 2}}{32\pi}\Big).
\end{equation*}
\end{lemma}
\begin{proof}

Let $r(x):=p(x)-(2\slash\sqrt{\beta})B_x$ for all $x\geq 0$. Then $r(x)$ satisfies 
\begin{equation*}
    r'(x)=x-a-\Big(r(x)+\frac{2}{\sqrt{\beta}}B_x\Big)^2,\quad r(0)=\infty.
\end{equation*}
By the AM-GM inequality, we have 
\begin{equation}\label{Eqn20}
    \Big(r(x)+\frac{2}{\sqrt{\beta}}B_x\Big)^2\leq (1+\epsilon)r^2(x)+\Big(1+\frac{1}{\epsilon}\Big)\frac{4B_x^2}{\beta}.
\end{equation}
Thus, when the event $\mathcal{C}_{\epsilon,\delta}$ holds, for any $x\in [0,\min\{2\pi\slash\sqrt{a},\Delta\}]$, we have
\begin{equation}\label{Eqn23}
    x-a-\Big(r(x)+\frac{2}{\sqrt{\beta}}B_x\Big)^2\geq -(1+\delta)a-(1+\epsilon)r^2(x).
\end{equation}
We define $w(x),x\geq 0$ by
\begin{equation}\label{Eqn5}
    w'(x)=-(1+\delta)a-(1+\epsilon)w^2(x),\quad w(0)=\infty.
\end{equation}
Solving (\ref{Eqn5}), we obtain that
\begin{equation*}
    w(x)=\sqrt{\frac{(1+\delta)a}{1+\epsilon}}\tan\Big(\frac{\pi}{2}-\sqrt{(1+\epsilon)(1+\delta)a}x\Big).
\end{equation*}
Hence $w(x)$ has first blow-up time $\pi\slash\sqrt{(1+\epsilon)(1+\delta)a}\leq \pi\slash \sqrt{a}$. By (\ref{Eqn23}), when the event $\mathcal{C}_{\epsilon,\delta}$ holds and $x\in [0,\min\{\pi\slash\sqrt{(1+\epsilon)(1+\delta)a},\Delta\}]$, $r(x)$ is lower bounded by $w(x)$. Hence 
\begin{equation*}
    \Delta\geq\frac{\pi}{\sqrt{(1+\epsilon)(1+\delta)a}}.
\end{equation*}

Now we define $\tilde{w}(x), x\geq 0$ by
\begin{equation}\label{Eqn6}
\tilde{w}'(x)=\frac{2\pi}{\sqrt{a}}-(1-\delta)a-(1-\epsilon)\tilde{w}^2(x), \quad  \tilde{w}(0)=\infty.
\end{equation}
From our assumption $a>(12\pi)^{2\slash 3}$, we can deduce that 
\begin{equation*}
    (1-\delta)a-\frac{2\pi}{\sqrt{a}}\geq \frac{2}{3}a-\frac{2\pi}{\sqrt{a}}\geq \frac{a}{2} >0.
\end{equation*}
Solving (\ref{Eqn6}), we obtain that
\begin{equation*}
    \tilde{w}(x)=\sqrt{\frac{(1-\delta)a-2\pi\slash \sqrt{a}}{1-\epsilon}}\tan\Big(\frac{\pi}{2}-\sqrt{(1-\epsilon)((1-\delta)a-2\pi\slash\sqrt{a})}x\Big).
\end{equation*}
Hence $\tilde{w}(x)$ has first blow-up time $\pi\slash\sqrt{(1-\epsilon)\left((1-\delta)a-2\pi\slash\sqrt{a}\right)}\leq 2\pi\slash\sqrt{a}$. Following a similar argument as in the previous paragraph, we can deduce that when the event $\mathcal{C}_{\epsilon,\delta}$ holds and $x\in [0,\min\{\pi\slash\sqrt{(1-\epsilon)\left((1-\delta)a-2\pi\slash\sqrt{a}\right)},\Delta\}]$, $r(x)$ is upper bounded by $\tilde{w}(x)$. Hence we have
\begin{equation*}
    \Delta\leq\frac{\pi}{\sqrt{(1-\epsilon)((1-\delta)a-2\pi\slash\sqrt{a})}}.
\end{equation*}

By standard estimates for Brownian motion, for any $t\geq 0$, 
\begin{equation}\label{Eqn26}
    \mathbb{P}\Big(\sup_{x\in [0,1]}|B_x|>t\Big)\leq  4 e^{-t^2\slash 2}.
\end{equation}
As $\left(c^{-1\slash 2}B_{cx}, x\geq 0\right)$ is a standard Brownian motion for any $c>0$, we have
\begin{eqnarray*}
\mathbb{P}(\mathcal{C}_{\epsilon,\delta}^c) &=& \mathbb{P}\bigg(\sup_{x\in[0,1]}\{(2\pi \slash\sqrt{a})^{-1 \slash 2}|B_{2\pi x \slash\sqrt{a}}|\}>\sqrt{\frac{\beta\delta\epsilon a^{3\slash 2}}{16\pi}}\bigg) \nonumber\\
&\leq& 4\exp\Big(-\frac{\beta\delta\epsilon a^{3\slash 2}}{32\pi}\Big).
\end{eqnarray*}

\end{proof}

\begin{lemma}\label{LemmaAiry3}
For $a>(32\pi)^{2\slash 3}$, consider the diffusion $p(x),x\geq 0$ given by
\begin{equation*}
    dp(x)=(x-a-p^2(x))dx+\frac{2}{\sqrt{\beta}}dB_x, \quad p(0)=\infty.
\end{equation*}
Let $\Delta>0$ be the first time when $p(x)$ blows up to $-\infty$.
Then for any $M\in\mathbb{R}$ satisfying $\pi\sqrt{a}\leq M\leq a^2\slash 100$,
\begin{equation*}
    \mathbb{P}\Big(\Delta>\frac{4\pi}{\sqrt{a}}+\frac{4M}{a}\Big)\leq 4\exp\Big(-\frac{\beta M a}{64}\Big).
\end{equation*}
\end{lemma}
\begin{proof}

Let $t_1:=\pi\slash\sqrt{a}, t_2:=\pi\slash\sqrt{a}+4M\slash a, t_3:=4\pi\slash\sqrt{a}+4M\slash a\leq 8M\slash a$, and let $\mathcal{C}_M$ be the event that $\sup\limits_{t\in [0,8M\slash a]}|B_t|\leq M\sqrt{\beta}\slash 2$. By (\ref{Eqn26}),
\begin{equation}\label{EL1}
\mathbb{P}(\mathcal{C}_M^c)\leq 4\exp(-\beta M a\slash 64).
\end{equation}

Below we assume that $\mathcal{C}_M$ holds. As $M\leq a^2\slash 100$, we have $t_3\leq a\slash 4$. Let $r(x):=p(x)-(2\slash\sqrt{\beta})B_x$ for all $x\geq 0$. Then $r(x)$ satisfies
\begin{equation*}
    r'(x)=x-a-\Big(r(x)+
    \frac{2}{\sqrt{\beta}} B_x\Big)^2, \quad r(0)=\infty.
\end{equation*}
Note that $r(x)$ is monotone decreasing on $[0,t_3]$ until it blows up. We define $w(x),x\geq 0$ by
\begin{equation*}
    w'(x)=-\frac{3}{4}a-\left(w(x)-M\right)^2, \quad w(0)=\infty.
\end{equation*}
Then 
\begin{equation*}
    w(x)=M+\frac{\sqrt{3a}}{2}\tan   \Big(\frac{\pi}{2}-\frac{\sqrt{3a}}{2}x\Big).
\end{equation*}
Note that $w(x)$ does not blow up on $[0,t_1]$ and $w(t_1)<M$.

Suppose that $r(x)$ does not blow up on $[0,t_1]$ and $r(t_1)\geq M$. Then on $[0,t_1]$, $r(x)$ is dominated by $w(x)$. However, $w(t_1)<M\leq r(t_1)$, which leads to a contradiction. Hence either $r(x)$ blows up on $[0,t_1]$, or $r(x)$ does not blow up on $[0,t_1]$ and $r(t_1)<M$. In the former case, we have $\Delta\leq t_1\leq t_3$. Below we assume that the latter case holds. Note that for $x\in [t_1,t_2]$,
\begin{equation*}
    x-a-\Big(r(x)+\frac{2}{\sqrt{\beta}}B_x\Big)^2\leq -\frac{3}{4}a.
\end{equation*}
Hence either $r(x)$ blows up on $[t_1,t_2]$ (which leads to $\Delta\leq t_2\leq t_3$), or $r(x)$ does not blow up on $[t_1,t_2]$ and
\begin{equation*}
    r(t_2)\leq r(t_1)-\frac{3}{4}a(t_2-t_1)<-M.
\end{equation*}
Below we consider the latter case. Let $s(x),x\in [t_2,t_3]$ be defined by 
\begin{equation*}
    s'(x)=-\frac{3}{4}a-(s(x)+M)^2, \quad s(t_2)=r(t_2).
\end{equation*}
Note that for $x\in [t_2,t_3]$, we have
\begin{equation*}
    s(x)=-M+\frac{\sqrt{3a}}{2}\tan\Big(\arctan\Big(\frac{2}{\sqrt{3a}}(r(t_2)+M)\Big)-\frac{\sqrt{3a}}{2}(x-t_2)\Big).
\end{equation*}
On $[t_2,\min\{t_3,\Delta\}]$, $r(x)$ is dominated by $s(x)$. As $s(x)$ blows up on $[t_2,t_3]$, $r(x)$ also blows up on $[t_2,t_3]$. Hence $\Delta\leq t_3$.

We conclude that $\mathcal{C}_M\subseteq\{\Delta\leq t_3\}$. Noting (\ref{EL1}), we obtain the conclusion of the lemma.

\end{proof}

\subsection{Proof of Proposition \ref{PropAiry1}}\label{Sect.2.3}

In this subsection, we give the proof of Proposition \ref{PropAiry1}.

\begin{proof}[Proof of Proposition \ref{PropAiry1}]

We assume that $\lambda\geq 15$. Let $N_1$ and $N_2$ be the number of blow-ups of the diffusion $p_{\lambda}(x)$ (defined in (\ref{Eqn16})) that lie entirely in the intervals $[0,2\lambda]$ and $[2\lambda,\infty)$, respectively. Then by Proposition \ref{Diffu}, $N(\lambda)\leq N_1+N_2+1$. We define the stopping times $\tau_1\leq\tau_2\leq\cdots$ to be the blow-up times of $p_{\lambda}(x)$ in $[0,\infty)$, and define the stopping times $\tau_1'\leq \tau_2'\leq \cdots$ to be the blow-up times of $p_{\lambda}(x)$ in $[2\lambda,\infty)$. Here, if $p_{\lambda}(x)$ has $j$ blow-ups in $[0,\infty)$ in total, we define $\tau_{j+1}=\tau_{j+2}=\cdots=\infty$; a similar rule applies to $\{\tau_i'\}_{i\geq 1}$. We let $\tau_0:=0$. We also let $\Delta_i=\tau_{i+1}-\tau_i$ for every $i\in\{0,1,2,\cdots\}$ such that $\tau_i<\infty$ and $\Delta_i=0$ for every $i\in\{0,1,2,\cdots\}$ such that $\tau_i=\infty$.

We define the following events:
\begin{equation*}
    \mathcal{C}_1:=\Big\{N_1\geq \frac{1}{4}\eta\lambda^{3\slash 2}\Big\}, \quad \mathcal{C}_2:=\Big\{N_2\geq \frac{1}{4}\eta \lambda^{3\slash 2}\Big\}.
\end{equation*}
Then we have
\begin{equation}\label{Eqn32}
    \mathbb{P}(N(\lambda)\geq \eta \lambda^{3\slash 2})\leq \mathbb{P}(\mathcal{C}_1)+\mathbb{P}(\mathcal{C}_2).
\end{equation}

Now we define the following events for every $i\in\mathbb{N}$:
\begin{equation*}
\mathcal{A}_i:=\Big\{\tau_i<\infty,\Delta_i\geq \frac{4\pi}{5\sqrt{\lambda}}\Big\}\cup\{\tau_i=\infty\},
\end{equation*}
\begin{equation*}
\mathcal{A}_i':=\Big\{\tau_i<\infty,\sup_{x\in [0,2\pi\slash \sqrt{\lambda}]}|B_{x+\tau_i}-B_{\tau_i}|\leq \frac{1}{8}\sqrt{\frac{\beta \lambda}{2}}\Big\}\cup\{\tau_i=\infty\}.
\end{equation*}
For each $i\in\mathbb{N}$, if $\tau_i<\infty$, we define $q(x),x\geq 0$ by \begin{equation}
    dq(x)=(x-\lambda-q^2(x))dx+\frac{2}{\sqrt{\beta}}d\tilde{B}_x, \quad q(0)=\infty,
\end{equation}
where $\tilde{B}_x:=B_{x+\tau_i}-B_{\tau_i}$. Note that if $\tau_i<\infty$, for $x\in [0,\tau_{i+1}-\tau_i]$ and before $q(x)$ blows up, $q(x)$ is a lower bound on $p_{\lambda}(x+\tau_i)$ due to the following argument: Taking
\begin{equation*}
    h_1(x):=p_{\lambda}(x+\tau_i)-\frac{2}{\sqrt{\beta}}\tilde{B}_x, \quad h_2(x):=q(x)-\frac{2}{\sqrt{\beta}}\tilde{B}_x
\end{equation*}
for any $x\geq 0$, we can deduce that
\begin{equation*}
    h_1'(x)=x+\tau_i-\lambda-\bigg(h_1(x)+\frac{2}{\sqrt{\beta}}\tilde{B}_x\bigg)^2, \quad h_1(0)=\infty,
\end{equation*}
\begin{equation*}
    h_2'(x)=x-\lambda-\bigg(h_2(x)+\frac{2}{\sqrt{\beta}}\tilde{B}_x\bigg)^2, \quad h_2(0)=\infty,
\end{equation*}
hence before $h_1(x)$ or $h_2(x)$ blows up, $h_2(x)$ is a lower bound on $h_1(x)$, and $q(x)$ is a lower bound on $p_{\lambda}(x+\tau_i)$. Therefore, by Lemma \ref{LemmaAiry1} (applied to $q(x)$, with $\epsilon=\delta=1\slash 4$ and $a=\lambda$) and the strong Markov property, for any $i\in\mathbb{N}$, we have
\begin{equation}
    \mathbb{P}(\mathcal{A}_i^c|\mathcal{F}_{\tau_i})\leq \mathbb{P}((\mathcal{A}_i')^c|\mathcal{F}_{\tau_i})\leq 4\exp\Big(-\frac{\beta \lambda^{3\slash 2}}{512\pi}\Big).
\end{equation}
Below, we let 
\begin{equation*}
    P_0:=4\exp\Big(-\frac{\beta \lambda^{3\slash 2}}{512\pi}\Big).
\end{equation*}

Let $l_0:=\lceil \eta \lambda^{3\slash 2} \slash 4 \rceil$ and $l:=\lfloor (\eta\slash 4-5\pi^{-1})\lambda^{3\slash 2} \rfloor \geq \eta \lambda^{3\slash 2}\slash 8$. We have
\begin{equation}\label{Eqn30}
    \mathcal{C}_1\subseteq \bigcup_{0\leq i_1<\cdots<i_l\leq l_0-1}\mathcal{A}_{i_1}^c\cap\cdots\cap\mathcal{A}_{i_l}^c.
\end{equation}
This is because if the event on the right-hand side of (\ref{Eqn30}) does not hold and the event on the left-hand side of (\ref{Eqn30}) holds, then we have
\begin{equation*}
    \sum_{i=0}^{l_0-1}\Delta_i\geq \frac{5}{\pi}\lambda^{3\slash 2}\frac{4\pi}{5\sqrt{\lambda}}=4\lambda,
\end{equation*}
\begin{equation*}
    \sum_{i=0}^{l_0-1}\Delta_i\leq 2\lambda,
\end{equation*}
which lead to a contradiction. Now by iterative conditioning, for any choice of $(i_1,\cdots,i_l)\in\{0,1,\cdots\}^l$ such that
\begin{equation}\label{Eqn7}
    0\leq i_1<\cdots<i_l\leq l_0-1,
\end{equation}
we have
\begin{eqnarray}\label{Iteratively}
\mathbb{P}(\mathcal{A}_{i_1}^c\cap\cdots\cap\mathcal{A}_{i_l}^c)&=&\mathbb{E}[\mathbb{E}[\mathbbm{1}_{\mathcal{A}_{i_1}^c}\cdots\mathbb{E}[\mathbbm{1}_{\mathcal{A}_{i_l}^c}|\mathcal{F}_{\tau_{i_l}}]\cdots|\mathcal{F}_{\tau_{i_1}}]]\nonumber\\
&\leq& P_0 \mathbb{E}[\mathbb{E}[\mathbbm{1}_{\mathcal{A}_{i_1}^c}\cdots\mathbb{E}[\mathbbm{1}_{\mathcal{A}_{i_{l-1}}^c}|\mathcal{F}_{\tau_{i_{l-1}}}]\cdots|\mathcal{F}_{\tau_{i_1}}]]\nonumber\\
&\leq&\cdots\leq P_0^l.
\end{eqnarray}
Now note that the number of choices of $(i_1,\cdots,i_l)\in\{0,1,\cdots\}^l$ that satisfies (\ref{Eqn7}) is $\binom{l_0}{l}\leq 2^{l_0}$. Hence by the union bound,
\begin{equation}\label{Eq1}
    \mathbb{P}(\mathcal{C}_1)\leq 2^{l_0}P_0^l\leq 8^{\eta \lambda^{3\slash 2}\slash 2}\exp\Big(-\frac{\beta\eta \lambda^3}{4096\pi}\Big).
\end{equation}

Now for every $i\in\mathbb{N}_{+}$, we define
\begin{equation*}
    \mathcal{B}_i:=\{\tau_{i+1}'<\infty\}.
\end{equation*}
For each $i\in\mathbb{N}_{+}$, if $\tau_i'<\infty$, we define $r(x),x\geq 0$ by
\begin{equation}
    dr(x)=(x+\lambda-r^2(x))dx+\frac{2}{\sqrt{\beta}}d\tilde{B}_x, \quad r(0)=\infty,
\end{equation}
where $\tilde{B}_x:=B_{x+\tau_i'}-B_{\tau_i'}$. Note that if $\tau_i'<\infty$, for $x\in [0,\tau_{i+1}'-\tau_i']$ and before $r(x)$ blows up, $r(x)$ is a lower bound on $p_{\lambda}(x+\tau_i')$ due to the following argument: Taking
\begin{equation*}
    \tilde{h}_1(x):=p_{\lambda}(x+\tau_i')-\frac{2}{\sqrt{\beta}}\tilde{B}_x, \quad \tilde{h}_2(x):=r(x)-\frac{2}{\sqrt{\beta}}\tilde{B}_x
\end{equation*}
for any $x\geq 0$, we can deduce that 
\begin{equation*}
    \tilde{h}_1'(x)=x+\tau_i'-\lambda-\bigg(\tilde{h}_1(x)+\frac{2}{\sqrt{\beta}}\tilde{B}_x\bigg)^2,\quad \tilde{h}_1(0)=\infty,
\end{equation*}
\begin{equation*}
    \tilde{h}_2'(x)=x+\lambda-\bigg(\tilde{h}_2(x)+\frac{2}{\sqrt{\beta}}\tilde{B}_x\bigg)^2,\quad \tilde{h}_2(0)=\infty,
\end{equation*}
hence noting that $\tau_i'\geq 2\lambda$, we conclude that before $\tilde{h}_1(x)$ or $\tilde{h}_2(x)$ blows up, $\tilde{h}_2(x)$ is a lower bound on $\tilde{h}_1(x)$, and $r(x)$ is a lower bound on $p_{\lambda}(x+\tau_i')$. Therefore, by Corollary \ref{Tails} and the strong Markov property, for any $i\in\mathbb{N}_{+}$, we have
\begin{eqnarray}
   &&\mathbb{P}(\mathcal{B}_i|\mathcal{F}_{\tau_i'})=\mathbbm{1}_{\tau_i'<\infty}\mathbb{P}(\tau'_{i+1}<\infty|\mathcal{F}_{\tau_i'})\nonumber\\
   &\leq& \mathbbm{1}_{\tau_i'<\infty} \mathbb{P}(r(x)\text{ blows up in a finite time}|\mathcal{F}_{\tau_i'})\nonumber\\
   &\leq& \mathbb{P}(p_{-\lambda}(x)\text{ blows up in a finite time}) \leq C\exp\Big(-\frac{2}{3}\beta \lambda^{3\slash 2}\Big).
\end{eqnarray}
By iterative conditioning, we get
\begin{eqnarray}\label{Eq2}
\mathbb{P}(\mathcal{C}_2)&\leq&\mathbb{P}(\tau_1'<\infty,\cdots,\tau_{l_0+1}'<\infty)\nonumber\\
&\leq&\mathbb{E}[\mathbb{E}[\mathbbm{1}_{\mathcal{B}_{1}}\cdots\mathbb{E}[\mathbbm{1}_{\mathcal{B}_{l_0}}|\mathcal{F}_{\tau_{l_0}'}]\cdots|\mathcal{F}_{\tau_{1}'}]]\nonumber\\
&\leq& C^{\eta \lambda^{3\slash 2} \slash 4}\exp\Big(-\frac{1}{6}\beta \eta  \lambda^3\Big).
\end{eqnarray}

By (\ref{Eqn32}), (\ref{Eq1}), and (\ref{Eq2}), we obtain the conclusion of the proposition.

\end{proof}

\subsection{Proof of Proposition \ref{PropAiry2}}\label{Sect.2.4}

In this subsection, we give the proof of Proposition \ref{PropAiry2}.

\begin{proof}[Proof of Proposition \ref{PropAiry2}]
The proof is based on dyadic decomposition and comparison with the Airy operator. We recall from Section \ref{Sect.n1} that the eigenvalues of the Airy operator are denoted by $\gamma_1<\gamma_2<\cdots$. We also note that the number of eigenvalues of the Airy operator that are less than or equal to $\lambda$ is equal to the number of blow-ups of $q(x),x\geq 0$, where $q(x)$ is defined by
\begin{equation*}
    q'(x)=x-\lambda-q^2(x), \quad q(0)=\infty.
\end{equation*}
In the following, we assume that $K\geq 1000$ and is sufficiently large (depending on $\beta$).

We fix $d=120$, and denote by $N_1,N_2,N_3$ the number of blow-ups of $p_{\lambda}(x)$ (as defined in (\ref{Eqn16})) in the intervals
\begin{equation*}
    [0,\lambda-\gamma_{\lfloor k'\slash d \rfloor}], \quad (\lambda-\gamma_{\lfloor k'\slash d \rfloor},\lambda+\gamma_{\lfloor k'\slash d \rfloor}], \quad (\lambda+\gamma_{\lfloor k'\slash d \rfloor},\infty),
\end{equation*}
respectively. We also denote by $\tilde{N}_1,\tilde{N}_2,\tilde{N}_3$ the number of blow-ups of $q(x)$ in the above three intervals, respectively. By Proposition \ref{Diffu}, $N(\lambda)=N_1+N_2+N_3$ and $N_0(\lambda)=\tilde{N}_1+\tilde{N}_2+\tilde{N}_3$. 

Taking $u(x):=q(x+\lambda-\gamma_{\lfloor k'\slash d\rfloor})$ for $x\geq 0$, we obtain that
\begin{equation*}
    u'(x)=x-\gamma_{\lfloor k'\slash d\rfloor}-u^2(x), \quad u(0)=q(\lambda-\gamma_{\lfloor k'\slash d\rfloor}). 
\end{equation*}
Note that $\tilde{N}_2+\tilde{N}_3$ is equal to the number of blow-ups of $u(x)$ on $[0,\infty)$. Now we define $\tilde{u}(x),x\geq 0$ by
\begin{equation*}
    \tilde{u}'(x)=x-\gamma_{\lfloor k'\slash d\rfloor}-\tilde{u}^2(x), \quad \tilde{u}(0)=\infty,
\end{equation*}
and denote by $\tilde{N}$ the number of blow-ups of $\tilde{u}(x)$. By the above discussion, $\tilde{N}$ is equal to the number of eigenvalues of the Airy operator that are less than or equal to $\gamma_{\lfloor k'\slash d\rfloor}$, hence $\tilde{N}=\lfloor k'\slash d\rfloor$. Thus by monotonicity,
\begin{equation}\label{nnq4}
    \tilde{N}_2+\tilde{N}_3\leq \tilde{N}+1=\lfloor k'\slash d\rfloor+1\leq k'\slash 20.
\end{equation}

Below, we bound $\mathbb{P}(N_2\geq k'\slash 4)$, $\mathbb{P}(N_3\geq k'\slash 4)$, and $\mathbb{P}(|N_1-\tilde{N}_1|\geq k'\slash 4)$ in \textbf{Steps 1-3}, respectively.

\paragraph{Step 1}

In this step, we bound $\mathbb{P}(N_2\geq k'\slash 4)$. We let $\tau_1'<\cdots<\tau_{\lceil k'\slash 4\rceil} '$ be the first $\lceil k'\slash 4\rceil$ blow-up times of the diffusion $p_{\lambda}(x)$ in $(\lambda-\gamma_{\lfloor k'\slash d\rfloor},\lambda+\gamma_{\lfloor k'\slash d\rfloor}]$ (if the total number of blow-ups of $p_{\lambda}(x)$ in $(\lambda-\gamma_{\lfloor k'\slash d\rfloor},\lambda+\gamma_{\lfloor k'\slash d\rfloor}]$ is $j<\lceil k'\slash 4\rceil$, we let $\tau'_{j+1}=\cdots=\tau_{\lceil k'\slash 4 \rceil} '=\infty$). For any $i\in\{1,\cdots,\lceil k'\slash 4\rceil-1\}$ such that $\tau_i'<\infty$, we let $\Delta_i'=\tau'_{i+1}-\tau'_i$; for any $i\in\{1,\cdots,\lceil k'\slash 4\rceil-1\}$ such that $\tau_i'=\infty$, we let $\Delta_i'=0$.

We note that by Proposition \ref{Airyop}, we can take $K$ sufficiently large such that (note that $k'\geq(\log{k})^3\geq (\log{K})^3$)
\begin{equation}\label{Eqn40}
  \frac{3\pi k'}{4d}\leq (\gamma_{\lfloor k'\slash d\rfloor})^{3\slash 2}\leq \frac{3\pi k'}{d}.
\end{equation}

For every $i\in\{1,\cdots,\lceil k'\slash 4\rceil-1\}$, we define 
\begin{equation}
   \mathcal{K}_i:=\{\tau_i'<\infty, \Delta_i'\geq 5(k')^{-1\slash 3}\}\cup \{\tau_i'=\infty\}.
\end{equation}
For each $i\in\{1,\cdots,\lceil k'\slash 4\rceil-1\}$, if $\tau_i'<\infty$, we define $r(x),x\geq 0$ by 
\begin{equation*}
    dr(x)=(x-\gamma_{\lfloor k'\slash d\rfloor}-r^2(x))dx+\frac{2}{\sqrt{\beta}}d\tilde{B}_x, \quad r(0)=\infty,
\end{equation*}
where $\tilde{B}_x:=B_{x+\tau_i'}-B_{\tau_i'}$. Note that if $\tau_i'<\infty$, for $x\geq 0$, before $p_{\lambda}(x+\tau_i')$ or $r(x)$ blows up, $p_{\lambda}(x+\tau'_i)$ is lower bounded by $r(x)$ due to the following argument: Taking
\begin{equation*}
    h_1(x):=p_{\lambda}(x+\tau'_i)-\frac{2}{\sqrt{\beta}}\tilde{B}_x, \quad h_2(x):=r(x)-\frac{2}{\sqrt{\beta}}\tilde{B}_x
\end{equation*}
for any $x\geq 0$, we can deduce that
\begin{equation*}
    h_1'(x)=x+\tau_i'-\lambda-\bigg(h_1(x)+\frac{2}{\sqrt{\beta}}\tilde{B}_x\bigg)^2,\quad h_1(0)=\infty,
\end{equation*}
\begin{equation*}
    h_2'(x)=x-\gamma_{\lfloor k'\slash d\rfloor}-\bigg(h_2(x)+\frac{2}{\sqrt{\beta}}\tilde{B}_x\bigg)^2,\quad h_2(0)=\infty,
\end{equation*}
hence noting that $\tau_i'\geq \lambda-\gamma_{\lfloor k'\slash d \rfloor}$, we conclude that before $h_1(x)$ or $h_2(x)$ blows up, $h_2(x)$ is a lower bound on $h_1(x)$, and $r(x)$ is a lower bound on $p_{\lambda}(x+\tau_i')$. Therefore, by Lemma \ref{LemmaAiry1} (applied to $r(x)$, with $\epsilon=\delta=1\slash 4$ and $a=\gamma_{\lfloor k'\slash d\rfloor}$) and the strong Markov property, we have
\begin{equation*}
    \mathbb{P}\Big(\Delta_i'<\frac{4\pi}{5}(\gamma_{\lfloor  k'\slash d 
 \rfloor})^{-1\slash 2}\Big|\mathcal{F}_{\tau_i'}\Big)\mathbbm{1}_{\tau_i'<\infty} \leq 4\exp\bigg(-\frac{\beta}{512\pi}(\gamma_{\lfloor  k'\slash d \rfloor})^{3\slash 2}\bigg).
\end{equation*}
By (\ref{Eqn40}), we have
\begin{equation*}
    \mathbb{P}(\mathcal{K}_i^c|\mathcal{F}_{\tau_i'})\leq C_1\exp(-c_1 k'),
\end{equation*}
where $C_1,c_1$ are positive constants that only depend on $\beta$. 

Now we let $l_0:=\lceil k'\slash 4 \rceil$ and $l:=\lfloor k'\slash 20 \rfloor$. As $K\geq 1000$, we have that $k'\geq(\log{k})^3\geq (\log{K})^3 \geq 300$. Note that we have
\begin{equation}\label{Eqn42}
    \Big\{N_2\geq \frac{k'}{4}\Big\} \subseteq \bigcup_{1\leq i_1<\cdots<i_l\leq l_0-1} \mathcal{K}_{i_1}^c\cap\cdots \cap \mathcal{K}_{i_l}^c.
\end{equation}
This is because if the event on the right-hand side of (\ref{Eqn42}) does not hold and the event on the left-hand side of (\ref{Eqn42}) holds, then by (\ref{Eqn40}), we have 
\begin{equation*}
    \sum_{i=1}^{l_0-1} \Delta_i'\geq \Big(\frac{k'}{4}-\frac{k'}{20}-1\Big)(5(k')^{-1\slash 3}) \geq \frac{5}{6}(k')^{2\slash 3}.
\end{equation*}
\begin{equation*}
    \sum_{i=1}^{l_0-1}\Delta_i'\leq 2\gamma_{\lfloor k'\slash d \rfloor}\leq 2\Big(\frac{3\pi k'}{d}\Big)^{2\slash 3}\leq \frac{1}{2}(k')^{2\slash 3}.
\end{equation*}
The above two equations lead to a contradiction. By iterative conditioning as in (\ref{Iteratively}), for any $(i_1,\cdots,i_l)\in \{1,2,\cdots\}^l$ such that $1\leq i_1<\cdots<i_l\leq l_0-1$, we have 
\begin{eqnarray*}
   \mathbb{P}(\mathcal{K}_{i_1}^c\cap\cdots \cap    \mathcal{K}_{i_l}^c)\leq C_2\exp(-c_2(k')^2),
\end{eqnarray*}
where $C_2,c_2$ are positive constants that only depend on $\beta$. Hence by the union bound,
\begin{equation}\label{nn1q}
    \mathbb{P}(N_2\geq  k'\slash 4)  \leq 2^{l_0-1} C_2\exp(-c_2(k')^2)\leq C\exp(-c(k')^2),
\end{equation}
where $C,c$ are positive constants that only depend on $\beta$.

\paragraph{Step 2}

In this step, we bound $\mathbb{P}(N_3\geq k'\slash 4)$. We let $\tau''_1<\cdots<\tau''_{\lceil k'\slash 4\rceil}$ be the first $\lceil k'\slash 4\rceil$ blow-up times of $p_{\lambda}(x)$ in $(\lambda+\gamma_{\lfloor k'\slash d\rfloor},\infty)$ (if the total number of blow-ups of $p_{\lambda}(x)$ in $(\lambda+\gamma_{\lfloor k'\slash d\rfloor},\infty)$ is $j<\lceil k'\slash 4\rceil$, we let $\tau''_{j+1}=\cdots=\tau''_{\lceil k'\slash 4\rceil}=\infty$). 

For every $i\in\{1,\cdots,\lceil k'\slash 
 4\rceil-1\}$, we define 
\begin{equation*}
    \mathcal{K}_i':=\{\tau_{i+1}''<\infty\}.
\end{equation*}
For each $i\in\{1,\cdots,\lceil k'\slash 4 \rceil-1\}$, if $\tau_i''<\infty$, we define $w(x), x\geq 0$ by
\begin{equation*}
    dw(x)=(x+\gamma_{\lfloor k'\slash d \rfloor}-w^2(x))dx+\frac{2}{\sqrt{\beta}}d\tilde{B}_x, \quad w(0)=\infty,
\end{equation*}
where $\tilde{B}_x:=B_{x+\tau''_i}-B_{\tau''_i}$. If $\tau_i''<\infty$, arguing similarly as in \textbf{Step 1} and noting that $\tau_i''\geq \lambda+\gamma_{\lfloor k'\slash d  \rfloor}$, we can deduce that for $x\geq 0$, before $p_{\lambda}(x+\tau_i'')$ or $w(x)$ blows up, $p_{\lambda}(x+\tau_i'')$ is lower bounded by $w(x)$. Therefore, by Corollary \ref{Tails}, the strong Markov property, and (\ref{Eqn40}), we have
\begin{eqnarray*}
&& \mathbb{P}(\mathcal{K}_i'|\mathcal{F}_{\tau''_i})=\mathbbm{1}_{\tau_i''<\infty}\mathbb{P}(\tau''_{i+1}<\infty|\mathcal{F}_{\tau_i''})\nonumber\\
&\leq& \mathbbm{1}_{\tau_i''<\infty}\mathbb{P}(w(x)\text{ blows up in a finite time}|\mathcal{F}_{\tau_i''}) \\
&\leq& \mathbb{P}(p_{-\gamma_{\lfloor k'\slash d\rfloor}}(x)\text{ blows up in a finite time})\\
&\leq& C_3\exp\Big(-\frac{2}{3}\beta (\gamma_{\lfloor k'\slash d\rfloor})^{3\slash 2}\Big)\leq C_3\exp(-c_3 k'),
\end{eqnarray*}
where $C_3,c_3$ are positive constants that only depend on $\beta$.
By iterative conditioning as in (\ref{Eq2}), we obtain that
\begin{eqnarray}\label{nnq2}
   && \mathbb{P}(N_3\geq k'\slash 4)\leq \mathbb{P}(\tau''_1<\infty,\cdots,\tau''_{\lceil k'\slash 4 \rceil}<\infty) \nonumber\\
   &\leq& (C_3\exp(-c_3 k'))^{\lceil k'\slash 4\rceil-1} \leq C\exp(-c(k')^2),
\end{eqnarray}
where $C,c$ are positive constants that only depend on $\beta$.

\paragraph{Step 3}

In this step, we bound $\mathbb{P}(|N_1-\tilde{N}_1|\geq k'\slash 4)$. This step consists of the following three sub-steps:
\begin{itemize}
    \item Sub-step 3.1: Decompose the interval $[0,\lambda-\gamma_{\lfloor k'\slash d\rfloor}]$ into disjoint sub-intervals $\{J_i\}_{i=0}^I$, and relate $\mathbb{P}(|N_1-\tilde{N}_1|\geq k'\slash 4)$ to the probability of certain events $\mathcal{C}_{i,\delta_i}$ with $i\in \{0,\cdots,I\}$ using dyadic decomposition (see below for the definition of $\mathcal{C}_{i,\delta_i}$).
    \item Sub-step 3.2: Provide an upper bound on $\mathbb{P}(\mathcal{C}_{i,\delta_i})$ using dyadic decomposition and the preliminary lemmas in Section \ref{Sect.2.2}.
    \item Sub-step 3.3: Use the upper bound obtained in Sub-step 3.2 to provide an upper bound on $\mathbb{P}(|N_1-\tilde{N}_1|\geq k'\slash 4)$.
\end{itemize}

\subparagraph{Sub-step 3.1}

First we decompose the interval $[0,\lambda-\gamma_{\lfloor k'\slash d\rfloor}]$ as follows. We let $I\in\mathbb{N}_{+}$ be such that $\gamma_{\lfloor 2^{I+1}k'\slash d \rfloor}\leq \lambda\leq \gamma_{\lfloor 2^{I+2}k'\slash d\rfloor}$. Note that
\begin{equation*}
    \Big|I+1-\frac{\log(kd\slash k')}{ \log(2)}\Big|\leq 1.
\end{equation*}
We denote $J_i:=(\lambda-\gamma_{\lfloor 2^{i+1}k'\slash d\rfloor},\lambda-\gamma_{\lfloor 2^{i}k'\slash d\rfloor}]$ for each $i\in\{0,1,\cdots,I-1\}$, and denote $J_{I}:=[0,\lambda-\gamma_{{\lfloor 2^{I}k'\slash d\rfloor}}]$. Below we consider any $i\in\{0,1,\cdots,  I\}$. We denote by $a_i$ and $b_i$ the left and right endpoints of $J_i$, respectively. We denote by $M_i$ and $\tilde{M}_i$ the number of blow-ups of $p_{\lambda}(x)$ and $q(x)$, respectively, in $J_i$. We define $u_i(x)$ and $v_i(x)$ for $x\geq a_i$ by   
\begin{equation*}
    du_i(x)=(x-\lambda-u_i^2(x))dx+\frac{2}{\sqrt{\beta}}dB_x, \quad u_i(a_i)=\infty,
\end{equation*}
\begin{equation*}
    v_i'(x)=x-\lambda-v_i^2(x), \quad  v_i(a_i)=\infty,
\end{equation*}
and denote by $M_i'$ and $\tilde{M}_i'$ the number of blow-ups of $u_i(x)$ and $v_i(x)$, respectively, in $J_i$. By monotonicity, we have $|M_i-M_i'|\leq 1$ and $|\tilde{M}_i-\tilde{M}_i'|\leq 1$. We also note that the random variables $\{M_i'\}_{i=0}^{I}$ are mutually independent. For each $i\in\{0,1,\cdots,I\}$, we let $\Delta_i:=|M_i'-\tilde{M}_i'|$, $\alpha_i:=2^i k'\slash d$, and $t_i:=\tilde{M}_i'$. By Lemma \ref{LemmaAiry2}, there exist positive absolute constants $c_1',c_2'$, such that $c_1'\alpha_i\leq t_i\leq c_2'\alpha_i$ for all $i\in\{0,1,\cdots, I\}$. As 
\begin{equation*}
    |M_i-\tilde{M}_i|\leq |M_i-M_i'|+|M_i'-\tilde{M}_i'|+|\tilde{M}_i'-\tilde{M}_i|\leq \Delta_i+2,
\end{equation*}
when the event $\{|N_1-\tilde{N}_1|\geq k'\slash 4\}$ holds, we have
\begin{equation*}
    k'\slash 4\leq |N_1-\tilde{N}_1|\leq \sum_{i=0}^{I}\Delta_i+2(I+1).
\end{equation*}
For $k\geq K$ sufficiently large, this leads to $\sum_{i=0}^{
I}\Delta_i\geq k'\slash 8$.

Now we perform a dyadic decomposition. We take $M$ sufficiently large (depending on $\beta$); note that by Lemma \ref{LemmaAiry1}, for any $i\in\{0,1,\cdots,I\}$, we have
\begin{equation*}
    \mathbb{P}(\Delta_i\geq Mk)\leq \mathbb{P}(M_i'\geq Mk)\leq C\exp(-cMk^2).
\end{equation*}
We let $\mathcal{D}_1$ be the event that $\Delta_i\geq Mk$ for some $i\in\{0,1,\cdots,  I\}$. By the union bound, we have
\begin{equation}\label{EM5}
    \mathbb{P}(\mathcal{D}_1)\leq Ck\exp(-cMk^2)\leq C\exp\bigg(-\frac{c (k')^2}{\log(2k\slash k')}\bigg).
\end{equation}
Let $\Upsilon:=\{2^l:l\in\mathbb{N},l\leq \lceil\log_2(Mk)\rceil-1\}\cup\{0\}$. Consider any $i\in\{0,1,\cdots,I\}$. For any $\delta_i\in\Upsilon\backslash\{0\}$, we define $\mathcal{C}_{i,\delta_i}$ to be the event that $\delta_i\leq \Delta_i\leq 2\delta_i$; for $\delta_i=0$, we define $\mathcal{C}_{i,\delta_i}$ to be the event that $\Delta_i=0$. We also define
\begin{equation*}
 \mathcal{D}_2:=\bigcup_{(\delta_i)_{i=0}^I\in \Upsilon^{I+1}: \sum_{i=0}^I\delta_i\geq k'\slash 16}\bigcap_{i=0}^I \mathcal{C}_{i,\delta_i}.
\end{equation*} 

From the analysis above, we have 
\begin{equation}\label{EM1}
    \{|N_1-\tilde{N}_1|\geq k'\slash 4\}\subseteq \mathcal{D}_1\cup \mathcal{D}_2.
\end{equation}
Note that by independence, for any $\delta_0,\cdots,\delta_I\in \Upsilon$,
\begin{equation}\label{EM2}
    \mathbb{P}\Big(\bigcap_{i=0}^I\mathcal{C}_{i,\delta_i}\Big)=\prod_{i=0}^{I}\mathbb{P}(\mathcal{C}_{i,\delta_i}).
\end{equation}

\subparagraph{Sub-step 3.2}

In the following, we consider any $i\in\{0,1,\cdots,I\}$ and any $\delta_i\in\Upsilon$, and proceed to bound $\mathbb{P}(\mathcal{C}_{i,\delta_i})$. We define
\begin{equation*}
    \mathcal{A}_{i,\delta_i}:=\mathcal{C}_{i,\delta_i}\cap\{M_i'\geq \tilde{M}_i'\}, \quad  \mathcal{B}_{i,\delta_i}:=\mathcal{C}_{i,\delta_i}\cap\{M_i'<\tilde{M}_i'\}.
\end{equation*}
Note that $\mathcal{C}_{i,\delta_i}=\mathcal{A}_{i,\delta_i}\cup\mathcal{B}_{i,\delta_i}$.

We bound $\mathbb{P}(\mathcal{A}_{i,\delta_i})$ as follows. We denote by $\tau_1<\tau_2<\cdots <\tau_{t_i+\delta_i}$ the first $t_i+\delta_i$ blow-up times of $u_i(x)$ in $J_i$ (if the total number of blow-ups of $u_i(x)$ in $J_i$ is $l<t_i+\delta_i$, we let $\tau_{l+1}=\cdots=\tau_{t_i+\delta_i}=\infty$) and $\tilde{\tau}_1<\tilde{\tau}_2<\cdots<\tilde{\tau}_{t_i}$ the $t_i$ blow-up times of $v_i(x)$ in $J_i$. We also let $\tau_0=\tilde{\tau}_0=a_i$ (recall that $a_i$ is the left endpoint of $J_i$). For each $j\in\{0,1,\cdots,t_i-1\}$, if $\tau_j<\infty$, we define $w_{i,j} (x)$ for $x\geq \tau_j$ by $w_{i,j}'(x)=x-\lambda-w_{i,j}^2(x), w_{i,j}(\tau_j)=\infty$, and denote by $\Delta_{j,2}$ the time it takes from $+\infty$ to $-\infty$ for the first blow-up of $w_{i,j}(x)$ after $\tau_j$; we also denote by $\Delta_{j,1}$ and $\Delta_{j,3}$ the time it takes from $+\infty$ to $-\infty$ for the $(j+1)$st blow-up of $u_i(x)$ and $v_i(x)$, respectively, in $[a_i,\infty)$.

By Lemma \ref{LemmaAiry1} and the strong Markov property, there exist positive constants $C,c$ that only depend on $\beta$, such that for any $j\in\{0,1,\cdots,t_i-1\}, \epsilon\in (0,1\slash 3)$, we have 
\begin{equation}
    \mathbb{P}\bigg(\bigg|\Delta_{j,1}-\frac{\pi}{\sqrt{\lambda-\tau_j}}\bigg|>\frac{C\epsilon}{\alpha_i^{1\slash 3}}+\frac{C}{\alpha_i^{4\slash 3}}\bigg|\mathcal{F}_{\tau_j}\bigg)\mathbbm{1}_{\tau_j<\infty}\leq C\exp(-c\epsilon^2\alpha_i).
\end{equation}

Now we perform a dyadic decomposition again. Below we consider any $j\in\{0,1,\cdots,t_i-1\}$, and take $c_0,C_0,c_1,C_1>0$ to be uniformly bounded and bounded away from $0$ (for fixed $\beta$), such that $c_0,c_1$ are sufficiently small and $C_0,C_1$ are sufficiently large so that the estimates in this paragraph hold. For any $\delta_{i,j}\in\{C_0\slash \alpha_i,2C_0\slash \alpha_i,\cdots,c_0\alpha_i\}$ (the ratio between two consecutive numbers is $2$), we let $\mathcal{E}_{i,j,\delta_{i,j}}$ be the event that
\begin{equation*}
   \tau_j<\infty,\quad \frac{\delta_{i,j}}{\alpha_i^{1\slash 3}}\leq \bigg|\Delta_{j,1}-\frac{\pi}{\sqrt{\lambda-\tau_j}}\bigg|\leq \frac{2\delta_{i,j}}{\alpha_i^{1\slash 3}};
\end{equation*}
for $\delta_{i,j}=0$, we let $\mathcal{E}_{i,j,\delta_{i,j}}$ be the event that 
\begin{equation*}
   \tau_j<\infty, 
 \quad \bigg|\Delta_{j,1}-\frac{\pi}{\sqrt{\lambda-\tau_j}}\bigg|<\frac{C_0}{\alpha_i^{4\slash 3}}.
\end{equation*}
We also let $\tilde{\mathcal{A}}_{i,1}$ be the event that there exists some $j\in\{0,1,\cdots,t_i-1\}$ such that
\begin{equation*}
  \tau_j<\infty,  \quad  \bigg|\Delta_{j,1}-\frac{\pi}{\sqrt{\lambda-\tau_j}}\bigg|>2c_0\alpha_i^{2\slash 3}.
\end{equation*}
By Lemma \ref{LemmaAiry3} and the strong Markov property,
\begin{equation*}
    \mathbb{P}(\tilde{\mathcal{A}}_{i,1})\leq Ct_i\exp(-c \alpha_i^2)\leq C\exp(-c (k')^2).
\end{equation*}
Denote $\Theta_i:=\{0,C_0\slash\alpha_i,2C_0\slash \alpha_i,\cdots,c_0\alpha_i\}$, and consider any $\delta_{i,j}\in \Theta_i$. If $0<\delta_{i,j}\leq c_1$, by Lemma \ref{LemmaAiry1} and the strong Markov property,
\begin{equation*}
    \mathbb{P}(\mathcal{E}_{i,j,\delta_{i,j}}|\mathcal{F}_{\tau_j})\leq \mathbb{P}\bigg(\bigg|\Delta_{j,1}-\frac{\pi}{\sqrt{\lambda-\tau_j}}\bigg|\geq \frac{\delta_{i,j}}{\alpha_i^{1\slash 3}}\bigg|\mathcal{F}_{\tau_j}\bigg)\mathbbm{1}_{\tau_j<\infty}\leq C\exp(-c\delta_{i,j}^2\alpha_i).
\end{equation*}
If $c_1\leq \delta_{i,j}\leq C_1$, 
\begin{equation*}
    \mathbb{P}(\mathcal{E}_{i,j,\delta_{i,j}}|\mathcal{F}_{\tau_j})\leq \mathbb{P}\bigg(\bigg|\Delta_{j,1}-\frac{\pi}{\sqrt{\lambda-\tau_j}}\bigg|\geq \frac{c_1}{\alpha_i^{1\slash 3}}\bigg|\mathcal{F}_{\tau_j}\bigg)\mathbbm{1}_{\tau_j<\infty}\leq C\exp(-c\delta_{i,j}^2\alpha_i).
\end{equation*}
If $C_1\leq \delta_{i,j}\leq c_0\alpha_i$, by Lemma \ref{LemmaAiry3} and the strong Markov property,
\begin{equation*}
    \mathbb{P}(\mathcal{E}_{i,j,\delta_{i,j}}|\mathcal{F}_{\tau_j})\leq \mathbb{P}\bigg(\bigg|\Delta_{j,1}-\frac{\pi}{\sqrt{\lambda-\tau_j}}\bigg|\geq\frac{\delta_{i,j}}{\alpha_i^{1\slash 3}}\bigg|\mathcal{F}_{\tau_j}\bigg)\mathbbm{1}_{\tau_j<\infty}\leq C\exp(-c\delta_{i,j}\alpha_i).
\end{equation*}
In summary, if we denote $\phi(\delta_{i,j}):=\delta_{i,j}^2 \mathbbm{1}_{\delta_{i,j}\leq C_1}+\delta_{i,j} \mathbbm{1}_{\delta_{i,j}>C_1}$, then
\begin{equation*}
    \mathbb{P}(\mathcal{E}_{i,j,\delta_{i,j}}|\mathcal{F}_{\tau_j})\leq C\exp(-c\phi(\delta_{i,j})\alpha_i).
\end{equation*}
Note that when the event $\mathcal{E}_{i,j,\delta_{i,j}}$ holds,
\begin{equation*}
    \bigg|\Delta_{j,1}-\frac{\pi}{\sqrt{\lambda-\tau_j}}\bigg|\leq \frac{C\delta_{i,j}}{\alpha_i^{1\slash 3}}+\frac{C}{\alpha_i^{4\slash 3}}.
\end{equation*}
By Lemma \ref{LemmaAiry2}, if $\tau_j<\infty$, we have 
\begin{equation*}
    \frac{\pi}{\sqrt{\lambda-\tau_j}}\leq \Delta_{j,2}\leq \frac{\pi}{\sqrt{\lambda-\tau_j-2\pi\slash\sqrt{\lambda-\tau_j}}}.
\end{equation*}
Thus when the event $\mathcal{E}_{i,j,\delta_{i,j}}$ holds, we have $|\Delta_{j,1}-\Delta_{j,2}|\leq C\delta_{i,j}\alpha_i^{-1\slash 3}+C\alpha_i^{-4\slash 3}$.

Consider any $(\delta_{i,j})_{j=0}^{t_i-1}\in\Theta_i^{t_i}$. We let $\tilde{\mathcal{A}}_{i,2,(\delta_{i,j})_{j=0}^{t_i-1}}:=\bigcap_{j=0}^{t_i-1}\mathcal{E}_{i,j,\delta_{i,j}}
$, and assume that $\mathcal{A}_{i,\delta_i}\cap \tilde{\mathcal{A}}_{i,2,(\delta_{i,j})_{j=0}^{t_i-1}}$ holds. We also set $K_0:=\sum_{j=0}^{t_i-1}\delta_{i,j}\alpha_i^{-1\slash 3}$. Now we consider any $j\in\{0,1,\cdots,t_i-1\}$. Note that 
\begin{equation*}
  |\tau_j-\tilde{\tau}_j|\leq  \sum_{l=0}^{j-1}|\Delta_{l,1}-\Delta_{l,3}|\leq \sum_{l=0}^{j-1}|\Delta_{l,2}-\Delta_{l,3}|+\frac{C\sum_{l=0}^{j-1}\delta_{i,l}}{\alpha_i^{1\slash 3}}+\frac{C}{\alpha_i^{1\slash 3}}.
\end{equation*}
Hence by Lemma \ref{LemmaAiry2}, recalling the definitions of $v_i(x)$ and $w_{i,j}(x)$, we have 
\begin{equation*}
    |\Delta_{j,2}-\Delta_{j,3}|\leq \frac{C}{\alpha_i}\bigg(\sum_{l=0}^{j-1}|\Delta_{l,2}-\Delta_{l,3}|+\frac{\sum_{l=0}^{j-1}\delta_{i,l}}{\alpha_i^{1\slash 3}}+\frac{1}{\alpha_i^{1\slash 3}}\bigg),
\end{equation*}
Hence denoting $S_j:=\sum_{l=0}^j|\Delta_{l,2}-\Delta_{l,3}|$, we have 
\begin{equation*}
    S_j\leq \Big(1+\frac{C}{\alpha_i}\Big)S_{j-1}+C\Big(\frac{K_0}{\alpha_i}+\frac{1}{\alpha_i^{4\slash 3}}\Big),
\end{equation*}
which leads to $S_j+K_0+\alpha_i^{-1\slash 3}\leq (1+C\slash\alpha_i)(S_{j-1}+K_0+\alpha_i^{-1\slash 3})$. Thus
\begin{equation*}
  S_{t_i-1}\leq (1+C\slash \alpha_i)^{t_i}(K_0+\alpha_i^{-1\slash 3})\leq C(K_0+\alpha_i^{-1\slash 3}),
\end{equation*}
\begin{equation}\label{EM4}
     \sum_{j=0}^{t_i-1}|\Delta_{j,1}-\Delta_{j,3}|\leq C(K_0+\alpha_i^{-1\slash 3})=C \Big(\sum_{j=0}^{t_i-1}\delta_{i,j}+1\Big) \alpha_i^{-1\slash 3}.
\end{equation}

Now let $\tilde{\mathcal{A}}_{i,3}$ be the event that there exist at least $\lceil\delta_i\slash 4\rceil$ elements $l\in[t_i+\delta_i]$ such that $\tau_l<\infty$ and $\tau_{l}-\tau_{l-1}<c_2 \alpha_i^{-1\slash 3}$ (where $c_2>0$ is a small enough constant that only depends on $\beta$). Similar to the proof of Proposition \ref{PropAiry1}, we can deduce that
\begin{equation*}
    \mathbb{P}(\tilde{\mathcal{A}}_{i,3}) \leq \exp(-c \alpha_i\delta_i)\leq \exp(-ck'\delta_i).
\end{equation*}
For any $(\delta_{i,j})_{j=0}^{t_i-1}\in\Theta_i^{t_i}$, when the event $\mathcal{A}_{i,\delta_i}\cap (\tilde{\mathcal{A}}_{i,3})^c\cap\tilde{\mathcal{A}}_{i,2,(\delta_{i,j})_{j=0}^{t_i-1}}$ holds, we have $\tau_{t_i+\delta_i}\in J_i$, and there exist at least $\delta_i-(\lceil \delta_i\slash 4\rceil-1)\geq 3\delta_i\slash 4$ elements $l\in [t_i+1,t_i+\delta_i]\cap\mathbb{Z}$ such that $\tau_{l}-\tau_{l-1}\geq c_2\alpha_i^{-1\slash 3}$. Thus
\begin{equation*}
  b_i-\tau_{t_i}\geq  \tau_{t_i+\delta_i}-\tau_{t_i}\geq c_2\alpha_i^{-1\slash 3}\cdot \frac{3\delta_i}{4}\geq c\delta_i\alpha_i^{-1\slash 3}. 
\end{equation*}
By Lemma \ref{LemmaAiry2}, we have $b_i-\tilde{\tau}_{t_i}\leq C\alpha_i^{-1\slash 3}$ (recall that $b_i$ is the right endpoint of $J_i$). Hence $\tilde{\tau}_{t_i}-\tau_{t_i}\geq (c\delta_i-C)\alpha_i^{-1\slash 3}$. Thus by (\ref{EM4}),
\begin{equation*}
    (c\delta_i-C)\alpha_i^{-1\slash 3}\leq \tilde{\tau}_{t_i}-\tau_{t_i}\leq  \sum_{j=0}^{t_i-1}|\Delta_{j,1}-\Delta_{j,3}|\leq C \Big(\sum_{j=0}^{t_i-1}\delta_{i,j}+1\Big) \alpha_i^{-1\slash 3},
\end{equation*}
which leads to $\sum_{j=0}^{t_i-1}\delta_{i,j} \geq c\delta_i-C$. Hence
\begin{equation*}
    \mathcal{A}_{i,\delta_i}\subseteq \tilde{\mathcal{A}}_{i,1}\bigcup \tilde{\mathcal{A}}_{i,3}\bigcup\bigg(\bigcup_{(\delta_{i,j})_{j=0}^{t_i-1}\in \Theta_i^{t_i}:\sum_{j=0}^{t_i-1}\delta_{i,j}\geq c\delta_i-C}\tilde{\mathcal{A}}_{i,2,(\delta_{i,j})_{j=0}^{t_i-1}}\bigg).
\end{equation*}

Now consider any $(\delta_{i,j})_{j=0}^{t_i-1}\in \Theta_i^{t_i}$ such that $\sum_{j=0}^{t_i-1}\delta_{i,j}\geq c\delta_i-C$. We let $T_i:=\{j\in \{0,1,\cdots,t_i-1\}:\delta_{i,j} \leq C_1\}$. Note that when $\sum_{j=0}^{t_i-1}\delta_{i,j}\geq c\delta_i-C$, either $\sum_{j\in T_i}\delta_{i,j}\geq (c\delta_i-C)\slash 2$ or $\sum_{j\in\{0,1,\cdots,t_i-1\}\backslash T_i}\delta_{i,j}\geq (c\delta_i-C)\slash 2$. For the former case, by the Cauchy-Schwarz inequality, we have
\begin{equation*}
    \sum_{j=0}^{t_i-1} \phi(\delta_{i,j})\geq  \sum_{j\in T_i}\delta_{i,j}^2\geq \frac{1}{t_i}\Big(\sum_{j\in T_i}\delta_{i,j}\Big)^2\geq \frac{1}{4 t_i}(c\delta_i-C)_{+}^2;
\end{equation*}
for the latter case, we have 
\begin{equation*}
     \sum_{j=0}^{t_i-1} \phi(\delta_{i,j})\geq \sum_{j\in \{0,1,\cdots,t_i-1\}\backslash  T_i}\delta_{i,j} \geq  \frac{1}{2}(c\delta_{i}-C).
\end{equation*}
Hence similar to the proof of Proposition \ref{PropAiry1} (using the strong Markov property), we obtain that
\begin{eqnarray}\label{neeeeq}
    \mathbb{P}(\mathcal{A}_{i,\delta_i}) &\leq& 
    \mathbb{P}(\tilde{\mathcal{A}}_{i,1})+\mathbb{P}(\tilde{\mathcal{A}}_{i,3})+\sum_{(\delta_{i,j})_{j=0}^{t_i-1}\in \Theta_i^{t_i}:\sum_{j=0}^{t_i-1}\delta_{i,j}\geq c\delta_i-C}\mathbb{P}\big(\tilde{\mathcal{A}}_{i,2,(\delta_{i,j})_{j=0}^{t_i-1}}\big)\nonumber\\
    &\leq& C\exp(-c (k')^2)+C\exp(-ck'\delta_i)\nonumber\\
    && + C^{t_i}\sum_{(\delta_{i,j})_{j=0}^{t_i-1}\in \Theta_i^{t_i}:\sum_{j=0}^{t_i-1}\delta_{i,j}\geq c\delta_i-C}\exp\Big(-c\sum_{j=0}^{t_i-1}\phi(\delta_{i,j})\alpha_i\Big)\nonumber\\
    &\leq& \exp(Ct_i\log\log{k})(\exp(-ck'\min\{\delta_i,k'\})+\exp(-c\delta_i^2)).
\end{eqnarray}

To bound $\mathbb{P}(\mathcal{B}_{i,\delta_i})$, we replace $t_i$ by $t_i-\delta_i$ in the previous argument and obtain the same upper bound as in (\ref{neeeeq}). Thus by the union bound, we have
\begin{equation}\label{EM3}
    \mathbb{P}(\mathcal{C}_{i,\delta_i})\leq
    \exp(Ct_i\log\log{k})(\exp(-ck'\min\{\delta_i,k'\})+\exp(-c\delta_i^2)).
\end{equation}

\subparagraph{Sub-step 3.3}

By (\ref{EM1}) and the union bound, we have 
\begin{equation}\label{EM7}
    \mathbb{P}(|N_1-\tilde{N}_1|\geq k'\slash 4)\leq \mathbb{P}(\mathcal{D}_1)+\sum_{(\delta_i)_{i=0}^I\in \Upsilon^{I+1}:\sum_{i=0}^I \delta_i\geq k'\slash 16}\mathbb{P}\Big(\bigcap_{i=0}^I \mathcal{C}_{i,\delta_i}\Big).
\end{equation}

Now consider any $(\delta_i)_{i=0}^I\in \Upsilon^{I+1}$ such that $\sum_{i=0}^I\delta_i\geq k'\slash 16$.
For any $S\subseteq\{0,1,\cdots,I\}$, at least one of $\sum_{i\in S}\delta_i$ and $\sum_{i\in \{0,1,\cdots,I\}\backslash S}\delta_i$ is greater than or equal to $k'\slash 32$. Hence by the Cauchy-Schwarz inequality and the fact that $\sum_{i=0}^I t_i\leq Ck$,
\begin{eqnarray*}
&&ck'\sum_{i\in S}\min\{\delta_i,k'\}+c\sum_{i\in \{0,1,\cdots,I\}\backslash S}\delta_i^2-C\sum_{i=0}^I t_i\log\log{k}\\
&\geq& ck'\sum_{i\in S}\min\{\delta_i,k'\}+\frac{c}{I+1}\Big(\sum_{i\in\{0,1,\cdots,I\}\backslash S}\delta_i\Big)^2-Ck\log\log{k}\\
&\geq& c\frac{(k')^2}{\log(2k\slash k')}-Ck\log\log{k}.
\end{eqnarray*}
Hence by (\ref{EM2}) and (\ref{EM3}), we have
\begin{eqnarray}\label{EM6}  \mathbb{P}\Big(\bigcap_{i=0}^I\mathcal{C}_{i,\delta_i}\Big)
&\leq& \sum_{S\subseteq \{0,1,\cdots,I\}}\exp\Big(-ck'\sum_{i\in S}\min\{\delta_i,k'\}-c\sum_{i\in\{0,1,\cdots,I\}\backslash S}\delta_i^2\Big)  \nonumber \\
&&\quad\quad\quad\quad \times\exp\Big(C\sum_{i = 0}^I t_i \log\log{k}\Big)\nonumber\\
&\leq& \exp(Ck\log\log{k})\exp\Big(-\frac{c (k')^2}{ \log(2k\slash k')}\Big).
\end{eqnarray}

Now note that the number of choices of $(\delta_i)_{i=0}^I \in \Upsilon^{I+1}$ is at most
\begin{equation*}
    (C\log{k})^{C\log(2k\slash k')}\leq \exp(C(\log{k})^2).
\end{equation*}
Hence by (\ref{EM7}), (\ref{EM5}), and (\ref{EM6}),
\begin{equation}\label{nnq3}
    \mathbb{P}(|N_1-\tilde{N}_1|\geq k'\slash 4)\leq \exp(Ck\log{k})\exp\bigg(-
    \frac{c(k')^2}{\log(2k\slash k')}\bigg).
\end{equation}

\bigskip

By (\ref{nnq4}), (\ref{nn1q}), (\ref{nnq2}), and (\ref{nnq3}), we conclude that
\begin{eqnarray}
   && \mathbb{P}(|N(\lambda)-N_0(\lambda)|\geq k')\nonumber\\
   &\leq& \mathbb{P}(N_2\geq  k'\slash 4)+\mathbb{P}(N_3\geq  k'\slash 4)+ \mathbb{P}(|N_1-\tilde{N}_1|\geq k'\slash 4)\nonumber\\
    &\leq& \exp(Ck\log{k})\exp\bigg(-
    \frac{c(k')^2}{\log(2k\slash k')}\bigg).
\end{eqnarray}

\end{proof}

\subsection{Proof of Theorem \ref{Main1}, parts (b) and (c)}\label{Sect.3.4}

In this subsection, we present the proof of parts (b) and (c) of Theorem \ref{Main1} based on Propositions \ref{PropAiry1} and \ref{PropAiry2}. We first give the definition of the function $I_2$ appearing in the statement of Theorem \ref{Main1} as follows.

\begin{definition}\label{I2}
Recall the definition of $\mathcal{X}$ in (\ref{Def:X}). For any $\mu\in\mathcal{X}$ and any $x\in [-R_0,R_0]$, we define $\phi_{\mu}(x):=\mu([-R_0,x])$ and
\begin{equation*}
\psi_{\mu}(x):=
\begin{cases}
    \min\Big\{\frac{2}{3}R_0^{3\slash 2}|\phi_{\mu}(x)|,\frac{c\min\{\phi_{\mu}(x)^2,x^3\slash 6\}}{8\log(\max\{2,x^{3\slash 2}\slash |\phi_{\mu}(x)|\})}\Big\} & \text{ if }  x\in [0,R_0],\phi_{\mu}(x)\neq 0  \\
    0 & \text{ if } x\in [0,R_0],\phi_{\mu}(x)= 0\\
    \frac{4}{3}|x|^{3\slash 2}|\phi_{\mu}(x)| & \text{ if } x\in[-R_0,0)
\end{cases},
\end{equation*}
where $c>0$ is the constant appearing in Proposition \ref{PropAiry2}.
We further define $I_2(\mu):=\sup_{x\in [-R_0,R_0]}\psi_{\mu}(x)$ for any $\mu\in \mathcal{X}$.
\end{definition}
\begin{remark}
The definition of $\psi_{\mu}(x)$ stems from Propositions \ref{PropAiry1} and \ref{PropAiry2}.
\end{remark}

\begin{proof}[Proof of Theorem \ref{Main1}, part (c)]
As $[-R_0,R_0]$ is a compact set, $K_{\eta}$ is uniformly tight. Note that for any $\mu\in K_{\eta}$, $|\mu|([-R_0,R_0])\leq \eta R_0^{3\slash 2}$. Hence by \cite[Theorems 8.4.7 and 8.6.2]{Bog}, $K_{\eta}$ is a compact subset of $\mathcal{X}$. 

Below we consider any $\eta\geq 20$. For each $i\in\mathbb{N}_{+}$, we take $\lambda_i:=-a_i$. Note that by Proposition \ref{Eigens}, $(\lambda_i)_{i=1}^{\infty}$ has the same distribution as the ordered eigenvalues of $H_2$. Recalling (\ref{nukr}) and (\ref{Eq2.14}), we obtain that
\begin{eqnarray*}
    |\nu_{k;R_0}|([-R_0,R_0])&\leq& \frac{1}{k}\#\{i\in\mathbb{N}_{+}:k^{-2\slash 3}\lambda_i\in[-R_0,R_0]\}+\nu_{0;R_0}([-R_0,R_0])\nonumber\\
    &\leq& \frac{1}{k}N(k^{2\slash 3}R_0)+\frac{2}{3\pi}R_0^{3\slash 2}\leq \frac{1}{k}N(k^{2\slash 3}R_0)+R_0^{3\slash 2}.
\end{eqnarray*}
Hence by Proposition \ref{PropAiry1}, for $k$ sufficiently large, we have 
\begin{eqnarray*}
 && \mathbb{P}(\nu_{k;R_0}\notin K_{\eta}) = \mathbb{P}(|\nu_{k;R_0}|([-R_0,R_0])>\eta R_0^{3\slash 2}) \nonumber\\
 &\leq& \mathbb{P}(N(k^{2\slash 3}R_0)\geq (\eta-1)kR_0^{3\slash 2})\leq C\exp(-c\eta R_0^3 k^2),
\end{eqnarray*}
from which (\ref{exponential.tightness}) follows.

\end{proof}

\begin{proof}[Proof of Theorem \ref{Main1}, part (b)]

Consider a fixed $\mu\in\mathcal{X}$ and any $x_0\in [-R_0,R_0]$ such that $\phi_{\mu}(x_0)\neq 0$. Without loss of generality, we assume that $\phi_{\mu}(x_0)>0$. For each $i\in\mathbb{N}_{+}$, we take $\lambda_i:=-a_i$. By Proposition \ref{Eigens}, $(\lambda_i)_{i=1}^{\infty}$ has the same distribution as the ordered eigenvalues of $H_2$.

In this paragraph, we assume that $d_{R_0}(\nu_{k;R_0},\mu)\leq \delta$, where $\delta\in (0,1)$. Recall (\ref{nukr}) and (\ref{Eq2.14}). If $x_0=R_0$, we take $h_1(x):=1$ for any $x\in [-R_0,R_0]$ (note that $\|h_1\|_{BL}\leq 1$), and obtain that
\begin{eqnarray}\label{nnee1}
  && \frac{1}{k}(N(R_0 k^{2\slash 3})-N(-R_0 k^{2\slash 3})+1)-\frac{2}{3\pi} R_0^{3\slash 2} \geq
    \nu_{k;R_0}([-R_0,R_0])\nonumber\\
    &=&\int_{[-R_0,R_0]} h_1 
 d\nu_{k;R_0} \geq \int_{[-R_0,R_0]} h_1 d\mu -\delta =\phi_{\mu}(R_0)-\delta. 
\end{eqnarray}
Now we assume that $x_0<R_0$. It can be checked that $\phi_{\mu}(\cdot)$ is right-continuous on $[-R_0,R_0]$. Hence for any $\epsilon\in (0,\phi_{\mu}(x_0)\slash 2)$, there exists $t\in (0,\min\{1,R_0-x_0\})$ such that $\phi_{\mu}(x)>\phi_{\mu}(x_0)-\epsilon$ for all $x\in [x_0,x_0+t]$. We take $\epsilon':=\delta^{1\slash 2}R_0^{-1\slash 4}$, and let $h_2(x):=\max\{1-(x-x_0-t+\epsilon')_{+}\slash\epsilon',0\}$ for any $x\in [-R_0,R_0]$. Below we assume that $\delta$ is sufficiently small, so that $\epsilon'<\min\{1,t\}$. As $d_{R_0}(\nu_{k;R_0},\mu)\leq \delta$, $\|h_2\|_{BL}\leq (\epsilon')^{-1}$, and $\mu+\nu_{0;R_0}$ is a positive measure on $[-R_0,R_0]$, we have 
\begin{eqnarray}
&& \nu_{k;R_0}([-R_0,x_0+t])+\nu_{0;R_0}([-R_0,x_0+t])\nonumber\\
&\geq&\int_{[-R_0,R_0]} h_2 d(\nu_{k;R_0}+\nu_{0;R_0}) \geq\int_{[-R_0,R_0]} h_2 d(\mu+\nu_{0;R_0})-\frac{\delta}{\epsilon'}\nonumber\\
&\geq&\mu([-R_0,x_0+t-\epsilon'])+\nu_{0;R_0}([-R_0,x_0+t-\epsilon'])-\frac{\delta}{\epsilon'}.
\end{eqnarray}
Hence noting that $\mu([-R_0,x_0+t-\epsilon'])=\phi_{\mu}(x_0+t-\epsilon')> \phi_{\mu}(x_0)-\epsilon$, we have
\begin{eqnarray}\label{neee2}
 && \frac{1}{k}(N((x_0+t)k^{2\slash 3})-N(-R_0 k^{2\slash 3})+1)-\frac{2}{3\pi}(x_0+t)_{+}^{3\slash 2}\nonumber\\
 &\geq&
  \nu_{k;R_0}([-R_0,x_0+t])
\geq\phi_{\mu}(x_0)-\epsilon-\frac{\delta}{\epsilon'}-\nu_{0;R_0}((x_0+t-\epsilon',x_0+t])\nonumber\\
    &\geq& \phi_{\mu}(x_0)-\epsilon-\frac{\delta}{\epsilon'}-\sqrt{R_0}\epsilon' = \phi_{\mu}(x_0)-\epsilon-2\sqrt{\delta}R_0^{1\slash 4}.
\end{eqnarray}
Let $\alpha(x_0):=\phi_{\mu}(x_0)-\epsilon-2\sqrt{\delta}R_0^{1\slash 4}$. Below we assume that $\epsilon,\delta$ are sufficiently small, so that $\alpha(x_0)>0$.

If $x_0<0$, for $t>0$ sufficiently small, we have $x_0+t<0$. Thus following the proof of Proposition \ref{PropAiry1} (for bounding $\mathbb{P}(\mathcal{C}_2)$), by Proposition \ref{Diffu} and Corollary \ref{Tails} (with $\beta=2$), we obtain that
\begin{eqnarray}\label{nnee4}
   && \mathbb{P}\left(N((x_0+t)k^{2\slash 3})\geq \alpha(x_0) k-1\right)\nonumber\\
   &\leq& \mathbb{P}(\text{the number of blow-ups of }p_{(x_0+t)k^{2\slash 3}}(x)\text{ on }[0,\infty)\text{ is at least }\alpha(x_0) k-1)\nonumber\\
   &\leq& C^{\alpha(x_0)k}\exp\left(-\frac{4}{3} |x_0+t|^{3\slash 2}k(\alpha(x_0)k-1)\right),
\end{eqnarray}
where we recall (\ref{Eqn16}). By (\ref{neee2}) and (\ref{nnee4}),
\begin{eqnarray}
   && \mathbb{P}(d_{R_0}(\nu_{k;R_0},\mu)\leq \delta)\leq\mathbb{P}\left(N((x_0+t)k^{2\slash 3})-N(-R_0 k^{2\slash 3})\geq \alpha(x_0)k-1\right)\nonumber\\
   &\leq& \mathbb{P}\left(N((x_0+t)k^{2\slash 3})\geq \alpha(x_0)k-1\right)\nonumber\\
   &\leq& C^{\alpha(x_0)k}\exp\left(-\frac{4}{3} |x_0+t|^{3\slash 2}k(\alpha(x_0)k-1)\right).
\end{eqnarray}
First taking $k\rightarrow\infty$, then taking $\delta \rightarrow 0^{+}$, then taking $t\rightarrow 0^{+}$, and finally taking $\epsilon\rightarrow 0^{+}$, we obtain that
\begin{equation}\label{x0}
\limsup\limits_{\delta\rightarrow 0^{+}}\limsup\limits_{k\rightarrow\infty}\frac{1}{k^2}\log{\mathbb{P}(d_{R_0}(\nu_{k;R_0},\mu)\leq\delta)}\leq   -\psi_{\mu}(x_0).
\end{equation}

Now if $x_0\in (0,R_0)$, similar to the proof of Proposition \ref{PropAiry1} (for bounding $\mathbb{P}(\mathcal{C}_2)$), by Proposition \ref{Diffu} and Corollary \ref{Tails} (with $\beta=2$), we have  
\begin{equation}\label{TW}
\mathbb{P}\left(N(-R_0k^{2\slash 3})\geq \alpha(x_0) k\slash 2-2\right)\leq C^{\alpha(x_0)k}\exp\left(-\frac{4}{3} R_0^{3\slash 2}k\left(\frac{\alpha(x_0)k}{2}-2\right)\right).
\end{equation}
By Proposition \ref{Airyop}, for sufficiently large $k$,
\begin{equation}
    \Big|N_0((x_0+t)k^{2\slash 3})-\frac{2}{3\pi}(x_0+t)^{3\slash 2}k\Big|\leq  1.
\end{equation}
By Proposition \ref{PropAiry2}, for sufficiently large $k$, we have
\begin{eqnarray}\label{NNNE1}
&&\mathbb{P}\bigg(|N((x_0+t)k^{2\slash 3})-N_0((x_0+t)k^{2\slash 3})|\geq \alpha(x_0) k\slash 2\bigg)\nonumber\\
&\leq& C\exp\bigg(-\frac{c}{8}\frac{\min\{\alpha(x_0)^2,(x_0+t)^3\slash 6\}}{\log(\max\{2,(x_0+t)^{3\slash 2}\slash \alpha(x_0)\})}k^2\bigg).
\end{eqnarray}
By (\ref{neee2}), (\ref{TW})-(\ref{NNNE1}), and the union bound, for sufficiently large $k$,  
\begin{eqnarray*}
&&\mathbb{P}(d_{R_0}(\nu_{k;R_0},\mu)
\leq\delta)\nonumber\\
&\leq&\mathbb{P}\left(N((x_0+t)k^{2\slash 3})-\frac{2}{3\pi}(x_0+t)^{3\slash 2}k-N(-R_0k^{2\slash 3})\geq \alpha(x_0)k-1\right)\nonumber\\
&\leq& C^{\alpha(x_0)k}\exp\left(-\frac{4}{3} R_0^{3\slash 2} k\left(\frac{\alpha(x_0)k}{2}-2\right)\right)\nonumber\\
&&+C\exp\bigg(-\frac{c}{8}\frac{\min\{\alpha(x_0)^2,(x_0+t)^3\slash 6\}}{\log(\max\{2,(x_0+t)^{3\slash 2}\slash \alpha(x_0)\})}k^2\bigg).
\end{eqnarray*}
Hence similarly as above, (\ref{x0}) also holds for $x_0\in (0,R_0)$. If $x_0=R_0$, replacing $x_0+t$ by $R_0$ in the above argument, noting (\ref{nnee1}) and $\phi_{\mu}(R_0)-\delta\geq \alpha(R_0)$, we conclude that (\ref{x0}) holds. Taking the infimum over $x_0\in [-R_0,R_0]$, we obtain
\begin{equation}\label{localLDP}
    \limsup\limits_{\delta\rightarrow 0^{+}}\limsup\limits_{k\rightarrow\infty}\frac{1}{k^2}\log{\mathbb{P}(d_{R_0}(\nu_{k;R_0},\mu)\leq\delta)}\leq   -I_2(\mu).
\end{equation}

We note that (\ref{localLDP}) establishes a local LDP upper bound for $(\nu_{k;R_0})_{k=1}^{\infty}$. From (\ref{localLDP}), following the proof of \cite[Theorem 4.1.11]{DZ}, we can deduce that (\ref{LDP_UBD}) holds for all compact sets $F\subseteq\mathcal{X}$. Combining this with the exponential tightness of $(\nu_{k;R_0})_{k=1}^{\infty}$ (Theorem \ref{Main1}, part (c)) established above, by \cite[Lemma 1.2.18]{DZ}, we conclude that (\ref{LDP_UBD}) holds for all closed sets $F\subseteq \mathcal{X}$.

\end{proof}

\section{Proof of Theorem \ref{KPZ1}}\label{Sect.3.5}

In this section, we provide the proof of Theorem \ref{KPZ1} based on Propositions \ref{PropAiry1} and \ref{PropAiry2}. The main structure of the proof is parallel to \cite[Section 5]{Cor}, so we only indicate the main changes in the proof. First we note the following corollary of an identity obtained in \cite{BG}, which is also the starting point of the study of the lower tail of the KPZ equation using the Airy point process in \cite{Cor}, \cite{CGKLT}, and \cite{Tsa}.

\begin{proposition}[Corollary of Theorem 2.1 of \cite{BG}]\label{M}
Let $a_1>a_2>\cdots$ denote the ordered points of the Airy point process. Then for any $T>0$,
\begin{equation}
    \mathbb{E}_{\mathrm{SHE}}\left[\exp\left(-\exp\Big(T^{1\slash 3}(\gamma_T+s)\Big)\right)\right]=\mathbb{E}_{\mathrm{Airy}}\left[\prod_{k=1}^{\infty}\frac{1}{1+\exp(T^{1\slash 3}(s+a_k))}\right],
\end{equation}
where $\mathbb{E}_{\mathrm{SHE}}$ means taking expectation with respect to the stochastic heat equation (\ref{stochastic_heat_eq}) with $Z(0,X)=\delta_{X=0}$, and $\mathbb{E}_{\mathrm{Airy}}$ means taking expectation with respect to the Airy point process.
\end{proposition}

By Proposition \ref{M}, in order to prove Theorem \ref{KPZ1}, it suffices to establish the following:

\begin{proposition}\label{KPZP}
Let $a_1>a_2>\cdots$ denote the ordered points of the Airy point process. For any $s,x\in\mathbb{R}$, we let 
\begin{equation*}
    \mathcal{I}_s(x):=(1+\exp(T^{1\slash 3}(s+x)))^{-1}.
\end{equation*}
Fix $\epsilon\in (0,1\slash 3)$ and $T_0>0$. Then there exist $S=S(\epsilon,T_0)>0,C=C(T_0)>0$, and $K=K(\epsilon, T_0)>0$, such that for any $s\geq S$ and $T\geq T_0$,
\begin{equation}
    \mathbb{E}\Big[\prod_{k=1}^{\infty}\mathcal{I}_s(a_k)\Big]
    \leq e^{-\frac{4(1-C\epsilon)}{15\pi}T^{1\slash 3}s^{5\slash 2}}+e^{-K s^3-\epsilon T^{1\slash 3}s}+e^{-\frac{(1-C\epsilon)}{12}s^3},
\end{equation}
\begin{equation}
    \mathbb{E}\Big[\prod_{k=1}^{\infty}\mathcal{I}_s(a_k)\Big]\geq e^{-\frac{4(1+C\epsilon)}{15\pi}T^{1\slash 3}s^{5\slash 2}}+e^{-\frac{(1+C\epsilon)}{12}s^3}.
\end{equation}
\end{proposition}

In order to prove Proposition \ref{KPZP}, we state and prove several lemmas. We recall from Section \ref{Sect.n1} that $\gamma_1<\gamma_2<\cdots$ are the eigenvalues of the Airy operator. For any $k\in\mathbb{N}_{+}$, we denote $\mathcal{D}_k:=\max\{-\gamma_k-a_k,0\}$, and take $\lambda_k:=-a_k$. Note that by Proposition \ref{Eigens}, $(\lambda_k)_{k=1}^{\infty}$ has the same distribution as the ordered eigenvalues of $H_2$.

We first establish the following refinement of \cite[Lemma 5.3]{Cor}. 

\begin{lemma}\label{KL}
Let $\theta_0:=\lceil 2s^{3\slash 2}\slash (3\pi) \rceil$, and fix $\epsilon\in (0,1\slash 3)$. There exist positive constants $S_0=S_0(\epsilon)$, $K_1=K_1(\epsilon)$ that only depend on $\epsilon$, such that the following holds for all $s\geq S_0$. Divide the interval $[-s,0)$ into $\lceil 2\epsilon^{-1} \rceil$ segments 
\begin{equation*}
    \mathcal{Q}_j:=\left[-\frac{j s}{\lceil 2\epsilon^{-1}\rceil},-\frac{(j-1)s}{\lceil 2\epsilon^{-1}\rceil}\right)
\end{equation*}
for $j=1,\cdots,\lceil 2\epsilon^{-1} \rceil$. Denote the left and right endpoints of $\mathcal{Q}_j$ by $p_j$ and $q_j$. Define $k_j:=\sup\{k\in\mathbb{N}_{+}: -\gamma_k\geq q_j\}$. Then
\begin{equation}
    \mathbb{P}(a_{k_j}\leq p_j)\leq \exp(-K_1 s^3), \quad \forall j=2,\cdots,\lceil 2\epsilon^{-1} \rceil,
\end{equation}
\begin{equation}
    \mathbb{P}\Big(\bigcup_{k=1}^{\theta_0} \{\mathcal{D}_k\geq 2\epsilon s\}\Big)\leq \exp(-K_1 s^3).
\end{equation}
\end{lemma}
\begin{proof}
For any $j\in\{2,\cdots,\lceil 2\epsilon^{-1} \rceil\}$, we let $\lambda_j:=-p_j$ and $\lambda_j':=-q_j$. We recall (\ref{Eq2.14}), and note that the event $a_{k_j}\leq p_j$ implies
\begin{equation*}
    N(\lambda_j)-N_0(\lambda_j)\leq N_0(\lambda'_j)-N_0(\lambda_j).
\end{equation*}
Note that by Proposition \ref{Airyop}, for any $i\in\mathbb{N}_{+}$,
\begin{equation*}
    \gamma_i=\left(\frac{3\pi}{2}\left(i-\frac{1}{4}+R(i)\right)\right)^{2\slash 3},
\end{equation*}
where $|R(i)|\leq C\slash i$ for some positive absolute constant $C$. Hence there exist $S_1, c>0$ (which only depend on $\epsilon$), such that when $s\geq S_1$, we have
\begin{equation*}
    N_0(\lambda_j')-N_0(\lambda_j)\leq  -cs^{3\slash 2}
\end{equation*}
for all $j=2,\cdots,\lceil 2\epsilon^{-1} \rceil$. Therefore, when $s\geq S_1$, 
\begin{equation}\label{KL1}
    \mathbb{P}(a_{k_j}\leq  p_j )\leq \mathbb{P}(|N(\lambda_j)-N_0(\lambda_j)|\geq cs^{3\slash 2}).
\end{equation}
By Proposition \ref{PropAiry2}, there exist $S_0\geq S_1 $ and $K_1>0$ (depending only on $\epsilon$), such that when $s\geq S_0$,
\begin{equation}\label{KE2}
    \mathbb{P}(|N(\lambda_j)-N_0(\lambda_j)|\geq cs^{3\slash 2})\leq \exp(-K_1 s^3)
\end{equation}
for all $j\in \{2,\cdots,\lceil 2\epsilon^{-1} \rceil\}$. The first conclusion of the lemma follows from (\ref{KL1}) and (\ref{KE2}). The second conclusion follows by the same argument as in the proof of \cite[Lemma 5.3]{Cor}.
\end{proof}

We then establish the following lemma, which improves upon \cite[Equation (5.28)]{Cor}.

\begin{lemma}\label{KL2}
Fix $\epsilon\in (0,1\slash 3)$ and $T_0>0$. Then there exist $C=C(T_0)>0$ and $S_0=S_0(\epsilon,T_0)>0$, such that for all $T\geq T_0$ and $s\geq S_0$,
\begin{equation}
    \mathbb{E}\bigg[\mathbbm{1}_{a_1< -s}\prod_{k=1}^{\infty}\mathcal{I}_s(a_k)\bigg]\geq \exp\Big(-\frac{1+C\epsilon}{12}s^3\Big).
\end{equation}
\end{lemma}
\begin{proof}
Below we fix $\delta=1\slash 3$. We divide the interval $(-\infty,-s)$ into three parts: $I_1=(-\infty,-s^{3+\delta}), I_2=[-s^{3+\delta},-2s), I_3=[-2s,-s)$. By Propositions \ref{PropAiry1}-\ref{PropAiry2}, there exist positive absolute constants $C,c_1,c_2,S_1$, such that for any $s\geq S_1$,
\begin{equation}
    \mathbb{P}(|N(2s)-N_0(2s)|\geq c_1 s^2)\leq C\exp(-c_2 s^{7\slash 2}),
\end{equation}
\begin{equation}
    \mathbb{P}(|N(s^{3+\delta})-N_0(s^{3+\delta})|\geq c_1 s^{3(3+\delta)\slash 2})\leq C\exp(-c_2   s^{6+2\delta}).
\end{equation}
Therefore, there exists an absolute constant $C'>0$, such that if we denote  
\begin{equation}
    \mathcal{A}_1:=\{N(2s)\leq C's^2, N(s^{3+\delta})\leq C' s^{3(3+\delta)\slash 2}\},
\end{equation}
then
\begin{equation}
    \mathbb{P}(\mathcal{A}_1^c)\leq 2C\exp(-c_2 s^{7\slash 2}).
\end{equation}

Below, we denote $\mathcal{J}_s(x):=\log(1+\exp(T^{1\slash 3}(s+x)))$ for any $s,x\in\mathbb{R}$. Note that when $x\in I_3$, $\mathcal{J}_s(x)\leq \log(2)$; when $x\in I_2$, $\mathcal{J}_s(x)\leq \log(1+\exp(-T^{1\slash 3}s))$. Hence there exist $S_2=S_2(T_0)>0$ and $C''=C''(T_0)>0$, such that when the event $\{a_1<-s\}\cap\mathcal{A}_1$ holds, we have for any $T\geq T_0, s\geq S_2$,
\begin{eqnarray*}
\sum_{k\in\mathbb{N}_{+}: a_k\in I_2\cup I_3} \mathcal{J}_s(a_k)&\leq& C'\log(2)s^2+C' s^{3(3+\delta)\slash 2}\log(1+\exp(-T^{1\slash 3}s))\\
&\leq& C'' s^2.
\end{eqnarray*}
By the same argument as in \cite[Section 5.2]{Cor}, there exists $C_2=C_2(T_0)>0$, such that if we denote
\begin{equation}
    \mathcal{A}_2:=\Big\{\sum_{k\in\mathbb{N}_{+}: a_k\in I_1}\mathcal{J}_s(a_k)\leq C_2 s^{5\slash 2}\Big\},
\end{equation}
then there exist $C_3=C_3(T_0)>0, c_3=c_3(T_0)>0, S_3=S_3(T_0)>0$, such that for all $s\geq S_3$,
\begin{equation}
    \mathbb{P}(\mathcal{A}_2^c)\leq  C_3\exp(-c_3 s^{3+\delta\slash 4}).
\end{equation}

Take $\mathcal{A}:=\mathcal{A}_1 \cap \mathcal{A}_2$. By the tail behavior of $a_1$ (see e.g. \cite[Theorem1.3]{RRV}) and the union bound, there exist $\tilde{C}=\tilde{C}(T_0)>0, \tilde{S}_0=\tilde{S}_0(\epsilon,T_0)>0$, such that for any $s\geq \tilde{S}_0$,
\begin{equation*}
    \mathbb{P}(\{a_1< -s\}\cap \mathcal{A})\geq \mathbb{P}(a_1< -s)+\mathbb{P}(\mathcal{A})-1\geq \exp\Big(-\frac{1+\tilde{C}\epsilon}{12}s^3\Big).
\end{equation*}
Note that when the event $\{a_1< -s\}\cap\mathcal{A}$ holds, $\sum_{k=1}^{\infty}\mathcal{J}_s(a_k)\leq C''s^2+C_2s^{5\slash 2}$. Hence there exist $C=C(T_0)>0$ and $S_0=S_0(\epsilon,T_0)>0$, such that when $T\geq T_0, s\geq S_0$,
\begin{eqnarray*}
    && \mathbb{E}\Big[\mathbbm{1}_{a_1<-s}\prod_{k=1}^{\infty}\mathcal{I}_s(a_k)\Big] \geq \mathbb{E}\Big[\mathbbm{1}_{\{a_1<-s\}\cap\mathcal{A}}\exp\Big(-\sum_{k=1}^{\infty}\mathcal{J}_s(a_k)\Big)\Big]\\
    &\geq& \exp(-(C''s^2+C_2s^{5\slash 2}))\mathbb{P}(\{a_1<-s\}\cap\mathcal{A}) \geq  \exp\Big(-\frac{1+C\epsilon}{12}s^3\Big).
\end{eqnarray*}
\end{proof}

Now we finish the proof of Proposition \ref{KPZP}.

\begin{proof}[Proof of Proposition \ref{KPZP}]
The proof of the upper bound part follows by combining Lemma \ref{KL} and the arguments in \cite[Section 5.1]{Cor} except Lemma 5.3. The proof of the lower bound part follows by combining Lemma \ref{KL2} and the arguments in \cite[Section 5.2]{Cor} before equation (5.28).
\end{proof}

\section{Approximation of the Airy point process by the Gaussian $\beta$-ensemble}\label{Sect.3}

In this section, we establish an approximation of the Airy point process (and more generally, the eigenvalues of the stochastic Airy operator $H_{\beta}$) using the Gaussian $\beta$-ensemble up to an exponentially small probability, thereby proving Theorem \ref{Main3}. We refer the reader to Section \ref{Sect.n} for relevant background materials, and assume the notations there throughout this section.

We briefly describe the proof strategy of Theorem \ref{Main3} as follows. We fix an arbitrary $e_0\in (0,1\slash 10000)$, and consider any $k,n\in  \mathbb{N}_{+}$ with $n^{e_0}\leq k\leq n^{1\slash 10000}$. We first show exponential decay of the top $k$ eigenfunctions (eigenvectors, respectively) of the stochastic Airy operator $H_{\beta}$ (the Gaussian $\beta$-ensemble $H_{\beta,n}$, respectively) up to an exponentially small probability. This is achieved by analyzing the Riccati transform $\bar{p}_j$ (defined in Section \ref{Sect.n1}) of the $j$th eigenfunction of the stochastic Airy operator for every $j\in [k]$ together with a discrete analogue. Using this information, we can couple $H_{\beta}$ and $H_{\beta,n}$ along their eigenfunctions\slash eigenvectors so that their eigenvalues are close up to an exponentially small probability. To achieve this, we further rely on certain a priori estimates that are based on the eigenvalue-eigenfunction (eigenvector) equations for $H_{\beta}$ and $H_{\beta,n}$. 

The rest of this section is devoted to the proof of Theorem \ref{Main3}.

\subsection{Exponential decay of the eigenfunctions of $H_{\beta}$}
\label{sec:3.1}

In this subsection, we show exponential decay of the top eigenfunctions of the stochastic Airy operator $H_{\beta}$. Recall from Section \ref{Sect.n1} that for each $j\in \mathbb{N}_{+}$, the $j$th eigenfunction of $H_{\beta}$ is denoted by $f_j$. The main result of this subsection is the following proposition.

\begin{proposition}\label{AiryDecay}
Assume that $k\in\mathbb{N}_{+}$. There exist positive constants $C,c,T_0$ that only depend on $\beta$, such that for any $T\geq T_0$, there exists an event $\mathcal{A}(T)\in \mathcal{F}$, such that $\mathbb{P}(\mathcal{A}(T)^c)\leq C\exp(-c(kT)^{1  \slash 4})$, and when the event $\mathcal{A}(T)$ holds, the following properties hold for every $j\in [k]$:
\begin{itemize}
    \item [(a)] For any $(x_1,x_2)\in\mathbb{R}^2$ such that $kT\leq x_1\leq x_2$, we have $f_j(x_1)f_j(x_2)>0$ and $|f_j(x_2)|\leq |f_j(x_1)|\exp(-(x_2^{3\slash 2}-x_1^{3 \slash 2 })\slash 3)$.
    \item [(b)] For any $x\geq kT$, we have $|f_j'(x)|\leq 4\sqrt{x}|f_j(x)|$.
\end{itemize}
\end{proposition}

Throughout the rest of this subsection, we fix an arbitrary $k\in\mathbb{N}_{+}$. We briefly describe the proof strategy of Proposition \ref{AiryDecay} as follows. Consider an arbitrary $j\in [k] $, and let $\bar{p}_j$ be the Riccati transform (defined in Section \ref{Sect.n1}) of $f_j$; then, we argue that up to an exponentially small probability, once $\bar{p}_j(x)$ deviates a bit upward from the curve $y=-\sqrt{x}\slash 2$, it will quickly move above the curve $y=4\sqrt{x}\slash 5$ and will be ``trapped'' above the curve $y=\sqrt{x}\slash 2$ forever. This leads to a contradiction, as $f_j$ has unit $L^2$ norm. Therefore, $\bar{p}_j(x)$ has to be ``trapped'' below the curve $y=-\sqrt{x}\slash 2$, which leads to the exponential decay result in Proposition \ref{AiryDecay}. The main challenge is to establish this result up to an \emph{exponentially small probability}, which requires a delicate analysis of the behavior of $\bar{p}_j(x)$. 

Throughout this subsection, we fix three positive numbers $K_1, K_2, L_1$ that are at least $10$. We define the following sets:
\begin{itemize}
    \item $\mathcal{D}_1:=\{(x,y)\in [0,\infty)\times (\mathbb{R}\cup \{-\infty\}):y\geq 4\sqrt{x}\slash 5\}$;
    \item $\mathcal{D}_2:=\{(x,y)\in [0,\infty)\times (\mathbb{R}\cup \{-\infty\}):y\geq \sqrt{x}\slash 2\}$;
    \item $\mathcal{D}_3:=\{(x,y)\in [0,\infty)\times (\mathbb{R}\cup \{-\infty\}):y\leq -\sqrt{x}\slash 2\}$;
    \item $\mathcal{D}_4:=\{(x,y)\in [0,\infty)\times (\mathbb{R}\cup \{-\infty\}):y\leq -4\sqrt{x}\slash 5\}$.
\end{itemize}
For every $j\in [k]$, we define the following events (recall from Section \ref{Sect.n1} that $\lambda_j$ is the $j$th eigenvalue of $H_{\beta}$):
\begin{itemize}
    \item $\mathcal{A}_1(j)$: the event that $-K_1\leq\lambda_j\leq K_2$;
    \item $\mathcal{A}_2(j)$: the event that there exists $x_1\geq L_1$, such that for all $x\geq x_1$, $(x,\bar{p}_j(x))\in\mathcal{D}_2$;
    \item $\mathcal{A}_3(j)$: the event that there exists $x_1\geq L_1$ such that $(x_1,\bar{p}_j(x_1))\in \mathcal{D}_1$;
    \item $\mathcal{A}_4(j)$: the event that there exists $x_1\geq L_1$ such that $(x_1,\bar{p}_j(x_1))\in \mathcal{D}_3^c$;
    \item $\mathcal{A}_5(j)$: the event that there exists $x_1\geq L_1$ such that $\bar{p}_j(x_1)\leq -4\sqrt{x_1}$.
\end{itemize}

Roughly speaking, the following lemma shows that for any $j\in [k]$, once $\bar{p}_j(x)$ hits the region $\mathcal{D}_1$, up to an event with exponentially small probability, it will never leave the region $\mathcal{D}_2$.

\begin{lemma}\label{AiryDecayL1}
Assume that $L_1>32K_2$. There exist an event $\mathcal{N}_1\in \mathcal{F}$ and positive constants $C,c$ that only depend on $\beta$, such that $\mathbb{P}(\mathcal{N}_1) \leq C\exp(-c L_1^{1\slash 4})$, and for every $j\in [k]$,
\begin{equation*}
\mathcal{A}_1(j)\cap\mathcal{A}_2(j)^c\cap\mathcal{A}_3(j)\subseteq \mathcal{N}_1, \text{ hence } \mathcal{A}_1(j)\cap\mathcal{A}_3(j)\subseteq \mathcal{N}_1.
\end{equation*}
\end{lemma}

\begin{proof}

We let $\Delta:=L_1^{1\slash 4}$.
The strategy of the proof is to break the interval $[L_1,\infty)$ into sub-intervals $I_i:=[L_1+i\Delta,L_1+(i+1)\Delta)$ for $i\in \mathbb{N}$ and control the behavior of $\bar{p}_j(x)$ (where $j\in [k]$) on these sub-intervals inductively.

For every $i\in \mathbb{N}$, we let
\begin{equation*}
    \psi_1(i):=\frac{1}{3}-\frac{1}{16}(i+1)^{-1\slash 4}, \quad \psi_2(i):=\frac{1}{3}-\frac{1}{8}(i+1)^{-1\slash 4};
\end{equation*}
\begin{equation*}
    M_i:=\frac{1}{16}(i+2)^{-1\slash 4}\sqrt{L_1+i\Delta}.
\end{equation*}
Note that $M_i\leq (\psi_1(i+1)-\psi_2(i))\sqrt{L_1+i\Delta}$. We also define the following events for every $i\in \mathbb{N}$: 
\begin{equation*}
    \mathcal{E}_{1,i}:=\Big\{\sup_{t\in [0,\Delta]}\left\{\left|B_{L_1+i\Delta+t}-B_{L_1+i\Delta}\right|\right\}>\frac{\sqrt{\beta}}{8}M_i\Big\},
\end{equation*}
\begin{eqnarray*}
    \mathcal{E}_{2,i}:=&&\Big\{\inf_{t\geq 0}\Big\{\frac{1}{32}(L_1+i\Delta)t+\frac{2}{\sqrt{\beta}}\left(B_{L_1+i\Delta+t}-B_{L_1+i\Delta}\right)\Big\}\\
    &&\quad \leq -\frac{1}{64}(i+2)^{-5\slash 4}\sqrt{L_1+i\Delta}\Big\}.
\end{eqnarray*}
Let $\mathcal{N}_1:=\big(\bigcup_{i=0}^{\infty}\mathcal{E}_{1,i} \big)\bigcup\big(\bigcup_{i=0}^{\infty} \mathcal{E}_{2,i}\big) $. Note that $\mathcal{N}_1\in\mathcal{F}$.

By the estimate (\ref{Eqn26}) for Brownian motion, for any $i\in\mathbb{N}$,
\begin{eqnarray*}
    \mathbb{P}(\mathcal{E}_{1,i})&\leq& 4\exp\Big(-\frac{\beta M_i^2}{128\Delta}\Big)\leq 4\exp(-c\Delta^{-1}\sqrt{L_1+i\Delta})\\
    &\leq & 4\exp(-c(L_1+i\Delta)^{1\slash 4}).
\end{eqnarray*}
Applying the Girsanov theorem to Brownian motion with drift, we obtain that for any $i\in\mathbb{N}$,
\begin{equation*}
    \mathbb{P}(\mathcal{E}_{2,i})\leq \exp(-c(i+2)^{-5\slash 4}(L_1+i\Delta)^{3\slash 2})\leq \exp(-c(L_1+i\Delta)^{1\slash 4}).
\end{equation*}
Hence by the union bound,
\begin{equation}\label{NE11}
\mathbb{P}(\mathcal{N}_1)\leq \sum_{i=0}^{\infty}\mathbb{P}(\mathcal{E}_{1,i})+\sum_{i=0}^{\infty} \mathbb{P}(\mathcal{E}_{2,i})
\leq C\exp(-cL_1^{1\slash 4}).
\end{equation}

In the following, we fix an arbitrary $j\in [k]$, and assume that the event $\mathcal{N}_1^c\cap\mathcal{A}_1(j)\cap\mathcal{A}_3(j)$ holds. As the event $\mathcal{A}_3(j)$ holds, there exist some $i_0\in \mathbb{N}$ and some $x_1\in I_{i_0}$ such that $\bar{p}_j(x_1)\geq 4\sqrt{x_1}\slash 5$. Below we show by induction that for any $i\geq i_0$ and any $x\in [\max\{x_1,L_1+i\Delta\},L_1+(i+1)\Delta]$, we have $\bar{p}_j(x)\geq (1-\psi_1(i+1))\sqrt{x}$.

\textbf{First consider the case where $i=i_0$.} Suppose that there exists some $x_2\in [x_1,L_1+(i_0+1)\Delta]$, such that $\bar{p}_j(x_2)\leq(1-\psi_1(i_0+1))\sqrt{x_2}$, and for any $x\in [x_1,x_2]$, $\bar{p}_j(x)\geq (1-\psi_1(i_0+1))\sqrt{x}$. Let
\begin{equation*}
    x_3:=\inf \{x\in[x_1,x_2]:\bar{p}_j(y)\leq (1-\psi_2(i_0))\sqrt{y}\text{ for all }y\in [x,x_2] \}.
\end{equation*}
As $\psi_1(i_0+1)\geq \psi_2(i_0)$, we have $\bar{p}_j(x_2)\leq (1-\psi_1(i_0+1))\sqrt{x_2}\leq (1-\psi_2(i_0))\sqrt{x_2}$, hence $x_3\in [x_1,x_2]$. As $\bar{p}_j(x_1)\geq 4\sqrt{x_1}\slash 5 >(1-\psi_2(i_0))\sqrt{x_1}$, by the continuity of $\bar{p}_j$ on $[x_1,x_2]$, we have $\bar{p}_j(x_3)=(1-\psi_2(i_0))\sqrt{x_3}$. For any $x\in \bar{I}_{i_0}$, let $r(x):=\bar{p}_j(x)-(1-\psi_1(i_0+1))\sqrt{x}$. By (\ref{pbar}), for any $x\in [x_3,x_2]$, we have 
\begin{equation}
    r(x)-r(x_3)=\int_{x_3}^x\Big(t-\lambda_j-\bar{p}_j(t)^2-\frac{1-\psi_1(i_0+1)}{2\sqrt{t}}\Big)dt +\frac{2}{\sqrt{\beta}}(B_x-B_{x_3}).
\end{equation}
As the event $\mathcal{A}_1(j)$ holds, we have $\lambda_j\in [-K_1,K_2]$. For any $x\in [x_3,x_2]$, by the definition of $x_3$, $(1-\psi_1(i_0+1))\sqrt{x}\leq \bar{p}_j(x)\leq (1-\psi_2(i_0))\sqrt{x}$, hence $x-\bar{p}_j(x)^2\geq 0.36 x$; as $L_1\geq 32K_2\geq 32$, we have
\begin{equation*}
    (1-\psi_1(i_0+1))\slash(2\sqrt{x})\leq 1\slash (2\sqrt{L_1})\leq L_1\slash 300\leq x\slash 300.
\end{equation*}
Hence for any $x\in [x_3,x_2]$, using $K_2\leq L_1\slash 32\leq (L_1+i_0\Delta)\slash 32$, we have
\begin{eqnarray*}
    &&x-\lambda_j-\bar{p}_j(x)^2-(1-\psi_1(i_0+1))\slash(2\sqrt{x})
    \geq 0.35 x-\lambda_j\\
    &\geq& 0.35(L_1+i_0\Delta)-K_2
    \geq (0.35-1\slash 32)(L_1+i_0\Delta)\geq \frac{1}{16}(L_1+i_0\Delta).
\end{eqnarray*}
Hence for $x\in [x_3,x_2]$, $r(x)$ is lower bounded by $w(x)$, where $w(x)$ is defined by
\begin{equation*}
    w(x)=r(x_3)+\frac{1}{16}(L_1+i_0\Delta)(x-x_3)+\frac{2}{\sqrt{\beta}}(B_x-B_{x_3})
\end{equation*}
for any $x\geq x_3$. As the event $(\mathcal{E}_{1,i_0})^c$ holds, we have
\begin{eqnarray}\label{NEQ1.1}
 r(x_2)-r(x_3)&\geq& w(x_2)-r(x_3)
\geq  \frac{2}{\sqrt{\beta}}(B_{x_2}-B_{x_3})\nonumber\\
&\geq& -\frac{2}{\sqrt{\beta}}(|B_{x_2}-B_{L_1+i_0\Delta}|+|B_{x_3}-B_{L_1+i_0\Delta}|)\nonumber\\
&\geq& -\frac{4}{\sqrt{\beta}}\sup_{t\in [0,\Delta]}|B_{L_1+i_0\Delta+t}-B_{L_1+i_0\Delta}|\geq -\frac{1}{2}M_{i_0}\nonumber\\
&\geq& -\frac{1}{2}(\psi_1(i_0+1)-\psi_2(i_0))\sqrt{L_1+i_0\Delta}.
\end{eqnarray}
As $\bar{p}_j(x_2)\leq (1-\psi_1(i_0+1))\sqrt{x_2}$ and $\bar{p}_j(x_3)=(1-\psi_2(i_0))\sqrt{x_3}$, we have
\begin{equation}\label{NEQ2.2}
r(x_2)-r(x_3)\leq -(\psi_1(i_0+1)-\psi_2(i_0))\sqrt{x_3}\leq -(\psi_1(i_0+1)-\psi_2(i_0))\sqrt{L_1+i_0\Delta}.
\end{equation}
By (\ref{NEQ1.1}) and (\ref{NEQ2.2}), we have $\psi_1(i_0+1)-\psi_2(i_0)\leq 0$, which leads to a contradiction. Therefore, for any $x\in [x_1,L_1+(i_0+1)\Delta]$, we have $\bar{p}_j(x)\geq (1-\psi_1(i_0+1))\sqrt{x}$.

\textbf{Now consider the case where $i\geq i_0+1$.} By the induction hypothesis,
\begin{equation}\label{NEQ6}
    \bar{p}_j(L_1+i\Delta)\geq (1-\psi_1(i))\sqrt{L_1+i\Delta}.
\end{equation}
If there exists $x_2\in [L_1+i\Delta, L_1+(i+1)\Delta]$ such that $\bar{p}_j(x_2)\geq (1-\psi_2(i))\sqrt{x_2}$ and $\bar{p}_j(x)\geq (1-\psi_1(i+1))\sqrt{x}$ for all $x\in[L_1+i\Delta,x_2]$,
then since the event $(\mathcal{E}_{1,i})^c$ holds, arguing as in the $i=i_0$ case and using $\bar{p}_j(x_2)\geq (1-\psi_2(i))\sqrt{x_2}$, we obtain that $\bar{p}_j(x)\geq (1-\psi_1(i+1))\sqrt{x}$ for all $x\in [x_2,L_1+(i+1)\Delta]$ (hence for all $x\in \bar{I}_i$). 

Now suppose that there exists $x_2\in \bar{I}_i$ such that $\bar{p}_j(x_2)\leq (1-\psi_1(i+1))\sqrt{x_2}$ and $(1-\psi_1(i+1))\sqrt{x}\leq \bar{p}_j(x)\leq (1-\psi_2(i))\sqrt{x}$ for all $x\in [L_1+i\Delta,x_2]$. For any $x\in \bar{I}_i$, let $r(x):=\bar{p}_j(x)-(1-\psi_1(i+1))\sqrt{x}$. Arguing as in the $i=i_0$ case, we obtain that for any $x\in [L_1+i\Delta,x_2]$, 
\begin{equation}\label{NEQ3}
    r(x)\geq r(L_1+i\Delta)+\frac{1}{32}(L_1+i\Delta)(x-(L_1+i\Delta))+\frac{2}{\sqrt{\beta}}(B_x-B_{L_1+i\Delta}).
\end{equation}
For any $t\geq 0$, let $\tilde{B}_t:=(L_1+i\Delta)t\slash 32+(2\slash \sqrt{\beta})(B_{L_1+i\Delta+t}-B_{L_1+i\Delta})$. By (\ref{NEQ3}),
\begin{equation}\label{NEQ4}
    r(x_2)-r(L_1+i\Delta)\geq \tilde{B}_{x_2-(L_1+i\Delta)}. 
\end{equation}
By (\ref{NEQ6}), as $\bar{p}_j(x_2)\leq (1-\psi_1(i+1))\sqrt{x_2}$, we have
\begin{equation*}
    r(x_2)-r(L_1+i\Delta)\leq -(\psi_1(i+1)-\psi_1(i))\sqrt{L_1+i\Delta},
\end{equation*}
which combined with (\ref{NEQ4}) gives
\begin{equation*}
    \inf_{t\in [0,\Delta]}\tilde{B}_t\leq -(\psi_1(i+1)-\psi_1(i))\sqrt{L_1+i\Delta}\leq -\frac{1}{64}(i+2)^{-5\slash 4}\sqrt{L_1+i\Delta}.
\end{equation*}
This contradicts the fact that the event $(\mathcal{E}_{2,i})^c$ holds. Therefore, for all $x\in \bar{I}_i$, we have $\bar{p}_j(x)\geq (1-\psi_1(i+1))\sqrt{x}$.

By induction, we conclude that for any $i\geq i_0$ and any $x\in \bar{I}_i\cap [x_1,\infty)$, $\bar{p}_j(x)\geq (1-\psi_1(i+1))\sqrt{x}$. In particular, $\bar{p}_j(x)\geq \sqrt{x}\slash 2$ for any $x\geq x_1$. Hence $\mathcal{N}_1^c\cap\mathcal{A}_1(j)\cap\mathcal{A}_3(j)\subseteq \mathcal{A}_2(j)$, which leads to 
\begin{equation}\label{NNEQ1.1}
\mathcal{A}_1(j)\cap\mathcal{A}_2(j)^c\cap\mathcal{A}_3(j)\subseteq \mathcal{N}_1.
\end{equation}

Now suppose that the event $\mathcal{A}_2(j)$ holds. Then there exists $x_1'\geq L_1$, such that for any $x\geq x_1'$, $f_j'(x)\slash f_j(x)=\bar{p}_j(x)\geq \sqrt{x}\slash 2$. Without loss of generality, we assume that $f_j(x_1')>0$. Hence for any $x\geq x_1'$, we have
\begin{equation*}
    f_j(x)\geq f_j(x_1')\exp((x^{3\slash 2}-(x_1')^{3\slash 2})\slash 3).
\end{equation*}
This contradicts the fact that $\int_{0}^{\infty}f_j(x)^2 dx=1$. Hence $\mathcal{A}_2(j)=\emptyset$. By (\ref{NNEQ1.1}), we conclude that $\mathcal{A}_1(j)\cap\mathcal{A}_3(j)\subseteq \mathcal{N}_1$.

\end{proof}

The next lemma roughly says that if $\bar{p}_j(x)$ deviates a bit upward from the curve $y=-\sqrt{x}\slash 2$, it will hit the region $\mathcal{D}_1$ up to an event with exponentially small probability.

\begin{lemma}\label{AiryDecayL2}
Assume that $L_1\geq \max\{10K_1,100K_2,10^8   \beta^{-1}\}$. There exist an event $\mathcal{N}_2 \in \mathcal{F}$ and positive constants $C,c$ that only depend on $\beta$, such that $\mathbb{P}(\mathcal{N}_2)\leq C\exp(-c L_1^{1\slash 2})$, and for every $j\in [k]$,
\begin{equation}\label{Step12}
    \mathcal{A}_1(j)\cap \mathcal{A}_3(j)^c\cap \mathcal{A}_4(j)\subseteq \mathcal{N}_2, \quad \mathcal{A}_1(j)\cap\mathcal{A}_3(j)^c\cap\mathcal{A}_5(j)\subseteq \mathcal{N}_2.
\end{equation}
\end{lemma}
\begin{proof}

We take $\Delta:=L_1^{1\slash 4}$. For each $i\in \mathbb{N}$, we let $I_i:=[L_1+i\Delta,L_1+(i+1)\Delta)$ and $n_i:=\lceil \Delta\sqrt{L_1+i\Delta}\rceil$; as $L_1\geq 10$, we have $n_i\geq 5$. For any $i\in \mathbb{N}$ and $l\in [n_i]$, we construct $I_{i,l}$ as follows: 
\begin{itemize}
    \item [(a)] If $\Delta\sqrt{L_1+i\Delta}\in\mathbb{N}_{+}$, for every $l\in [n_i]$, we let
    \begin{equation*}
        I_{i,l}:=\Big[L_1+i\Delta+\frac{l-1}{\sqrt{L_1+i\Delta}}, L_1+i\Delta+\frac{l}{\sqrt{L_1+i\Delta}}\Big).
    \end{equation*}
    \item [(b)] If $\Delta\sqrt{L_1+i\Delta}\notin\mathbb{N}_{+}$, for every $l\in [1,n_i-2]\cap\mathbb{N}_{+}$, we let
    \begin{equation*}
        I_{i,l}:=\Big[L_1+i\Delta+\frac{l-1}{\sqrt{L_1+i\Delta}}, L_1+i\Delta+\frac{l}{\sqrt{L_1+i\Delta}}\Big);
    \end{equation*}
    we also let
    \begin{equation*}
        I_{i,n_i-1}:=\Big[L_1+i\Delta+\frac{n_i-2}{\sqrt{L_1+i\Delta}},L_1+\Big(i+\frac{1}{2}\Big)\Delta+\frac{n_i-2}{2\sqrt{L_1+i\Delta}}\Big),
    \end{equation*}
    \begin{equation*}
        I_{i,n_i}:=\Big[L_1+\Big(i+\frac{1}{2}\Big)\Delta+\frac{n_i-2}{2\sqrt{L_1+i\Delta}},L_1+(i+1)\Delta\Big).
    \end{equation*}
    Note that from $|I_{i,n_i-1}|=|I_{i,n_i}|=\Delta\slash 2-(n_i-2)\slash(2\sqrt{L_1+i\Delta})$ and $\Delta\sqrt{L_1+i\Delta}\leq n_i\leq \Delta\sqrt{L_1+i\Delta}+1$, we can deduce that\\
    $(L_1+i\Delta)^{-1\slash 2}\slash 2\leq |I_{i,n_i-1}|=|I_{i,n_i}|\leq (L_1+i\Delta)^{-1\slash 2}$.
\end{itemize}
Hereafter, we denote the left and right endpoints of $I_{i,l}$ by $a_{i,l}$ and $b_{i,l}$, respectively.

Note that for any $i\in \mathbb{N}$, $I_{i,1},\cdots,I_{i,n_i}$ are pairwise disjoint, and 
\begin{equation*}
    I_i=\bigcup_{l=1}^{n_i}I_{i,l}.
\end{equation*}
Moreover, for any $i\in\mathbb{N}$ and $l\in [n_i]$,
\begin{equation}\label{iil}
\frac{1}{2} (L_1+i\Delta)^{-1\slash 2} \leq |I_{i,l}|\leq (L_1+i\Delta)^{-1\slash 2}.
\end{equation}

For any $i\in\mathbb{N}$ and $l\in [n_i]$, we denote by $\mathcal{C}_{i,l}$ the event that
\begin{equation*}
    \sup_{t\in \bar{I}_{i,l}}|B_t-B_{a_{i,l}}|\geq 1.
\end{equation*}
By (\ref{Eqn26}), we have 
\begin{equation*}
    \mathbb{P}(\mathcal{C}_{i,l})\leq 4\exp(-\sqrt{L_1+i\Delta}\big\slash 2).
\end{equation*}
Letting $\mathcal{N}_2:=\bigcup_{i=0}^{\infty}\bigcup_{l=1}^{n_i}\mathcal{C}_{i,l}\in\mathcal{F}$, by the union bound, we have
\begin{equation}\label{NEQ18}
    \mathbb{P}(\mathcal{N}_2)\leq \sum_{i=0}^{\infty}\sum_{l=1}^{n_i}\mathbb{P}(\mathcal{C}_{i,l})\leq C\exp(-c L_1^{1\slash 2}).
\end{equation}

In the following, we fix an arbitrary $j\in [k]$, and show (\ref{Step12}) in \textbf{Steps 1 and 2}.

\paragraph{Step 1}

In the following, we show that $\mathcal{N}_2^c\cap \mathcal{A}_1(j)\cap \mathcal{A}_3(j)^c\cap\mathcal{A}_4(j)=\emptyset$ by contradiction. Suppose that the event $\mathcal{N}_2^c\cap \mathcal{A}_1(j)\cap \mathcal{A}_3(j)^c\cap\mathcal{A}_4(j)$ holds. Then there exists $x_1\geq L_1$ such that $\bar{p}_j(x_1)\geq -\sqrt{x_1}\slash 2$. Suppose that $x_1\in I_{i,l}$, where $i\in\mathbb{N}$ and $l\in [n_i]$.

Below we show that $(x,\bar{p}_j(x))\notin \mathcal{D}_4$ for any $x\in [x_1,x_1+16\slash\sqrt{L_1+i\Delta}]$. Suppose there exists $x_2\in [x_1,x_1+16\slash\sqrt{L_1+i\Delta}]$ such that $\bar{p}_j(x_2)=-4\sqrt{x_2}\slash 5$ and $\bar{p}_j(x)\geq -4\sqrt{x}\slash 5$ for all $x\in [x_1,x_2]$. Then there exists $x_3\in [x_1,x_2]$ such that $\bar{p}_j(x_3)\geq -\sqrt{x_3}\slash 2$ and $\bar{p}_j(x)\in\left[-4\sqrt{x}\slash 5,-\sqrt{x}\slash 2\right]$ for all $x\in [x_3,x_2]$. As the event $\mathcal{A}_1(j)$ holds, we have $\lambda_j\in [-K_1,K_2]$. By the assumptions of the lemma, $K_2\leq L_1\slash 100$. Hence when $x\geq L_1+i\Delta$ and $\bar{p}_j(x)\in [-4\sqrt{x}\slash 5,4\sqrt{x}\slash 5]$, we have
\begin{equation}\label{Eqn47}
    x-\lambda_j-\bar{p}_j(x)^2\geq x-K_2-\frac{16}{25}x\geq \frac{9}{25}(L_1+i\Delta)-\frac{1}{100}L_1\geq \frac{1}{4}(L_1+i\Delta).
\end{equation}
By (\ref{pbar}) and (\ref{Eqn47}), we have
\begin{eqnarray*}
    \bar{p}_j(x_2)-\bar{p}_j(x_3)&=&\int_{x_3}^{x_2}(t-\lambda_j-\bar{p}_j(t)^2)dt+\frac{2}{\sqrt{\beta}}(B_{x_2}-B_{x_3})\nonumber \\
    &\geq&  \frac{1}{4}(L_1+i\Delta)(x_2-x_3)+\frac{2}{\sqrt{\beta}}(B_{x_2}-B_{x_3}).
\end{eqnarray*}
By the definitions of $x_2$ and $x_3$, we have
\begin{eqnarray*}
&& \bar{p}_j(x_2)-\bar{p}_j(x_3)\leq \frac{1}{2}\sqrt{x_3}-\frac{4}{5}\sqrt{x_2}\leq -\frac{3}{10}\sqrt{x_3}\leq -\frac{3}{10}\sqrt{L_1+i\Delta}.
\end{eqnarray*}
By the above two inequalities and the assumption that $L_1\geq 10^8\beta^{-1}$, we have
\begin{equation}\label{NEQ15}
    B_{x_3}-B_{x_2}\geq \frac{\sqrt{\beta}}{10}\sqrt{L_1+i\Delta}\geq  1000.
\end{equation}
Suppose that $x_2\in I_{i',l'}$, where $i'\in\mathbb{N}$ and $l'\in [n_{i'}]$. As $L_1\geq 100K_2\geq 1000$, we have $16\slash \sqrt{L_1+i\Delta}\leq 16L_1^{-1\slash 2}\leq \Delta\slash 2$, hence $i'\in \{i,i+1\}$. As
\begin{equation}\label{NEQ17}
    \sqrt{L_1+(i+1)\Delta}\slash \sqrt{L_1+i\Delta}\leq \sqrt{1+\Delta L_1^{-1}}\leq 101\slash 100,
\end{equation}
by (\ref{iil}) and the fact that $0\leq x_2-x_1\leq 16\slash\sqrt{L_1+i\Delta}$, we have
\begin{eqnarray*}
  &&  (i'-i)n_i+l'-l-1 \leq \frac{x_2-x_1}{(L_1+(i+1)\Delta)^{-1\slash 2}\slash 2}\\
  &\leq & \frac{16\slash\sqrt{L_1+i\Delta}}{(L_1+(i+1)\Delta)^{-1\slash 2}\slash 2}=\frac{32\sqrt{L_1+(i+1)\Delta}}{\sqrt{L_1+i\Delta}}\leq 33.
\end{eqnarray*}
As the event $\mathcal{N}_2^c$ holds, we have $|B_{x_2}-B_{a_{i,l}}|\leq (i'-i)n_i+l'-l+1\leq 35$. Similarly, we have $|B_{x_3}-B_{a_{i,l}}|\leq 35$. Hence $|B_{x_2}-B_{x_3}|\leq 70$, which contradicts (\ref{NEQ15}). Therefore, for any $x\in [x_1,x_1+16\slash \sqrt{L_1+i\Delta}]$, $(x,\bar{p}_j(x))\notin \mathcal{D}_4$.

As the event $\mathcal{A}_3(j)^c$ holds, for any $x\in [x_1,x_1+16\slash \sqrt{L_1+i\Delta}]$, we have $(x,\bar{p}_j(x))\notin \mathcal{D}_1$. Let $x_4:=x_1+16\slash\sqrt{L_1+i\Delta}$. Arguing as in the previous paragraph, we can deduce that $x_4\in I_{i}\cup I_{i+1}$ and $|B_{x_4}-B_{x_1}|\leq 70$. By (\ref{pbar}) and (\ref{Eqn47}), noting that $\bar{p}_j(x)\in [-4\sqrt{x}\slash 5,4\sqrt{x}\slash 5]$ for any $x\in [x_1,x_4]$, we have
\begin{eqnarray*}
 \bar{p}_j(x_4)&\geq& \bar{p}_j(x_1)+\frac{1}{4}(L_1+i\Delta)(x_4-x_1)+\frac{2}{\sqrt{\beta}}(B_{x_4}-B_{x_1})\\
&\geq& -\frac{1}{2}\sqrt{L_1+(i+1)\Delta}+4\sqrt{L_1+i\Delta}-\frac{140}{\sqrt{\beta}}\nonumber\\
&\geq& 2 \sqrt{L_1+(i+1)\Delta}\geq \sqrt{L_1+(i+2)\Delta}\geq \sqrt{x_4},
\end{eqnarray*}
where we use (\ref{NEQ17}), $L_1\geq 10^8\beta^{-1}$, and
\begin{equation*}
    \sqrt{L_1+(i+2)\Delta}\slash \sqrt{L_1+(i+1)\Delta}\leq \sqrt{1+\Delta L_1^{-1}}\leq 101\slash 100
\end{equation*}
in the last line. This contradicts the fact that $(x,\bar{p}_j(x))\notin\mathcal{D}_1$ for any $x\in [x_1,x_4]$. Hence we have
\begin{equation*}
    \mathcal{N}_2^c\cap \mathcal{A}_1(j)\cap \mathcal{A}_3(j)^c\cap\mathcal{A}_4(j)=\emptyset,\text{ thus } \mathcal{A}_1(j)\cap \mathcal{A}_3(j)^c\cap\mathcal{A}_4(j)\subseteq \mathcal{N}_2.
\end{equation*}

\paragraph{Step 2}

In the following, we show that $\mathcal{A}_1(j)\cap\mathcal{A}_5(j)\cap \mathcal{N}_2^c\subseteq \mathcal{A}_3(j)$. 

Below we assume that the event $\mathcal{A}_1(j)\cap\mathcal{A}_5(j)\cap \mathcal{N}_2^c$ holds. Then there exists $x_1\geq L_1$, such that $\bar{p}_j(x_1)\leq -4\sqrt{x_1}$. If $\bar{p}_j(x_1)=-\infty$, then $\bar{p}_j(x)$ restarts at $+\infty$ immediately after $x_1$, hence the event $\mathcal{A}_3(j)$ holds. Below we assume that $\bar{p}_j(x_1)\neq -\infty$ and $x_1\in I_i=[L_1+i\Delta,L_1+(i+1)\Delta)$ (where $i\in\mathbb{N}$). For any $x\geq x_1$, we let 
\begin{equation*}
    r(x):=\bar{p}_j(x)+2\sqrt{x}-\frac{2}{\sqrt{\beta}}(B_x-B_{x_1}).
\end{equation*}
Note that $r(x_1)\leq -2\sqrt{x_1}$. Let $\tau:=\inf\{x\geq x_1: \bar{p}_j(x)=-\infty\}$. Note that $\tau=\inf\{x\geq x_1: r(x)=-\infty\}$. By (\ref{pbar}), for any $x\in [x_1,\tau)$,
\begin{equation}\label{NEQ3.1}
    r'(x)= x-\lambda_j-\bigg(r(x)-2\sqrt{x}+\frac{2}{\sqrt{\beta}}(B_x-B_{x_1})\bigg)^2 +x^{-1\slash 2}.
\end{equation}
If $\tau<\infty$, then $\bar{p}_j(x)$ blows up to $-\infty$ at $x=\tau$ and restarts immediately at $+\infty$; as $\tau\geq x_1\geq L_1$, the event $\mathcal{A}_3(j)$ holds. In the following, we assume that $\tau=\infty$. 

For any $x\geq x_1$, we let $q(x):=\bar{p}_j(x)+2\sqrt{x}=r(x)+(2\slash\sqrt{\beta})(B_x-B_{x_1})$. Let $I':=[x_1,x_1+16\slash \sqrt{L_1+i\Delta}]$. If there exists some $x\in I'$ such that $q(x)=0$, we let $x_2:=\inf\{x\in I': q(x)=0\}$; otherwise we let $x_2:=x_1+16\slash \sqrt{L_1+i\Delta}$. For any $x\in I'$, as the event $\mathcal{N}_2^c$ holds, arguing as in \textbf{Step 1}, we obtain that $|B_x-B_{x_1}|\leq 70$. As the event $\mathcal{A}_1(j)$ holds, $\lambda_j\geq -K_1$. For any $x\in [x_1,x_2]$, $q(x)\leq 0$. By the inequality $(a+b)^2\geq a^2\slash 2-b^2,\forall a,b\in\mathbb{R}$, for any $x\in [x_1,x_2]$, we have
\begin{equation*}
    q(x)^2\geq r(x)^2\slash 2-4\beta^{-1}(B_x-B_{x_1})^2 \geq r(x)^2\slash 2-2\cdot 10^4 \beta^{-1},
\end{equation*}
which leads to
\begin{eqnarray}\label{NEQ3.2}
 &&   x-\lambda_j-\bigg(r(x)-2\sqrt{x}+\frac{2}{\sqrt{\beta}}(B_x-B_{x_1})\bigg)^2 +x^{-1\slash 2} \nonumber\\
 &=& x-\lambda_j-(q(x)-2\sqrt{x})^2+x^{-1\slash 2}\nonumber\\
 &=&-\lambda_j-q(x)^2+4\sqrt{x}q(x)-3x+x^{-1\slash 2}\nonumber\\
 &\leq& K_1-r(x)^2\slash 2+2\cdot 10^4\beta^{-1}-3x+x^{-1\slash 2}\nonumber\\
 &\leq&-r(x)^2\slash 2-2x  < -(r(x)^2+L_1+i\Delta)\slash 2,
\end{eqnarray}
where we use the fact that $x\geq L_1\geq \max\{10 K_1,10^8\beta^{-1}\}$ (note that $L_1\geq 10$) in the fourth line and the fact that $x\geq L_1+i\Delta$ in the fifth line. Let $\tilde{r}(x), x\geq x_1$ be defined by $\tilde{r}(x_1)=r(x_1)$ and 
\begin{equation}\label{NEQ3.3}
    \tilde{r}'(x)=-(\tilde{r}(x)^2+L_1+i\Delta)\slash 2, \quad  \text{for any } x\geq x_1.
\end{equation}
By (\ref{NEQ3.1}) and (\ref{NEQ3.2}), $r(x)$ is dominated by $\tilde{r}(x)$ for $x\in [x_1,x_2]$.  

Suppose that $q(x_2)=0$ and $\tilde{r}(x)$ does not blow up to $-\infty$ for $x\in [x_1,x_2]$. Then $\tilde{r}(x_2)\leq \tilde{r}(x_1)=r(x_1)\leq -2\sqrt{L_1+i\Delta}$. Note that $r(x)$ is dominated by 
$\tilde{r}(x)$ on $[x_1,x_2]$, $q(x_2)=0$, and $|B_{x_2}-B_{x_1}|\leq 70$. Combining these with the fact that $L_1\geq 10^8\beta^{-1}$, we obtain that  
\begin{equation*}
    \tilde{r}(x_2)\geq r(x_2)= -\frac{2}{\sqrt{\beta}}(B_{x_2}-B_{x_1})\geq -\frac{140}{\sqrt{\beta}}> -2\sqrt{L_1+i\Delta},
\end{equation*}
which leads to a contradiction. Therefore, either $\tilde{r}(x)$ blows up to $-\infty$ for some $x\in [x_1,x_2]$ or $q(x_2)<0$. For the former case, $r(x)$ also blows up to $-\infty$ for some $x\in [x_1,x_2]$. For the latter case, $x_2=x_1+16\slash\sqrt{L_1+i\Delta}$. Solving the Riccati equation (\ref{NEQ3.3}), we obtain that for $x\geq x_1$,
\begin{equation*}
    \tilde{r}(x)=\sqrt{L_1+i\Delta}\tan\left(\arctan\left(\frac{r(x_1)}{\sqrt{L_1+i\Delta}}\right)-\frac{1}{2}\sqrt{L_1+i\Delta}(x-x_1) \right),
\end{equation*}
which blows up for some $x\in [x_1,x_2]$. Hence $r(x)$ also blows up for some $x\in [x_1,x_2]$. Therefore, we conclude that $r(x)$ blows up to $-\infty$ for some $x\in [x_1,x_2]$, which implies that $\bar{p}_j(x)$ blows up to $-\infty$ for some $x\in [x_1,x_2]$. This contradicts our assumption that $\tau=\infty$.

Thus we conclude that $\tau<\infty$ and the event $\mathcal{A}_3(j)$ holds. Therefore,
\begin{equation*}
    \mathcal{A}_1(j)\cap\mathcal{A}_5(j)\cap \mathcal{N}_2^c\subseteq \mathcal{A}_3(j), \text{ thus } \mathcal{A}_1(j)\cap \mathcal{A}_3(j)^c\cap\mathcal{A}_5(j)\subseteq \mathcal{N}_2.
\end{equation*}

\end{proof}

\begin{lemma}\label{AiryDecayL3}
There exist positive constants $T_0,C,c$ that only depend on $\beta$, such that when $T\geq T_0$, for any $j\in [k]$,
\begin{equation}
    \mathbb{P}(\lambda_j>kT)\leq C\exp(-c(kT)^3).
\end{equation}
Moreover, there exist positive constants $C,c$ that only depend on $\beta$, such that for any $j\in [k]$ and $x\geq 1$,
\begin{equation}
    \mathbb{P}(\lambda_j\leq -x)\leq C\exp(-c x^{3\slash 2}).
\end{equation}
\end{lemma}
\begin{proof}
Recall the definitions of $N(\lambda)$ and $N_0(\lambda)$ in (\ref{Eq2.14}). From Proposition \ref{Airyop}, we can deduce that for $T$ sufficiently large,
\begin{equation*}
    \Big|N_0(k T)-\frac{2}{3\pi}(kT)^{3\slash 2}\Big|\leq 2.
\end{equation*}
Hence by Proposition \ref{PropAiry2}, there exists a positive constant $T_0$ that only depends on $\beta$, such that when $T\geq T_0$, for any $j\in [k]$,
\begin{eqnarray*}
&& \mathbb{P}(\lambda_j>kT)\leq \mathbb{P}(\lambda_k>kT)=\mathbb{P}(N(kT)<k) \\
& \leq & \mathbb{P}(|N(kT)-N_0(kT)|\geq (3\pi)^{-1}(kT)^{3\slash 2})
\leq C\exp(-c (kT)^3).
\end{eqnarray*}
Moreover, by Proposition \ref{TailBound}, for any $j\in [k]$ and $x\geq 1$, 
\begin{equation}
    \mathbb{P}(\lambda_j\leq -x)\leq \mathbb{P}(\lambda_1\leq -x)\leq C\exp(-c x^{3\slash 2}).
\end{equation}
\end{proof}

Now we finish the proof of Proposition \ref{AiryDecay}.

\begin{proof}[Proof of Proposition \ref{AiryDecay}]

We take $K_1=K_2=kT\slash 100$ and $L_1=kT$. By Lemma \ref{AiryDecayL3}, there exists a positive constant $T_1$ that only depends on $\beta$, such that when $T\geq T_1$, for any $j\in [k]$,
\begin{equation}\label{NEQ5.1}
    \mathbb{P}(\mathcal{A}_1(j)^c)=\mathbb{P}(\lambda_j<-kT\slash 100)+\mathbb{P}(\lambda_j>kT\slash 100)\leq C\exp(-c(kT)^{3\slash 2}).
\end{equation}

Let $\mathcal{N}_1,\mathcal{N}_2\in\mathcal{F}$ be defined as in Lemmas \ref{AiryDecayL1} and \ref{AiryDecayL2}, respectively. For any $j\in [k]$, we define $\mathcal{A}_0(j):=\mathcal{A}_1(j)\cap \mathcal{A}_4(j)^c\cap \mathcal{A}_5(j)^c$. We also define
\begin{equation*}
    \mathcal{A}_0:=\bigcap_{j=1}^k \mathcal{A}_0(j), \quad \mathcal{N}_0:=\Big( 
 \bigcup_{j=1}^k\mathcal{A}_1(j)^c\Big)  \bigcup\mathcal{N}_1 \bigcup \mathcal{N}_2\in\mathcal{F}.
\end{equation*}
In the rest of this proof, we assume that $T$ is sufficiently large (depending on $\beta$).

Note that for any $j\in [k]$, by Lemmas \ref{AiryDecayL1}-\ref{AiryDecayL2},
\begin{eqnarray*}
   \mathcal{A}_0(j)^c &=& \mathcal{A}_1(j)^c\cup \mathcal{A}_4(j)\cup\mathcal{A}_5(j)\nonumber\\
  &\subseteq& \mathcal{A}_1(j)^c\cup (\mathcal{A}_1(j)\cap\mathcal{A}_3(j)) \cup (\mathcal{A}_1(j)\cap \mathcal{A}_3(j)^c\cap\mathcal{A}_4(j))\nonumber\\
  &&\cup (\mathcal{A}_1(j)\cap \mathcal{A}_3(j)^c\cap\mathcal{A}_5(j))\nonumber\\
  & \subseteq & \mathcal{A}_1(j)^c\cup \mathcal{N}_1\cup \mathcal{N}_2 \subseteq \mathcal{N}_0.
\end{eqnarray*}
Hence taking $\mathcal{A}(T):=\mathcal{N}_0^c\in\mathcal{F}$, we have 
\begin{equation}\label{eee1}
    \mathcal{A}_0^c\subseteq \mathcal{N}_0 = \mathcal{A}(T)^c, \text{ thus } \mathcal{A}(T)\subseteq \mathcal{A}_0.
\end{equation}
By Lemmas \ref{AiryDecayL1}-\ref{AiryDecayL2}, (\ref{NEQ5.1}), and the union bound, 
\begin{eqnarray}
 \mathbb{P}(\mathcal{A}(T)^c)  &= &
    \mathbb{P}(\mathcal{N}_0)\leq\sum_{j=1}^k \mathbb{P}(\mathcal{A}_1(j)^c)+\mathbb{P}(\mathcal{N}_1)+\mathbb{P}(\mathcal{N}_2)\nonumber\\
    &\leq& C\exp\big(-c(kT)^{1\slash 4}\big).
\end{eqnarray}

Below we assume that the event $\mathcal{A}(T)$ holds. By (\ref{eee1}), the event $\mathcal{A}_0$ holds. Consider an arbitrary $j\in [k]$. Recall from Section \ref{Sect.n1} that $\bar{p}_j(x)=f_j'(x)\slash f_j(x)$ for any $x\geq 0$. As the event $\mathcal{A}_5(j)^c$ holds, $\bar{p}_j(x)$ never blows up to $-\infty$ for $x\geq kT$. Hence on $[kT,\infty)$, $f_j(x)\neq 0$ and $f_j(x)$ does not change sign. As the event $\mathcal{A}_4(j)^c$ holds, we have $\bar{p}_j(x)\leq -\sqrt{x}\slash 2$
for any $x\geq kT$, from which we can deduce that for any $(x_1,x_2)\in\mathbb{R}^2$ such that $x_2\geq x_1\geq kT$,
\begin{equation}
    |f_j(x_2)|\leq |f_j(x_1)|\exp\Big(-\frac{1}{3}\big(x_2^{3\slash 2}-x_1^{3\slash 2}\big)\Big).
\end{equation}
Moreover, as the events $\mathcal{A}_4(j)^c$ and $\mathcal{A}_5(j)^c$ hold, for any $x\geq kT$, we have $|f_j'(x)\slash f_j(x)|=|\bar{p}_j(x)|\leq 4\sqrt{x}$, which leads to $|f_j'(x)|\leq 4\sqrt{x}|f_j(x)|$.

\end{proof}

\subsection{Exponential decay of the eigenvectors of $H_{\beta,n}$}\label{sec:3.2}

In this subsection, we show exponential decay of the top eigenvectors of $H_{\beta,n}$, the Gaussian $\beta$-ensemble of size $n$ (recall Section \ref{Sect.n2}). The main result is the following proposition, the proof of which uses a discrete analogue of the argument in Section \ref{sec:3.1} (but with a more delicate analysis). Recall from Section \ref{Sect.n2} that for each $j\in[n]$, $g_j\in\mathbb{R}^n$ is the normalized eigenvector of $H_{\beta,n}$ that corresponds to the $j$th largest eigenvalue $\lambda_j^{(n)}$. 

\begin{proposition}\label{DiscreteDecay}
Assume that $\beta>0$, $n, k\in\mathbb{N}_{+}$, and $n^{e_0}\leq k\leq n^{1\slash 2000}$ for a fixed constant $e_0\in (0,1\slash 2000)$. There exist positive constants $C,c,K$ that only depend on $\beta, e_0$ and a positive absolute constant $C'$, such that the following holds. Let $\mathcal{B}$ be the event that for any $i_1,i_2\in\mathbb{N}_{+}$ with $n^{1\slash 3}k^{12}\leq i_1\leq i_2\leq n$ and any $j\in [k]$, the following properties hold: 
\begin{itemize}
    \item[(a)] $g_j(i_1)$ and $g_j(i_2)$ are either both positive or both negative;
    \item[(b)] $|g_j(i_2)|\leq C'|g_j(i_1)| \exp\big(-\big(i_2^{3\slash 2} -i_1^{3\slash 2}
    \big)n^{-1\slash 2}\big\slash 12 \big)$.
\end{itemize}
When $k\geq K$, we have $\mathbb{P}\left(\mathcal{B}^c\right)\leq C\exp(-c k^3)$.
\end{proposition}

Throughout this subsection, we fix $\beta>0$, $n, k\in\mathbb{N}_{+}$, and $e_0\in (0,1\slash 2000)$ such that 
\begin{equation}\label{Connk}
    n^{e_0}\leq k\leq n^{1\slash 2000}.
\end{equation}
We also assume that $n\geq 2$ (note that Proposition \ref{DiscreteDecay} holds trivially for $n=1$). We recall (\ref{tilde_l}) and the definition of $\{Y_j\}_{j=1}^{n-1}$ from Section \ref{Sect.n2}. Note that for any $j\in [n]$, $\hat{H}_{\beta,n}g_j=\tilde{\lambda}_j^{(n)}g_j$. Hence for every $i\in [2,n-1]\cap\mathbb{Z}$,
\begin{equation}\label{Eigen}
   n^{1\slash 6}\bigg(\sqrt{\frac{2}{\beta}}\xi_i g_j(i)+\frac{Y_{n-i}}{\sqrt{\beta}}g_j(i+1)+\frac{Y_{n-i+1}}{\sqrt{\beta}} g_j(i-1)-2\sqrt{n} g_j(i)\bigg)=\tilde{\lambda}_j^{(n)} g_j(i).
\end{equation}
Moreover,
\begin{equation}\label{EigenN}
      n^{1\slash 6}\bigg(\sqrt{\frac{2}{\beta}}\xi_n g_j(n)+\frac{Y_{1}}{\sqrt{\beta}} g_j(n-1)-2\sqrt{n} g_j(n)\bigg)=\tilde{\lambda}_j^{(n)} g_j(n).
\end{equation}

Now for any $i\in [n-1]$, we define
\begin{equation}\label{psum}
    Z_i:=Y_{n-i}\slash\sqrt{\beta}-\sqrt{n-i},\quad \mu_i:=\mathbb{E}[Y_{n-i}]\slash\sqrt{\beta}, \quad \kappa_i:=Y_{n-i}\slash \sqrt{\beta}-\mu_i.
\end{equation}
By the Cauchy-Schwarz inequality, $\mu_i\leq \sqrt{\mathbb{E}[Y_{n-i}^2]}\slash \sqrt{\beta}=\sqrt{n-i}$, hence
\begin{equation}\label{Muine}
    n^{1\slash 6}(\sqrt{n}-\mu_i)=n^{1\slash 6}\frac{n-\mu_i^2}{\sqrt{n}+  \mu_i}\geq \frac{1}{2}(n-\mu_i^2)n^{-1\slash 3} \geq \frac{1}{2}  i n^{-1\slash 3}.
\end{equation}

\subsubsection{Preliminary bounds on partial sums of $\{Z_i\}_{i=1}^{n-1}$}

In this part, we establish some preliminary bounds on certain partial sums of $\{Z_i\}_{i=1}^{n-1}$ as defined in (\ref{psum}). These bounds will be used in the proof of Proposition \ref{DiscreteDecay}. For any $i\in [n-1]$ and $M\geq 0$, we define 
\begin{equation}\label{npsum}
    \tilde{Z}_{i,M}:=\min\{\max\{Z_i,-M\},M\}.
\end{equation}

The following lemma bounds $Z_i$ and $\tilde{Z}_{i,M}$ for all $i\in [n-1]$ and $M\geq 1$.

\begin{lemma}\label{ZL1}
For any $i\in [n-1]$ and $M\geq 1$, we have
\begin{equation}\label{ZL1E1}
    \mathbb{P}(|Z_i|\geq M)\leq 2\exp\big(-\beta M \min\{M,\sqrt{n-i}\}\slash 8\big),
\end{equation}
\begin{equation}\label{ZL1E2}
|\mathbb{E}[\tilde{Z}_{i,M}]-\mathbb{E}[Z_i]|\leq 8\beta^{-1}\exp(-\beta M\slash 8),
\end{equation}
\begin{equation}\label{ZL1E3}
    \mathbb{E}[(\tilde{Z}_{i,M})^2]\leq \mathbb{E}[Z_i^2]\leq 2\beta^{-1}.
\end{equation}
\end{lemma}
\begin{proof}

We note that 
\begin{equation}\label{EEq1.4}
   \sqrt{\beta}Z_i = Y_{n-i}-\sqrt{\beta(n-i)}=\frac{Y_{n-i}^2-\beta(n-i)}{Y_{n-i}+\sqrt{\beta(n-i)}}.
\end{equation}
Hence 
\begin{equation}\label{EEq1.1}
    \mathbb{P}(Z_i\geq M)\leq\mathbb{P}(Y_{n-i}^2\geq \beta(n-i)+\beta M\sqrt{n-i}),
\end{equation}
\begin{equation}
    \mathbb{P}(Z_i\leq -M)\leq \mathbb{P}(Y_{n-i}^2\leq \beta(n-i)-\beta M\sqrt{n-i}).
\end{equation}
As $Y_{n-i}\sim \chi_{(n-i)\beta}$, by Markov's inequality, for any $\theta\in [0,1\slash 2)$, we have
\begin{eqnarray}\label{EEq1.2}
  && \mathbb{P}(Y_{n-i}^2\geq \beta(n-i)+\beta M\sqrt{n-i})\nonumber\\
  &\leq& \exp(-\theta\beta(n-i)-\theta\beta M\sqrt{n-i})\mathbb{E}[\exp(\theta Y_{n-i}^2)]\nonumber\\
  &=& \exp(-\theta\beta(n-i)-\theta\beta M\sqrt{n-i}-\beta(n-i)\log(1-2\theta)\slash 2),
\end{eqnarray}
\begin{eqnarray}\label{EEq1.3}
 && \mathbb{P}(Y_{n-i}^2\leq \beta(n-i)-\beta M\sqrt{n-i})\nonumber\\
  &\leq& \exp(\theta\beta(n-i)-\theta\beta M\sqrt{n-i})\mathbb{E}[\exp(-\theta Y_{n-i}^2)]\nonumber\\
  &=& \exp(\theta\beta(n-i)-\theta\beta M\sqrt{n-i}-\beta(n-i)\log(1+2\theta)\slash 2).
\end{eqnarray}
Take $\theta=\min\{M\slash (4\sqrt{n-i}),1\slash 4\}\in [0, 1\slash 4]$. As $\log(1-x)\geq -x-x^2$ for any $x\in [0,1\slash 2]$ and $\log(1+x)\geq x-x^2$ for any $x\geq 0$ (which can be verified by taking the derivative), by (\ref{EEq1.1})-(\ref{EEq1.3}), we have
\begin{eqnarray}\label{estimate1.1}
   &&  \max\{\mathbb{P}(Z_i\geq M),\mathbb{P}(Z_i\leq -M)\}\leq \exp(-\theta\beta M\sqrt{n-i}+2\theta^2(n-i)\beta) \nonumber\\
   &\leq& \exp(-\theta\beta M\sqrt{n-i}\slash 2) = \exp\big(-\beta M \min\{M,\sqrt{n-i}\}\slash 8\big),
\end{eqnarray}
from which (\ref{ZL1E1}) follows.

By (\ref{estimate1.1}) (replacing $M$ by $M+y$), we obtain that 
\begin{eqnarray*}
  &&  \mathbb{E}[Z_i \mathbbm{1}_{Z_i\geq M}]-M\mathbb{P}(Z_i\geq M)=\mathbb{E}[(Z_i-M)\mathbbm{1}_{Z_i\geq M}]\\
  &=& \int_0^{\infty} \mathbb{P}((Z_i-M)\mathbbm{1}_{Z_i\geq M}> y) dy=\int_0^{\infty}  \mathbb{P}(Z_i>M+y) dy\\
  &\leq& \int_{0}^{\infty}\exp(-\beta (M+y)\slash 8)dy=8\beta^{-1}\exp(-\beta M\slash 8).
\end{eqnarray*}
As $\tilde{Z}_{i,M}\geq Z_i$ when $Z_i<M$, we have
\begin{eqnarray}\label{estimate1.3}
   \mathbb{E}[Z_i]&= &\mathbb{E}[Z_i   \mathbbm{1}_{Z_i<M}]+\mathbb{E}[Z_i \mathbbm{1}_{Z_i\geq M}]\leq \mathbb{E}[\tilde{Z}_{i,M}\mathbbm{1}_{Z_i<M}]+\mathbb{E}[Z_i \mathbbm{1}_{Z_i\geq M}]\nonumber\\
    &=&\mathbb{E}[\tilde{Z}_{i,M}]-\mathbb{E}[\tilde{Z}_{i,M}\mathbbm{1}_{Z_i\geq M}]+ \mathbb{E}[Z_i \mathbbm{1}_{Z_i\geq M}]\nonumber\\
    &=&\mathbb{E}[\tilde{Z}_{i,M}]-M\mathbb{P}(Z_i\geq M)+ \mathbb{E}[Z_i \mathbbm{1}_{Z_i\geq M}]\nonumber\\
    &\leq& \mathbb{E}[\tilde{Z}_{i,M}]+8\beta^{-1}\exp(-\beta M\slash 8),
\end{eqnarray}
Similarly, we can deduce that
\begin{equation}\label{estimate1.4}
  \mathbb{E}[Z_i]\geq\mathbb{E}[\tilde{Z}_{i,M}]-8\beta^{-1}\exp(-\beta M\slash 8).
\end{equation}
Combining (\ref{estimate1.3}) and (\ref{estimate1.4}), we obtain (\ref{ZL1E2}).

By (\ref{npsum}), $(\tilde{Z}_{i,M})^2\leq Z_i^2$. Hence by (\ref{EEq1.4}), noting $Y_{n-i}\sim\chi_{(n-i)\beta}$, we have 
\begin{equation*}
    \mathbb{E}[(\tilde{Z}_{i,M})^2]\leq \mathbb{E}[Z_i^2] \leq \frac{\mathbb{E}\big[\big(Y_{n-i}^2-\beta(n-i)\big)^2\big]}{\beta^2(n-i)} = 2\beta^{-1}. 
\end{equation*}

\end{proof}

Based on Lemma \ref{ZL1}, we establish the following lemma, which bounds partial sums of $\tilde{Z}_{i,M}-\mathbb{E}[\tilde{Z}_{i,M}]$.

\begin{lemma}\label{ZL2}
Assume that $m_1,m_2\in \{0,1,\cdots,n-1\}$, $m_1<m_2$, and $M\geq 1$. Let
\begin{equation*}
    \tilde{S}_{m_1, m_2,M}:=\sum\limits_{i=m_1+1}^{m_2}\big(\tilde{Z}_{i,M}-\mathbb{E}[\tilde{Z}_{i,M}]\big).
\end{equation*}
For any $t\geq 0$, we have
\begin{equation}
    \mathbb{P}(|\tilde{S}_{m_1,m_2,M}|\geq t)\leq 2\exp\left(-\frac{\beta t^2}{4(m_2-m_1)}\psi_{Benn}\left(\frac{\beta Mt}{2(m_2-m_1)}\right)\right),
\end{equation}
where
\begin{equation*}
    \psi_{Benn}(t):=
    \begin{cases}
         \frac{2(1+t) \log(1+t)-2t}{t^2} & \text{ if } t>0\\
         1 & \text{ if } t=0
    \end{cases}.
\end{equation*}
In particular, when $0\leq t \leq 4\beta^{-1}M^{-1}(m_2-m_1)$, we have
\begin{equation}
    \mathbb{P}(|\tilde{S}_{m_1,m_2,M}|\geq t)\leq 2\exp\left(-\frac{\beta t^2}{8(m_2-m_1)}\right);
\end{equation}
when $t\geq 4\beta^{-1}M^{-1}(m_2-m_1)$, we have
\begin{equation}
    \mathbb{P}(|\tilde{S}_{m_1,m_2,M}|\geq t)\leq 2\exp(-t\slash (2M)).
\end{equation}
\end{lemma}
\begin{proof}
Note that $|\tilde{Z}_{i,M}|\leq M$ for any $i\in [n-1]$. Denote 
\begin{equation*}
    v:=\sum_{i=m_1+1}^{m_2}\mathbb{E}[(\tilde{Z}_{i,M})^2].
\end{equation*}
Applying Bennett's inequality (see e.g. \cite[Theorem 2.9]{BLM}) to the two sets of independent random variables $\{\tilde{Z}_{i,M}\}_{i=m_1+1}^{m_2}$ and $\{-\tilde{Z}_{i,M}\}_{i=m_1+1}^{m_2}$, we obtain that for any $t\geq 0$,
\begin{equation}\label{Es5.2}
    \mathbb{P}(|\tilde{S}_{m_1,m_2,M}|\geq t)\leq 2\exp\left(-\frac{t^2}{2v}\psi_{Benn}\left(\frac{Mt}{v}\right)\right).
\end{equation}
Now for any $t>0$, we define
\begin{equation*}
    h(t):=t\psi_{Benn}(t)=2(1+t^{-1})\log(1+t)-2.
\end{equation*}
Note that for any $t>0$, $h'(t)=2t^{-2}(t-\log(1+t))\geq 0$. Hence $h(t)$ is non-decreasing on $(0,\infty)$. By Lemma \ref{ZL1}, $v\leq 2(m_2-m_1)\beta^{-1}$. Hence for any $t>0$,
\begin{eqnarray}\label{Es5.1}
    \frac{t^2}{2v}\psi_{Benn}\left(\frac{Mt}{v}\right) &=&\frac{t}{2M}h\left(\frac{Mt}{v}\right)\geq \frac{t}{2M}h\left(\frac{\beta M t}{2(m_2-m_1)}\right)\nonumber\\
    &=& \frac{\beta t^2}{4(m_2-m_1)}\psi_{Benn}\left(\frac{\beta M t}{2(m_2-m_1)}\right).
\end{eqnarray}
Plugging (\ref{Es5.1}) into (\ref{Es5.2}), we obtain that for any $t>0$,
\begin{equation}\label{Es5.3}
    \mathbb{P}(|\tilde{S}_{m_1,m_2,M}|\geq t)\leq 2\exp\left(-\frac{\beta t^2}{4(m_2-m_1)}\psi_{Benn}\left(\frac{\beta M t}{2(m_2-m_1)}\right)\right).
\end{equation}
Note that (\ref{Es5.3}) also holds when $t=0$.

Now note that for any $t>0$, 
\begin{equation*}
    \psi'_{Benn}(t)=2t^{-3}(t+2)\left(-\log(1+t)+\frac{2t}{t+2}\right).
\end{equation*}
For any $t\geq 0$, let $g(t):=\log(1+t)-2t\slash (t+2)$. Note that for any $t\geq 0$,
\begin{equation*}
    g'(t)=\frac{t^2}{(t+1)(t+2)^2}\geq 0.
\end{equation*}
Hence $g(t)\geq g(0)=0$. Therefore, $\psi'_{Benn}(t)\leq 0$ for any $t>0$, and $\psi_{Benn}(t)$ is non-increasing on $(0,\infty)$.

When $0\leq t\leq 4\beta^{-1}M^{-1}(m_2-m_1)$, we have $0\leq \beta M t\slash(2(m_2-m_1))\leq 2$, which leads to
\begin{equation*}
    \psi_{Benn}\left( \frac{\beta M t}{2(m_2-m_1)}\right)\geq \min\{\psi_{Benn}(0),\psi_{Benn}(2)\}\geq \frac{1}{2}.
\end{equation*}
Hence when $0\leq t\leq 4\beta^{-1}M^{-1}(m_2-m_1)$, 
\begin{equation}
    \mathbb{P}(|\tilde{S}_{m_1,m_2,M}|\geq t)\leq 2\exp\left(-\frac{\beta t^2}{8(m_2-m_1)}\right).
\end{equation}

When $t\geq 4\beta^{-1}M^{-1}(m_2-m_1)$, we have $\beta M t\slash (2(m_2-m_1))\geq 2$. Recalling that $h(t)=t\psi_{Benn}(t)$ is non-decreasing on $(0,\infty)$, we obtain
\begin{equation}
   \frac{\beta M t}{2(m_2-m_1)}\psi_{Benn} \left( \frac{\beta M t}{2(m_2-m_1)}\right)\geq 2\psi_{Benn}(2)\geq 1.
\end{equation}
Hence when $t \geq 4\beta^{-1}M^{-1}(m_2-m_1)$,
\begin{eqnarray}
    \mathbb{P}(|\tilde{S}_{m_1,m_2,M}|\geq t)&\leq &2\exp\left( -\frac{\beta t^2}{4(m_2-m_1)}\cdot\frac{2(m_2-m_1)}{\beta M t}\right)\nonumber\\
    &=& 2\exp\left(-t\slash(2M)\right).
\end{eqnarray}

\end{proof}

Based on Lemmas \ref{ZL1} and \ref{ZL2}, we obtain the following result, which bounds partial sums of $Z_i-\mathbb{E}[Z_i]$.

\begin{lemma}\label{ZL3}
Assume that $m_1,m_2\in \{0,1,\cdots,n-1\}$ satisfy $m_1<m_2$. Let
\begin{equation*}
    S_{m_1,m_2}:=\sum\limits_{i=m_1+1}^{m_2}(Z_i-\mathbb{E}[Z_i]).
\end{equation*}
For any $0\leq t\leq 4\beta^{-1}n^{-1\slash 72}(m_2-m_1)$, we have
\begin{eqnarray}
    &&\mathbb{P}\Big(|S_{m_1,m_2}|\geq t+8\beta^{-1}(m_2-m_1)\exp\big(-\beta n^{1\slash 72}\slash 8\big)\Big) \nonumber\\
    &\leq& 2(m_2-m_1)\exp\big(-\beta n^{1\slash 72}\slash 8\big)+2\exp\left(-\frac{\beta t^2}{8(m_2-m_1)}\right).
\end{eqnarray}
For any $t\geq 4\beta^{-1}n^{-1\slash 72}(m_2-m_1)$, we have
\begin{eqnarray}
    &&\mathbb{P}\Big(|S_{m_1,m_2}|\geq t+8\beta^{-1}(m_2-m_1)\exp\big(-\beta n^{1\slash 72}\slash 8\big)\Big) \nonumber\\
    &\leq& 2(m_2-m_1)\exp\big(-\beta n^{1\slash 72}\slash 8\big)+2\exp\left(-\frac{t}{2 n^{1\slash 72}}\right).
\end{eqnarray}
\end{lemma}
\begin{proof}
Let $M:=n^{1\slash 72}\geq 1$. By Lemma \ref{ZL1}, we have
\begin{eqnarray*}
  &&  \bigg|\sum_{i=m_1+1}^{m_2} \big(\mathbb{E}[Z_{i}]-\mathbb{E}[\tilde{Z}_{i,M}]\big)\bigg|\leq \sum_{i=m_1+1}^{m_2}\big|\mathbb{E}[Z_{i}]-\mathbb{E}[\tilde{Z}_{i,M}]\big|
    \\
    &\leq& 8\beta^{-1}\exp(-\beta M\slash 8)(m_2-m_1)=8\beta^{-1}(m_2-m_1)\exp\big(-\beta n^{1\slash 72}\slash 8\big).
\end{eqnarray*}

Let $\mathcal{H}$ be the event that there exists some $i\in [m_1+1,m_2]\cap\mathbb{Z}$ such that $Z_i\neq \tilde{Z}_{i,M}$. By Lemma \ref{ZL1} and the union bound, we have
\begin{eqnarray}\label{Es6.2}
 && \mathbb{P}(\mathcal{H})\leq \sum_{i=m_1+1}^{m_2}\mathbb{P}(Z_i\neq \tilde{Z}_{i,M}) = \sum_{i=m_1+1}^{m_2}\mathbb{P}(|Z_i|>M)\nonumber\\
 &  \leq & 2(m_2-m_1)\exp(-\beta M\slash 8) =2(m_2-m_1)\exp\big(-\beta n^{1\slash 72}\slash 8\big).
\end{eqnarray}
When the event $\mathcal{H}^c$ holds, $Z_i=\tilde{Z}_{i,M}$ for every $i\in [m_1+1,m_2]\cap\mathbb{Z}$, hence recalling the notations in Lemma \ref{ZL2}, we have
\begin{eqnarray}\label{Es6.1}
   |S_{m_1,m_2}|&\leq& |\tilde{S}_{m_1,m_2,M}|+\bigg|\sum_{i=m_1+1}^{m_2}\big( \mathbb{E}[Z_{i}]-\mathbb{E}[\tilde{Z}_{i,M}]\big)\bigg|\nonumber\\
    &\leq& |\tilde{S}_{m_1,m_2,M}|+8\beta^{-1}(m_2-m_1)\exp\big(-\beta n^{1\slash 72}\slash 8\big).
\end{eqnarray}

Therefore, for any $t\geq 0$, we have
\begin{eqnarray}\label{Es6.3}
&& \mathbb{P}\Big(|S_{m_1,m_2}|\geq t+8\beta^{-1}(m_2-m_1)\exp\big(-\beta n^{1\slash 72}\slash 8\big)\Big)  \nonumber\\
&\leq& \mathbb{P}(\mathcal{H})+\mathbb{P}\Big(\Big\{|S_{m_1,m_2}|\geq t+8\beta^{-1}(m_2-m_1)\exp\big(-\beta n^{1\slash 72}\slash 8\big)\Big\}\cap\mathcal{H}^c\Big)\nonumber\\
&\leq& 2(m_2-m_1)\exp\big(-\beta n^{1\slash 72}\slash 8\big)+\mathbb{P}(|\tilde{S}_{m_1,m_2,M}|\geq t),
\end{eqnarray}
where we use the union bound in the first inequality and (\ref{Es6.2})-(\ref{Es6.1}) in the second inequality. The conclusion of the lemma follows from (\ref{Es6.3}) and Lemma \ref{ZL2}.

\end{proof}

\subsubsection{Proof of Proposition \ref{DiscreteDecay}}

In this part, we present the proof of Proposition \ref{DiscreteDecay}. We start by setting up some notations and definitions that will be used throughout the rest of this subsection. We denote $\widehat{\mathbb{R}}:=\mathbb{R}\cup \{\infty\}$. For each $j\in [n]$, we define $P_j\in\widehat{\mathbb{R}}^n$ by setting $P_j(1):=0$ and 
\begin{equation}\label{Pji}
     P_j(i):=
     \begin{cases}
          \frac{n^{1\slash 3}(g_j(i)-g_j(i-1))}{g_j(i-1)} & \text{ if } g_j(i-1)\neq 0\\
          \infty & \text{ if } g_j(i-1)=0
     \end{cases}
\end{equation}
for each $i\in [n]\backslash\{1\}$.

Throughout the rest of this subsection, we fix $K_3,L_2,\Delta>0$ such that $n^{1\slash 3}L_2, n^{1\slash 3}\Delta\in\mathbb{Z}$ and $L_2^{1\slash 4}\slash 4\leq \Delta\leq  L_2^{1\slash 4}$. We define the following events:
\begin{itemize}
    \item $\mathcal{U}_1$: the event that $|\kappa_i|\leq n^{1\slash 72}\slash 24$ for every $i\in[n-1]$ (recall the definition of $\kappa_i$ from (\ref{psum}));
    \item $\mathcal{U}_2$: the event that $|\xi_i|< n^{1\slash 144}\slash 24$ for every $i\in [n]$ and $|Z_i|<n^{1\slash 72}\slash 24$ for every $i\in [n-1]$ (recall the definition of $\xi_i$ from Section \ref{Sect.n2} and the definition of $Z_i$ from (\ref{psum})).
\end{itemize}
For every $j\in [k]$, we define the following events:
\begin{itemize}
    \item $\mathcal{B}_1(j)$: the event that $\tilde{\lambda}_j^{(n)}\geq -K_3$.
    \item $\mathcal{B}_2(j)$: the event that the following properties hold:
    \begin{itemize}
        \item[(a)] for any $i\in \left[\max\left\{2,n-\sqrt{n}\right\},n-1\right]\cap\mathbb{Z}$ that satisfies
    \begin{equation*}
        g_j(i-1)>0, \quad g_j(i)>0, \quad g_j(i-1)\leq 2g_j(i),
    \end{equation*}
    we have $g_j(i+1)>0$;
        \item[(b)] for any $i\in \left[\max\left\{2,n-\sqrt{n}\right\},n-1\right]\cap\mathbb{Z}$ that satisfies
    \begin{equation*}
        g_j(i-1)<0, \quad g_j(i)<0, \quad g_j(i-1)\geq 2g_j(i),
    \end{equation*}
    we have $g_j(i+1)<0$;
    \item[(c)] if $g_j(n-1)>0$ and $g_j(n)>0$, then $g_j(n-1)>2g_j(n)$; if $g_j(n-1)<0$ and $g_j(n)< 0$, then $g_j(n-1)<2g_j(n)$.
    \end{itemize}
    \item $\mathcal{B}_3(j)$: the event that there exists some $i\in [n^{1\slash 3}L_2,n]\cap\mathbb{Z}$, such that $P_j(i)\neq \infty$ and $P_j(i)\geq \sqrt{i n^{-1\slash 3}}\slash 8$.
    \item $\mathcal{B}_4(j)$: the event that there exists some $i\in [n^{1\slash 3}L_2,n-50]\cap\mathbb{Z}$, such that $P_j(i)\neq \infty$ and $P_j(i)\geq -\sqrt{i n^{-1\slash 3}}  \slash  8$.
\end{itemize}

\begin{lemma}\label{PL1}
There exist positive constants $C,c$ that only depend on $\beta$, such that
\begin{equation}
    \mathbb{P}(\mathcal{U}_1^c)\leq C\exp\big(-cn^{1\slash 72}\big), \quad \mathbb{P}(\mathcal{U}_2^c)\leq C\exp\big(-cn^{1\slash 72}\big).
\end{equation}
\end{lemma}
\begin{proof}
For any $i\in [n-1]$, by (\ref{psum}), $\kappa_i=Z_i-\mathbb{E}[Z_i]$, hence $|\kappa_i|\leq |Z_i|+|\mathbb{E}[Z_i]|$. Setting $M=1$ in Lemma \ref{ZL1}, noting that $|\tilde{Z}_{i,1}|\leq 1$, we obtain that
\begin{equation*}
|\mathbb{E}[Z_i]|\leq 1+8\beta^{-1}\exp(-\beta\slash 8).
\end{equation*}
Hence there exists $N\geq 1$ that only depends on $\beta$, such that for any $n\geq N$ and $i\in [n-1]$, $|\mathbb{E}[Z_i]|\leq n^{1\slash 72}\slash 48$. Hence by Lemma \ref{ZL1} and the union bound, for any $n\geq N$,
\begin{equation}\label{Esti1.1}
\mathbb{P}(\mathcal{U}_1^c)\leq \sum_{i=1}^{n-1}\mathbb{P}\big(|\kappa_i|> n^{1\slash 72}\slash 24\big)\leq  \sum_{i=1}^{n-1}\mathbb{P}\big(|Z_i| \geq n^{1\slash 72}\slash 48\big)\leq C\exp\big(-c n^{1\slash 72}\big).
\end{equation}
By enlarging $C$, we conclude that (\ref{Esti1.1}) holds for all $n\in\mathbb{N}_{+}$. 

By Lemma \ref{ZL1} and the union bound, we have
\begin{eqnarray*}
    \mathbb{P}(\mathcal{U}_2^c)&\leq& \sum_{i=1}^n\mathbb{P}\big(|\xi_i|\geq n^{1\slash 144}\slash 24\big)+\sum_{i=1}^{n-1}\mathbb{P}\big(|Z_i|\geq n^{1\slash 72}\slash 24\big)\\
    &\leq & C\exp\big(-c n^{1\slash 72}\big).
\end{eqnarray*}

\end{proof}

\begin{lemma}\label{PL2}
Assume that $K_3\leq n^{2\slash 3}\slash 3$. Then there exist positive constants $C,c$ that only depend on $\beta$, such that for any $j\in [k]$,
\begin{equation}
\mathbb{P}(\mathcal{B}_1(j)\cap\mathcal{B}_2(j)^c)\leq C\exp\big(-c n^{1\slash 2}\big).
\end{equation}
\end{lemma}
\begin{proof}
Let
\begin{equation*}
    \mathcal{C}_1:=\Big\{\xi_i\leq \sqrt{\beta n \slash 2}\text{ for any }i\in [\max\{2,n-\sqrt{n}\},n]\cap\mathbb{Z}\Big\},
\end{equation*}
\begin{equation*}
    \mathcal{C}_2:=\Big\{Y_{n-i+1}\leq\sqrt{\beta n}\slash 4\text{ for any }i\in [\max\{2,n-\sqrt{n}\},n]\cap \mathbb{Z}\Big\}.
\end{equation*}
Below we consider any $j\in [k]$, and assume that the event $\mathcal{C}_1\cap \mathcal{C}_2\cap\mathcal{B}_1(j)$ holds. 

Suppose that there exists $i_0\in \left[\max\left\{2,n-\sqrt{n}\right\},n-1\right]\cap\mathbb{Z}$, such that 
\begin{equation*}
    g_j(i_0-1)>0,\quad g_j(i_0)>0, \quad g_j(i_0-1)\leq 2g_j(i_0), \quad g_j(i_0+1)\leq 0.
\end{equation*}
By (\ref{Eigen}), we have
\begin{eqnarray}\label{Esti1.2}
  \sqrt{\frac{2}{\beta}  
}\xi_{i_0} g_j(i_0)+\frac{Y_{n-i_0+1}}{\sqrt{\beta}}g_j(i_0-1)  &\geq&  2\sqrt{n}g_j(i_0)+\tilde{\lambda}_j^{(n)} n^{-1\slash 6}g_j(i_0)\nonumber\\
&\geq & \frac{5}{3}\sqrt{n}g_j(i_0),
\end{eqnarray}
where in the second inequality we use $\tilde{\lambda}_j^{(n)}\geq -K_3\geq -n^{2\slash 3}\slash 3$. As $\xi_{i_0}\leq \sqrt{\beta n \slash 2}$ and $Y_{n-i_0+1} \leq \sqrt{\beta n}\slash 4$, by (\ref{Esti1.2}), we have
\begin{equation}\label{Esti1.3}
    \frac{2}{3}\sqrt{n}g_j(i_0)\leq \frac{1}{4}\sqrt{n}g_j(i_0-1)\leq \frac{1}{2}\sqrt{n}g_j(i_0), 
\end{equation}
where we use $g_j(i_0-1)\leq 2g_j(i_0)$ in the second inequality. As $g_j(i_0)>0$, (\ref{Esti1.3}) leads to a contradiction. Similarly, if there exists $i_0\in [\max\{2,n-\sqrt{n}\},n-1]\cap\mathbb{Z}$ such that
\begin{equation*}
    g_j(i_0-1)<0,\quad g_j(i_0)<0, \quad g_j(i_0-1)\geq 2g_j(i_0), \quad g_j(i_0+1)\geq 0,
\end{equation*}
we are led to a contradiction.

Suppose that $g_j(n-1)>0$, $g_j(n)>0$, and $g_j(n-1)\leq 2g_j(n)$. By (\ref{EigenN}), we can deduce that
\begin{equation*}
    \sqrt{\frac{2}{\beta}}\xi_n g_j(n)+\frac{Y_1}{\sqrt{\beta}}g_j(n-1)=2\sqrt{n}g_j(n)+\tilde{\lambda}_j^{(n)}n^{-1\slash 6}g_j(n)\geq \frac{5}{3}\sqrt{n}g_j(n).
\end{equation*}
As $\xi_n\leq \sqrt{\beta n\slash 2}$, $Y_1\leq \sqrt{\beta n}\slash 4$, and $g_j(n-1)\leq 2g_j(n)$, we have 
\begin{equation*}
    \frac{2}{3}\sqrt{n}g_j(n)\leq \frac{1}{4}\sqrt{n} g_j(n-1)\leq \frac{1}{2}\sqrt{n} g_j(n),
\end{equation*}
which leads to a contradiction as $g_j(n)>0$. Similarly, if $g_j(n-1)<0$, $g_j(n)<0$, and $g_j(n-1)\geq 2 g_j(n)$, we are led to a contradiction.

Hence the event $\mathcal{B}_2(j)$ holds. Therefore, we have $\mathcal{C}_1\cap\mathcal{C}_2\cap\mathcal{B}_1(j)\subseteq \mathcal{B}_2(j)$, which implies
\begin{equation}\label{Esti1.4}
    \mathcal{B}_1(j)\cap 
 \mathcal{B}_2(j)^c\subseteq \mathcal{C}_1^c\cup\mathcal{C}_2^c.
\end{equation}

Now by (\ref{psum}), Lemma \ref{ZL1}, and the union bound, we have
\begin{equation*}
    \mathbb{P}(\mathcal{C}_1^c)\leq C\exp(-cn),
\end{equation*}
\begin{equation*}
\mathbb{P}(\mathcal{C}_2^c)\leq \sum_{i\in [n-\sqrt{n}-1, n-1]\cap\mathbb{N}_{+}}\mathbb{P}\left(Z_i>\frac{1}{4}\sqrt{n}-\sqrt{n-i}\right)\leq C\exp\big(-c n^{1\slash 2}\big).
\end{equation*}
Hence by (\ref{Esti1.4}) and the union bound, we have
\begin{equation}
    \mathbb{P}(\mathcal{B}_1(j)\cap \mathcal{B}_2(j)^c)\leq \mathbb{P}(\mathcal{C}_1^c)+\mathbb{P}(\mathcal{C}_2^c)\leq C\exp\big(-c n^{1\slash 2}\big).
\end{equation}

\end{proof}

\begin{lemma}\label{PL3}
There exist positive constants $C,c,C_1,C_2$ (with $C_2\geq 1$) that only depend on $\beta$, such that when $L_2\geq \max\big\{C_1,C_2 K_3,(\log{n})^8\big\}$ and $K_3\leq n^{2\slash 3}\slash 3$, for any $j\in [k]$, we have
\begin{eqnarray}\label{Esti5.1}
   && \mathbb{P}\left(\mathcal{B}_3(j)\cap\mathcal{B}_1(j)\cap\mathcal{B}_2(j)\cap\mathcal{U}_1\cap\mathcal{U}_2\right)\nonumber\\
   &  \leq & C\exp\big(-c n^{1\slash 72}\big)+C\exp\big(-c L_2^{1\slash 4}\big).
\end{eqnarray}
\end{lemma}
\begin{proof}

We fix an arbitrary $j\in [k]$ throughout the proof. When $L_2>n^{2\slash 3}$, we have $\mathcal{B}_3(j)=\emptyset$ and (\ref{Esti5.1}) automatically holds. Throughout the rest of the proof, we assume that $L_2\leq n^{2\slash 3}$. We also assume that 
\begin{equation}\label{Assu}
    L_2\geq \max\big\{80^4,10^4 K_3,(\log{n})^8\big\}, \quad  n\geq \max\{10^4,(2\beta)^{-72}\}.
\end{equation}
The conclusion (\ref{Esti5.1}) for $n\in [1,\max\{10^4,(2\beta)^{-72}\})\cap\mathbb{Z}$ can be obtained by enlarging the constant $C$. The proof proceeds through \textbf{Steps 1-4} as detailed below.

\paragraph{Step 1}

In this step, we set up some notations and definitions that will be used in later parts of the proof. 

For any $i\in [2,n-1]\cap\mathbb{Z}$, we define
\begin{equation}\label{Qd}
    Q_i:=\sqrt{\frac{2}{\beta}}\xi_i+\kappa_i+\kappa_{i-1}, \quad \tilde{Q}_i:=\sqrt{\frac{2}{\beta}}\xi_i+Z_i+Z_{i-1},
\end{equation}
where we recall the definitions of $\xi_i,\kappa_i,Z_i$ from Section \ref{Sect.n2} and (\ref{psum}). Let $L_0:=\max\big\{\lfloor (n^{2\slash 3}-L_2)\Delta^{-1}-1\rfloor,0\big\}$. For any $l\in [0,L_0-1]\cap\mathbb{Z}$, we let
\begin{equation*}
    I_l:=\big[(L_2+l\Delta)n^{1\slash 3},(L_2+(l+1)\Delta)n^{1\slash 3}\big]\cap\mathbb{Z}.
\end{equation*}
We also let
\begin{equation*}
    I_{L_0}:=\big[(L_2+L_0\Delta)n^{1\slash 3},n\big]\cap\mathbb{Z}.
\end{equation*}
For every $l\in\mathbb{N}$, we let
\begin{equation*}
    \phi_1(l):=\frac{1}{16}+\frac{1}{16}(l+1)^{-1\slash 4},\quad \phi_2(l):=\frac{1}{16}+\frac{1}{32}(l+1)^{-1\slash 4}.
\end{equation*}
Note that for any $l\in\mathbb{N}$, we have $\phi_2(l+1)<\phi_2(l)<\phi_1(l)\leq 1\slash 8$. For every $l\in\mathbb{N}$, we also define
\begin{equation}\label{dh}
    h_l:=(\phi_2(l)-\phi_2(l+1))\sqrt{L_2+l\Delta}.
\end{equation}
For any $l\in [0,L_0]\cap\mathbb{Z}$, we let $\mathcal{V}_l$ be the event that for any $m_1,m_2\in [2,n-1]\cap\mathbb{Z}$ such that $0\leq m_2-m_1\leq 400 h_ln^{1\slash 3}(L_2+l\Delta)^{-1}$, we have \begin{equation}\label{Hl}
    \left|n^{-1\slash 6}\sum_{i=m_1}^{m_2}Q_i\right|\leq \frac{1}{2}h_l.
\end{equation}
We also define the event $\mathcal{V}$ as
\begin{equation}\label{VV}
    \mathcal{V}:=\bigcap_{l\in [0,L_0]\cap\mathbb{Z}: L_2+l\Delta<10 n^{13\slash 72}}\mathcal{V}_l.
\end{equation}

For every $l\in [0,L_0]\cap\mathbb{Z}$, we define the events $\mathcal{E}_{0,l},\mathcal{E}_{1,l},\mathcal{E}_{2,l}$ as follows:
\begin{itemize}
    \item $\mathcal{E}_{0,l}$: the event that for any $i_1\in I_l$ such that $P_j(i_1)\geq \phi_2(l)\sqrt{i_1 n^{-1\slash 3}}$ and $P_{j}(i_1)\neq \infty$, we have $P_j(i)\geq \phi_2(l+1)\sqrt{i n^{-1\slash 3}}$ and $P_j(i)\neq \infty$ for all $i\in I_l$ with $i\geq i_1$.
    \item $\mathcal{E}_{1,l}$: the event that there exist $i_1,i_2\in I_l$ such that $i_1<i_2$ and the following properties hold:
    \begin{itemize}
        \item[(a)] $P_j(i)\neq \infty$ for any $i\in [i_1,i_2-1]\cap\mathbb{Z}$;
        \item[(b)] $P_j(i_1)\geq \phi_2(l)\sqrt{i_1 n^{-1\slash 3}}$;
        \item [(c)] $P_j(i_2)=\infty$ or $P_j(i_2)<\phi_2(l+1)\sqrt{i_2 n^{-1\slash 3}}$;
        \item [(d)] $\phi_2(l+1)\sqrt{i n^{-1\slash 3}}\leq P_j(i) < \phi_1(l)\sqrt{i n^{-1\slash 3}}$ for any $i\in [i_1,i_2-1]\cap\mathbb{Z}$.
    \end{itemize}
    \item $\mathcal{E}_{2,l}$: the event that there exist $i_1,i_2\in I_l$ such that $i_1<i_2$ and the following properties hold:
    \begin{itemize}
        \item[(a)] $P_j(i)\neq \infty$ for any $i\in [i_1,i_2-1]\cap\mathbb{Z}$;
        \item[(b)] $P_j(i_1)\geq \phi_1(l)\sqrt{i_1 n^{-1\slash 3}}$;
        \item[(c)] $P_j(i_2)=\infty$ or $P_j(i_2)< \phi_2(l+1)\sqrt{i_2 n^{-1\slash 3}}$;
        \item[(d)] $\phi_2(l+1)\sqrt{i n^{-1\slash 3}}\leq P_j(i)<\phi_1(l)\sqrt{i n^{-1\slash 3}}$ for any $i\in [i_1+1,i_2-1]\cap\mathbb{Z}$.
    \end{itemize}
\end{itemize}
We note that by definition, for any $l\in [0,L_0]\cap\mathbb{Z}$,
\begin{equation}\label{E10.1}
    (\mathcal{E}_{0,l})^c\subseteq \mathcal{E}_{1,l}\cup\mathcal{E}_{2,l}. 
\end{equation}
Let
\begin{equation}\label{Ed}
    \mathcal{E}_0:=\bigcap_{l=0}^{L_0}\mathcal{E}_{0,l}, \quad\mathcal{E}_1:=\bigcup_{l=0}^{L_0}\mathcal{E}_{1,l}, \quad \mathcal{E}_2:=\bigcup_{l=0}^{L_0}\mathcal{E}_{2,l}. 
\end{equation}
By (\ref{E10.1}), we have
\begin{equation}\label{E10.2}
    \left(\mathcal{E}_{0}\right)^c\subseteq \mathcal{E}_{1}\cup\mathcal{E}_{2}. 
\end{equation}

\bigskip

\paragraph{Step 2}

In this step, we give an upper bound on
$\mathbb{P}(\mathcal{B}_1(j)\cap\mathcal{B}_2(j)\cap\mathcal{U}_1\cap\mathcal{U}_2\cap\mathcal{E}_1)$. We proceed through the following two sub-steps:
\begin{itemize}
    \item \textbf{Sub-step 2.1}: Show that $\mathcal{B}_1(j)\cap\mathcal{B}_2(j)\cap\mathcal{U}_1\cap\mathcal{U}_2\cap\mathcal{E}_1\subseteq \mathcal{V}^c$.
    \item \textbf{Sub-step 2.2}: Provide an upper bound on $\mathbb{P}(\mathcal{V}^c)$, which yields an upper bound on $\mathbb{P}(\mathcal{B}_1(j)\cap\mathcal{B}_2(j)\cap\mathcal{U}_1\cap\mathcal{U}_2\cap\mathcal{E}_1)$ by \textbf{Sub-step 2.1}.
\end{itemize}

\subparagraph{Sub-step 2.1}

As mentioned above, in this sub-step, we show that
\begin{equation*}
    \mathcal{B}_1(j)\cap\mathcal{B}_2(j)\cap\mathcal{U}_1\cap\mathcal{U}_2\cap\mathcal{E}_1\subseteq \mathcal{V}^c.
\end{equation*}
The proof consists of \textbf{Parts 2.1.1-2.1.4} as detailed below.

\smallskip

\textbf{Part 2.1.1}

In this part, we derive some preliminary estimates on $g_j$ and $P_j$.

We first consider any $i\in [2,n-1]\cap\mathbb{Z}$ such that $g_j(i)\neq 0, g_j(i-1)\neq 0$. By (\ref{Pji}), we have
\begin{eqnarray}\label{PP1}
    P_j(i+1)-P_j(i)&=&-n^{1\slash 3}\frac{\left(g_j(i)-g_j(i-1)\right)^2}{g_j(i)g_j(i-1)}-2n^{1\slash 3}\nonumber\\
    && +n^{1\slash 3}\frac{g_j(i-1)+g_j(i+1)}{g_j(i)}.
\end{eqnarray}
By (\ref{Eigen}), we have
\begin{eqnarray}\label{PP2}
    2n^{1\slash 3}&=&n^{-1\slash 6}\sqrt{\frac{2}{\beta}}\xi_i-n^{-1\slash 3}\tilde{\lambda}_j^{(n)}+n^{-1\slash 6}\frac{Y_{n-i}}{\sqrt{\beta}}\frac{g_j(i+1)}{g_j(i)}\nonumber\\
    &&+n^{-1\slash 6}\frac{Y_{n-i+1}}{\sqrt{\beta}}\frac{g_j(i-1)}{g_j(i)}.
\end{eqnarray}
Plugging (\ref{PP2}) into (\ref{PP1}) and using the definitions of $\mu_i,\kappa_i$ (see (\ref{psum})) and $Q_i$ (see (\ref{Qd})), we obtain that 
\begin{eqnarray}\label{PP3}
    && P_j(i+1)-P_j(i)=-n^{1\slash 3}\frac{\left(g_j(i)-g_j(i-1)\right)^2}{g_j(i)g_j(i-1)}+n^{-1\slash 3}\tilde{\lambda}_j^{(n)}-n^{-1\slash 6} Q_i\nonumber\\
    &&\quad\quad+n^{-1\slash 6}(\sqrt{n}-\mu_i)\frac{g_j(i+1)}{g_j(i)}+n^{-1\slash 6}(\sqrt{n}-\mu_{i-1})\frac{g_j(i-1)}{g_j(i)}\nonumber\\
    &&\quad \quad  -n^{-1\slash 6}\kappa_i\frac{g_j(i+1)-g_j(i)}{g_j(i)}-n^{-1\slash 6}\kappa_{i-1}\frac{g_j(i-1)-g_j(i)}{g_j(i)}.
\end{eqnarray}

Throughout the rest of this sub-step, we assume that the events $\mathcal{B}_1(j)$, $\mathcal{B}_2(j)$, $\mathcal{U}_1$, and $\mathcal{U}_2$ hold. We also fix an arbitrary $l\in [0,L_0]\cap\mathbb{Z}$ and assume that the event $\mathcal{E}_{1,l}$ holds.

For any $i\in [2,n-1]\cap\mathbb{Z}$ such that either $g_j(i)>0,g_j(i-1)>0$ or $g_j(i)<0,g_j(i-1)<0$, by (\ref{PP3}), (\ref{Muine}), and the definition of $\mathcal{B}_1(j)$, we can deduce that
\begin{eqnarray}\label{PP5}
    && P_j(i+1)-P_j(i)\geq -n^{1\slash 3}\frac{\left(g_j(i)-g_j(i-1)\right)^2}{g_j(i)g_j(i-1)}-n^{-1\slash 3}K_3-n^{-1\slash 6} Q_i\nonumber\\
    &&\quad\quad+n^{-1\slash 6}(\sqrt{n}-\mu_i)\frac{g_j(i+1)}{g_j(i)}+\frac{1}{2}(i-1)n^{-2\slash 3}\frac{g_j(i-1)}{g_j(i)}\nonumber\\
    &&\quad \quad  -n^{-1\slash 6}\kappa_i\frac{g_j(i+1)-g_j(i)}{g_j(i)}-n^{-1\slash 6}\kappa_{i-1}\frac{g_j(i-1)-g_j(i)}{g_j(i)}.
\end{eqnarray}

As the event $\mathcal{E}_{1,l}$ holds, there exist $i_1,i_2\in I_l$, such that $i_1<i_2$ and the properties (a)-(d) in the definition of $\mathcal{E}_{1,l}$ are satisfied. We fix a choice of such $i_1,i_2$ for the remainder of this sub-step. Note that for any $i\in [i_1,i_2-1]\cap\mathbb{Z}$, we have $g_j(i-1)\neq 0$ and 
\begin{equation}\label{PP6}
    \left|\frac{g_j(i)-g_j(i-1)}{g_j(i-1)}\right|=n^{-1\slash 3}|P_j(i)|\leq \frac{1}{8}n^{-1\slash 3}\sqrt{i n^{-1\slash 3}}\leq \frac{1}{2},
\end{equation}
hence $1\slash 2\leq g_j(i)\slash g_j(i-1)\leq 3\slash 2$. Without loss of generality, we assume that $g_j(i_1-1)>0$ in the following (the case $g_j(i_1-1)<0$ can be handled similarly). Thus for any $i\in [i_1,i_2-1]\cap\mathbb{Z}$, we have \begin{equation}\label{Es7.1}
    g_j(i)>0,\quad g_j(i-1)\leq 2 g_j(i), \quad  g_j(i)\leq \frac{3}{2}g_j(i-1).
\end{equation}

\smallskip

\textbf{Part 2.1.2}

In this part, we show that $g_j(i_2)>0$. Note that $i_2\geq L_2 \geq 10$. Thus if $n-\sqrt{n}+1\leq i_2\leq n$, as the event $\mathcal{B}_2(j)$ holds, we can deduce from (\ref{Es7.1}) that $g_j(i_2)>0$. Now consider the case $i_2< n-\sqrt{n}+1$. Plugging (\ref{PP2}) into (\ref{PP1}) and using the definitions of $Z_i$ (see (\ref{psum})) and $\tilde{Q}_i$ (see (\ref{Qd})), we obtain that for any $i\in [2,n-1]\cap\mathbb{Z}$ such that $g_j(i)\neq 0,  g_j(i-1)\neq 0$,
\begin{eqnarray}\label{P10.1}
  &&  n^{1\slash 3}\frac{g_j(i+1)}{g_j(i)}-n^{1\slash 3}\frac{g_j(i)}{g_j(i-1)}= P_j(i+1)-P_j(i) \nonumber\\
  &=& -n^{1\slash 3}\frac{\left(g_j(i)-g_j(i-1)\right)^2}{g_j(i)g_j(i-1)}+n^{-1\slash 3}\tilde{\lambda}_j^{(n)}-n^{-1\slash 6} \tilde{Q}_i\nonumber\\
    && +n^{-1\slash 6}(\sqrt{n}-\sqrt{n-i})\frac{g_j(i+1)}{g_j(i)}+n^{-1\slash 6}(\sqrt{n}-\sqrt{n-i+1})\frac{g_j(i-1)}{g_j(i)}\nonumber\\
    && -n^{-1\slash 6} Z_i\frac{g_j(i+1)-g_j(i)}{g_j(i)}-n^{-1\slash 6} Z_{i-1}\frac{g_j(i-1)-g_j(i)}{g_j(i)}.
\end{eqnarray}
As the event $\mathcal{B}_1(j)$ holds, we have $n^{-1\slash 3}\tilde{\lambda}_j^{(n)}\geq -n^{-1\slash 3}K_3$. By (\ref{PP6}) and (\ref{Es7.1}), we have for any $i\in [i_1,i_2-1]\cap\mathbb{Z}$,
\begin{equation}\label{PP8}
    \left|\frac{g_j(i-1)-g_j(i)}{g_j(i)}\right|=\left|\frac{g_j(i-1)-g_j(i)}{g_j(i-1)}\right|\left|\frac{g_j(i-1)}{g_j(i)}\right|\leq \sqrt{\frac{i}{n}}.
\end{equation}
As the event $\mathcal{U}_2$ holds, we have $|Z_{i_2-2}|\leq n^{1\slash 72}\slash 24$, hence
\begin{equation*}
    \left| n^{-1\slash 6} Z_{i_2-2}\frac{g_j(i_2-2)-g_j(i_2-1)}{g_j(i_2-1)}\right|\leq  i_2^{1\slash 2}n^{-47\slash 72}\slash 24. 
\end{equation*}
Now by (\ref{PP6}) and (\ref{Es7.1}), we have for any $i\in [i_1,i_2-1]\cap\mathbb{Z}$,
\begin{eqnarray}\label{P10.11}
    n^{1\slash 3}\frac{\left(g_j(i)-g_j(i-1)\right)^2}{g_j(i)g_j(i-1)}&=&n^{1\slash 3}\left(\frac{g_j(i)-g_j(i-1)}{g_j(i-1)}\right)^2\frac{g_j(i-1)}{g_j(i)}\nonumber\\
    &\leq& \frac{1}{4} \frac{g_j(i-1)}{g_j(i)}i n^{-2\slash 3}.
\end{eqnarray}
As $i_2-2\geq 3i_2\slash 4$ (as $i_2\geq 10$), we have 
\begin{eqnarray*}
  &&  n^{-1\slash 6}\left(\sqrt{n}-\sqrt{n-i_2+2}\right)\frac{g_j(i_2-2)}{g_j(i_2-1)}\nonumber\\
  &=& n^{-1\slash 6}\frac{i_2-2}{\sqrt{n}+\sqrt{n-i_2+2}}\frac{g_j(i_2-2)}{g_j(i_2-1)}\geq \frac{3}{8}\frac{g_j(i_2-2)}{g_j(i_2-1)}i_2 n^{-2\slash 3}
\end{eqnarray*}
By (\ref{Assu}), we have $i_2\geq L_2 n^{1\slash 3}\geq 10^4\max\{1,K_3\}n^{1\slash 3}$, hence
\begin{equation*}
    i_2 n^{-2\slash 3}\slash 1000\geq  n^{-1\slash 3}K_3, 
\end{equation*}
\begin{equation*}
    i_2 n^{-2\slash 3}\slash 1000 \geq i_2^{1\slash 2}\cdot\sqrt{10^4 n^{1\slash 3}}\cdot n^{-2\slash 3}\slash 1000 \geq     
    i_2^{1\slash 2}n^{-47\slash 72}\slash 24.
\end{equation*}
Combining (\ref{Es7.1})-(\ref{P10.1}) with the above estimates, we obtain that
\begin{eqnarray}\label{P10.4}
    n^{1\slash 3}\frac{g_j(i_2)}{g_j(i_2-1)}-n^{1\slash 3}\frac{g_j(i_2-1)}{g_j(i_2-2)}\geq n^{-1\slash 6}(\sqrt{n}-\sqrt{n-i_2+1})\frac{g_j(i_2)}{g_j(i_2-1)}\nonumber\\
    +\frac{1}{20}i_2 n^{-2\slash 3}-n^{-1\slash 6}\tilde{Q}_{i_2-1}-n^{-1\slash 6}\left|Z_{i_2-1} \right| \left| \frac{g_j(i_2)-g_j(i_2-1)}{g_j(i_2-1)}\right|.
\end{eqnarray}

Suppose that $g_j(i_2)\leq 0$. As the event $\mathcal{U}_2$ holds, we have $|Z_{i_2-1}|\leq n^{1\slash 72}$. Using this and (\ref{Es7.1}), we can deduce from (\ref{P10.4}) that 
\begin{equation}\label{P10.5}
    \tilde{Q}_{i_2-1}\geq \frac{1}{2}\sqrt{n}-n^{1\slash 72}+\big(n^{1\slash 72}-\sqrt{n-i_2+1}\big)\frac{g_j(i_2)}{g_j(i_2-1)}.
\end{equation}
As $i_2<n-\sqrt{n}+1$, we have $\sqrt{n-i_2+1}>n^{1\slash 4}\geq n^{1\slash 72}$. Combining this with (\ref{P10.5}) and noting (\ref{Assu}), we obtain that
\begin{equation}\label{P10.6}
    \tilde{Q}_{i_2-1}\geq \frac{1}{2}\sqrt{n}-n^{1\slash 72}> \frac{1}{4}\sqrt{n}.
\end{equation}
As the event $\mathcal{U}_2$ holds, noting (\ref{Assu}), we have 
\begin{equation}\label{P10.7}
    \tilde{Q}_{i_2-1}=\sqrt{\frac{2}{\beta}}\xi_{i_2-1}+Z_{i_2-1}+Z_{i_2-2}\leq \frac{1}{6} n^{1\slash 72}\leq \frac{1}{6}\sqrt{n}.
\end{equation}
Note that (\ref{P10.6}) and (\ref{P10.7}) lead to a contradiction. Hence we conclude that
\begin{equation}\label{P10.10}
    g_j(i_2)>0.
\end{equation}

\smallskip

\textbf{Part 2.1.3}

For any $i\in I_l$, we define
\begin{equation}\label{Defin}
    R_j(i):=\begin{cases}
     P_j(i)-\phi_2(l+1)\sqrt{i n^{-1\slash 3}} & \text{ if } P_j(i)\neq \infty \\
    \infty & \text{ if } P_j(i)= \infty
    \end{cases}.
\end{equation}
In this part, we derive certain inequalities on $R_j$ (see (\ref{P10.17}) and (\ref{P10.18}) below), which will be used in \textbf{Part 2.1.4}. Now consider any $i\in [i_1,i_2-1]\cap\mathbb{Z}$. By (\ref{Muine}) and (\ref{P10.10}), we have
\begin{equation*}
    n^{-1\slash 6}(\sqrt{n}-\mu_i)\frac{g_j(i+1)}{g_j(i)}\geq 0,   \quad \frac{1}{2}(i-1)n^{-2\slash 3}\frac{g_j(i-1)}{g_j(i)}\geq \frac{3}{8}\frac{g_j(i-1)}{g_j(i)}in^{-2\slash 3}.
\end{equation*}
As the event $\mathcal{U}_1$ holds, we have $|\kappa_{i'}|\leq n^{1\slash 8}$ for every $i'\in [n-1]$. Hence by (\ref{PP8}), we have 
\begin{equation*}
   \left| n^{-1\slash 6}\kappa_{i-1}\frac{g_j(i-1)-g_j(i)}{g_j(i)}\right|\leq n^{-1\slash 6} n^{1\slash 8}\sqrt{i n^{-1}}=i^{1\slash 2}n^{-13\slash 24}.
\end{equation*}
If $i\in [i_1,i_2-2]\cap\mathbb{Z}$, by (\ref{PP6}), noting that $i_1\geq L_2 n^{1\slash 3}\geq 10$, we have
\begin{equation*}
    \left| n^{-1\slash 6}\kappa_i\frac{g_j(i+1)-g_j(i)}{g_j(i)}    \right|\leq n^{-1\slash 6} n^{1\slash 8} \cdot \frac{1}{8}\sqrt{(i+1) n^{-1}}\leq \frac{1}{2}i^{1\slash 2}n^{-13\slash 24};
\end{equation*}
if $i=i_2-1$, we have
\begin{equation*}
    \left| n^{-1\slash 6}\kappa_i\frac{g_j(i+1)-g_j(i)}{g_j(i)}    \right|\leq n^{-1\slash 6} n^{1\slash 8} ( n^{-1\slash 3}|P_j(i+1)|)=n^{-3\slash 8}|P_j(i+1)|.
\end{equation*}
From $i\geq L_2 n^{1\slash 3}$ and (\ref{Assu}), we can deduce that
\begin{equation*}
 i^{1\slash 2}n^{-13\slash 24}\leq in^{-2\slash 3}\slash 100, \quad  n^{-1\slash 3}K_3\leq n^{-1\slash 3}L_2\slash 100\leq in^{-2\slash 3}\slash 100. 
\end{equation*}
Plugging the above estimates, (\ref{Es7.1}), and (\ref{P10.11}) into (\ref{PP5}) and using the fact that $i\geq (L_2+l\Delta)n^{1\slash 3}$, we obtain that for any $i\in [i_1,i_2-2]\cap\mathbb{Z}$,
\begin{eqnarray}\label{P10.13}
  &&  P_j(i+1)-P_j(i)\nonumber\\
  &\geq& 
  \frac{1}{8}\frac{g_j(i-1)}{g_j(i)}i n^{-2\slash 3}-n^{-1\slash 3}K_3-2 i^{1\slash 2} n^{-13\slash 24}-n^{-1\slash 6} Q_i
  \nonumber\\
  &\geq& \Big(\frac{1}{12}-\frac{3}{100}\Big)in^{-2\slash 3}-n^{-1\slash 6}Q_i\geq \frac{1}{20}(L_2+l\Delta)n^{-1\slash 3}-n^{-1\slash 6}Q_i.  \nonumber\\
  &&
\end{eqnarray}
Similarly, we have
\begin{eqnarray}\label{P10.14}
  &&  n^{-3\slash 8}|P_j(i_2)|+P_j(i_2)-P_j(i_2-1)\geq  \frac{1}{20} (i_2-1)n^{-2\slash 3}-n^{-1\slash 6}Q_{i_2-1}\nonumber\\
  &&\quad  \geq \frac{1}{20}(L_2+l\Delta)n^{-1\slash 3}-n^{-1\slash 6} Q_{i_2-1}.
\end{eqnarray}
Note that for any $i,i'\in I_l$ such that $i<i'$, we have $i,i'\geq (L_2+l\Delta)n^{1\slash 3}$, hence
\begin{eqnarray*}
  && \phi_2(l+1)(\sqrt{i'}-\sqrt{i})n^{-1\slash 6}=\phi_2(l+1)\frac{i'-i}{\sqrt{i'}+\sqrt{i}}n^{-1\slash 6}\\
  &\leq& \frac{1}{16}(L_2+l\Delta)^{-1\slash 2}n^{-1\slash 3}(i'-i)\leq \frac{1}{100}(L_2+l\Delta)n^{-1\slash 3}(i'-i),
\end{eqnarray*}
where we use (\ref{Assu}) in the last inequality. Therefore, for any $i,i'\in I_l$ such that $i<i'$ and $P_j(i),P_j(i')\neq \infty$, we have
\begin{equation}\label{P10.15}
    R_j(i')-R_j(i)\geq P_j(i')-P_j(i)-\frac{1}{100}(L_2+l\Delta)n^{-1\slash 3}(i'-i).
\end{equation}
Now note that for any $i\in I_l$, by (\ref{Assu}), 
\begin{equation*}
    \phi_2(l+1)i^{1\slash 2}n^{-1\slash 6}\leq \frac{1}{8}\sqrt{L_2+(l+2)\Delta}\leq \frac{1}{4}\sqrt{L_2+l\Delta}.
\end{equation*}
Hence by (\ref{Assu}) again,
\begin{eqnarray}\label{P10.16}
   && n^{-3\slash 8}|R_j(i_2)|-n^{-3\slash 8}|P_j(i_2)|\geq -n^{-3\slash 8}\phi_2(l+1)i_2^{1\slash 2}n^{-1\slash 6} \nonumber\\
  &\geq& -\frac{1}{4}(L_2+l\Delta)^{1\slash 2}n^{-3\slash 8} \geq -\frac{1}{100}(L_2+l\Delta)n^{-1\slash 3}. 
\end{eqnarray}
Combining (\ref{P10.15})-(\ref{P10.16}) with (\ref{P10.13})-(\ref{P10.14}), we conclude that for any\\ $i\in [i_1,i_2-2]\cap\mathbb{Z}$,
\begin{equation}\label{P10.17}
    R_j(i+1)-R_j(i)\geq \frac{1}{40}(L_2+l\Delta)n^{-1\slash 3}-n^{-1\slash 6}Q_i.
\end{equation}
Moreover,
\begin{equation}\label{P10.18}
    n^{-3\slash 8}|R_j(i_2)|+R_j(i_2)-R_j(i_2-1)\geq \frac{1}{40}(L_2+l\Delta)n^{-1\slash 3}-n^{-1\slash 6}Q_{i_2-1}.
\end{equation}

\smallskip

\textbf{Part 2.1.4}

As the events $\mathcal{U}_1$ and $\mathcal{U}_2$ hold, by (\ref{Assu}), for any $i\in [i_1,i_2-1]\cap\mathbb{Z}$, we have
\begin{equation}\label{neqeq1}
     |Q_i|\leq \sqrt{\frac{2}{\beta}}\cdot \frac{n^{1\slash 144}}{24}+\frac{n^{1\slash 72}}{12} \leq \frac{n^{1\slash 72}}{6}.
\end{equation}
By property (b) in the definition of $\mathcal{E}_{1,l}$,
\begin{equation}\label{P10.19}
    R_j(i_1)=P_j(i_1)-\phi_2(l+1)\sqrt{i_1 n^{-1\slash 3}}\geq (\phi_2(l)-\phi_2(l+1))\sqrt{i_1 n^{-1\slash 3}}>0.
\end{equation}
By (\ref{Es7.1}), we have $P_j(i_2)\neq \infty$, which combined with property (b) in the definition of $\mathcal{E}_{1,l}$ implies
\begin{equation}\label{P10.20}
    R_j(i_2)=P_j(i_2)-\phi_2(l+1)\sqrt{i_2 n^{-1\slash 3}}<0.
\end{equation}

First consider the case where $L_2+l\Delta\geq 10 n^{13\slash 72}$. By (\ref{P10.17}) and (\ref{neqeq1}), we have $R_j(i+1)>R_j(i)$ for all $i\in [i_1,i_2-2]\cap\mathbb{Z}$. By (\ref{P10.18}) and (\ref{neqeq1}), we have $n^{-3\slash 8}|R_j(i_2)|+R_j(i_2)-R_j(i_2-1)>0$. Combining these with (\ref{P10.19}) gives $R_j(i_2)>0$. Noting (\ref{P10.20}), we are led to a contradiction. 

Now consider the case where $L_2+l\Delta<10 n^{13\slash 72}$. From (\ref{Assu}), we can deduce that $L_2+l\Delta\geq 40+20l$, hence $l+2\leq (L_2+l\Delta)\slash 20$. Thus by (\ref{dh}), 
\begin{eqnarray}\label{Hliq}
  &&  h_l=\frac{1}{32}\big((l+1)^{-1\slash 4}-(l+2)^{-1\slash 4}\big)\sqrt{L_2+l\Delta}\nonumber\\
  &\geq& \frac{1}{128}(l+2)^{-5\slash 4}\sqrt{L_2+l\Delta}\geq \frac{1}{4}(L_2+l\Delta)^{-3\slash 4},
\end{eqnarray}
which leads to
\begin{eqnarray}\label{P11.2}
   \frac{400 h_l n^{1\slash 3}}{L_2+l\Delta}\geq 100(L_2+l\Delta)^{-7\slash 4}n^{1\slash 3} \geq n^{1\slash 100},
\end{eqnarray}
where we use $L_2+l\Delta<10 n^{13\slash 72}$ in the last inequality. Now suppose that the event $\mathcal{V}_l$ holds. Let 
\begin{equation}\label{HL}
    H_l:=\lfloor 400 h_l n^{1\slash 3}(L_2+l\Delta)^{-1} \rfloor.
\end{equation}
Note that 
\begin{equation}\label{P11.5}
    H_l\geq 1,\quad H_l\geq \frac{200 h_l n^{1\slash 3}}{L_2+l\Delta}
\end{equation}
due to (\ref{P11.2}) and the following argument:
\begin{itemize}
    \item If $H_l=1$, then $400 h_l n^{1\slash 3} (L_2+l\Delta)^{-1}<2=2H_l$.
    \item If $H_l\geq 2$, then $400 h_l n^{1\slash 3}(L_2+l\Delta)^{-1}\geq 2$ and\\ $H_l\geq 400 h_l n^{1\slash 3}(L_2+l\Delta)^{-1}-1\geq 200 h_l n^{1\slash 3}(L_2+l\Delta)^{-1}$.
\end{itemize}
Let $t_0:=\max\{t\in \mathbb{N}:i_1+t H_l\leq i_2-1\}\geq 0$. For any $t\in [0,t_0]\cap\mathbb{Z}$ and any $i\in [i_1+t H_l,\min\{i_1+(t+1)H_l,i_2-1\}]\cap\mathbb{Z}$ (note that $i\in I_l$), by (\ref{P10.17}), 
\begin{eqnarray}\label{P11.1}
  &&  R_j(i)-R_j(i_1+t H_l)\nonumber\\
  &\geq& \frac{1}{40}(L_2+l\Delta)n^{-1\slash 3}\left(i-(i_1+t H_l)\right)-n^{-1\slash 6}\sum_{s=i_1+t H_l}^{i-1} Q_s \nonumber  \\
    &\geq& \frac{1}{40}(L_2+l\Delta) n^{-1\slash 3}(i-(i_1+t H_l))-\frac{1}{2}h_l,
\end{eqnarray}
where we use (\ref{Hl}) in the second inequality. Using (\ref{P11.5}) and (\ref{P11.1}), we obtain that for any $t\in [0,t_0-1]\cap\mathbb{Z}$, 
\begin{equation}\label{P11.4}
    R_j(i_1+(t+1)H_l)-R_j(i_1+t H_l)\geq \frac{1}{40}(L_2+l\Delta) n^{-1\slash 3}H_l-\frac{1}{2}h_l\geq h_l.
\end{equation}
By (\ref{P10.19}), we have $R_j(i_1)\geq h_l$, which combined with (\ref{P11.4}) leads to
\begin{equation}\label{P11.6}
    R_j(i_1+ t_0 H_l)\geq h_l.
\end{equation}
By (\ref{P10.18}), (\ref{P11.1}), and (\ref{Hl}), we have
\begin{equation*}
    n^{-3\slash 8}|R_j(i_2)|+R_j(i_2)-R_j(i_1+t_0 H_l)\geq -n^{-1\slash 6} \sum_{s=i_1+t_0 H_l}^{i_2-1}Q_{s}\geq -\frac{1}{2}h_l,
\end{equation*}
which combined with (\ref{P11.6}) leads to $R_j(i_2)>0$. Noting (\ref{P10.20}), we are led to a contradiction. Hence we conclude that $\mathcal{V}_l$ does not hold if $L_2+l\Delta<10 n^{13\slash 72}$.

We conclude that for any $l\in [0,L_0]\cap\mathbb{Z}$ such that $L_2+l\Delta\geq 10n^{13\slash 72}$,
\begin{equation}
    \mathcal{B}_1(j)\cap\mathcal{B}_2(j)\cap\mathcal{U}_1\cap\mathcal{U}_2\cap\mathcal{E}_{1,l}=\emptyset;
\end{equation}
for any $l\in [0,L_0]\cap\mathbb{Z}$ such that $L_2+l\Delta< 10n^{13\slash 72}$,
\begin{equation}
    \mathcal{B}_1(j)\cap\mathcal{B}_2(j)\cap\mathcal{U}_1\cap\mathcal{U}_2\cap\mathcal{E}_{1,l}\subseteq \mathcal{V}_l^c.
\end{equation}
Noting (\ref{VV}) and (\ref{Ed}), we obtain that
\begin{equation}\label{Con}
    \mathcal{B}_1(j)\cap\mathcal{B}_2(j)\cap\mathcal{U}_1\cap\mathcal{U}_2\cap\mathcal{E}_1\subseteq \mathcal{V}^c.
\end{equation}

\medskip

\subparagraph{Sub-step 2.2}

In this sub-step, we give an upper bound on $\mathbb{P}(\mathcal{V}^c)$, which by (\ref{Con}) yields an upper bound on $\mathbb{P}(\mathcal{B}_1(j)\cap\mathcal{B}_2(j)\cap\mathcal{U}_1\cap\mathcal{U}_2\cap\mathcal{E}_1)$. We fix an arbitrary $l\in [0,L_0]\cap\mathbb{Z}$ such that $L_2+l\Delta<10 n^{13\slash 72}$. Consider any $m_1,m_2\in [2,n-1]\cap\mathbb{Z}$ such that $0\leq m_2-m_1\leq 400 h_l n^{1\slash 3}(L_2+l\Delta)^{-1}$. By (\ref{Qd}), we have
\begin{equation}\label{Qform}
    \sum_{i=m_1}^{m_2}Q_i=\sqrt{\frac{2}{\beta}}\left(\sum_{i=m_1}^{m_2}\xi_i\right)+\sum_{i=m_1}^{m_2}\kappa_i+\sum_{i=m_1-1}^{m_2-1}\kappa_i.
\end{equation}
We also note that by (\ref{Hliq}),
\begin{equation*}
    h_{l}\geq \frac{1}{4}(L_2+l\Delta)^{-3\slash 4}\geq \frac{1}{40}n^{-13\slash 96},
\end{equation*}
where we use $L_2+l\Delta<10 n^{13\slash 72}$ in the second inequality. Hence
\begin{equation}\label{neqq2}
    \frac{1}{10}h_l n^{1\slash 6}\geq \frac{1}{400}n^{1\slash 32}\geq \frac{1}{400}n^{1\slash 36}.
\end{equation}

Now note that for any $i\in [n-1]$, by (\ref{psum}), $\kappa_i=Z_i-\mathbb{E}[Z_i]$. Hence by Lemma \ref{ZL3}, for any $0\leq t\leq 4\beta^{-1}(m_2-m_1+1)n^{-1\slash 72}$, 
\begin{eqnarray}\label{P12.1}
  && \mathbb{P}\left(\left|\sum_{i=m_1}^{m_2}\kappa_i\right|\geq t+8\beta^{-1}(m_2-m_1+1)\exp\big(-\beta n^{1\slash 72}\slash 8\big)\right)\nonumber\\
  &\leq& 2(m_2-m_1+1)\exp\big(-\beta n^{1\slash 72}\slash 8\big)+2\exp\left(-\frac{\beta t^2}{8(m_2-m_1+1)}\right);\nonumber\\
  &&
\end{eqnarray}
for any $t\geq 4\beta^{-1}(m_2-m_1+1)n^{-1\slash 72}$,
\begin{eqnarray}\label{P12.2}
    && \mathbb{P}\left(\left|\sum_{i=m_1}^{m_2}\kappa_i\right|\geq t+8\beta^{-1}(m_2-m_1+1)\exp\big(-\beta n^{1\slash 72}\slash 8\big)\right)\nonumber\\
  &\leq& 2(m_2-m_1+1)\exp\big(-\beta n^{1\slash 72}\slash 8\big)+2\exp\left(-\frac{t}{2n^{1\slash 72}}\right).
\end{eqnarray}
Take $t=h_l n^{1\slash 6}\slash 10$. First consider the case where $t\leq 4\beta^{-1}(m_2-m_1+1)n^{-1\slash 72}$. Noting (\ref{P11.5}), we have 
\begin{equation*}
    \frac{\beta t^2}{8(m_2-m_1+1)}\geq \frac{\beta (h_l n^{1\slash 6}\slash 10)^2 }{8\cdot 800 h_l n^{1\slash 3}(L_2+l\Delta)^{-1}}\geq 10^{-6}\beta h_l(L_2+l\Delta).
\end{equation*}
Hence by (\ref{P12.1}) and (\ref{neqq2}), we can deduce that
\begin{equation*}
     \mathbb{P}\left(n^{-1\slash 6}\left|\sum_{i=m_1}^{m_2}\kappa_i\right|\geq \frac{1}{5} h_l\right)\leq C\exp\big(-c n^{1\slash 72}\big)+C\exp(-c h_l (L_2+l\Delta)).
\end{equation*}
Now consider the case where $t> 4\beta^{-1}(m_2-m_1+1)n^{-1\slash 72}$. By (\ref{P12.2}) and (\ref{neqq2}), we can deduce that
\begin{equation*}
     \mathbb{P}\left(n^{-1\slash 6}\left|\sum_{i=m_1}^{m_2}\kappa_i\right|\geq \frac{1}{5} h_l\right)\leq C\exp\big(-c n^{1\slash 72}\big).
\end{equation*}

The above argument also applies to $\sum_{i=m_1-1}^{m_2-1}\kappa_i$. Hence we have
\begin{equation}\label{P12.3}
    \mathbb{P}\left(n^{-1\slash 6}\left|\sum_{i=m_1}^{m_2}\kappa_i\right|\geq \frac{1}{5} h_l\right)\leq C\exp\big(-c n^{1\slash 72}\big)+C\exp(-c h_l (L_2+l\Delta)),
\end{equation}
\begin{equation}\label{P12.4}
    \mathbb{P}\left(n^{-1\slash 6}\left|\sum_{i=m_1-1}^{m_2-1}\kappa_i\right|\geq \frac{1}{5} h_l\right)\leq C\exp\big(-c n^{1\slash 72}\big)+C\exp(-c h_l (L_2+l\Delta)).
\end{equation}

As $\{\xi_i\}_{i=1}^n$ are i.i.d. $N(0,1)$ random variables (see Section \ref{Sect.n2}), we have $\sum_{i=m_1}^{m_2}\xi_i\sim N(0,m_2-m_1+1)$. As $m_2-m_1\leq 400 h_l n^{1\slash 3}(L_2+l\Delta)^{-1}$, by (\ref{P11.2}), we have
\begin{equation*}
    \frac{1}{10}h_l n^{1\slash 6}\sqrt{\frac{\beta}{2(m_2-m_1+1)}}\geq \frac{1}{400}\beta^{1\slash 2}h_l^{1\slash 2} (L_2+l\Delta)^{1\slash 2}.
\end{equation*}
Hence letting $W\sim N(0,1)$, we have
\begin{eqnarray}\label{P12.5}
   \mathbb{P}\left(n^{-1\slash 6}\sqrt{\frac{2}{\beta}}\left|\sum_{i=m_1}^{m_2}\xi_i\right|\geq \frac{h_l}{10}\right)&=&\mathbb{P}\left(|W|\geq \frac{1}{10}h_l n^{1\slash 6}\sqrt{\frac{\beta}{2(m_2-m_1+1)}}\right)\nonumber\\
 &\leq& C\exp(-c h_l(L_2+l\Delta)).
\end{eqnarray}

Combining (\ref{P12.3})-(\ref{P12.5}), we conclude that there exist positive constants $C_0,c_0$ that only depend on $\beta$, such that
\begin{equation*}
    \mathbb{P}\left(n^{-1\slash 6}\left|\sum_{i=m_1}^{m_2}Q_i\right|\geq  \frac{1}{2}h_l\right)\leq C_0\exp\big(-c_0 n^{1\slash 72}\big)+C_0\exp(-c_0 h_l(L_2+l\Delta)). 
\end{equation*}
By the union bound, for any $l\in [0,L_0]\cap\mathbb{Z}$ such that $L_2+l\Delta<10 n^{13\slash 72}$,
\begin{eqnarray}\label{V0}
    &&   \mathbb{P}(\mathcal{V}_l^c)  \leq C_0 n^2\exp\big(-c_0 n^{1\slash 72}\big)+C_0 n^2\exp\big(-c_0 h_l(L_2+l\Delta)\big)\nonumber\\
    &\leq&    C_0 n^2\exp\big(-c_0 n^{1\slash 72}\big)+C_0 n^2\exp\left(-\frac{1}{4}c_0 (L_2+l\Delta)^{1\slash 4}\right),
\end{eqnarray}
where we use (\ref{Hliq}) in the last inequality.

By (\ref{Assu}), (\ref{VV}), (\ref{V0}), and the union bound, we have
\begin{eqnarray}
&& \mathbb{P}(\mathcal{V}^c)\leq \sum_{l\in [0,L_0]\cap\mathbb{Z}: L_2+l\Delta<10 n^{13\slash 72}}\mathbb{P}(\mathcal{V}_l^c) \nonumber\\
&\leq& C_0 n^2(L_0+1)\exp\big(-c_0 n^{1\slash 72}\big)+C_0 n^2\sum_{l=0}^{\infty}\exp\left(-\frac{1}{4}c_0 (L_2+l\Delta)^{1\slash 4}\right) \nonumber\\
&\leq& C\exp\big(-c n^{1\slash 72}\big)+C\exp\big(-c L_2^{1\slash 4}\big).
\end{eqnarray}
Therefore, by (\ref{Con}), we conclude that
\begin{equation}\label{Conc1}
    \mathbb{P}( \mathcal{B}_1(j)\cap\mathcal{B}_2(j)\cap\mathcal{U}_1\cap\mathcal{U}_2\cap\mathcal{E}_1 )\leq C\exp\big(-c n^{1\slash 72}\big)+C\exp\big(-c L_2^{1\slash 4}\big).
\end{equation}

\bigskip

\paragraph{Step 3}

In this step, we give an upper bound on 
$\mathbb{P}(\mathcal{B}_1(j)\cap\mathcal{B}_2(j)\cap\mathcal{U}_1\cap\mathcal{U}_2\cap\mathcal{E}_2)$. Throughout the rest of this step, we assume that the events $\mathcal{B}_1(j)$, $\mathcal{B}_2(j)$, $\mathcal{U}_1$, and $\mathcal{U}_2$ hold. 

\subparagraph{Sub-step 3.1}

In the following, we assume that the event $\mathcal{E}_{2,l}$ holds, where $l\in [0,L_0]\cap\mathbb{Z}$. Let $i_1,i_2\in I_l$ be such that $i_1<i_2$ and the properties (a)-(d) in the definition of $\mathcal{E}_{2,l}$ hold. 

\textbf{Part 3.1.1}

We first consider the case where $P_j(i_1)\geq \sqrt{i_1 n^{-1\slash 3}}\slash 2$. Note that either $g_j(i_1-1)>0,g_j(i_1)>0$ or $g_j(i_1-1)<0,g_j(i_1)<0$. Without loss of generality, we assume that $g_j(i_1-1)>0,g_j(i_1)>0$. 
In this case we also have 
\begin{equation}\label{geq}
    g_j(i_1-1)< g_j(i_1).
\end{equation}
Note that by (\ref{Eigen}), we have 
\begin{equation}\label{M11}
    \frac{Y_{n-i_1}}{\sqrt{\beta}}\frac{g_j(i_1+1)}{g_j(i_1)}= 2\sqrt{n}+n^{-1\slash 6}\tilde{\lambda}_j^{(n)}-\sqrt{\frac{2}{\beta}}\xi_{i_1}-\frac{Y_{n-i_1+1}}{\sqrt{\beta}}\frac{g_j(i_1-1)}{g_j(i_1)}.
\end{equation}
As the event $\mathcal{B}_1(j)$ holds and $K_3\leq n^{2\slash 3}\slash 3$, we have 
\begin{equation*}
    n^{-1\slash 6}\tilde{\lambda}_j^{(n)}\geq -n^{-1\slash 6}K_3\geq -\frac{1}{3}\sqrt{n}.
\end{equation*}
As the event $\mathcal{U}_2$ holds, by (\ref{psum}) and (\ref{Assu}), we have 
\begin{equation*}
    \sqrt{\frac{2}{\beta}}\left|\xi_{i_1}\right|\leq n^{1\slash 72}, \quad 0\leq \frac{Y_{n-i_1+1}}{\sqrt{\beta}}=\sqrt{n-i_1+1}+Z_{i_1-1}\leq \sqrt{n-i_1+1}+n^{1\slash 72}.
\end{equation*}
Hence by (\ref{geq}), we have
\begin{equation*}
   0\leq \frac{Y_{n-i_1+1}}{\sqrt{\beta}}\frac{g_j(i_1-1)}{g_j(i_1)}\leq \frac{Y_{n-i_1+1}}{\sqrt{\beta}}\leq \sqrt{n-i_1+1}+n^{1\slash 72}.
\end{equation*}
Plugging the above estimates into (\ref{M11}) and noting (\ref{Assu}), we have
\begin{equation}
    \frac{Y_{n-i_1}}{\sqrt{\beta}}\frac{g_j(i_1+1)}{g_j(i_1)}\geq \frac{5}{3}\sqrt{n}-2n^{1\slash 72}-\sqrt{n-i_1+1}\geq \frac{1}{2}\sqrt{n}>0,
\end{equation}
hence $g_j(i_1+1)>0$. 

Note that
\begin{equation*}
    P_j(i_1)-n^{1\slash 3} \frac{(g_j(i_1)-g_j(i_1-1))^2}{g_j(i_1)g_j(i_1-1)}=n^{1\slash 3}\frac{g_j(i_1)-g_j(i_1-1)}{g_j(i_1)}.
\end{equation*}
As the event $\mathcal{B}_1(j)$ holds and $g_j(i_1),g_j(i_1-1),g_j(i_1+1)>0$, by (\ref{PP3}), we have
\begin{eqnarray}\label{L10.1}
    P_j(i_1+1)&\geq& (n^{1\slash 3}+n^{-1\slash 6}\kappa_{i_1-1})\cdot \frac{g_j(i_1)-g_j(i_1-1)}{g_j(i_1)}-n^{-1\slash 3}K_3\nonumber\\
    && -n^{-1\slash 6}Q_{i_1}-n^{-1\slash 6}\kappa_{i_1}\frac{g_j(i_1+1)-g_j(i_1)}{g_j(i_1)}.
\end{eqnarray}
As
\begin{equation*}
    n^{1\slash 3}\frac{g_j(i_1)-g_j(i_1-1)}{g_j(i_1-1)}=P_j(i_1)\geq \frac{1}{2}\sqrt{i_1 n^{-1\slash 3}},
\end{equation*}
we have $g_j(i_1)\slash g_j(i_1-1)\geq 1+(i_1\slash n)^{1\slash 2}\slash 2$. Hence
\begin{equation*}
    \frac{g_j(i_1)-g_j(i_1-1)}{g_j(i_1)}=1-\frac{g_j(i_1-1)}{g_j(i_1)}\geq 1-\left(1+\frac{1}{2}\sqrt{\frac{i_1}{n}}\right)^{-1}\geq \frac{1}{3}\sqrt{\frac{i_1}{n}}.
\end{equation*}
As the event $\mathcal{U}_1$ holds, we have $|\kappa_{i_1-1}|,|\kappa_{i_1}|\leq n^{1\slash 8}$. By (\ref{Assu}) and the fact that $K_3\leq n^{2\slash 3}\slash 3$, we have $K_3^{1\slash 2}\leq L_2^{1  \slash 2}\slash 100$ and $K_3^{1\slash 2}\leq n^{1\slash 3}$, hence
\begin{equation*}
    K_3 n^{-1\slash 3}\leq \frac{1}{100}L_2^{1\slash 2}\leq \frac{1}{100}i_1^{1\slash 2}n^{-1\slash 6}.
\end{equation*}
As the events $\mathcal{U}_1$ and $\mathcal{U}_2$ hold, noting (\ref{Assu}), we have
\begin{equation*}
    |Q_{i_1}|\leq \sqrt{\frac{2}{\beta}}|\xi_{i_1}|+|\kappa_{i_1}|+|\kappa_{i_1-1}|\leq 3n^{1\slash 8},
\end{equation*}
hence
\begin{equation*}
    n^{-1\slash 6}|Q_{i_1}|\leq 3n^{-1\slash 24}\leq L_2^{1\slash 2}\slash 20\leq i_1^{1\slash 2}n^{-1\slash 6}\slash 20.
\end{equation*}
Plugging the above estimates into (\ref{L10.1}) and using (\ref{Assu}), we obtain that
\begin{eqnarray*}
  &&  P_j(i_1+1)+n^{-3\slash 8}|P_j(i_1+1)|\nonumber\\
  &\geq & \frac{1}{3}(1-n^{-3\slash 8} )i_1^{1\slash 2} n^{-1\slash 6} -n^{-1\slash 3}K_3-n^{-1\slash 6}Q_{i_1}\geq\frac{1}{4}\sqrt{i_1 n^{-1\slash 3}},
\end{eqnarray*}
which leads to
\begin{equation}
    P_j(i_1+1)\geq \frac{1}{6}\sqrt{(i_1+1) n^{-1\slash 3}}.
\end{equation}
We also note that $P_j(i_1+1)\neq \infty$ (as $g_j(i_1)>0$). However, by the property (c) in the definition of $\mathcal{E}_{2,l}$, either $P_j(i_1+1)=\infty$ or
\begin{equation*}
    P_j(i_1+1)<\phi_1(l)\sqrt{(i_1+1)n^{-1\slash 3}}\leq \frac{1}{8}\sqrt{(i_1+1)n^{-1\slash 3}}.
\end{equation*}
Hence we are led to a contradiction.

\textbf{Part 3.1.2}

By the argument in \textbf{Part 3.1.1}, we have $ P_j(i_1)<\sqrt{i_1 n^{-1\slash 3}}\slash 2$. Combining this with the definition of $\mathcal{E}_{l,2}$, we conclude that for any $i\in [i_1,i_2-1]\cap\mathbb{Z}$, $g_j(i-1)\neq 0$ and
\begin{equation}
    |P_j(i)|\leq \frac{1}{2}\sqrt{i n^{-1\slash 3}}. 
\end{equation}
Following the same argument in \textbf{Parts 2.1.2-2.1.3} and recalling the definition of $R_j$ from (\ref{Defin}), we obtain that $g_j(i_2)>0$ and 
\begin{equation}\label{P16.1}
    R_j(i+1)-R_j(i)\geq \frac{1}{40}(L_2+l\Delta)n^{-1\slash 3}-n^{-1\slash 6}Q_i,\quad\forall i\in [i_1,i_2-2]\cap\mathbb{Z},
\end{equation}
\begin{equation}\label{P16.2}
    n^{-3\slash 8}|R_j(i_2)|+R_j(i_2)-R_j(i_2-1)\geq \frac{1}{40}(L_2+l\Delta)n^{-1\slash 3}-n^{-1\slash 6}Q_{i_2-1}.
\end{equation}
As $P_j(i_1)\geq \phi_1(l)\sqrt{i_1 n^{-1\slash 3}}$, we have $R_j(i_1)\geq (\phi_1(l)-\phi_2(l+1))\sqrt{i_1 n^{-1\slash 3}}$. Hence by (\ref{P16.1}) and (\ref{P16.2}),
\begin{eqnarray}\label{P16.3}
   && n^{-3\slash 8}|R_j(i_2)|+R_j(i_2) \geq R_j(i_1)-n^{-1\slash 6}\sum_{i=i_1}^{i_2-1}Q_i\nonumber\\
   &\geq  & (\phi_1(l)-\phi_2(l+1))\sqrt{i_1 n^{-1\slash 3}}-n^{-1\slash 6}\sum_{i=i_1}^{i_2-1}Q_i.
\end{eqnarray}
As $P_j(i_2)<\phi_2(l+1)\sqrt{i_2n^{-1\slash 3}}$, we have $R_j(i_2)<0$. Hence by (\ref{P16.3}),
\begin{eqnarray}
    n^{-1\slash 6}\sum_{i=i_1}^{i_2-1}Q_i&\geq& (\phi_1(l)-\phi_2(l+1))\sqrt{i_1 n^{-1\slash 3}}\nonumber\\
    &\geq& (\phi_1(l)-\phi_2(l+1))\sqrt{L_2+l\Delta}.
\end{eqnarray}

For any $l\in [0,L_0]\cap\mathbb{Z}$, we let $\mathcal{V}'_l$ be the event that there exist $i_1,i_2\in I_l$ such that $i_1<i_2$ and 
\begin{equation*}
    n^{-1\slash 6}\sum_{i=i_1}^{i_2-1}Q_i\geq (\phi_1(l)-\phi_2(l+1))\sqrt{L_2+l\Delta}.
\end{equation*}
We conclude by the preceding argument that
\begin{equation}
    \mathcal{B}_1(j)\cap\mathcal{B}_2(j)\cap\mathcal{U}_1\cap\mathcal{U}_2\cap \mathcal{E}_{2,l}\subseteq \mathcal{V}'_l. 
\end{equation}
Letting $\mathcal{V}':=\bigcup_{l=0}^{L_0}\mathcal{V}_l'$, we have
\begin{equation}\label{VVV}
    \mathcal{B}_1(j)\cap\mathcal{B}_2(j)\cap\mathcal{U}_1\cap\mathcal{U}_2\cap \mathcal{E}_2\subseteq \mathcal{V}'.
\end{equation}

\subparagraph{Sub-step 3.2}

In this sub-step, we give an upper bound on $\mathbb{P}(\mathcal{V}')$, which via (\ref{VVV}) yields an upper bound on $\mathbb{P}(\mathcal{B}_1(j)\cap\mathcal{B}_2(j)\cap\mathcal{U}_1\cap\mathcal{U}_2\cap \mathcal{E}_2)$.

Consider any $l\in [0,L_0]\cap\mathbb{Z}$ and any $m_1,m_2\in\mathbb{N}_{+}$ such that $m_1\leq m_2$ and $m_1,m_2+1\in I_l$. Arguing similarly as in \textbf{Sub-step 2.2}, we obtain (\ref{P12.1}) and (\ref{P12.2}) for the current choice of $m_1,m_2$. In (\ref{P12.1}) and (\ref{P12.2}), we take
\begin{equation*}
    t=\frac{1}{8}(\phi_1(l)-\phi_2(l+1))\sqrt{L_2+l\Delta} n^{1\slash 6}\geq \frac{1}{256}(L_2+l\Delta)^{1 \slash 4} n^{1\slash 6}. 
\end{equation*}
Considering both cases in (\ref{P12.1}) and (\ref{P12.2}) and noting that $\Delta\leq L_2^{1\slash 4}$ and $m_2-m_1+1\leq 2\Delta n^{1\slash 3}$, we obtain that
\begin{eqnarray*}
  &&  \mathbb{P}\left(n^{-1\slash 6}\left|\sum_{i=m_1}^{m_2}\kappa_i\right|\geq \frac{1}{4}(\phi_1(l)-\phi_2(l+1))\sqrt{L_2+l\Delta}\right)\nonumber\\
  &\leq& C\exp\big(-c n^{1\slash 72}\big)+ C\exp\big(-c(L_2+l\Delta)^{1\slash 4}\big),
\end{eqnarray*}
Similarly, we have
\begin{eqnarray*}
  &&  \mathbb{P}\left(n^{-1\slash 6}\left|\sum_{i=m_1-1}^{m_2-1}\kappa_i\right|\geq \frac{1}{4}(\phi_1(l)-\phi_2(l+1))\sqrt{L_2+l\Delta}\right)\nonumber\\
  &\leq& C\exp\big(-c n^{1\slash 72}\big)+ C\exp\big(-c(L_2+l\Delta)^{1\slash 4}\big).
\end{eqnarray*}
Arguing similarly as in \textbf{Sub-step 2.2}, we obtain that
\begin{eqnarray*}
  &&  \mathbb{P}\left(n^{-1\slash 6}\sqrt{\frac{2}{\beta}}\left|\sum_{i=m_1}^{m_2}\xi_i\right|\geq \frac{1}{4}(\phi_1(l)-\phi_2(l+1))\sqrt{L_2+l\Delta}\right)\\
  &\leq& C\exp\big(-c(L_2+l\Delta)^{1\slash 4}\big).
\end{eqnarray*}
Noting (\ref{Qform}), we have
\begin{eqnarray}\label{P20.1}
  &&  \mathbb{P}\left(n^{-1\slash 6}\left|\sum_{i=m_1}^{m_2} Q_i\right|\geq (\phi_1(l)-\phi_2(l+1))\sqrt{L_2+l\Delta}\right)\nonumber\\
  &\leq& C_0\exp\big(-c_0 n^{1\slash 72}\big)+C_0\exp\big(-c_0 (L_2+l\Delta)^{1\slash 4}\big),
\end{eqnarray}
where $C_0,c_0$ are positive constants that only depend on $\beta$. 

By (\ref{P20.1}) and the union bound, for any $l\in [0,L_0]\cap\mathbb{Z}$, we have
\begin{equation}
    \mathbb{P}\left(\mathcal{V}'_l\right)\leq C_0 n^2\exp\big(-c_0 n^{1\slash 72}\big)+C_0 n^2\exp\big(-c_0 (L_2+l\Delta)^{1\slash 4}\big).
\end{equation}
By the union bound, noting (\ref{Assu}), we have
\begin{eqnarray}
    \mathbb{P}\left(\mathcal{V}'\right)
    &\leq &  C_0 n^2(L_0+1)\exp\big(-c_0 n^{1\slash 72}\big)+C_0 n^2\sum_{l=0}^{\infty}\exp\big(-c_0 (L_2+l\Delta)^{1\slash 4}\big)\nonumber\\
    &\leq& C\exp\big(-c n^{1\slash 72}\big)+C\exp\big(-c L_2^{1\slash 4}\big).
\end{eqnarray}
Hence by (\ref{VVV}), we conclude that
\begin{equation}\label{Conc2}
  \mathbb{P}\left(\mathcal{B}_1(j)\cap\mathcal{B}_2(j)\cap\mathcal{U}_1\cap\mathcal{U}_2\cap\mathcal{E}_2\right)\leq  C\exp\big(-c n^{1\slash 72}\big)+C\exp\big(-c L_2^{1\slash 4}\big).
\end{equation}

\paragraph{Step 4}

In this step, we finish the proof of the lemma. 

Let $\mathcal{T}$ be the event that there exists some $i_0\in [L_2 n^{1\slash 3},n]\cap\mathbb{Z}$, such that for any $i\in [i_0,n]\cap\mathbb{Z}$, $P_j(i)\neq\infty$ and $P_j(i)\geq \sqrt{i n^{-1\slash 3}}\slash 16$. We show $\mathcal{E}_0\cap\mathcal{B}_3(j)\subseteq \mathcal{T}$ as follows. Suppose that the events $\mathcal{E}_0$ and $\mathcal{B}_3(j)$ hold. Then there exists $i_0\in [L_2 n^{1\slash 3},n]\cap\mathbb{Z}$, such that $P_j(i_0)\neq \infty$ and $P_j(i_0)\geq\sqrt{i_0 n^{-1\slash 3}}\slash 8$. Note that there exists $l_0\in [0,L_0]\cap\mathbb{Z}$, such that $i_0\in I_{l_0}$. As $\phi_2(l_0)\leq 1\slash 8$, we have that $P_j(i_0)\geq \phi_2(l_0)\sqrt{i_0 n^{-1\slash 3}}$. By the definition of $\mathcal{E}_0$ and induction, we can deduce that for any $i\in I_{l_0}$ such that $i\geq i_0$, $P_j(i)\neq\infty$ and $P_j(i)\geq \phi_2(l_0+1)\sqrt{i n^{-1\slash 3}}$; for any $l\in [l_0+1,L_0]\cap\mathbb{Z}$ and any $i\in I_l$, $P_j(i)\neq \infty$ and $P_j(i)\geq \phi_2(l+1)\sqrt{i n^{-1\slash 3}}$. As $\phi_2(l)\geq 1\slash 16$ for any $l\in\mathbb{N}$, we have $P_j(i)\geq \sqrt{i n^{-1\slash 3}}\slash 16$ for any $i\in [i_0,n]\cap\mathbb{Z}$. Hence the event $\mathcal{T}$ holds. We conclude that
\begin{equation}\label{P14.1}
    \mathcal{E}_0\cap\mathcal{B}_3(j)\subseteq \mathcal{T}.
\end{equation}

Suppose that the event $\mathcal{T}$ holds. Then we have 
\begin{equation*}
    P_j(n)\neq \infty, \quad P_j(n)\geq \frac{1}{16}n^{1\slash 3},
\end{equation*}
which leads to $g_j(n-1)\neq 0$ and $(g_j(n)-g_j(n-1))\slash g_j(n-1)\geq 1\slash 16$. Hence if $g_j(n-1)>0$, then $g_j(n)>0$ and $g_j(n-1)\leq g_j(n)$; if $g_j(n-1)<0$, then $g_j(n)<0$ and $g_j(n-1)\geq g_j(n)$. Noting the property (c) in the definition of $\mathcal{B}_2(j)$, we conclude that the event $\mathcal{B}_2(j)$ does not hold. Hence
\begin{equation}\label{P14.2}
    \mathcal{T}\subseteq \mathcal{B}_2(j)^c.
\end{equation}
Combining (\ref{P14.1}) and (\ref{P14.2}) gives
\begin{equation}
    \mathcal{E}_0\cap\mathcal{B}_3(j)\subseteq \mathcal{B}_2(j)^c.
\end{equation}
Hence
\begin{equation}\label{Pa1}
    \mathcal{B}_3(j)\cap\mathcal{B}_1(j)\cap\mathcal{B}_2(j)\cap\mathcal{U}_1\cap\mathcal{U}_2\cap\mathcal{E}_0=\emptyset.
\end{equation}

Now by (\ref{E10.2}), we have
\begin{eqnarray*}
 &&   \mathcal{B}_3(j)\cap\mathcal{B}_1(j)\cap\mathcal{B}_2(j)\cap\mathcal{U}_1\cap\mathcal{U}_2\cap \mathcal{E}_0^c \nonumber \\
&\subseteq & \left(\mathcal{B}_1(j)\cap\mathcal{B}_2(j)\cap\mathcal{U}_1\cap\mathcal{U}_2\cap\mathcal{E}_1\right) \cup \left(\mathcal{B}_1(j)\cap\mathcal{B}_2(j)\cap\mathcal{U}_1\cap\mathcal{U}_2\cap\mathcal{E}_2\right).
\end{eqnarray*}
By (\ref{Conc1}), (\ref{Conc2}), and the union bound, we have
\begin{eqnarray}\label{Pa2}
  &&  \mathbb{P}\left(\mathcal{B}_3(j)\cap\mathcal{B}_1(j)\cap\mathcal{B}_2(j)\cap\mathcal{U}_1\cap\mathcal{U}_2\cap\mathcal{E}_0^c\right)\nonumber\\
  &\leq& C\exp\big(-c n^{1\slash 72}\big)+C\exp\big(-c L_2^{1\slash 4}\big).
\end{eqnarray}

Combining (\ref{Pa1}) and (\ref{Pa2}), we conclude that
\begin{eqnarray}
   && \mathbb{P}\left(\mathcal{B}_3(j)\cap\mathcal{B}_1(j)\cap\mathcal{B}_2(j)\cap\mathcal{U}_1\cap\mathcal{U}_2\right)\nonumber\\
   &  \leq & C\exp\big(-c n^{1\slash 72}\big)+C\exp\big(-c L_2^{1\slash 4}\big).
\end{eqnarray}

\end{proof}

\begin{lemma}\label{PL4}
The exist positive constants $C,c,C_1,C_2$ (with $C_2\geq 1$) that only depend on $\beta$, such that when $L_2\geq \max\big\{C_1,C_2K_3,(\log{n})^8\big\}$, for any $j\in [k]$,
\begin{eqnarray}\label{ConCPL4}
 &&   \mathbb{P}(\mathcal{B}_3(j)^c\cap\mathcal{B}_4(j)\cap\mathcal{B}_1(j)\cap\mathcal{B}_2(j)\cap\mathcal{U}_1\cap\mathcal{U}_2)\nonumber\\
 &\leq & C\exp\big(-cn^{1\slash 72}\big)+C\exp\big(-c L_2^{3\slash 2}\big).
\end{eqnarray}
\end{lemma}
\begin{proof}

We fix an arbitrary $j\in [k]$ throughout the proof. We adopt the notations used in the proof of Lemma \ref{PL3}, and assume that $L_2\leq n^{2\slash 3}$ and (\ref{Assu}) holds (if $L_2>n^{2\slash 3}$, then $\mathcal{B}_4(j)=\emptyset$ and (\ref{ConCPL4}) automatically holds).

Recall the definitions of $L_0$ and $\{I_l\}_{l\in [0,L_0]\cap\mathbb{Z}}$ from the proof of Lemma \ref{PL3}. For any $l\in [0,L_0]\cap\mathbb{Z}$, we define the following events:
\begin{itemize}
    \item $\mathcal{G}_{1,l}$: the event that there exists $(i_1,i_2)\in [n]^2$, such that the following properties hold:
\begin{itemize}
        \item[(a)] $i_1\in I_l$, $1\leq i_2-i_1\leq 40 n^{1\slash 3}(L_2+l \Delta)^{-1\slash 2}$;
        \item[(b)] $P_j(i_1)\in \big[-\sqrt{i_1 n^{-1\slash 3}}\slash 8,\sqrt{i_1 n^{-1\slash 3}}\slash 8\big]$;
        \item[(c)] $P_j(i)\in \big[-\sqrt{i n^{-1\slash 3}}\slash 2,\sqrt{i n^{-1\slash 3}}\slash 8\big]$ for any $i\in [i_1+1,i_2-1]\cap\mathbb{Z}$;
        \item[(d)] $P_j(i_2)=\infty$ or $P_j(i_2)<-\sqrt{i_2n^{-1\slash 3}}\slash 2$.
\end{itemize}
\item $\mathcal{G}_{2,l}$: the event that there exists $(i_1,i_2)\in [n]^2$, such that the following properties hold:
\begin{itemize}
         \item[(a)] $i_1\in I_l$, $1\leq i_2-i_1\leq 40 n^{1\slash 3}(L_2+l\Delta)^{-1\slash 2}$;
        \item[(b)] $n^{-1\slash 6}\sum_{i=i_1}^{i_2-1}Q_i\geq \sqrt{L_2+l\Delta}\slash 4$ (see (\ref{Qd}) for the definition of $Q_i$). 
      \end{itemize} 
\end{itemize}
We also let $\mathcal{G}_1:=\bigcup_{l=0}^{L_0}\mathcal{G}_{1,l}$ and $\mathcal{G}_2:=\bigcup_{l=0}^{L_0}\mathcal{G}_{2,l}$.

Throughout the rest of the proof, we assume that the events $\mathcal{B}_1(j)$, $\mathcal{B}_2(j)$, $\mathcal{U}_1$, and $\mathcal{U}_2$ hold. 

Consider an arbitrary $l\in [0,L_0]\cap\mathbb{Z}$. Suppose that the event $\mathcal{G}_{1,l}$ holds and $i_1,i_2$ satisfy the properties in the definition of $\mathcal{G}_{1,l}$. Arguing similarly as in \textbf{Parts 2.1.2 and 2.1.3} in the proof of Lemma \ref{PL3}, we obtain that $P_j(i_2)\neq \infty$ and  
\begin{equation*}
    n^{-3\slash 8}|P_j(i_2)|+P_j(i_2)-P_j(i_1)\geq \frac{1}{20}(L_2+l\Delta)n^{-1\slash 3}(i_2-i_1)-n^{-1\slash 6}\sum_{i=i_1}^{i_2-1} Q_i. 
\end{equation*}
By (\ref{Assu}) and the definition of $\mathcal{G}_{1,l}$, we have
\begin{equation*}
    n^{-3\slash 8}|P_j(i_2)|+P_j(i_2)-P_j(i_1)\leq -\frac{1}{4}\sqrt{i_2 n^{-1\slash 3}}\leq -\frac{1}{4}\sqrt{L_2+l\Delta}.
\end{equation*}
Hence we have
\begin{equation*}
    n^{-1\slash 6}\sum_{i=i_1}^{i_2-1}Q_i\geq \frac{1}{4}\sqrt{L_2+l\Delta}.
\end{equation*}
Therefore, for any $l\in [0,L_0]\cap\mathbb{Z}$, $\mathcal{G}_{1,l}\cap\mathcal{B}_1(j)\cap\mathcal{B}_2(j)\cap\mathcal{U}_1\cap\mathcal{U}_2\subseteq \mathcal{G}_{2,l}$, which leads to
\begin{equation}\label{E.21.6}
    \mathcal{G}_1\cap \mathcal{B}_1(j)\cap\mathcal{B}_2(j)\cap\mathcal{U}_1\cap\mathcal{U}_2\subseteq \mathcal{G}_2.
\end{equation}

Consider any $l\in [0,L_0]\cap\mathbb{Z}$ and any $(m_1,m_2)\in [n]^2$ such that
\begin{equation*}
    m_1\in I_l, \quad 1\leq m_2-m_1\leq 40 n^{1\slash 3}(L_2+l\Delta)^{-1\slash 2}.
\end{equation*}
By Lemma \ref{ZL3} (considering both cases), noting (\ref{psum}), we obtain that \begin{eqnarray}\label{E21.1}
  &&  \mathbb{P}\left(\left|n^{-1\slash 6} \sum_{i=m_1}^{m_2-1}\kappa_i\right|\geq \frac{1}{12}\sqrt{L_2+l\Delta}  \right)\nonumber\\
   & \leq &  C\exp\big(-c n^{1\slash 72}\big)+C\exp\big(-c(L_2+l\Delta)^{3\slash 2}\big).
\end{eqnarray}
Similarly, (\ref{E21.1}) holds with $m_1,m_2-1$ replaced by $m_1-1,m_2-2$. As $\{\xi_i\}_{i=1}^n$ are i.i.d. $N(0,1)$ random variables, we have $\sum_{i=m_1}^{m_2-1}\xi_i\sim N(0,m_2-m_1)$. Hence letting $W$ be a $N(0,1)$ random variable, we have
\begin{eqnarray*}
&& \mathbb{P}\left(\left|n^{-1\slash 6}\sum_{i=m_1}^{m_2-1}\sqrt{\frac{2}{\beta}}\xi_i\right|\geq\frac{1}{24}\sqrt{L_2+l\Delta}\right)\nonumber\\
&=&\mathbb{P}\left(|W|\geq \frac{1}{24} n^{1\slash 6}\sqrt{\frac{\beta(L_2+l\Delta)}{2(m_2-m_1)}}\right)
   \leq C\exp\big(-c (L_2+l\Delta)^{3\slash 2}\big).
\end{eqnarray*}
Hence by (\ref{Qform}),
\begin{eqnarray*}
  && \mathbb{P}\left(\left|n^{-1\slash 6} \sum_{i=m_1}^{m_2-1}Q_i\right|\geq \frac{1}{4}\sqrt{L_2+l\Delta}\right)\nonumber\\
  &\leq &  C_0\exp\big(-c_0 n^{1\slash 72}\big)+ C_0\exp\big(-c_0 (L_2+l\Delta)^{3\slash 2}\big),
\end{eqnarray*}
where $C_0,c_0$ are positive constants that only depend on $\beta$. By the union bound, we have
\begin{eqnarray*}
  &&  \mathbb{P}(\mathcal{G}_{2,l}\cap\mathcal{B}_1(j)\cap\mathcal{B}_2(j)\cap\mathcal{U}_1\cap\mathcal{U}_2)\\
  &\leq &  C_0 n^2\exp\big(-c_0 n^{1\slash 72}\big)+ C_0 n^2\exp\big(-c_0 (L_2+l\Delta)^{3\slash 2}\big),
\end{eqnarray*}
\begin{eqnarray}\label{E22.8}
 &&   \mathbb{P}(\mathcal{G}_2\cap\mathcal{B}_1(j)\cap\mathcal{B}_2(j)\cap\mathcal{U}_1\cap\mathcal{U}_2)\nonumber\\
 &\leq & C_0 n^3 \exp\big(-c_0 n^{1\slash 72}\big)+\sum_{l=0}^{\infty}C_0 n^2 \exp\big(-c_0 (L_2+l\Delta)^{3\slash 2}\big)\nonumber\\
 &\leq& C\exp\big(-c n^{1\slash 72}\big)+C\exp\big(-c L_2^{3\slash 2}\big).
\end{eqnarray}

Below we assume that the events $\mathcal{G}_2^c$, $\mathcal{B}_3(j)^c$, and $\mathcal{B}_4(j)$ hold. By (\ref{E.21.6}), the event $\mathcal{G}_1^c$ holds. From the definition of $\mathcal{B}_4(j)$, there exist $l\in [0,L_0]\cap\mathbb{Z}$ and $i_1\in [L_2 n^{1\slash 3},n-50]\cap I_l$, such that
\begin{equation*}
    P_j(i_1)\neq\infty, \quad P_j(i_1)\geq -\sqrt{i_1 n^{-1\slash 3}}\slash 8.
\end{equation*}
As the event $\mathcal{B}_3(j)^c$ holds, for any $i\in [i_1,n]\cap\mathbb{Z}$,
\begin{equation*}
    \text{ either }P_j(i)=\infty\text{ or }P_j(i)\leq \sqrt{i n^{-1\slash 3}}\slash 8.
\end{equation*}
As the event $\mathcal{G}_1^c$ holds, for any $i\in [i_1,i_1+40n^{1\slash 3}(L_2+l\Delta)^{-1\slash 2}]\cap [n]$, we have
\begin{equation}\label{neqq3}
    -\sqrt{i n^{-1\slash 3}}\slash 2\leq P_j(i)\leq \sqrt{i n^{-1\slash 3}}\slash 8.
\end{equation}
Following the argument as given in \textbf{Parts 2.1.2 and 2.1.3} in the proof of Lemma \ref{PL3}, for any $i_2\in [n]$ such that $i_2\in [i_1+1,i_1+40n^{1\slash 3}(L_2+l\Delta)^{-1\slash 2}]$, we have $P_j(i_2)\neq \infty$ and 
\begin{equation}\label{E.22.3}
   P_j(i_2)-P_j(i_1)\geq \frac{1}{20}(L_2+l\Delta) n^{-1\slash 3}(i_2-i_1)-n^{-1\slash 6} \sum_{i=i_1}^{i_2-1}Q_i.
\end{equation}

First consider the case where $i_1+40n^{1\slash 3}(L_2+l\Delta)^{-1\slash 2}>n$. As $i_1\in I_l$, we have $i_1\leq (L_2+(l+2)\Delta)n^{1\slash 3}$. Hence
\begin{equation*}\label{E.20.2}
    (L_2+(l+2)\Delta) n^{1\slash 3}+40n^{1\slash 3}(L_2+l\Delta)^{-1\slash 2}>n.
\end{equation*}
If $L_2+(l+2)\Delta\geq (9\slash 10)n^{2\slash 3}$, as $\Delta\leq L_2^{1\slash 4}\leq L_2\slash 1000$ (by (\ref{Assu})), we have
\begin{equation*}
    L_2+l\Delta\geq \frac{499}{500}(L_2+(l+2)\Delta)\geq \frac{499}{500}\cdot\frac{9}{10}n^{2\slash 3}\geq \frac{4}{5} n^{2\slash 3},
\end{equation*}
hence $40n^{1\slash 3} (L_2+l\Delta)^{-1\slash 2}\leq 48$ and $i_1>n-48$. This results in a contradiction, as $i_1\leq n-50$. If $L_2+(l+2)\Delta<(9\slash 10)n^{2\slash 3}$, by (\ref{Assu}), we have
\begin{eqnarray*}
    i_1+40n^{1\slash 3}(L_2+l\Delta)^{-1\slash 2}&\leq& (L_2+(l+2)\Delta)n^{1\slash 3} + 40n^{1\slash 3}(L_2+l\Delta)^{-1\slash 2}\nonumber\\
    &\leq&   \frac{9}{10} n+\frac{1}{10} n^{1\slash 3}\leq n,
\end{eqnarray*}
which again leads to a contradiction. Therefore, we conclude that
\begin{equation}\label{E.21.5}
    i_1+40n^{1\slash 3}(L_2+l\Delta)^{-1\slash 2}\leq n.
\end{equation}

If $20 n^{1  \slash 3}(L_2+l\Delta)^{-1\slash 2}\geq 1$, by (\ref{E.21.5}), we have
\begin{equation}\label{E.22.2}
    [i_1+20n^{1\slash 3}(L_2+l\Delta)^{-1\slash 2},i_1+40 n^{1\slash 3}(L_2+l\Delta)^{-1\slash 2}]\cap [n]\neq \emptyset.
\end{equation}
Take any $i_2\in [i_1+20n^{1\slash 3}(L_2+l\Delta)^{-1\slash 2},i_1+40 n^{1\slash 3}(L_2+l\Delta)^{-1\slash 2}]\cap [n]$. Note that by (\ref{Assu}),
\begin{equation}\label{E.22.1}
    i_2\leq i_1+40 n^{1\slash 3}(L_2+l\Delta)^{-1\slash 2}\leq i_1+\Delta n^{1\slash 3}\leq (L_2+(l+3)\Delta) n^{1\slash 3}.
\end{equation}
By (\ref{Assu}) again, we have
\begin{equation}\label{E.22.7}
    L_2+(l+3)\Delta\leq L_2+l\Delta+\frac{3}{1000}L_2\leq \frac{1003}{1000}(L_2+l\Delta).
\end{equation}
By (\ref{neqq3}), we have $P_j(i_2)\leq \sqrt{i_2 n^{-1\slash 3}}\slash 8$. By (\ref{E.22.3}), (\ref{E.22.1}), (\ref{E.22.7}), and the fact that $P_j(i_1)\geq -\sqrt{i_1 n^{-1\slash 3}}\slash 8$, we have 
\begin{eqnarray*}
  &&  n^{-1\slash 6}\sum_{i=i_1}^{i_2-1}Q_i\geq \sqrt{L_2+l\Delta}+P_j(i_1)-P_j(i_2)\\
  &\geq& \sqrt{L_2+l\Delta}-\frac{1}{4}\sqrt{L_2+(l+3)\Delta}>\frac{1}{4}\sqrt{L_2+l\Delta}.
\end{eqnarray*}
As the event $\mathcal{G}_2^c$ holds, we are led to a contradiction.

Now consider the case where $20 n^{1\slash 3}(L_2+l\Delta)^{-1\slash 2}<1$. Note that
\begin{equation}\label{neqq4}
    \frac{1}{20}(L_2+l\Delta)n^{-1\slash 3}>\sqrt{L_2+l\Delta}.
\end{equation}
Following the argument in \textbf{Parts 2.1.2 and 2.1.3} in the proof of Lemma \ref{PL3}, we have $P_j(i_1+1)\neq\infty$ and 
\begin{equation}\label{E.22.10}
   n^{-3\slash 8}|P_j(i_1+1)|+P_j(i_1+1)-P_j(i_1)\geq \frac{1}{20}(L_2+l\Delta) n^{-1\slash 3}-n^{-1\slash 6} Q_{i_1}.
\end{equation}
As the events $\mathcal{U}_1$ and $\mathcal{U}_2$ hold, we have $|Q_{i_1}|\leq 3n^{1\slash 8}$. Hence by (\ref{Assu}) and (\ref{E.22.7})-(\ref{E.22.10}), we have
\begin{eqnarray*}
 && n^{-3\slash 8}|P_j(i_1+1)|+P_j(i_1+1)\\
 &\geq& -\frac{1}{8}\sqrt{i_1 n^{-1\slash 3}}+\sqrt{L_2+l\Delta}-3\geq \frac{1}{4}\sqrt{(i_1+1) n^{-1\slash 3}},
\end{eqnarray*}
which via (\ref{Assu}) leads to $P_j(i_1+1)\geq \sqrt{(i_1+1)n^{-1\slash 3}}\slash 8$. As the event $\mathcal{B}_3(j)^c$ holds, we are led to a contradiction.

We conclude that
\begin{equation}
    \mathcal{B}_3(j)^c\cap\mathcal{B}_4(j)\cap\mathcal{B}_1(j)\cap\mathcal{B}_2(j)\cap\mathcal{U}_1\cap\mathcal{U}_2\subseteq \mathcal{G}_2\cap\mathcal{B}_1(j)\cap\mathcal{B}_2(j)\cap\mathcal{U}_1\cap\mathcal{U}_2.
\end{equation}
Hence by (\ref{E22.8}), we obtain the conclusion of the lemma. 

\end{proof}

\begin{lemma}\label{PL5}
Recall the definitions of $\{\tilde{\lambda}_{j}^{(n)}\}_{j=1}^n$ and $\{\gamma_j^{(n)}\}_{j=1}^n$ from Section \ref{Sect.n2} and the condition (\ref{Connk}) on $n$ and $k$. There exist positive constants $N_0,C,c$ that only depend on $\beta$ and $e_0$, such that the following holds. If $n\geq N_0$, then for any $j\in [k]$ and any $T\geq 10$, we have
\begin{equation}
    \mathbb{P}\big(\tilde{\lambda}_j^{(n)}\leq -k^{11\slash 3} T\big)\leq C\exp\big(-c \min\{n^{1\slash 4},k T^{1\slash 4}\}\big).
\end{equation}
\end{lemma}
\begin{proof}

We apply Proposition \ref{Edge} and Lemma \ref{gamma} in Sections \ref{Sect.4} and \ref{Sect.5}. Note that the proofs of these results are independent of the rest part of this article. We assume that $n\geq 4$ and $T\geq 10$ in the following. By (\ref{Connk}), $n^{e_0}\leq k\leq n\slash 2$. By Lemma \ref{gamma}, we have $|\gamma_k^{(n)}-2|\leq 4(k\slash n)^{2\slash 3}$. In Proposition \ref{Edge}, we take $a_0=e_0\in (0,1)$, $a=\min\{1,\log(k^{4}T\slash 2)\slash \log{n}\}\in [a_0,1]$, $\epsilon=1\slash 2$, and conclude that there exist positive constants $N_1,C,c$ that only depend on $\beta$ and $e_0$, such that for any $n\geq \max\{N_1,4\}$ and $j\in [k]$,
\begin{eqnarray*}
  &&  \mathbb{P}\big(\tilde{\lambda}_j^{(n)}\leq -k^{11\slash 3}T\big)\leq \mathbb{P}\big(\tilde{\lambda}_k^{(n)}\leq -k^{11\slash 3}T\big)\\
  &\leq&  \mathbb{P}\left(\big|n^{-1\slash 2}\lambda_k^{(n)}-\gamma_k^{(n)}\big|\geq \frac{1}{2}k^{11\slash 3}T n^{-2\slash 3}\right)\nonumber\\
  &\leq& C\exp\big(-c \min\{n^{1\slash 4},k T^{1\slash 4}\}\big).
\end{eqnarray*}

\end{proof}

Now we finish the proof of Proposition \ref{DiscreteDecay}.

\begin{proof}[Proof of Proposition \ref{DiscreteDecay}]

Recall the condition (\ref{Connk}) on $n$ and $k$. Below we assume that $k\geq 10$. For any $j\in [k]$, let
\begin{equation*}
    \mathcal{B}_0(j):=\mathcal{B}_1(j)\cap\mathcal{B}_3(j)^c\cap\mathcal{B}_4(j)^c\cap\mathcal{U}_1\cap\mathcal{U}_2.
\end{equation*}
We take
\begin{equation*}
    L_2=n^{-1\slash 3}\lfloor n^{1\slash 3} k^{12}\rfloor, \quad \Delta=n^{-1\slash 3} (\lfloor n^{1\slash 3} k^3 \rfloor-1), \quad K_3=k^{12}\slash (2C_2),
\end{equation*}
where $C_2\geq 1$ is taken to be the maximum of the constant $C_2$ appearing in Lemma \ref{PL3} and the constant $C_2$ appearing in Lemma \ref{PL4}. Note that $L_2,\Delta,K_3>0$ and $n^{1\slash 3}L_2,n^{1\slash 3} \Delta\in\mathbb{Z}$. As   
\begin{eqnarray*}
    \Delta^4 &\leq& \big(k^3-n^{-1\slash 3}\big)^4 \leq k^{12}\big(1-n^{-1\slash 3}k^{-3}\big)\leq k^{12}-n^{-1\slash 3}\leq L_2,
\end{eqnarray*}
\begin{eqnarray*}
    \Delta^4\geq \big(k^3-2n^{-1\slash 3}\big)^4\geq \big(  2^{-1\slash 2} k^3 \big)^4 = \frac{1}{4}k^{12}\geq \frac{1}{4}L_2,
\end{eqnarray*}
we have $L_2^{1\slash 4}\slash 2\leq \Delta\leq L_2^{1\slash 4}$. 

By Lemma \ref{PL5}, there exist positive constants $N_1,C_1,c_1$ that only depend on $\beta$ and $e_0$, such that the following holds: If $k\geq N_1$, then for any $j\in [k]$, 
\begin{equation}\label{Bound1}
    \mathbb{P}(\mathcal{B}_1(j)^c)\leq C_1\exp\big(-c_1 k^3\big).
\end{equation}
Now note that
\begin{eqnarray*}
   \mathcal{B}_0(j)^c&=&\mathcal{B}_1(j)^c\cup\mathcal{B}_3(j)\cup\mathcal{B}_4(j)\cup(\mathcal{U}_1\cap\mathcal{U}_2)^c\\
    &\subseteq & \left(\mathcal{U}_1\cap\mathcal{U}_2\right)^c\cup \mathcal{B}_1(j)^c\cup
    \left(\mathcal{B}_1(j)\cap\mathcal{B}_2(j)^c\right)\\
    && \cup \left(\mathcal{B}_3(j)\cap\mathcal{B}_1(j)\cap\mathcal{B}_2(j)\cap\mathcal{U}_1\cap\mathcal{U}_2\right) \\
    &&  \cup \left(\mathcal{B}_3(j)^c\cap\mathcal{B}_4(j)\cap\mathcal{B}_1(j)\cap\mathcal{B}_2(j)\cap\mathcal{U}_1\cap\mathcal{U}_2\right).
\end{eqnarray*}

Hence by Lemmas \ref{PL1}-\ref{PL4}, (\ref{Connk}), (\ref{Bound1}), and the union bound, there exist positive constants $N_0,C_0,c_0$ that only depend on $\beta$ and $e_0$, such that the following holds: If $k\geq N_0$, then for any $j\in [k]$,
\begin{equation*}
    \mathbb{P}(\mathcal{B}_0(j)^c)\leq C_0\exp(-c_0 k^3).
\end{equation*}
Let $\mathcal{B}_0:=\bigcap_{j=1}^k \mathcal{B}_0(j)$. By the union bound, if $k\geq N_0$, then
\begin{equation}\label{Neq}
    \mathbb{P}(\mathcal{B}_0^c)\leq C\exp(-c k^3),
\end{equation}
where $C,c$ are positive constants that only depend on $\beta,e_0$.

We take $K=\max\{N_0,10,(2\beta)^{-1}\}$. In the following, we assume that $k\geq K$ and the event $\mathcal{B}_0$ holds. Note that this combined with (\ref{Connk}) implies
\begin{equation}\label{Conn}
    k\geq \max\{10,(2\beta)^{-1}\}, \quad n\geq 10^8.
\end{equation}

Suppose that there exist some $j\in [k]$ and $i_0\in [n^{1\slash 3}k^{12}-1,n]\cap\mathbb{Z}$, such that $g_j(i_0)=0$. If $g_j(n)=0$, then by (\ref{EigenN}), $g_j(n-1)=0$, which via (\ref{Eigen}) leads to $g_j(i)=0$ for any $i\in [n]$. We are led to a contradiction, as $g_j$ is an eigenvector. Hence $g_j(n)\neq 0$. If $g_j(n-1)=0$, then by (\ref{Connk}), (\ref{EigenN}), (\ref{Conn}), and the fact that the events $\mathcal{B}_1(j)$ and $\mathcal{U}_2$ hold, we have
\begin{equation*}
   2\sqrt{n}= \sqrt{\frac{2}{\beta}}\xi_n-\tilde{\lambda}_j^{(n)} n^{-1\slash 6}\leq  \frac{1}{10}n^{1\slash 72}+\frac{1}{2}n^{-1\slash 6}k^{12}\leq \sqrt{n},
\end{equation*}
which leads to a contradiction. Hence $i_0\leq n-2$. If $g_j(i_0+1)=0$, then by (\ref{Eigen}), $g_j(i)=0$ for any $i\in [n]$, which contradicts the fact that $g_j$ is an eigenvector. If $g_j(i_0+1)\neq 0$, by (\ref{Eigen}), we have
\begin{equation*}
    \sqrt{\frac{2}{\beta}}\xi_{i_0+1}g_j(i_0+1)+\frac{Y_{n-i_0-1}}{\sqrt{\beta}}g_j(i_0+2)-2\sqrt{n}g_j(i_0+1)=n^{-1\slash 6}\tilde{\lambda}_j^{(n)}g_j(i_0+1). 
\end{equation*}
As the events $\mathcal{B}_1(j)$ and $\mathcal{U}_2$ hold, by (\ref{Conn}), we can deduce the following:
\begin{itemize}
    \item[(a)] $\sqrt{2\slash\beta}\left|\xi_{i_0+1}\right|\leq \sqrt{n}\slash 100$;
    \item[(b)] $Y_{n-i_0-1}\slash \sqrt{\beta}=Z_{i_0+1}+\sqrt{n-i_0-1}\leq (25\slash 24)\sqrt{n}$;
    \item[(c)] $n^{-1\slash 6}\tilde{\lambda}_j^{(n)}\geq -n^{-1\slash 6}k^{12}\slash 2\geq -\sqrt{n}\slash 100$.
\end{itemize}
Hence $g_j(i_0+2)\slash g_j(i_0+1)\geq 3\slash 2$, which leads to 
\begin{equation*}
    P_j(i_0+2)\geq  \frac{1}{2}n^{1\slash 3}\geq \frac{1}{2}\sqrt{(i_0+2)n^{-1\slash 3}}.
\end{equation*}
As the event $\mathcal{B}_3(j)^c$ holds, we are led to a contradiction. We conclude that $g_j(i)\neq 0$ for any $j\in [k]$ and $i\in [n^{1\slash 3}k^{12}-1,n]\cap\mathbb{Z}$.

Now suppose that there exist some $j\in [k]$ and $i_0\in [n^{1\slash 3}k^{12}-1,n-1]\cap\mathbb{Z}$, such that either $g_j(i_0)<0$, $g_j(i_0+1)>0$ or $g_j(i_0)>0$, $g_j(i_0+1)<0$. Without loss of generality, we assume that $g_j(i_0)<0$, $g_j(i_0+1)>0$. If $i_0\leq n-2$, by (\ref{Eigen}), we have
\begin{equation*}
    \sqrt{\frac{2}{\beta}}\xi_{i_0+1}g_j(i_0+1)+\frac{Y_{n-i_0-1}}{\sqrt{\beta}}g_j(i_0+2)-2\sqrt{n}g_j(i_0+1)\geq n^{-1\slash 6}\tilde{\lambda}_j^{(n)}g_j(i_0+1). 
\end{equation*}
Arguing as in the previous paragraph, we can deduce that
\begin{equation*}
    P_j(i_0+2)\geq \frac{1}{2}\sqrt{(i_0+2)n^{-1\slash 3}},
\end{equation*}
which leads to a contradiction (as the event $\mathcal{B}_3(j)^c$ holds). If $i_0=n-1$, by (\ref{EigenN}), arguing as in the previous paragraph, we obtain that 
\begin{eqnarray*}
    0\geq \frac{Y_1}{\sqrt{\beta}}g_j(n-1) &= &\left(2\sqrt{n}-\sqrt{\frac{2}{\beta}}\xi_n+n^{-1\slash 6}\tilde{\lambda}_j^{(n)}\right)g_j(n)\\
    &\geq & \left(2-\frac{1}{50}\right)\sqrt{n} g_j(n)>0,
\end{eqnarray*}
which leads to a contradiction. Hence we conclude that the property (a) in the definition of the event $\mathcal{B}$ (in the statement of the proposition) holds.

Below we fix an arbitrary $j\in [k]$. For any $i\in [n^{1\slash 3}k^{12},n]\cap\mathbb{Z}$, as $g_j(i-1)\neq 0$, we have $P_j(i)\neq \infty$. Without loss of generality, we assume that $g_j(i)>0$ for any $i\in [n^{1\slash 3}k^{12}-1,n]\cap\mathbb{Z}$ in the following. As the event $\mathcal{B}_4(j)^c$ holds, for any $i\in [n^{1\slash 3}k^{12},n-50]\cap\mathbb{Z}$, we have $P_j(i)\leq -\sqrt{i n^{-1\slash 3}}\slash 8$, which leads to
\begin{equation*}
    g_j(i)   \leq \left(1-\frac{1}{8}\sqrt{\frac{i}{n}}\right) g_j(i-1).
\end{equation*}
For any $i\in [n-49,n]\cap\mathbb{Z}$, as the event $\mathcal{B}_3(j)^c$ holds, we can deduce that $P_j(i)\leq \sqrt{i n^{-1\slash 3}}\slash 8$, which leads to
\begin{equation*}
    g_j(i)\leq \left(1+\frac{1}{8}\sqrt{\frac{i}{n}}\right) g_j(i-1)\leq \frac{9}{8}g_j(i-1).
\end{equation*}
Hence for any $(i_1,i_2)\in\mathbb{N}_{+}^2$ such that $n^{1\slash 3} k^{12}\leq i_1\leq i_2\leq n$, we have 
\begin{eqnarray}
    |g_j(i_2)|&\leq& \left(\frac{9}{8}\right)^{50} \prod_{i=i_1+1}^{i_2-50}\left(1-\frac{1}{8}\sqrt{\frac{i}{n}}\right) |g_j(i_1)|\nonumber\\
    &\leq & 3\cdot 10^5\prod_{i=i_1+1}^{i_2}\left(1-\frac{1}{8}\sqrt{\frac{i}{n}}\right) |g_j(i_1)|\nonumber\\
    &\leq& 3\cdot 10^5 \exp\left(-\frac{1}{12}\left(i_2^{3\slash 2}-i_1^{3\slash 2}\right) n^{-1\slash 2}\right)|g_j(i_1)|.
\end{eqnarray}
Hence taking $C'=3\cdot 10^5$, we conclude that the property (b) in the definition of the event $\mathcal{B}$ holds.

Therefore, we conclude that $\mathcal{B}_0\subseteq \mathcal{B}$. Hence by (\ref{Neq}), when $k\geq K$,
\begin{equation}
    \mathbb{P}\left(\mathcal{B}^c\right)\leq C\exp \left(-c k^3 \right). 
\end{equation}

\end{proof}

\subsection{Proof of Theorem \ref{Main3}}
\label{sec:3.3}
In this subsection, we finish the proof of Theorem \ref{Main3} based on the results established in Sections \ref{sec:3.1}-\ref{sec:3.2}. We first cite a result from \cite{Cor}.

\begin{lemma}[\cite{Cor}, Lemma 4.7]\label{LemmaBM}
Recall that $(B_x,x\geq 0)$ is the standard Brownian motion. Define
\begin{equation*}
    Z:=\sup_{x>0}\sup_{y\in[0,1)}\frac{|B_{x+y}-B_x|}{6\sqrt{\log(3+x)}}.
\end{equation*}
For any $x\geq 0$, let $Q_x:=B_{x+1}-B_x$ and $R_x:=B_x-\int_{x}^{x+1}B_ydy$. Then for any $x\geq 0$,
\begin{equation}
    \max\{|R_x|,|Q_x|\}\leq 6Z\sqrt{\log(3+x)}.
\end{equation}
Moreover, there exist positive absolute constants $C,c,s_0$, such that for all $s\geq s_0$,
\begin{equation}
    \mathbb{P}(Z\geq s)\leq C\exp(-c s^2).
\end{equation}
\end{lemma}

\begin{proof}[Proof of Theorem \ref{Main3}]

We proceed by approximating $H_{\beta}$ and $H_{\beta,n}$ along their eigenfunctions\slash eigenvectors. Recall from Section \ref{Sect.n} that $f_1,\cdots,f_k$ are the top $k$ eigenfunctions of $H_{\beta}$ and $g_1,\cdots,g_k$ are the top $k$ normalized eigenvectors of $H_{\beta,n}$. We let $L:=k^{20}$ and $L':=\lfloor L n^{1\slash 3}\rfloor n^{-1\slash 3}$. Throughout the proof, we assume that $k$ is sufficiently large, so that $k\geq 100$ and $L n^{1\slash 3}\leq n\slash 100$ (note that $n\geq k^{10000}$; for smaller $k$, the conclusion of the theorem can be obtained by enlarging the constant $C_1$ in the statement of the theorem). Throughout this proof, we use $C,c$ to denote positive constants that only depend on $\beta$ and $e_0$; the values of such constants may change from line to line.

The proof consists of two steps. In \textbf{Step 1}, for each $j\in [k]$, we construct a discrete approximation $\hat{f}_j$---which is a vector in $\mathbb{R}^n$---of the eigenfunction $f_j$, based on which we show that with probability at least $1-C\exp(-c k^3)$, $\tilde{\lambda}_j^{(n)}\geq a_j-Cn^{-1\slash 24}$ for every $j\in [k]$. In \textbf{Step 2}, for each $j\in [k]$, we construct an approximation $\hat{g}_j$---which is a continuous function on $[0,\infty)$---of the eigenvector $g_j$, based on which we show that with probability at least $1-C\exp(-c k^3)$, $\tilde{\lambda}_j^{(n)}\leq a_j+Cn^{-1\slash 24}$ for every $j\in [k]$. We obtain the conclusion of the theorem by combining the results from \textbf{Steps 1-2}. 

\paragraph{Step 1}

For every $j\in [k]$, we construct a discrete approximation $\hat{f}_j$---which is a vector in $\mathbb{R}^n$---of the eigenfunction $f_j$ as follows. For any $i\in [1,L n^{1\slash 3}]\cap\mathbb{Z}$, we let $\hat{f}_j(i)=n^{-1\slash 6}f_j(in^{-1\slash 3})$; for any $i\in (L n^{1\slash 3},n]\cap\mathbb{Z}$, we let $\hat{f}_j(i)=0$.

Recall from Section \ref{Sect.n2} that $\hat{H}_{\beta,n}=n^{1\slash 6}(H_{\beta,n}-2\sqrt{n} I_n)$ and that $\{\tilde{\lambda}_j^{(n)}\}_{j=1}^n$ are the eigenvalues of $\hat{H}_{\beta,n}$. We also recall the definitions of $\{\xi_i\}_{i=1}^n$ and $\{Y_i\}_{i=1}^{n-1}$ from Section \ref{Sect.n2}. For any $j\in [k]$,
\begin{equation}\label{NEQ20.7}
  (\hat{f}_j)^{\intercal}\hat{H}_{\beta,n}\hat{f}_j  =n^{1\slash 6}\sum_{i=1}^n\Big(\sqrt{\frac{2}{\beta}}\xi_i \hat{f}_j(i)^2-2\sqrt{n}\hat{f}_j(i)^2\Big)+2n^{1\slash 6}\sum_{i=1}^{n-1} \frac{Y_{n-i}}{\sqrt{\beta}}\hat{f}_j(i)\hat{f}_j(i+1).
\end{equation}
For each $i\in [n-1]$, we couple $Y_{n-i}$ with $W_{i,l}\stackrel{i.i.d.}{\sim}N(0,1)$ for $l\in [\beta(n-i)]$ such that $Y_{n-i}=\sqrt{\sum_{l=1}^{\beta(n-i)}W_{i,l}^2}$. By the Koml\'os-Major-Tusn\'ady theorem (see \cite{Cha, KMT}), we can couple $\{W_{i,l}\}_{l=1}^{\beta(n-i)}$ with $\zeta_i \stackrel{i.i.d.}{\sim}N(0,1)$ for all $i\in [1,L n^{1\slash 3}]\cap\mathbb{Z}$, such that with probability at least $1-C\exp(-c n^{1\slash 4})$, for any $i\in [1,Ln^{1\slash 3}]\cap\mathbb{Z}$,
\begin{equation}\label{NEQ20.25}
    \bigg|\frac{\sum_{l=1}^{\beta(n-i)}(W_{i,l}^2-1)}{\sqrt{2\beta(n-i)}}-\zeta_i\bigg|\leq n^{-1\slash 4}.
\end{equation}
Now we recall that $(B_x,x\geq 0)$ is the standard Brownian motion. We couple $S_i:=2^{-1\slash 2}n^{-1\slash 6}\sum_{l=1}^i (\xi_l+\zeta_l)$ for all $i\in [1,L n^{1\slash 3}]\cap\mathbb{Z}$ with $(B_x,x\geq 0)$ such that
\begin{equation}\label{NEQ21.5}
    S_i=-B_{i n^{-1\slash 3}}, \quad \text{for all }  i\in [1,L n^{1\slash 3}]\cap\mathbb{Z}.
\end{equation}

Now we consider an arbitrary $i\in [n-1]$. Let
\begin{equation}\label{NEQ20.9}
    \Delta_i:=\frac{Y_{n-i}}{\sqrt{\beta}}-\sqrt{n-i}-\frac{1}{\sqrt{2\beta}}\frac{\sum_{l=1}^{\beta(n-i)}(W_{i,l}^2-1)}{\sqrt{2\beta(n-i)}}
\end{equation}
and $\bar{Z}_i:=\big(\sum_{l=1}^{\beta(n-i)}(W_{i,l}^2-1)\big)\slash \big(\beta\sqrt{n-i}\big)$. Note that 
\begin{eqnarray}\label{NEQ20.1}
    |\Delta_i| &=& \bigg|\frac{1}{\sqrt{\beta}}\bigg(\frac{\sum_{l=1}^{\beta(n-i)}(W_{i,l}^2-1)}{\sqrt{\beta(n-i)}+\sqrt{\sum_{l=1}^{\beta(n-i)}W_{i,l}^2}}-\frac{\sum_{l=1}^{\beta(n-i)}(W_{i,l}^2-1)}{2\sqrt{\beta(n-i)}}\bigg)\bigg|\nonumber\\
    &=& \frac{\big(\sum_{l=1}^{\beta(n-i)}(W_{i,l}^2-1)\big)^2}{2\beta\sqrt{n-i}\Big(\sqrt{\beta(n-i)}+\sqrt{\sum_{l=1}^{\beta(n-i)}W_{i,l}^2}\Big)^2}\leq \frac{\big(\bar{Z}_i\big)^2}{2\sqrt{n-i}}.
\end{eqnarray}
Following the proof of Lemma \ref{ZL1}, we can deduce that for any $x\geq 1$,
\begin{equation}\label{NEQ37.11}
    \mathbb{P}(|\bar{Z}_i|\geq x)\leq 2\exp(-\beta x\slash 8). 
\end{equation}
Hence for any $i\in [n-1]$,
\begin{equation}\label{NEQ37.9}
    \mathbb{P}(|\Delta_i|\geq n)\leq C \exp(-c n^{1\slash 2});
\end{equation}
for any $i\in [1,L n^{1\slash 3}]\cap\mathbb{Z}$, as $n-i\geq 99 n\slash 100$, we have
\begin{equation}\label{NEQ22.44}
    \mathbb{P}(|\Delta_i|\geq n^{-1\slash 3})\leq \mathbb{P}(|\bar{Z}_i|\geq \sqrt{2}(n-i)^{1\slash 4}n^{-1\slash 6})\leq C\exp(-c n^{1\slash 12}).
\end{equation}

\subparagraph{Sub-step 1.1}

In this sub-step, we derive some preliminary estimates on $f_j$, where $j\in [k]$. These estimates will be used in later parts of the proof. 

For any $x\geq 0$, let $Q_x$ and $R_x$ be defined as in Lemma \ref{LemmaBM}. By Lemma \ref{LemmaBM}, with probability at least $1-C\exp(-c k^4)$, $\max\{|R_x|,|Q_x|\}\leq 6k^2\sqrt{\log(3+x)}$ for any $x\geq 0$. 

In Proposition \ref{AiryDecay}, we let $T=k^{11}\slash 8$, and obtain that with probability at least $1-C\exp(-c k^3)$, the following properties hold: (a) for any $(x_1,x_2)\in\mathbb{R}^2$ such that $x_2\geq x_1\geq k^{12}\slash 8$, $|f_j(x_2)|\leq |f_j(x_1)|\exp\big(-(x_2^{3\slash 2}-x_1^{3\slash 2})\slash 3\big)$; (b) for any $x\geq k^{12}\slash 8$, $|f_j'(x)|\leq 4\sqrt{x}|f_j(x)|$. As $\int_{k^{12}\slash 8}^{k^{12} \slash 8+1}f_j(x)^2 dx\leq \int_{0}^{\infty} f_j(x)^2 dx=1$, there exists some $x_0\in [k^{12}\slash 8,k^{12}\slash 8+1]$ such that $|f_j(x_0)|\leq 1$. Therefore, with probability at least $1-C\exp(-c k^3)$, for any $x\geq k^{12}\slash 2$, the following holds:
\begin{equation}\label{NEQ20.29}
    |f_j(x)|\leq |f_j(x_0)|\exp\big(-(x^{3\slash 2}-x_0^{3\slash 2})\slash 3\big)\leq \exp(-x^{3\slash 2}\slash 6),
\end{equation}
\begin{equation}\label{NEQ20.28}
    |f_j'(x)|\leq 4\sqrt{x}|f_j(x)|\leq 4\sqrt{x}\exp(-x^{3\slash 2}\slash 6).
\end{equation}

Therefore, with probability at least $1-C\exp(-c k^3)$, the following holds:
\begin{eqnarray}\label{NEQ20.10}
 &&  \Big|\int_{0}^{\infty} Q_x f_j(x)^2 dx\Big|\leq 6k^2\int_{0}^{\infty}\sqrt{\log(3+x)}f_j(x)^2 dx\nonumber\\
 &\leq & 6k^2\sqrt{\log(3+k^{12})}\int_{0}^{k^{12}}f_j(x)^2dx \nonumber\\
 && +6k^2\int_{k^{12}}^{\infty}\sqrt{\log(3+x)}\exp(-x^{3\slash 2}\slash 3)dx\leq C k^2\log{k},
\end{eqnarray}
\begin{eqnarray}\label{NEQ20.12}
    && \Big|\int_{0}^{\infty} R_x f_j(x) f_j'(x) dx\Big|\leq \int_0^{\infty}\Big(\frac{2}{\sqrt{\beta}}R_x^2f_j(x)^2+\frac{\sqrt{\beta}}{8}f_j'(x)^2\Big)dx\nonumber\\
    &\leq& Ck^4\int_{0}^{\infty}\log(3+x)f_j(x)^2 dx+\frac{\sqrt{\beta}}{8}\int_{0}^{\infty}f_j'(x)^2dx\nonumber\\
    &\leq& Ck^4\log{k}+\frac{\sqrt{\beta}}{8}\int_{0}^{\infty}f_j'(x)^2dx,
\end{eqnarray}
where we use the AM-GM inequality in the first inequality in (\ref{NEQ20.12}) and 
\begin{eqnarray*}
    \int_{0}^{\infty}\log(3+x)f_j(x)^2 dx  &\leq& \log(3+k^{12})\int_{0}^{k^{12}} f_j(x)^2 dx\\
   &&+\int_{k^{12}}^{\infty}\log(3+x)\exp(-x^{3\slash 2}\slash 3)dx\nonumber\\
   &\leq& C\log{k}
\end{eqnarray*}
in the third inequality in (\ref{NEQ20.12}).

With probability at least $1-C\exp(-ck^3)$, the statements below in this paragraph hold. By Lemma \ref{AiryDecayL3}, $|\lambda_j|\leq k^{2}$. Hence by (\ref{eigenv2}), (\ref{NEQ20.10}), and (\ref{NEQ20.12}), 
\begin{eqnarray*}
   &&  \int_0^{\infty} f_j'(x)^2dx =  \lambda_j-\int_0^{\infty}xf_j(x)^2dx-\frac{2}{\sqrt{\beta}}\int_0^{\infty}Q_x f_j(x)^2dx\\
    && \quad+\frac{4}{\sqrt{\beta}}\int_0^{\infty} R_x f_j(x)f_j'(x) dx\leq Ck^{4}\log{k}+\frac{1}{2}\int_{0}^{\infty}f_j'(x)^2dx,
\end{eqnarray*}
which leads to 
\begin{equation}\label{NEQ20.14}
    \int_0^{\infty}f_j'(x)^2 dx\leq C k^4\log{k}.
\end{equation}
By (\ref{NEQ20.14}) and the Cauchy-Schwarz inequality, for any $x\in [0,L+1]$,
\begin{equation}\label{NEQ20.19}
    |f_j(x)|=\Big|\int_{0}^x f_j'(y)dy\Big| \leq \Big(\int_0^xf_j'(y)^2 dy\Big)^{1\slash 2} x^{1\slash 2}\leq C k^4 (L+1)^{1\slash 2}\leq Ck^{14};
\end{equation}
for any $x_1,x_2\in [0,\infty)$ such that $x_1\leq x_2$ and $|x_1-x_2|\leq n^{-1\slash 3}$,
\begin{eqnarray}\label{NEQ20.17}
   && |f_j(x_1)-f_j(x_2)|=\Big|\int_{x_1}^{x_2}f_j'(x)dx\Big|\nonumber\\
    &\leq & \Big(\int_{x_1}^{x_2}f_j'(x)^2dx\Big)^{1\slash 2}|x_1-x_2|^{1\slash 2}
    \leq C k^4 n^{-1\slash 6}.
\end{eqnarray}

Now let $\mathcal{S}:=\{in^{-1\slash 3}\}_{i=0}^{\lfloor L n^{1\slash 3}\rfloor}$. By the standard estimate (\ref{Eqn26}) for Brownian motion and the union bound, we can deduce that with probability at least $1-C\exp(-c k^3)$,
\begin{equation}\label{NEQ20.15}
    \sup_{\substack{x_1\in [0,\infty),x_2\in\mathcal{S}:\\|x_1-x_2|\leq n^{-1\slash 3}}}|B_{x_1}-B_{x_2}|\leq Ck^2 n^{-1\slash 6}, \quad \sup_{x\in [0,L+1]}|B_x|\leq C k^{12};
\end{equation}
\begin{equation}\label{NEQ20.30}
    \sup_{x\in [0,1]}|B_{t+x}-B_t|\leq t ,\quad\forall t\in [L,\infty)\cap\mathbb{Z};\text{  hence } |B_x|\leq Cx^2,\quad\forall x\geq L.
\end{equation}
With probability at least $1-C\exp(-c k^3)$, the rest of the statements in this paragraph hold. Below we consider any $x_1\in [0,\infty)$ and any $x_2\in \mathcal{S}$ such that $|x_1-x_2|\leq n^{-1\slash 3}$. By (\ref{eigenv1}), 
\begin{eqnarray}\label{NEQ20.21}
    f_j'(x_1)-f_j'(x_2)&=& \int_{x_2}^{x_1}(y-\lambda_j)f_j(y)dy-\frac{2}{\sqrt{\beta}}\int_{x_2}^{x_1}B_yf_j'(y)dy\nonumber\\
    &&+\frac{2}{\sqrt{\beta}}(f_j(x_1)B_{x_1}-f_j(x_2)B_{x_2}).
\end{eqnarray}
Without loss of generality, we assume that $x_1\geq x_2$ below. As $|\lambda_j|\leq k^2$ (see the previous paragraph), by the AM-GM inequality and the Cauchy-Schwarz inequality, we have
\begin{eqnarray}\label{NEQ20.22}
  &&  \Big|\int_{x_2}^{x_1} (y-\lambda_j)f_j(y)dy  \Big|\leq \Big(\int_{x_2}^{x_1}f_j(y)^2dy\Big)^{1\slash 2}\Big(\int_{x_2}^{x_1}(y-\lambda_j)^2dy\Big)^{1\slash 2} \nonumber\\
  &\leq& \Big( 2\int_{x_2}^{x_1}(y^2+\lambda_j^2)dy \Big)^{1\slash 2}\leq C k^{20} n^{-1\slash 6}.
\end{eqnarray}
By (\ref{NEQ20.14}), (\ref{NEQ20.15}), and the Cauchy-Schwarz inequality,
\begin{equation}\label{NEQ20.23}
    \Big|\int_{x_2}^{x_1}B_y f'_j(y) dy\Big|\leq \Big(\int_{x_2}^{x_1}B_y^2dy\Big)^{1\slash 2}\Big(\int_{x_2}^{x_1}f'_j(y)^2dy\Big)^{1\slash 2}\leq C k^{16}n^{-1\slash 6}.
\end{equation}
By (\ref{NEQ20.19})-(\ref{NEQ20.15}), 
\begin{eqnarray}\label{NEQ20.24}
    |f_j(x_1)B_{x_1}-f_j(x_2)B_{x_2}|&\leq & |B_{x_1}||f_j(x_1)-f_j(x_2)|+|B_{x_1}-B_{x_2}||f_j(x_2)|\nonumber\\
    &\leq& C k^{16} n^{-1\slash 6}.
\end{eqnarray}
Combining (\ref{NEQ20.21})-(\ref{NEQ20.24}), we obtain that
\begin{equation}\label{NEQ22.5}
    |f_j'(x_1)-f_j'(x_2)|\leq Ck^{20} n^{-1\slash 6}. 
\end{equation}

\subparagraph{Sub-step 1.2}

In this sub-step, we bound $(\hat{f}_j)^{\intercal}\hat{H}_{\beta,n}\hat{f}_{j'}$ for any $j,j'\in [k]$. 

In the following, we fix an arbitrary $j\in [k]$. Recall the definitions of $\{Z_i\}_{i=1}^{n-1}$ and $\{\Delta_i\}_{i=1}^{n-1}$ from (\ref{psum}) and (\ref{NEQ20.9}). We define the following quantities:
\begin{equation*}
    \Pi_{1,1,j}=n^{1\slash 6}\sum_{i=1}^n\sqrt{\frac{2}{\beta}}\xi_i\hat{f}_j(i)^2+n^{1\slash 6}\sum_{i=1}^{n-1}\sqrt{\frac{2}{\beta}}\frac{\sum_{l=1}^{\beta(n-i)}(W_{i,l}^2-1)}{\sqrt{2\beta(n-i)}}\hat{f}_j(i)^2,
\end{equation*}
\begin{equation*}
    \Pi_{1,2,j}=2n^{1\slash 6}\sum_{i=1}^{n-1}\sqrt{n-i}\hat{f}_j(i)\hat{f}_j(i+1)-2n^{2\slash 3}\sum_{i=1}^n\hat{f}_j(i)^2;
\end{equation*}
\begin{equation*}
   \Pi_{2,1,j}=2n^{1\slash 6}\sum_{i=1}^{n-1} Z_i\hat{f}_j(i)(\hat{f}_j(i+1)-\hat{f}_j(i)), \quad  \Pi_{2,2,j}=2n^{1\slash 6}\sum_{i=1}^{n-1}\Delta_i\hat{f}_j(i)^2.
\end{equation*}
Note that by (\ref{NEQ20.7}) and (\ref{NEQ20.9}), 
\begin{equation}\label{NEQ22.50}
    (\hat{f}_j)^{\intercal}\hat{H}_{\beta,n}\hat{f}_j=\Pi_{1,1,j}+\Pi_{1,2,j}+\Pi_{2,1,j}+\Pi_{2,2,j}.
\end{equation}

In the following, we bound $\Pi_{1,1,j}$, $\Pi_{1,2,j}$, $\Pi_{2,1,j}$, $\Pi_{2,2,j}$ in \textbf{Parts 1.2.1-1.2.4}, respectively. These bounds are then used in \textbf{Part 1.2.5} to bound $(\hat{f}_j)^{\intercal}\hat{H}_{\beta,n}\hat{f}_j$.  

\textbf{Part 1.2.1}

In this part, we bound $\Pi_{1,1,j}$. Let
\begin{equation}
    \Upsilon_{1,j}:=n^{1\slash 6}\sqrt{\frac{2}{\beta}}\sum_{i=1}^{\lfloor L n^{1\slash 3} \rfloor}(\xi_i+\zeta_i)\hat{f}_j(i)^2.
\end{equation}
By (\ref{NEQ20.25}) and the fact that $\hat{f}_j(i)=0$ for any $i\in (L n^{1\slash 3},n]\cap\mathbb{Z}$, with probability at least $1-C\exp(-c n^{1\slash 4})$, we have
\begin{eqnarray}\label{NEQ21.1}
  |\Pi_{1,1,j}-\Upsilon_{1,j}|&\leq& n^{1\slash 6}\sqrt{2\slash\beta}\sum_{i=1}^{\lfloor L n^{1\slash 3}\rfloor} \Big|\zeta_i-\frac{\sum_{l=1}^{\beta(n-i)}(W_{i,l}^2-1)}{\sqrt{2\beta(n-i)}}\Big|\hat{f}_j(i)^2\nonumber\\
  &\leq& n^{-1\slash 12}\sqrt{2\slash\beta}\sum_{i=1}^{\lfloor L n^{1\slash 3}\rfloor}\hat{f}_j(i)^2.
\end{eqnarray}
By the definition of $\hat{f}_j$ and (\ref{NEQ20.19})-(\ref{NEQ20.17}), recalling the definition of $L'$ at the beginning of the proof, with probability at least $1-C\exp(-ck^3)$, we have
\begin{eqnarray}\label{NEQ21.2}
&& \Big|\sum_{i=1}^{\lfloor L n^{1\slash 3}\rfloor}\hat{f}_j(i)^2-\int_0^{L'}f_j(x)^2dx\Big|\leq\sum_{i=1}^{\lfloor L n^{1\slash 3}\rfloor}\Big|\hat{f}_j(i)^2-\int_{(i-1)n^{-1\slash 3}}^{i n^{-1\slash 3}}f_j(x)^2 dx\Big| \nonumber\\
 &\leq & \sum_{i=1}^{\lfloor L n^{1\slash 3}\rfloor}\int_{(i-1)n^{-1\slash 3}}^{i n^{-1\slash 3}}\Big|f_j(i n^{-1\slash 3})^2-f_j(x)^2\Big| dx\leq Ck^{38} n^{-1\slash 6},
\end{eqnarray}
which combined with (\ref{NEQ21.1}) leads to 
\begin{equation}\label{NEQ22.12}
    |\Pi_{1,1,j}-\Upsilon_{1,j}|\leq  C n^{-1\slash 12}.
\end{equation}

Now we bound $\Upsilon_{1,j}$. By (\ref{NEQ21.5}) and the Abel transformation,
\begin{eqnarray}\label{NEQ22.7}
   \Upsilon_{1,j} &=&  2\beta^{-1\slash 2}n^{1\slash 3}\sum_{i=1}^{\lfloor L n^{1\slash 3} \rfloor} S_i(\hat{f}_j(i)^2-\hat{f}_j(i+1)^2)\nonumber\\
    &=&-2\beta^{-1\slash 2}\sum_{i=1}^{\lfloor L n^{1\slash 3} \rfloor}B_{i n^{-1\slash 3}}\Big(f_j\big(in^{-1\slash 3}\big)^2-f_j\big((i+1) n^{-1\slash 3}\big)^2\Big)\nonumber\\
    && -2\beta^{-1\slash 2}B_{L'}f_j\big(L'+n^{-1\slash 3}\big)^2.
\end{eqnarray} 

With probability at least $1-C\exp(-ck^3)$, the rest of the statements in this part hold. By (\ref{NEQ20.29}), (\ref{NEQ20.28}), (\ref{NEQ20.14}), (\ref{NEQ20.19}), (\ref{NEQ20.15}), (\ref{NEQ20.30}), and the Cauchy-Schwarz inequality,
\begin{equation}\label{NEQ22.8}
    \Big|B_{L'}f_j\big(L'+n^{-1\slash 3}\big)^2\Big|\leq Ck^{12}\exp(-ck^{30})\leq Cn^{-1\slash 3},
\end{equation}
\begin{eqnarray}\label{NEQ22.9}
  &&  \Big|\int_{0}^{n^{-1\slash 3}}B_x f_j(x) f_j'(x)dx\Big|\leq Ck^{26}\int_{0}^{n^{-1\slash 3}}|f_j'(x)| dx\nonumber\\
  &\leq& Ck^{26}n^{-1\slash 6} \Big(\int_{0}^{n^{-1\slash 3}}f_j'(x)^2dx\Big)^{1\slash 2}\leq Ck^{30} n^{-1\slash 6},
\end{eqnarray}
\begin{equation}\label{NEQ22.10}
    \Big|\int_{L'+n^{-1\slash 3}}^{\infty} B_x f_j(x) f_j'(x)dx\Big|\leq C\int_{L}^{\infty}x^{5\slash 2} \exp\big(-c x^{3\slash 2}\big)dx\leq Cn^{-1\slash 3}.
\end{equation}

Now consider any $i\in [1,Ln^{1\slash 3}]\cap\mathbb{Z}$. By the mean value theorem, there exists some $\theta_i\in [0,1]$, such that
\begin{equation}\label{NEQ22.2}
    f_j\big(i n^{-1\slash 3}\big)^2-f_j\big((i+1) n^{-1\slash 3}\big)^2=-2 n^{-1\slash 3}f_j\big((i+\theta_i) n^{-1\slash 3}\big) f_j'\big((i+\theta_i)n^{-1\slash 3}\big).
\end{equation}
Let $\Phi_i:=B_{i n^{-1\slash 3}}\Big(f_j\big(in^{-1\slash 3}\big)^2-f_j\big((i+1) n^{-1\slash 3}\big)^2\Big)$. By (\ref{NEQ20.19})-(\ref{NEQ20.15}), (\ref{NEQ22.5}), (\ref{NEQ22.2}), and the triangle inequality, for any $x\in [i n^{-1\slash 3},(i+1)n^{-1\slash 3}]$, we have
$\Big| \Phi_i n^{1\slash 3}+ 2B_x f_j(x) f_j'(x)\Big|\leq Cn^{-1\slash 6}(k^{46}+k^{16}|f_j'(x)|)$. Note that by (\ref{NEQ20.14}) and the Cauchy-Schwarz inequality,
\begin{equation*}
    \int_{n^{-1\slash 3}}^{L'+n^{-1\slash 3}}|f_j'(x)|dx\leq \Big(\int_{n^{-1\slash 3}}^{L'+n^{-1\slash 3}}f_j'(x)^2 dx\Big)^{1\slash 2} (L')^{1\slash 2}\leq Ck^{14}.
\end{equation*}
Hence
\begin{eqnarray}\label{NEQ22.6}
&& \Big|   \sum_{i=1}^{\lfloor L n^{1\slash 3} \rfloor}\Phi_i+2\int_{n^{-1\slash 3}}^{L'+n^{-1\slash 3}}B_x f_j(x) f_j'(x)dx\Big|\nonumber\\
&\leq& \sum_{i=1}^{\lfloor L n^{1\slash 3}\rfloor}\Big|\int_{i n^{-1\slash 3}}^{(i+1)n^{-1\slash 3}} \big(\Phi_i n^{1\slash 3}+2 B_x f_j(x) f_j'(x)\big) dx\Big|\nonumber\\
&\leq& \sum_{i=1}^{\lfloor L n^{1\slash 3}\rfloor}\int_{i n^{-1\slash 3}}^{(i+1)n^{-1\slash 3}}\big| \Phi_i n^{1\slash 3}+2 B_x f_j(x) f_j'(x) \big|dx\nonumber\\
&\leq& C k^{66}n^{-1\slash 6}+C k^{16}n^{-1\slash 6}\int_{n^{-1\slash 3}}^{L'+n^{-1\slash 3}}|f_j'(x)|dx\leq Ck^{66}n^{-1\slash 6}.
\end{eqnarray}
By (\ref{NEQ22.7})-(\ref{NEQ22.10}) and (\ref{NEQ22.6}), $|\Upsilon_{1,j}-(4\slash\sqrt{\beta})\int_0^{\infty} B_x f_j(x) f_j'(x)dx|\leq Cn^{-1\slash 12}$, which combined with (\ref{NEQ22.12}) gives
\begin{equation}\label{NEQ22.51}
    \Big|\Pi_{1,1,j}-\frac{4}{\sqrt{\beta}}\int_0^{\infty} B_x f_j(x) f_j'(x)dx\Big|\leq Cn^{-1\slash 12}.
\end{equation}

\textbf{Part 1.2.2}

In this part, we bound $\Pi_{1,2,j}$. Let 
\begin{equation*}
  \Upsilon_{2,j}:=-n^{-1\slash 3}\sum_{i=1}^{n-1}i \hat{f}_j(i)\hat{f}_j(i+1),
\end{equation*}
\begin{equation*}
    \Upsilon_{3,j}:=-n^{-1\slash 3}\sum_{i=1}^{n-1}\frac{ \sqrt{n}-\sqrt{n-i}}{\sqrt{n}+\sqrt{n-i}} i\hat{f}_j(i)\hat{f}_j(i+1).
\end{equation*}
\begin{equation*}
    \Upsilon_{4,j}:=-n^{2\slash 3}\Big( \hat{f}_j(1)^2+\sum_{i=1}^{n-1}\big(\hat{f}_j(i+1)-\hat{f}_j(i)\big)^2 +\hat{f}_j(n)^2\Big),
\end{equation*}
Note that 
\begin{equation}\label{NEQ22.31}
    \Pi_{1,2,j}=\Upsilon_{2,j}+\Upsilon_{3,j}+\Upsilon_{4,j}.
\end{equation}
The rest of the statements in this part hold with probability greater than or equal to $1-C\exp(-c k^3)$.

We first bound $\Upsilon_{2,j}$. By the definition of $\hat{f}_j$, we have
\begin{equation}
    \Upsilon_{2,j}=-n^{-2\slash 3}\sum_{i=0}^{\lfloor Ln^{1\slash 3} \rfloor -1  }  i f_j\big( i n^{-1\slash 3}\big) f_j\big( (i+1)n^{-1\slash 3} \big).
\end{equation}
Now we consider any $i\in \{0,\cdots \lfloor L n^{1\slash 3}\rfloor -1\}$. Let
\begin{equation*}
    \Psi_i:=in^{-1\slash 3} f_j\big( i n^{-1\slash 3}\big) f_j\big( (i+1)n^{-1\slash 3} \big).
\end{equation*}
For any $x\in [i n^{-1\slash 3},(i+1)n^{-1\slash 3}]$, by (\ref{NEQ20.19}), (\ref{NEQ20.17}), and the triangle inequality, $\big|\Psi_i-x f_j(x)^2 \big|\leq Ck^{38}n^{-1\slash 6}$, which leads to
\begin{eqnarray}\label{NEQ22.30}
 &&  \Big| \Upsilon_{2,j}+\int_{0}^{L'}xf_j(x)^2  dx\Big|\leq \sum_{i=0}^{\lfloor L n^{1\slash 3}\rfloor -1  }\Big|\int_{i n^{-1\slash 3}}^{(i+1)n^{-1\slash 3}}(\Psi_i-x f_j(x)^2)dx\Big|\nonumber\\
 &\leq& \sum_{i=0}^{\lfloor L n^{1\slash 3}\rfloor -1  }\int_{i n^{-1\slash 3}}^{(i+1)n^{-1\slash 3}}|\Psi_i-x f_j(x)^2|dx\nonumber\\
 &\leq& Ck^{38}n^{-1\slash 6}L\leq Ck^{58} n^{-1\slash 6}.
\end{eqnarray}
By (\ref{NEQ20.29}), $\int_{L'}^{\infty}xf_j(x)^2 dx\leq \int_{L'}^{\infty} x\exp(-c x^{3\slash 2})dx\leq \exp(-c k^{30})\leq Cn^{-1\slash 12} $, which combined with (\ref{NEQ22.30}) gives
\begin{equation}\label{NEQ22.33}
    \Big| \Upsilon_{2,j}+\int_{0}^{\infty}xf_j(x)^2  dx\Big|\leq C n^{-1\slash 12}.
\end{equation}

Now we bound $\Upsilon_{3,j}$. By the AM-GM inequality and (\ref{NEQ21.2}), 
\begin{equation}\label{NEQ22.42}
    \sum_{i=1}^{\lfloor Ln^{1\slash 3}\rfloor -1}|\hat{f}_j(i)\hat{f}_j(i+1)|\leq \sum_{i=1}^{\lfloor L n^{1\slash 3}\rfloor}\hat{f}_j(i)^2\leq \int_0^{\infty} f_j(x)^2dx+C k^{38}n^{-1\slash 6}\leq C.
\end{equation}
As $\hat{f}_j(i)=0$ for any $i\in (L n^{1\slash 3},n]\cap\mathbb{Z}$, we have
\begin{eqnarray}\label{NEQ22.34}
 && |\Upsilon_{3,j}|=n^{-1\slash 3}\bigg|\sum_{i=1}^{\lfloor L n^{1\slash 3}\rfloor-1} \frac{i^2\hat{f}_j(i)\hat{f}_j(i+1)}{(\sqrt{n}+\sqrt{n-i})^2}\bigg|\nonumber\\
 &\leq& L^2 n^{-2\slash 3}\sum_{i=1}^{\lfloor Ln^{1\slash 3}\rfloor -1}|\hat{f}_j(i)\hat{f}_j(i+1)|\leq Ck^{40}n^{-2\slash 3}\leq C n^{-1\slash 12}.
\end{eqnarray}

Finally, we bound $\Upsilon_{4,j}$. Note that $f_j(0)=0$. By the definition of $\hat{f}_j$,
\begin{equation}\label{NEQ22.25}
    \Upsilon_{4,j}=-n^{1\slash 3}\Big(f_j\big(L'\big)^2+\sum_{i=0}^{\lfloor L n^{1\slash 3}\rfloor-1}\big(f_j\big((i+1)n^{-1\slash 3}\big)-f_j\big(i n^{-1\slash 3}\big)\big)^2\Big). 
\end{equation}
By (\ref{NEQ20.29}), as $L'\geq L-1\geq k^{20}\slash 2$ (by our assumption that $k\geq 100$), we have
\begin{equation}\label{NEQ22.26}
    n^{1\slash 3}f_j(L')^2\leq n^{1\slash 3}\exp(-c k^{30})\leq  Cn^{-1\slash 3}.
\end{equation}
Consider any $i\in \{0,1,\cdots,\lfloor L n^{1\slash 3}\rfloor -1\}$. By the mean value theorem, there exists $\theta_i'\in [0,1]$, such that
\begin{equation}\label{NEQ22.19}
    \big(f_j\big((i+1)n^{-1\slash 3}\big)-f_j\big(i n^{-1\slash 3}\big)\big)^2=n^{-2\slash 3}f_j'\big((i+\theta_i')n^{-1\slash 3}\big)^2.
\end{equation}
Note that by (\ref{NEQ20.14}) and the Cauchy-Schwarz inequality,
\begin{equation}\label{NEQ22.23}
    \int_{0}^{L'}|f_j'(x)|dx\leq \Big(\int_{0}^{L'}f_j'(x)^2 dx\Big)^{1\slash 2} (L')^{1\slash 2}\leq Ck^{14}.
\end{equation}
For any $x\in [in^{-1\slash 3},(i+1) n^{-1\slash 3}]$, by (\ref{NEQ22.5}),
\begin{eqnarray}
  && \big| f_j'\big((i+\theta_i')n^{-1\slash 3}\big)^2-f_j'(x)^2\big|\nonumber\\
  &\leq & \big| f_j'\big((i+\theta_i')n^{-1\slash 3}\big)-f_j'(x)\big|\big(2|f_j'(x)|+\big| f_j'\big((i+\theta_i')n^{-1\slash 3}\big)-f_j'(x)\big|\big)\nonumber\\
  &\leq& C k^{20} n^{-1\slash 6}\big(|f_j'(x)|+k^{20}n^{-1\slash 6}\big),
\end{eqnarray}
which combined with (\ref{NEQ22.19}) and (\ref{NEQ22.23}) gives
\begin{eqnarray}\label{NEQ22.27}
  && \bigg| n^{1\slash 3}\sum_{i=0}^{\lfloor L n^{1\slash 3}\rfloor-1}\big(f_j\big((i+1)n^{-1\slash 3}\big)-f_j\big(i n^{-1\slash 3}\big)\big)^2-\int_{0}^{L'}f_j'(x)^2 dx\bigg| \nonumber\\
  &\leq& Ck^{40}n^{-1\slash 3}L'+Ck^{20} n^{-1\slash 6}\int_{0}^{L'}|f_j'(x)|dx\leq C k^{34}n^{-1\slash 6}.
\end{eqnarray}
By (\ref{NEQ20.28}), we have
\begin{equation}\label{NEQ22.28}
    \int_{L'}^{\infty} f_j'(x)^2dx\leq C\int_{L'}^{\infty} x\exp(-c x^{3\slash 2})dx\leq C\exp(-c k^{30})\leq Cn^{-1\slash 3}.
\end{equation}
Combining (\ref{NEQ22.25}), (\ref{NEQ22.26}), (\ref{NEQ22.27}), and (\ref{NEQ22.28}), we have
\begin{equation}\label{NEQ22.35}
    \Big|\Upsilon_{4,j}+\int_{0}^{\infty}f_j'(x)^2dx\Big|\leq Cn^{-1\slash 12}. 
\end{equation}

By (\ref{NEQ22.31}), (\ref{NEQ22.33}), (\ref{NEQ22.34}), and (\ref{NEQ22.35}), we conclude that
\begin{equation}\label{NEQ22.52}
    \Big|\Pi_{1,2,j}+\int_0^{\infty} x f_j(x)^2dx+\int_{0}^{\infty}f_j'(x)^2dx\Big|\leq C n^{-1\slash 12}.
\end{equation}

\textbf{Part 1.2.3}

In this part, we bound $\Pi_{2,1,j}$. By Lemma \ref{ZL1}, with probability at least $1-C\exp(-c k^3)$, for any $i\in [n-1]$, $|Z_i|\leq k^3$. By the fact that $\hat{f}_j(i)=0$ for any $i\in (L n^{1\slash 3},n]\cap\mathbb{Z}$ and the Cauchy-Schwarz inequality,
\begin{eqnarray}\label{NEQ22.41}
  &&  |\Pi_{2,1,j}|\leq 2k^3 n^{1\slash 6} \sum_{i=1}^{\lfloor L n^{1\slash 3} \rfloor }\big|\hat{f}_j(i)\big|\big|\hat{f}_j(i+1)-\hat{f}_j(i)\big|\nonumber\\ 
  &\leq & 2k^3 n^{1\slash 6}\Big(\sum_{i=1}^{\lfloor L n^{1\slash 3}\rfloor} \hat{f}_j(i)^2\Big)^{1\slash 2}\Big(\sum_{i=1}^{\lfloor L n^{1\slash 3}\rfloor} \big(\hat{f}_j(i+1)-\hat{f}_j(i)\big)^2\Big)^{1\slash 2}.
\end{eqnarray}

For any $i\in [1,Ln^{1\slash 3}-1]\cap\mathbb{Z}$, by the mean value theorem there exists $\theta_i''\in [0,1]$, such that
\begin{eqnarray}\label{NEQ22.40} 
  && \sum_{i=1}^{\lfloor L n^{1\slash 3}\rfloor }\big(\hat{f}_j(i+1)-\hat{f}_j(i)\big)^2\nonumber\\
  &=& n^{-1\slash 3}\sum_{i=1}^{\lfloor L n^{1\slash 3}\rfloor-1}\big(f_j\big((i+1) n^{-1\slash 3}\big)-f_j\big(i n^{-1\slash 3}\big)\big)^2+n^{-1\slash 3}f_j(L')^2\nonumber\\
  &=& n^{-1}\sum_{i=1}^{\lfloor L n^{1\slash 3}\rfloor-1}f_j'\big((i+\theta_i'')n^{-1\slash 3}\big)^2+n^{-1\slash 3}f_j(L')^2.
\end{eqnarray}
With probability at least $1-C\exp(-c k^3)$, the following statements in this paragraph hold. For any $i\in [1,L n^{1\slash 3}-1]\cap\mathbb{Z}$ and any $x\in [i n^{-1\slash 3},(i+1)n^{-1\slash 3}]$, by (\ref{NEQ22.5}), we have
\begin{eqnarray}
    \big|f_j'\big((i+\theta_i'')n^{-1\slash 3}\big)^2-f_j'(x)^2\big|&\leq& \big|f_j'\big((i+\theta_i'')n^{-1\slash 3}\big)-f_j'(x)\big|\nonumber\\
 &\times& \big(2|f_j'(x)|+\big|f_j'\big((i+\theta_i'')n^{-1\slash 3}\big)-f_j'(x)\big|\big)\nonumber\\
 &\leq& C k^{20} n^{-1\slash 6} (|f'_j(x)|+k^{20}n^{-1\slash 6}),
\end{eqnarray}
which via (\ref{NEQ20.14}) and (\ref{NEQ22.23}) leads to
\begin{eqnarray*}
  && n^{-1\slash 3}\sum\limits_{i=1}^{\lfloor L n^{1\slash 3}\rfloor-1}f_j'\big((i+\theta_i'')n^{-1\slash 3}\big)^2 \nonumber\\
 &\leq& \int_{0}^{\infty} f_j'(x)^2 dx + C k^{40}n^{-1\slash 3}L +ck^{20}n^{-1\slash 6}\int_{0}^{L'}|f_j'(x)|dx\leq Ck^5.
\end{eqnarray*}
Hence by (\ref{NEQ22.26}) and (\ref{NEQ22.40}), $\sum\limits_{i=1}^{\lfloor L n^{1\slash 3}\rfloor }\big(\hat{f}_j(i+1)-\hat{f}_j(i)\big)^2\leq Ck^5n^{-2\slash 3}$, which combined with (\ref{NEQ22.42}) and (\ref{NEQ22.41}) gives
\begin{equation}\label{NEQ22.53}
    |\Pi_{2,1,j}|\leq C k^6 n^{-1\slash 6}\leq C n^{-1\slash 12}.
\end{equation}

\textbf{Part 1.2.4}

In this part, we bound $\Pi_{2,2,j}$. By (\ref{NEQ22.44}), (\ref{NEQ22.42}), and the fact that $\hat{f}_j(i)=0$ for any $i\in (L n^{1\slash 3},n]\cap\mathbb{Z}$, we have with probability at least $1-C\exp(-c k^3)$,
\begin{equation}\label{NEQ22.54}
    |\Pi_{2,2,j}|\leq 2n^{-1\slash 6}\sum_{i=1}^{\lfloor L n^{1\slash 3}\rfloor }\hat{f}_j(i)^2\leq Cn^{-1\slash 6}. 
\end{equation}

\textbf{Part 1.2.5}

In this part, we bound $(\hat{f}_j)^{\intercal}\hat{H}_{\beta,n}\hat{f}_{j'}$ for any $j,j'\in [k]$. The following statements in this part hold with probability at least $1-C\exp(-c k^3)$. By (\ref{NEQ22.50}), (\ref{NEQ22.51}), (\ref{NEQ22.52}), (\ref{NEQ22.53}), (\ref{NEQ22.54}), and the union bound, for any $j\in [k]$, we have 
\begin{eqnarray}
  &&  \Big|(\hat{f}_j)^{\intercal}\hat{H}_{\beta,n}\hat{f}_j-\frac{4}{\sqrt{\beta}}\int_0^{\infty} B_x f_j(x) f_j'(x)dx+\int_0^{\infty} x f_j(x)^2dx+\int_{0}^{\infty}f_j'(x)^2dx\Big|\nonumber\\
  && \leq C n^{-1\slash 12}. 
\end{eqnarray}
Hence by (\ref{eigenv3}) and the definition of $\{a_j\}_{j=1}^{\infty}$,
\begin{equation}
    \big|(\hat{f}_j)^{\intercal}\hat{H}_{\beta,n}\hat{f}_j-a_j\big|= \big|(\hat{f}_j)^{\intercal}\hat{H}_{\beta,n}\hat{f}_j+\lambda_j\big|\leq C n^{-1\slash 12}.
\end{equation}
Similarly, we can deduce that for any $j\in [k]$, $|(\hat{f}_j)^{\intercal}\hat{f}_j-1|\leq Cn^{-1\slash 12}$; for any $j,j'\in [k]$ such that $j\neq j'$, $|(\hat{f}_j)^{\intercal}\hat{H}_{\beta,n}\hat{f}_{j'}|\leq Cn^{-1\slash 12}$ and $|(\hat{f}_j)^{\intercal}\hat{f}_{j'}|\leq Cn^{-1\slash 12}$.

\subparagraph{Sub-step 1.3}

In this sub-step, we show that with probability at least $1-C\exp(-c k^3)$, for every $j\in [k]$, $\tilde{\lambda}_j^{(n)}\geq a_j-C n^{-1\slash 24}$. The rest of the statements in this sub-step hold with probability at least $1-C\exp(-c k^3)$. 

We start by showing that $\{\hat{f}_j\}_{j=1}^k$ are linearly independent in $\mathbb{R}^n$ when $n\geq C'$, where $C'$ is a positive constant that only depends on $\beta,e_0$. Suppose that there exist some $j_0\in [k]$ and $t_j\in \mathbb{R}$ for $j\in [k]\backslash\{j_0\}$, such that $\hat{f}_{j_0}=\sum_{j\in [k] \backslash\{j_0\}} t_j \hat{f}_j$. By the AM-GM inequality,
\begin{equation}\label{NEQ32.2}
    \sum_{j,j'\in [k]\backslash\{j_0\}: j\neq j'}|t_jt_{j'}|\leq \frac{1}{2}\sum_{j,j'\in [k]\backslash\{j_0\}: j\neq j'}(t_j^2+t_{j'}^2)\leq k\Big(\sum_{j\in [k]\backslash \{j_0\}} t_j^2\Big).
\end{equation}
Hence by the results in \textbf{Part 1.2.5}, we have
\begin{eqnarray}\label{NEQ31.1}
  &&  1-C n^{-1\slash 12}\leq  (\hat{f}_{j_0})^{\intercal}\hat{f}_{j_0} \nonumber\\
  &=& \sum_{j\in [k]\backslash\{j_0\}} t_j^2  (\hat{f}_j)^{\intercal}\hat{f}_j+\sum_{j,j'\in [k]\backslash\{j_0\}: j\neq j'}t_j t_{j'}(\hat{f}_j)^{\intercal}\hat{f}_{j'}\nonumber\\
  &\leq& (1+Cn^{-1\slash 12})\big(\sum_{j\in [k]\backslash\{j_0\}}t_j^2\big)+Cn^{-1\slash 12}\sum_{j,j'\in [k]\backslash\{j_0\}: j\neq j'}|t_jt_{j'}|\nonumber\\
  &\leq& (1+2C k n^{-1\slash 12})\big(\sum_{j\in [k]\backslash\{j_0\}}t_j^2\big).
\end{eqnarray}
Below we assume that $Ck n^{-1\slash 12}\leq 1\slash 4$. By (\ref{NEQ31.1}), we have
\begin{equation}
    \sum_{j\in [k]\backslash \{j_0\}}t_j^2\geq \frac{1-C n^{-1\slash 12}}{1+2Ck n^{-1\slash 12}}; \text{ similarly, }  \sum_{j\in [k]\backslash \{j_0\}}t_j^2\leq \frac{1+C n^{-1\slash 12}}{1-2Ck n^{-1\slash 12}}.
\end{equation}
Hence for any $j\in [k]\backslash \{j_0\}$, 
\begin{equation}\label{NEQ31.5}
    |t_j|\leq \sqrt{\frac{1+C n^{-1\slash 12}}{1-2Ck n^{-1\slash 12}}}.
\end{equation}
Moreover, there exists some $j_1\in [k]\backslash \{j_0\}$, such that 
\begin{equation}\label{NEQ31.8}
    |t_{j_1}|\geq \sqrt{\frac{1-C n^{-1\slash 12}}{k(1+2Ck n^{-1\slash 12})}}.
\end{equation}
By (\ref{NEQ31.5}), (\ref{NEQ31.8}), and the results in \textbf{Part 1.2.5}, we have 
\begin{eqnarray*}
 && C n^{-1\slash 12}\geq  \big|(\hat{f}_{j_0})^{\intercal}\hat{f}_{j_1}\big|= \big|t_{j_1}(\hat{f}_{j_1})^{\intercal}\hat{f}_{j_1}+\sum_{j\in [k]\backslash\{j_0,j_1\}}t_j (\hat{f}_{j})^{\intercal}\hat{f}_{j_1} \big|\nonumber\\
 &\geq& |t_{j_1}|\big|\hat{f}_{j_1})^{\intercal}\hat{f}_{j_1}\big|-\sum_{j\in [k]\backslash\{j_0,j_1\}}|t_j| \big|(\hat{f}_{j})^{\intercal}\hat{f}_{j_1} \big|\nonumber\\
 &\geq& (1-C n^{-1\slash 12})\sqrt{\frac{1-C n^{-1\slash 12}}{k(1+2Ck n^{-1\slash 12})}}-Ckn^{-1\slash 12}\sqrt{\frac{1+C n^{-1\slash 12}}{1-2Ck n^{-1\slash 12}}}.
\end{eqnarray*}
As $n^{e_0}\leq k\leq n^{1\slash 10000}$, there exists a positive constant $C'$ that only depends on $\beta,e_0$, such that for any $n\geq C'$, we have $Ckn^{-1\slash 12}\leq 1\slash 4$ and
\begin{equation*}
    C n^{-1\slash 12}<(1-C n^{-1\slash 12})\sqrt{\frac{1-C n^{-1\slash 12}}{k(1+2Ck n^{-1\slash 12})}}-Ckn^{-1\slash 12}\sqrt{\frac{1+C n^{-1\slash 12}}{1-2Ck n^{-1\slash 12}}},
\end{equation*}
which leads to a contradiction. Hence we conclude that when $n\geq C'$, $\{\hat{f}_j\}_{j=1}^k$ are linearly independent in $\mathbb{R}^n$. Below we assume that $n\geq C'$ (the case where $n<C'$ can be covered by enlarging the constant $C_1$ in the theorem). 

Now for any $l\in [k]$, consider the linear subspace $U_l$ of $\mathbb{R}^n$ spanned by $\{\hat{f}_j\}_{j=1}^l$. Note that $\dim(U_l)=l$. Consider an arbitrary vector $h_{\bm{\theta},l}=\sum_{j=1}^l \theta_j \hat{f}_j$ in $U_l$ with $\|h_{\bm{\theta},l}\|_2=1$, where $\bm{\theta}=(\theta_1,\theta_2,\cdots,\theta_l)\in\mathbb{R}^l$. By the results in \textbf{Part 1.2.5} and a similar derivation as in (\ref{NEQ32.2}),
\begin{eqnarray*}
  1=\sum_{j=1}^{l}\theta_j^2 (\hat{f}_j)^{\intercal}\hat{f}_j+\sum_{j,j'\in [l]:j\neq j'}\theta_j\theta_{j'} (\hat{f}_j)^{\intercal}\hat{f}_{j'}\leq(1+2C kn^{-1\slash 12})\big(\sum_{j=1}^l \theta_j^2\big),
\end{eqnarray*}
Hence by the results in \textbf{Part 1.2.5} and a similar derivation as in (\ref{NEQ32.2}),
\begin{eqnarray}
 &&(h_{\bm{\theta},l})^{\intercal} \hat{H}_{\beta,n} h_{\bm{\theta},l} = \sum_{j=1}^l \theta_j^2 (\hat{f}_j)^{\intercal} \hat{H}_{\beta,n}\hat{f}_j+\sum_{j,j'\in [l]:j\neq j'}\theta_j\theta_{j'}(\hat{f}_j)^{\intercal} \hat{H}_{\beta,n}\hat{f}_{j'}\nonumber\\
 &\geq& (a_l-2Ck n^{-1\slash 12})\big(\sum_{j=1}^l \theta_j^2\big)\geq \frac{a_l-2Ck n^{-1\slash 12}  }{1+2C kn^{-1\slash 12}}\geq a_l-C n^{-1\slash 24},
\end{eqnarray}
where we use the fact that $|a_j|=|\lambda_j|\leq k^2$ (as shown in \textbf{Sub-step 1.1}). By the Courant-Fischer-Weyl min-max principle, for any $l\in [k]$,
\begin{eqnarray*}
    \tilde{\lambda}_l^{(n)}=\max_{\dim(U)=l}\{\min_{\substack{\mathbf{x}\in U,\\ \|\mathbf{x}\|_2=1}}\{\mathbf{x}^{\intercal}\hat{H}_{\beta,n}\mathbf{x}\}\}\geq \min_{\substack{\mathbf{x}\in U_l,\\ \|\mathbf{x}\|_2=1}}\{\mathbf{x}^{\intercal}\hat{H}_{\beta,n}\mathbf{x}\}\geq a_l-C n^{-1\slash 24},
\end{eqnarray*}
where $U$ is a linear subspace of $\mathbb{R}^n$.

\paragraph{Step 2}

We recall that $L'=\lfloor L n^{1\slash 3}\rfloor n^{-1\slash 3}$. For every $j\in [k]$, we construct an approximation $\hat{g}_j$---a continuous function on $[0,\infty)$---of the eigenvector $g_j$ as follows. For every $i\in [1,L n^{1\slash 3}] \cap \mathbb{Z}$, we let $\hat{g}_j(i n^{-1\slash 3})=n^{1\slash 6} g_j(i)$. We also let $\hat{g}_j(0)=0$, $\hat{g}_j(L'+n^{-1\slash 3})=0$, and linearly interpolate between $(i n^{-1\slash 3},\hat{g}_j(i n^{-1\slash 3}))$ and $((i+1) n^{-1\slash 3},\hat{g}_j((i+1) n^{-1\slash 3}))$ on $[i n^{-1\slash 3},(i+1) n^{-1\slash 3}]$ for every $i\in \{0,1,\cdots,\lfloor L n^{1\slash 3}\rfloor\}$. For any $x>L'+n^{-1\slash 3}$, we let $\hat{g}_j(x)=0$.

We use the same coupling as described in \textbf{Step 1} throughout this step.

\subparagraph{Sub-step 2.1}

In this sub-step, we derive some preliminary bounds on $g_j\in\mathbb{R}^n$, where $j\in [k]$. These bounds will be used in later parts of the proof.

From $(g_j)^{\intercal}\hat{H}_{\beta,n}g_j=\tilde{\lambda}_j^{(n)}$ and the definition of $\{Z_i\}_{i=1}^{n-1}$ in (\ref{psum}), we have 
\begin{eqnarray}\label{EIGENNEQ}
&&    \tilde{\lambda}_j^{(n)} = -2n^{2\slash 3}\sum_{i=1}^n g_j(i)^2+2n^{2\slash 3}\sum_{i=1}^{n-1}g_j(i)g_j(i+1)+n^{1\slash 6}\sum_{i=1}^n\sqrt{\frac{2}{\beta}}\xi_i g_j(i)^2\nonumber\\
&-&2n^{1\slash 6}\sum_{i=1}^{n-1}(\sqrt{n}-\sqrt{n-i}) g_j(i)g_j(i+1)+2n^{1\slash 6}\sum_{i=1}^{n-1} Z_i g_j(i) g_j(i+1).   
\end{eqnarray}

The following statements in this sub-step hold with probability greater than or equal to $1-C\exp(-c k^3)$. By Proposition \ref{Edge} and Lemma \ref{gamma} (whose proofs are independent of other parts of the article), we have $|\tilde{\lambda}_j^{(n)}|\leq Ck^{20}$. Now by Proposition \ref{DiscreteDecay} and the fact that $|g_j(i)| \leq \|g_j\|_2= 1$ for all $i\in [n]$, we have 
\begin{equation}\label{NEQ37.1}
    |g_j(i)|\leq C|g_j(\lceil k^{12} n^{1\slash 3}\rceil)|\exp(-(i^{3\slash 2}n^{-1\slash 2}-k^{21})\slash 12)\leq C\exp(-c i^{3\slash 2} n^{-1\slash 2})
\end{equation}
for any $i\in [\lfloor L n^{1\slash 3}\rfloor,n]\cap\mathbb{Z}$. For any $i\in [n]$, as $\xi_i\sim N(0,1)$, we have $|\xi_i|\leq k^2$; for any $i\in [n-1]$, by Lemma \ref{ZL1}, $|Z_i|\leq k^4$. Thus $|n^{1 \slash 6}\sum_{i=1}^n\xi_ig_j(i)^2|\leq k^2n^{1\slash 6}$ and $|2n^{1\slash 6}\sum_{i=1}^{n-1} Z_{i}g_j(i)g_j(i+1)|\leq k^4n^{1\slash 6}\sum_{i=1}^{n-1}(g_j^2(i)+g_j^2(i+1))\leq 2k^4n^{1\slash 6}$. By (\ref{NEQ37.1}) and using $\sqrt{n}-\sqrt{n-i}\leq i n^{-1\slash 2}$ for any $i\in [n-1]$, we have
\begin{eqnarray*}
&&
\big|n^{1\slash 6}\sum_{i=1}^{n-1}(\sqrt{n}-\sqrt{n-i})g_j(i)g_j(i+1)\big|\leq n^{-1\slash 3}\sum_{i=1}^n i g_j(i)^2\\
&\leq & L\sum_{i=1}^{\lfloor Ln^{1\slash 3}\rfloor}g_j(i)^2+C\sum_{i=\lfloor L n^{1\slash 3}\rfloor+1}^{n}(in^{-1\slash 3})\exp\big(-c (in^{-1\slash 3})^{3\slash 2}\big)\leq C k^{20}.
\end{eqnarray*}
Plugging the above estimates into (\ref{EIGENNEQ}), we obtain that
\begin{equation}\label{NEQ37.10}
    g_j(1)^2+\sum_{i=1}^{n-1}(g_j(i+1)-g_j(i))^2+g_j(n)^2 \leq C k^4 n^{-1\slash 2}.
\end{equation}

\subparagraph{Sub-step 2.2}

Below we fix an arbitrary $j\in [k]$. For any $m,m'\in [2]$, we let $\tilde{\Pi}_{m,m',j}$ be obtained from $\Pi_{m,m',j}$ (as defined at the beginning of \textbf{Sub-step 1.2}) by replacing $\hat{f}_j$ with $g_j$. Note that $(g_j)^{\intercal}\hat{H}_{\beta,n} g_j=\sum\limits_{m,m'\in [2]}\tilde{\Pi}_{m,m',j}$. The following statements in this sub-step hold with probability greater than or equal to $1-C\exp(-c k^3)$. 

\textbf{Part 2.2.1}

We first bound $\tilde{\Pi}_{2,1,j}$ and $\tilde{\Pi}_{2,2,j}$. By (\ref{NEQ37.9}), (\ref{NEQ22.44}), (\ref{NEQ37.1}), (\ref{NEQ37.10}), the estimates on $\{Z_i\}_{i=1}^{n-1}$ in \textbf{Sub-step 2.1}, and the Cauchy-Schwarz inequality, 
\begin{equation}\label{NEQ42.1}
|\tilde{\Pi}_{2,1,j}|
\leq 2n^{1\slash 6}k^4\Big(\sum_{i=1}^{n-1} g_j(i)^2\Big)^{1\slash 2}\Big(\sum_{i=1}^{n-1}(g_j(i+1)-g_j(i))^2\Big)^{1\slash 2}
\leq Ck^6n^{-1\slash 12},
\end{equation}
\begin{eqnarray}\label{NEQ42.2}
|\tilde{\Pi}_{2,2,j}|&\leq& 2n^{1\slash 6}\sum_{i=1}^{\lfloor L n^{1\slash 3} \rfloor}|\Delta_i|g_j(i)^2+C n^{7\slash 6}\sum_{i=\lfloor L n^{1\slash 3}\rfloor +1}^{n}\exp\big(-c(i n^{-1\slash 3})^{3\slash 2}\big) \nonumber  \\
&\leq&  Cn^{-1\slash 6}+C\exp(-c k^{30})\leq Cn^{-1\slash 6}.
\end{eqnarray}

\textbf{Part 2.2.2}

Now for any $m\in [4]$, let $\tilde{\Upsilon}_{m,j}$ obtained from $\Upsilon_{m,j}$ (as defined in \textbf{Parts 1.2.1 and 1.2.2}) by replacing $\hat{f}_j$ with $g_j$. 

We bound $\tilde{\Pi}_{1,1,j}$ as follows. By (\ref{NEQ37.11}), $\big(\sum_{l=1}^{\beta(n-i)}(W_{i,l}^2-1)\big)\slash \big(\beta\sqrt{n-i}\big)\leq k^4$ for any $i\in [n-1]$. Hence by (\ref{NEQ20.25}), (\ref{NEQ37.1}), and the estimates on $\{\xi_i\}_{i=1}^n$ on \textbf{Sub-step 2.1}, we have
\begin{eqnarray}\label{NEQ37.18}
&&  \bigg|n^{1\slash 6}\sqrt{\frac{2}{\beta}}\sum_{i=\lfloor L n^{1\slash 3}\rfloor +1}^n \xi_i g_j(i)^2+n^{1\slash 6}\sum_{i=\lfloor L n^{1\slash 3}\rfloor+1}^{n-1} \frac{\sum_{l=1}^{\beta(n-i)}(W_{i,l}^2-1)}{\beta\sqrt{n-i}} g_j(i)^2\bigg| \nonumber\\
&\leq&  C n^{1\slash 6} k^4 \sum_{i=\lfloor L n^{1\slash 3 }\rfloor +1}^n\exp(-c i^{3\slash 2} n^{-1\slash 2})\leq C\exp(-c k^{30})\leq C n^{-1\slash 3}.
\end{eqnarray}
\begin{eqnarray}
n^{1\slash 6}\sqrt{2\slash \beta}\sum_{i=1}^{\lfloor L n^{1\slash 3}\rfloor} \Big|\zeta_i-\frac{\sum_{l=1}^{\beta(n-i)}(W_{i,l}^2-1)}{\sqrt{2\beta(n-i)}}\Big|g_j(i)^2\leq Cn^{-1\slash 12}.
\end{eqnarray}
Hence we have
\begin{equation}\label{NEQ37.17}
    |\tilde{\Pi}_{1,1,j}-\tilde{\Upsilon}_{1,j}|\leq C n^{-1\slash 12}.
\end{equation}

Below, we let $\tilde{g}_j(i):=g_j(i)$ for $i\in \{1,2,\cdots,\lfloor L n^{1\slash 3}\rfloor\}$ and $\tilde{g}_j(i):=0$ for $i\in\{0,\lfloor L n^{1\slash 3} \rfloor +1,\cdots,n\}$. Following the argument in (\ref{NEQ22.7}), we can deduce that $ \tilde{\Upsilon}_{1,j}=-2n^{1\slash 3}\beta^{-1\slash 2}\sum_{i=1}^{\lfloor L n^{1\slash 3}\rfloor} B_{i n^{-1\slash 3}} (\tilde{g}_j(i)^2-\tilde{g}_j(i+1)^2)$. We also let $\tilde{\Upsilon}_{1,j}':=4\beta^{-1\slash 2}\int_{0}^{\infty}\hat{g}_j'(x) \hat{g}_j(x)B_x dx$ (note that $\hat{g}_j$ is compactly supported and $\hat{g}_j'$ exists a.e. on $[0,\infty)$). Note that
\begin{equation}
    \tilde{\Upsilon}_{1,j}'=4n^{1\slash 2}\beta^{-1\slash 2}\sum_{i=0}^{\lfloor L n^{1\slash 3}\rfloor }\int_{i n^{-1\slash 3}}^{(i+1) n^{-1\slash 3}}B_x(\tilde{g}_j(i+1)-\tilde{g}_j(i))\hat{g}_j(x) dx.
\end{equation}
Let 
\begin{equation*}
    \Gamma_{1,j}:=n^{1\slash 2}\sum_{i=0}^{\lfloor L n^{1\slash 3}\rfloor }\int_{i n^{-1\slash 3}}^{(i+1) n^{-1\slash 3}}(\tilde{g}_j(i+1)-\tilde{g}_j(i))(B_x-B_{i n^{-1\slash 3}})\hat{g}_j(x)dx,
\end{equation*}
\begin{eqnarray*}
    \Gamma_{2,j}&:=&n^{1\slash 2}\sum_{i=0}^{\lfloor L n^{1\slash 3}\rfloor }\int_{i n^{-1\slash 3}}^{(i+1) n^{-1\slash 3}}\Big(B_{i n^{-1\slash 3}}(\tilde{g}_j(i+1)-\tilde{g}_j(i))\\
    &&\quad\quad\quad\quad\quad\quad\quad\quad\quad \times(n^{1\slash 6}(\tilde{g}_j(i)+\tilde{g}_j(i+1))-2\hat{g}_j(x))\Big)dx.
\end{eqnarray*}
Then $|\tilde{\Upsilon}_{1,j}-\tilde{\Upsilon}'_{1,j}|=4\beta^{-1\slash 2}|-\Gamma_{1,j}+\Gamma_{2,j}\slash 2|\leq 4\beta^{-1\slash 2}(|\Gamma_{1,j}|+|\Gamma_{2,j}|)$.

Consider any $i\in\{0,1,\cdots,\lfloor L n^{1\slash 3}\rfloor\}$. For any $x\in [i n^{-1\slash 3},(i+1) n^{-1\slash 3}]$, $|\hat{g}_j(x)| \leq  |\hat{g}_j(i n^{-1\slash 3})|+|\hat{g}_j((i+1)n^{-1\slash 3})|=n^{1\slash 6}(|\tilde{g}_j(i)|+|\tilde{g}_j(i+1)|)$; moreover,
\begin{equation*}
    \int_{i n^{-1\slash 3}}^{(i+1) n^{-1\slash 3}}|n^{1\slash 6}(\tilde{g}_j(i)+\tilde{g}_j(i+1))-2\hat{g}_j(x)|dx\leq |\tilde{g}_j(i+1)-\tilde{g}_j(i)|n^{-1\slash 6}.
\end{equation*}
By (\ref{NEQ37.1}), $|g_j(\lfloor L n^{1\slash 3}\rfloor +1)|\leq C\exp(-c L^{3\slash 2})$, which combined with (\ref{NEQ37.10}) gives
\begin{equation*}
    \sum_{i=0}^{\lfloor L n^{1\slash 3}\rfloor}(\tilde{g}_j(i+1)-\tilde{g}_j(i))^2\leq Ck^4 n^{-1\slash 2}.
\end{equation*}
Hence by (\ref{NEQ20.15}) and the Cauchy-Schwarz inequality,
\begin{eqnarray*}
    &&|\Gamma_{1,j}|\leq Cn^{1\slash 6}   k^2\sum_{i=0}^{\lfloor Ln^{1\slash 3}\rfloor}|\tilde{g}_j(i+1)-\tilde{g}_j(i)|(|\tilde{g}_j(i)|+|\tilde{g}_j(i+1)|)\nonumber\\
    &\leq& Cn^{1\slash 6} k^2\Big(\sum_{i=0}^{\lfloor L n^{1\slash 3}\rfloor}(\tilde{g}_j(i+1)-\tilde{g}_j(i))^2\Big)^{1\slash 2}\Big(\sum_{i=0}^{\lfloor L n^{1\slash 3}\rfloor}(|\tilde{g}_j(i)|+|\tilde{g}_j(i+1)|)^2\Big)^{1\slash 2}\nonumber\\
    &\leq& Ck^4 n^{-1\slash 12}\Big(\sum_{i=0}^{\lfloor L n^{1\slash 3}\rfloor}\tilde{g}_j(i)^2\Big)^{1\slash 2}\leq C k^4 n^{-1\slash 12},
\end{eqnarray*}
\begin{eqnarray*}
    |\Gamma_{2,j}|&\leq& Ck^{12}n^{1\slash 3}\sum_{i=0}^{\lfloor L n^{1\slash 3} \rfloor}(\tilde{g}_j(i+1)-\tilde{g}_j(i))^2\leq Ck^{16}n^{-1\slash 6}.
\end{eqnarray*}
Hence $|\tilde{\Upsilon}_{1,j}-4\beta^{-1\slash 2}\int_{0}^{\infty}\hat{g}_j'(x) \hat{g}_j(x)B_x dx|\leq C n^{-1\slash 16}$. Noting (\ref{NEQ37.17}), we have
\begin{equation}\label{NEQ42.3}
    \Big|\tilde{\Pi}_{1,1,j}-\frac{4}{\sqrt{\beta}}\int_{0}^{\infty}\hat{g}_j'(x) \hat{g}_j(x)B_x dx\Big|\leq C n^{-1\slash 16}. 
\end{equation}

\textbf{Part 2.2.3}

Finally, we bound $\tilde{\Pi}_{1,2,j}$. Note that $\tilde{\Pi}_{1,2,j}=\tilde{\Upsilon}_{2,j}+\tilde{\Upsilon}_{3,j}+\tilde{\Upsilon}_{4,j}$. 

We bound $\tilde{\Upsilon}_{2,j}$ as follows. By (\ref{NEQ37.1}), we have
\begin{eqnarray}\label{NEQ39.2}
 \Big| n^{-1\slash 3}\sum_{i=\lfloor L n^{1\slash 3}\rfloor +1}^{n-1} i g_j(i) g_j(i+1) \Big|&\leq & Cn\sum_{i=\lfloor L n^{1\slash 3}\rfloor +1}^{n-1}\exp(-c(i n^{-1\slash 3})^{3\slash 2})\nonumber\\
& \leq & C\exp(-c k^{30})\leq C n^{-1\slash 3},
\end{eqnarray}
\begin{eqnarray}\label{NEQ39.3}
   \int_{L'}^{\infty}x\hat{g}_j(x)^2dx & =&\int_{L'}^{L'+n^{-1\slash 3}}x\hat{g}_j(x)^2dx\leq (L'+n^{-1\slash 3})g_j(\lfloor L n^{1\slash 3}\rfloor)^2\nonumber\\
 &\leq& C\exp(-c k^{30})\leq C n^{-1\slash 3}.
\end{eqnarray}
Let 
\begin{equation*}
    \Gamma_{3,j}:=n^{-1\slash 3}\sum_{i=1}^{\lfloor L n^{1\slash 3}\rfloor}i g_j(i)(g_j(i+1)-g_j(i)),
\end{equation*}
\begin{equation*}
    \Gamma_{4,j}:=n^{-1\slash 3}\sum_{i=1}^{\lfloor L n^{1\slash 3} \rfloor}ig_j(i)^2-\int_{0}^{L'}x\hat{g}_j(x)^2 dx.
\end{equation*}
Note that for any $i\in [1,L n^{1\slash 3}]\cap\mathbb{Z}$ and any $x\in [(i-1)n^{-1\slash 3}, i n^{-1\slash 3}]$,
\begin{eqnarray}\label{NEQ39.1}
 &&|x\hat{g}_j(x)^2-i g_j(i)^2|\leq |x||n^{1\slash 6}
   g_j(i) -\hat{g}_j(x)|(2n^{1\slash 6}|g_j(i)|+|n^{1\slash 6}g_j(i)-\hat{g}_j(x)|)\nonumber\\
   &&\quad\quad\quad\quad\quad\quad\quad\quad\quad+n^{1\slash 3}|i n^{-1\slash 3}-x|g_j(i)^2\nonumber\\
   &&\leq  Ln^{1\slash 3} |\tilde{g}_j(i)-\tilde{g}_j(i-1)|(2|g_j(i)|+|\tilde{g}_j(i)-\tilde{g}_j(i-1)|)+g_j(i)^2.
\end{eqnarray}
By (\ref{NEQ37.10}), (\ref{NEQ39.1}), and the Cauchy-Schwarz inequality,
\begin{eqnarray*}
   &&\quad\quad |\Gamma_{3,j}|\leq L\sum_{i=1}^{\lfloor L n^{1\slash 3}\rfloor } |g_j(i)(g_j(i+1)-g_j(i))|\nonumber\\
  &\leq& L \Big(\sum_{i=1}^{\lfloor L n^{1\slash 3}\rfloor}g_j(i)^2\Big)^{1\slash 2}\Big(\sum_{i=1}^{\lfloor L n^{1\slash 3}\rfloor}(g_j(i+1)-g_j(i))^2\Big)^{1\slash 2}\leq C k^{22} n^{-1\slash 4},
\end{eqnarray*}
\begin{eqnarray*}
&& |\Gamma_{4,j}|\leq \sum_{i=1}^{\lfloor L n^{1\slash 3}\rfloor } \int_{(i-1) n^{-1\slash 3}}^{i n^{-1\slash 3}}|i g_j(i)^2-x\hat{g}_j(x)^2|dx\\
&\leq& n^{-1\slash 3}\sum_{i=1}^{\lfloor L n^{1\slash 3}  \rfloor}g_j(i)^2+L\sum_{i=1}^{\lfloor L n^{1\slash 3} \rfloor}(\tilde{g}_j(i)-\tilde{g}_j(i-1))^2\\
&& +2L \Big(\sum_{i=1}^{\lfloor L n^{1\slash 3} \rfloor}g_j(i)^2\Big)^{1\slash 2}\Big(\sum_{i=1}^{\lfloor L n^{1\slash 3}\rfloor}(\tilde{g}_j(i)-\tilde{g}_j(i-1))^2\Big)^{1\slash 2}\leq Ck^{22} n^{-1\slash 4}.
\end{eqnarray*}
Combining the above bounds on $\Gamma_{3,j}$ and $\Gamma_{4,j}$ with (\ref{NEQ39.2}) and (\ref{NEQ39.3}), we have
\begin{eqnarray}\label{NEQ41.1}
    \Big|\tilde{\Upsilon}_{2,j}+\int_{0}^{\infty}x\hat{g}_j(x)^2dx\Big|&=&\Big|\Gamma_{3,j}+\Gamma_{4,j}+n^{-1\slash 3}\sum_{i=\lfloor L n^{1\slash 3}\rfloor +1}^{n-1} i g_j(i) g_j(i+1) \nonumber\\
    &&-\int_{L'}^{\infty}x\hat{g}_j(x)^2dx\Big|\leq C n^{-1\slash 6}.
\end{eqnarray}

Now we bound $\tilde{\Upsilon}_{3,j}$. Following the argument in (\ref{NEQ39.2}), we have
\begin{equation}\label{NEQ39.4}
    \Big| n^{-1\slash 3}\sum_{i=\lfloor L n^{1\slash 3}\rfloor +1}^{n-1}\frac{\sqrt{n}-\sqrt{n-i}}{\sqrt{n}+\sqrt{n-i}} i g_j(i) g_j(i+1) \Big|\leq C\exp(-c k^{30})\leq C n^{-1\slash 3}.
\end{equation}
As $(\sqrt{n}-\sqrt{n-i})\slash (\sqrt{n}+\sqrt{n-i})=i (\sqrt{n}+\sqrt{n-i})^{-2}\leq i\slash n$ for any $i\in [n]$, we have
\begin{eqnarray}\label{NEQ39.5}
 &&   \bigg| n^{-1\slash 3}\sum_{i= 1}^{\lfloor L n^{1\slash 3}\rfloor}\frac{\sqrt{n}-\sqrt{n-i}}{\sqrt{n}+\sqrt{n-i}} i g_j(i) g_j(i+1) \bigg|\nonumber\\
 &\leq& n^{-4\slash 3}\sum_{i=1}^{\lfloor L n^{1\slash 3}\rfloor} i^2 |g_j(i)||g_j(i+1)|\nonumber\\
 &\leq& \frac{1}{2}L^2 n^{-2\slash 3}\sum_{i=1}^{\lfloor L n^{1\slash 3}\rfloor}(g_j(i)^2+g_j(i+1)^2)\nonumber\\
 &\leq& C k^{40}n^{-2\slash 3}\leq C n^{-1\slash 3}.
\end{eqnarray}
Combining (\ref{NEQ39.4}) and (\ref{NEQ39.5}) gives
\begin{equation}\label{NEQ41.2}
    |\tilde{\Upsilon}_{3,j}|\leq C n^{-1\slash 3}. 
\end{equation}

Below we bound $\tilde{\Upsilon}_{4,j}$. We note that
\begin{equation}
    \int_{0}^{\infty}(\hat{g}_j'(x))^2 dx=n^{2\slash 3}\Big(g_j(1)^2+\sum_{i=1}^{\lfloor L n^{1\slash 3}\rfloor -1}(g_j(i+1)-g_j(i))^2+g_j(\lfloor L n^{1\slash 3} \rfloor)^2\Big).
\end{equation}
Hence by (\ref{NEQ37.1}) and the AM-GM inequality, we have
\begin{eqnarray}\label{NEQ41.3}
 &&   \Big|\tilde{\Upsilon}_{4,j}+\int_{0}^{\infty}(\hat{g}_j'(x))^2dx\Big|\nonumber\\
 &=&n^{2\slash 3}\Big|\sum_{i=\lfloor L n^{1\slash 3}\rfloor}^{n-1}(g_j(i+1)-g_j(i))^2+g_j(n)^2-g_j(\lfloor L n^{1\slash 3}\rfloor)^2\Big|\nonumber\\
 &\leq& 8 n^{2\slash 3}\sum_{i=\lfloor L n^{1\slash 3}\rfloor}^n g_j(i)^2\leq Cn^{2\slash 3}\sum_{i=\lfloor L n^{1\slash 3}\rfloor}^n\exp(-c(i n^{-1\slash 3})^{3\slash 2})\nonumber\\
 &\leq& C\exp(-c k^{30})\leq C n^{-1\slash 3}.
\end{eqnarray}

Combining (\ref{NEQ41.1}), (\ref{NEQ41.2}), and (\ref{NEQ41.3}), we obtain that
\begin{equation}\label{NEQ42.4}
    \Big|\tilde{\Pi}_{1,2,j}+\int_{0}^{\infty}x\hat{g}_j(x)^2dx+\int_0^{\infty}(\hat{g}_j'(x))^2dx\Big|\leq Cn^{-1\slash 6}. 
\end{equation}

\textbf{Part 2.2.4}

Recall the definition of the bilinear form $\langle \cdot,\cdot \rangle_{H_{\beta}}$ from (\ref{Eqnn6}) in Section \ref{Sect.n}. Note that for any $j\in[k]$,
\begin{equation}
    \langle \hat{g}_j, \hat{g}_j\rangle_{H_{\beta}}=\int_{0}^{\infty}x\hat{g}_j(x)^2dx+\int_0^{\infty}(\hat{g}_j'(x))^2dx-\frac{4}{\sqrt{\beta}}\int_{0}^{\infty}\hat{g}_j'(x) \hat{g}_j(x)B_x dx.
\end{equation}
We denote by $\langle \cdot,\cdot\rangle_{L^2}$ the $L^2$ inner product on $[0,\infty)$. By (\ref{NEQ42.1}), (\ref{NEQ42.2}), (\ref{NEQ42.3}), and (\ref{NEQ42.4}), with probability at least $1-C\exp(-c k^3)$, for any $j\in[k]$, 
\begin{equation}
    |(g_j)^{\intercal}\hat{H}_{\beta,n}g_j+ \langle \hat{g}_j, \hat{g}_j\rangle_{H_{\beta}} |\leq Cn^{-1\slash 16}.
\end{equation}

Similarly, we can deduce that with probability at least $1-C\exp(-c k^3)$, the following holds: For any $j\in[k]$, $|\langle \hat{g}_j,\hat{g}_j\rangle_{L^2}-1|\leq Cn^{-1\slash 16}$; for any $j,j'\in[k]$ such that $j\neq j'$, $|\langle \hat{g}_j,\hat{g}_{j'}\rangle_{L^2}|\leq Cn^{-1\slash 16}$ and $|\langle \hat{g}_j, \hat{g}_{j'}\rangle_{H_{\beta}}|\leq C n^{-1\slash 16}$. Following the argument in \textbf{Sub-step 1.3} and noting (\ref{mini}), we can deduce that with probability at least $1-C\exp(-c k^3)$, for any $j\in [k]$, $\tilde{\lambda}_j^{(n)}\leq a_j+C n^{-1\slash 24}$. This combined with the result from \textbf{Step 1} shows that with probability at least $1-C\exp(-c k^3)$, for any $j\in [k]$, $|\tilde{\lambda}_j^{(n)} -a_j|  \leq   C n^{-1\slash 24}$.

\end{proof}

\section{Edge rigidity for the Gaussian $\beta$-ensemble}\label{Sect.4}

By adapting the proof of \cite[Proposition 6.2]{BEY}, we obtain the following edge rigidity result for the Gaussian $\beta$-ensemble. Note that the scaling of the Gaussian $\beta$-ensemble in this article (as given in Section \ref{Sect.n2}) is different from that in \cite{BEY}.

\begin{proposition}\label{Edge}
Assume that $\beta>0$ and $n\in\mathbb{N}_{+}$. Recall from Section \ref{Sect.n2} that the eigenvalues of $H_{\beta,n}$---the Gaussian $\beta$-ensemble of size $n$---are denoted by $\lambda_1^{(n)}>\lambda_2^{(n)}> \cdots>\lambda_n^{(n)}$. For any $\epsilon\in (0,1)$ and $a_0\in (0,1]$, there exist positive constants $N_0,C,c$ that only depend on $\beta,\epsilon,a_0$, such that for any $n\geq N_0$, $a\in [a_0,1]$, and $k\in [n]$, we have
\begin{equation}
    \mathbb{P}(|n^{-1\slash 2}\lambda_k^{(n)}-\gamma_{k}^{(n)}|\geq n^{-2\slash 3+a}(\hat{k})^{-1\slash 3})\leq C\exp(-cn^{a(1-\epsilon)\slash 2}),
\end{equation}
where $\hat{k}=\min\{n+1-k,k\}$ and $\{\gamma_j^{(n)}\}_{j=1}^n$ is defined in (\ref{classical}).
\end{proposition}
\begin{proof}
For any $j\in [n]$, we define $\hat{\gamma}_j^{(n)}$ by $\int_{\hat{\gamma}_j^{(n)}}^{\infty}\rho_1^{(n)}(x)dx=j\slash n$, where $\rho_1^{(n)}$ is the averaged density of the empirical spectral measure (defined as in \cite[Page 266]{BEY}). From large deviation results on the empirical spectral measure and extreme eigenvalues of $\beta$-ensembles (see, for example, \cite[Section 2.6]{AGZ} and equation (6.2) of \cite{BEY}), we have that for any fixed $\delta>0$, for some positive constants $C,c$ and any $k\in [n]$,
\begin{equation}
    \mathbb{P}\big(\big|n^{-1\slash 2}\lambda_k^{(n)}-\gamma_{k}^{(n)}\big|\geq \delta\big)\leq C\exp(-cn).
\end{equation}
Thus the conclusion holds for scale $a=1$.

In order to prove the result for all $a\in [a_0,1]$, we adapt a bootstrap argument in \cite{BEY}. First we use the accuracy result in the proof of \cite[Theorem 2.4]{BEY}. That is, for any $a\in [a_0,1]$, there is a constant $N_0$ such that for any $n>N_0$ and any $k\in [n]$,
\begin{equation}
    \big|\hat{\gamma}_k^{(n)}-
    \gamma_k^{(n)}\big|\leq n^{-2\slash 3+a}(\hat{k})^{-1\slash 3}.
\end{equation}
As in the proof of \cite[Theorem 2.4]{BEY}, we just need to prove a concentration result for all $a\in[a_0,1]$, that is, we need to show
\begin{equation}\label{Eq3}
    \mathbb{P}\big(\big|n^{-1\slash 2}\lambda_k^{(n)}-\mathbb{E}[n^{-1\slash 2}\lambda_k^{(n)}]\big|\geq n^{-2\slash 3+a}(\hat{k})^{-1\slash 3}\big)\leq C\exp\big(-cn^{a(1-\epsilon)\slash 2}\big).
\end{equation}

To show this concentration result, we adapt the bootstrap argument for \cite[Proposition 6.2]{BEY}. Below we state the adjustment that we need to make to get the conclusion. From the proof of \cite[Lemma 6.10]{BEY}, we can conclude that it suffices to prove the result for the convexified measure $\nu$ defined in \cite[Definition 6.8]{BEY} with the right-hand side of (\ref{Eq3}) replaced by $C\exp(-cn^{a(1-\epsilon)})$. We just need to apply the bootstrap argument to the convexified measure $\nu$. Fix $\epsilon\in(0,1\slash 10)$. Suppose that concentration at some scale $a\in [a_0,1]$ holds. Now we show for any $\epsilon'\in (0,\epsilon\slash 10)$, concentration at scale $(1-\epsilon')a$ holds. To avoid confusion, we denote by $\epsilon_0$ the $\epsilon$ that is used in the proof of \cite[Lemma 6.17]{BEY}, and here we take $\epsilon_0=a\epsilon\slash 2$. For $M=n^{a-\epsilon_0}$, from the proof of \cite[Lemma 6.16]{BEY}, we see that
\begin{equation*}
    S_{\omega^{(k,M)}}\big(d\nu/d\omega^{(k,M)}\big)\leq C\exp\big(-cn^{a(1-\epsilon)}\big).
\end{equation*}
In the proof of \cite[Lemma 6.18]{BEY}, replace $a$ by $a-\epsilon_0$, and take $r$ sufficiently small. This combined with the proof of \cite[Proposition 6.2]{BEY} shows that concentration at scale $a(1-\epsilon')$ holds.
\end{proof}

\section{Proof of large deviation bounds for the Airy point process}\label{Sect.5}

In this section, we finish the proof of the large deviation bounds for the Airy point process as stated in Theorems \ref{Main1} and \ref{Main2}. The strategy is to first use Theorem \ref{Main3}---which provides an approximation of the Airy point process using the Gaussian $\beta$-ensemble---to transfer the required bounds to large deviation bounds for the GUE, and then combine the estimates for the Airy point process in Section \ref{Sect.2} and the edge rigidity result in Section \ref{Sect.4} to obtain the desired conclusions. Throughout this section, we specialize $\beta=2$ and assume the notations in Sections \ref{Sect.n1} and \ref{Sect.n2} with $\beta=2$ (our argument applies to general $\beta\in\mathbb{N}_{+}$ as well; we specialize $\beta=2$ for ease of presentation). Unless otherwise specified, we use $C$ and $c$ to denote positive absolute constants, and the values of these constants may change from line to line.

Below we introduce some definitions that will be used for the remainder of this section. We define
\begin{equation}\label{Definitiony}
    \mathcal{Y}:=\text{ the set of finite signed Borel measures on }\mathbb{R}.
\end{equation}
For any $R>0$ and $\mu,\nu\in\mathcal{Y}$, we define
\begin{equation}
\tilde{d}_R(\mu,\nu):=\sup\limits_{\substack{f:\mathbb{R}\rightarrow\mathbb{R}:\|f\|_{BL}\leq 1,\\\supp(f)\subseteq [-R,R]}} \bigg|\int fd\mu-\int fd\nu \bigg|.
\end{equation}
Note that for any $\mu,\nu\in\mathcal{Y}$, $\tilde{d}_R(\mu,\nu)\leq d_R\big(\mu|_{[-R,R]},\nu|_{[-R,R]}\big)$. For the remainder of this section, we extend the definition of $d_R(\mu,\nu)$ to $\mu,\nu\in\mathcal{Y}$ by setting $d_R(\mu,\nu):=d_R(\mu|_{[-R,R]},\nu|_{[-R,R]})$.

\subsection{Proof of Theorem \ref{Main2}}
\label{sec:5.2}

In this subsection, we give the proof of Theorem \ref{Main2}. We start by introducing some definitions related to the GUE (which corresponds to the Gaussian $\beta$-ensemble with $\beta=2$). Throughout this section, we assume that $k,n\in\mathbb{N}_{+}$, and denote $\Delta:=\{(x,x):x\in\mathbb{R}\}$.

\begin{definition}\label{mu}
Let $\lambda_1^{(n)}>\lambda_2^{(n)}>\cdots>\lambda_n^{(n)}$ be the ordered eigenvalues of the GUE of size $n$. For any $i\in [n]$, we define $\tilde{\lambda}_i^{(n)}:=n^{1\slash 6}(\lambda_i^{(n)}-2\sqrt{n})$ and $b_i:=\left(n\slash k\right)^{2\slash 3}(2-n^{-1\slash 2}\lambda_i^{(n)})$; note that $k^{2\slash 3} b_i=-\tilde{\lambda}_i^{(n)}$. We also define $\mu_0$ to be the finite Borel measure on $\mathbb{R}$ with density 
\begin{equation*}
    \frac{1}{\pi} \sqrt{x}\Big(1-\frac{1}{4}(k\slash n)^{2\slash 3}x\Big)^{1\slash 2}\mathbbm{1}_{[0,4(n\slash k)^{2\slash 3}]}(x),\quad\forall x\in\mathbb{R}
\end{equation*}
with respect to the Lebesgue measure,
and let $\mu_{n,k}:=k^{-1}\sum_{i=1}^n\delta_{b_i}-\mu_0$.
\end{definition}
\begin{definition}\label{xi}
Let 
\begin{equation*}
    \rho_0(x)=\frac{1}{2\pi}\sqrt{4-x^2}\mathbbm{1}_{[-2, 2]}(x), \quad \forall x\in\mathbb{R}
\end{equation*}
be the density of the semicircle law. For any $x\in\mathbb{R}$, we define 
\begin{equation*}
    \xi(x):=-\int_{\mathbb{R}}  \log(|x-y|)\rho_0(y)dy+\frac{1}{4}x^2-\frac{1}{2},\quad \tilde{\xi}(x):=\frac{n}{k} \xi\big(2-(k\slash n)^{2\slash 3}x\big).
\end{equation*}
For any $\mu,\mu_1,\mu_2\in \mathcal{Y}$ (recall (\ref{Definitiony})) that are compactly supported, we define 
\begin{equation*}
    J(\mu):=-\int_{\mathbb{R}^2\backslash\Delta}\log(|x-y|)d\mu(x)d\mu(y),
\end{equation*}
\begin{equation*}
    J_0(\mu):=-\int_{\mathbb{R}^2\backslash\Delta}\log(|x-y|)d\mu(x)d\mu(y)+2\int_{\mathbb{R}}\tilde{\xi}(x)d\mu(x),
\end{equation*}
\begin{equation*}
    J(\mu_1,\mu_2):=-\int_{\mathbb{R}^2\backslash\Delta}\log(|x-y|)d\mu_1(x)d\mu_2(y).
\end{equation*}
\end{definition}

\begin{lemma}\label{L7.1}
The function $\xi(x)$ defined in Definition \ref{xi} satisfies the following properties:
\begin{enumerate}
     \item [(a)] For any $x\in\mathbb{R}$, $\xi(x)=\xi(-x)$.
    \item[(b)] For any $x\in [-2,2]$, $\xi(x)=0$.
    \item [(c)] For any $x>2$, $\xi'(x)=\sqrt{x^2-4}\slash 2$.
    \item [(d)] For any $x>2$, $2\left(x-2\right)^{3\slash 2}\slash 3\leq \xi(x)\leq (x+2)^{1\slash 2}(x-2)^{3\slash 2}\slash 3$.
\end{enumerate}
\end{lemma}
\begin{remark}
From part (d), we have $\lim\limits_{x\rightarrow 2^{+}}\xi(x)\slash (x-2)^{3\slash 2}=2\slash 3$.
\end{remark}

\begin{proof}

We first show part (a). By a change of variables, for any $x\in\mathbb{R}$,
\begin{eqnarray*}
   \xi(-x) &=& -\frac{1}{2\pi}\int_{-2}^2 \log(|-x-y|)\sqrt{4-y^2}dy+\frac{1}{4}x^2-\frac{1}{2}\\
   &=&  -\frac{1}{2\pi}\int_{-2}^2 \log(|x-y|)\sqrt{4-y^2}dy+\frac{1}{4}x^2-\frac{1}{2} = \xi(x).
\end{eqnarray*}

Now we show parts (b) and (c). For any $x\in\mathbb{R}$, we define 
\begin{equation*}
\rho_1(x):=\frac{1}{\pi}\sqrt{2-x^2}\mathbbm{1}_{[-\sqrt{2},\sqrt{2}]}(x),
\end{equation*}
\begin{equation*}
    \phi(x):=\int_{\mathbb{R}}  \log(|x-y|)\rho_0(y)dy,\quad  \psi(x):=\int_{\mathbb{R}} \log(|x-y|)\rho_1(y)dy.
\end{equation*}
By a change of variables, for any $x\in\mathbb{R}$, $\phi(x)=\log(2)\slash 2+\psi(x\slash \sqrt{2})$. By \cite[Lemma 2.7]{AG}, for any $x\in [-\sqrt{2},\sqrt{2}]$, $\psi(x)=x^2\slash 2-(\log(2)+1)\slash 2$. Hence for any $x\in [-2,2]$, $\phi(x)=x^2\slash 4-1\slash 2$ and $\xi(x)=0$. This establishes part (b). For any $x>2$, we have
\begin{equation}\label{Eqn1}
    \xi'(x)=-\frac{1}{2\pi}\int_{-2}^2\frac{\sqrt{4-y^2}}{x-y} dy+\frac{1}{2}x=-\frac{1}{\sqrt{2}\pi}\int_{-\sqrt{2}}^{\sqrt{2}}\frac{\sqrt{2-y^2}}{x\slash \sqrt{2}-y}dy+\frac{1}{2}x.
\end{equation}
By the discussion after equation (18) in \cite{AG}, for any $x>\sqrt{2}$,
\begin{equation}\label{Eqne1}
    \frac{1}{\pi}\int_{-\sqrt{2}}^{\sqrt{2}}\frac{\sqrt{2-y^2}}{x-y}dy=x-\sqrt{x^2-2}.
\end{equation}
Plugging (\ref{Eqne1}) into (\ref{Eqn1}), we obtain that $\xi'(x)=\sqrt{x^2-4}\slash 2$ for any $x>2$. This establishes part (c).

Finally, we show part (d). For any $x>2$, by parts (b) and (c), we have
\begin{equation}\label{Eqn2}
    \xi(x)=\xi(2)+\int_{2}^{x}\xi'(t)dt=\int_{2}^x\frac{\sqrt{t+2}\sqrt{t-2}}{2}dt.
\end{equation}
For any $2\leq t \leq x$, $2\leq\sqrt{t+2}\leq \sqrt{x+2}$. Hence by (\ref{Eqn2}), we obtain that
\begin{equation*}
    \xi(x)\geq \int_{2}^x\sqrt{t-2}dt=\frac{2}{3}(x-2)^{3\slash 2}, 
\end{equation*}
\begin{equation*}
    \xi(x)\leq \frac{\sqrt{x+2}}{2}\int_{2}^x\sqrt{t-2}dt=\frac{\sqrt{x+2}}{3}(x-2)^{3\slash 2}.
\end{equation*}
This establishes part (d).

\end{proof}

\begin{lemma}\label{gamma}
Recall the definitions of $\{\gamma_j\}_{j=1}^{\infty}$ and $\{\gamma_j^{(n)}\}_{j=1}^n$ from Sections \ref{Sect.n1} and \ref{Sect.n2}. For any $i\in [n]$ such that $i\leq n\slash 2$, we have
\begin{eqnarray}\label{estimate3}
&& \max\Big\{2- \Big(\frac{3\pi i}{\sqrt{2} n}\Big)^{2\slash 3},2-\Big(\frac{3\pi i}{2n}\Big)^{2\slash 3}\Big(1-\frac{1}{4}\Big(\frac{3\pi i}{\sqrt{2}n}\Big)^{2\slash 3}\Big)^{-1\slash 3}\Big\}\nonumber\\
&&   \leq \gamma_i^{(n)}\leq 2- \Big(\frac{3\pi i}{2n}\Big)^{2\slash 3}.
\end{eqnarray}
For any $e_0\in (0,1\slash 4)$, there exist positive constants $C_1,C_2$ that only depend on $e_0$, such that for any $n\geq C_1$ and any $i\in [n]$ satisfying $n^{e_0}\leq i\leq n^{1\slash 4}$, we have
\begin{equation}\label{estimate4}
\gamma_i\leq (2-\gamma_i^{(n)})n^{2\slash 3}< \min\big\{\gamma_i+C_2 i ^{2\slash 3}\big((i\slash n)^{2\slash 3}+i^{-1}\big),\gamma_{i+1}\big\}.
\end{equation}
\end{lemma}

\begin{proof}
We first consider the case where $i\in [n]$ and $i\leq n\slash 2$. By (\ref{classical}),
\begin{equation*}
    \frac{1}{2\pi}\int_{\gamma_i^{(n)}}^2\sqrt{(4-x^2)_{+}}dx=\frac{i}{n}\leq \frac{1}{2},
\end{equation*}
which leads to $\gamma_i^{(n)}\in [0,2]$. Hence we have
\begin{eqnarray*}
  &&  \frac{i}{n}=\frac{1}{2\pi}\int_{\gamma_i^{(n)}}^2\sqrt{(4-x^2)_{+}}dx=\frac{1}{2\pi}\int_{\gamma_i^{(n)}}^2\sqrt{2+x}\sqrt{2-x}dx\\
  &\leq& \frac{1}{\pi}\int_{\gamma_i^{(n)}}^2\sqrt{2-x}dx=\frac{2}{3\pi}(2-\gamma_i^{(n)})^{3\slash 2},
\end{eqnarray*}
\begin{eqnarray*}
 && \frac{i}{n}=\frac{1}{2\pi}\int_{\gamma_i^{(n)}}^2\sqrt{2+x}\sqrt{2-x}dx\geq \frac{\sqrt{2}}{2\pi}\int_{\gamma_i^{(n)}}^2\sqrt{2-x}dx=\frac{\sqrt{2}}{3\pi}(2-\gamma_i^{(n)})^{3\slash 2}.
\end{eqnarray*}
Hence
\begin{equation}\label{estimate1}
 2-\Big(\frac{3\pi i}{\sqrt{2} n}\Big)^{   2\slash 3} \leq \gamma_i^{(n)}\leq 2-\Big(\frac{3\pi i}{2 n}\Big)^{2\slash 3}.
\end{equation}
Now note that
\begin{eqnarray*}
   && \frac{i}{n}=\frac{1}{2\pi}\int_{\gamma_i^{(n)}}^2\sqrt{2+x}\sqrt{2-x}dx\geq \frac{1}{2\pi}\sqrt{2+\gamma_i^{(n)}}\int_{\gamma_i^{(n)}}^2\sqrt{2-x}dx\\
   &\geq& \frac{1}{2\pi}\Big(4-\Big(\frac{3\pi i}{\sqrt{2} n}\Big)^{2\slash 3}\Big)^{1\slash 2} \int_{\gamma_i^{(n)}}^2 \sqrt{2-x} dx\\
   &=& \frac{1}{3\pi}\Big(4-\Big(\frac{3\pi i}{\sqrt{2} n}\Big)^{2\slash 3}\Big)^{1\slash 2} (2-\gamma_i^{(n)})^{3\slash 2},
\end{eqnarray*}
where we use (\ref{estimate1}) in the second inequality. Hence 
\begin{equation}\label{estimate2}
    \gamma_i^{(n)}\geq 2-\Big(\frac{3\pi i}{2n}\Big)^{2\slash 3} \Big(1-\frac{1}{4}\Big(\frac{3\pi i}{\sqrt{2} n}\Big)^{2\slash 3}\Big)^{-1\slash 3}.
\end{equation}
Combining (\ref{estimate1}) and (\ref{estimate2}), we obtain (\ref{estimate3}).

Turning to the proof of (\ref{estimate4}), we 
let $C>0$ be the absolute constant appearing in Proposition \ref{Airyop}, and take $C_1=\max\{(4C)^{1\slash e_0},100^{1\slash e_0}\}$ and $C_2=10$. We consider any $n\geq C_1$ and $i\in [n]$ such that $n^{e_0}\leq i\leq n^{1\slash 4}$ throughout the rest of the proof. By Proposition \ref{Airyop}, $\gamma_i=\big(3\pi(i-1\slash 4+R(i))\slash 2\big)^{2\slash 3}$, where $|R(i)|\leq C i^{-1}$. As $i\geq n^{e_0}\geq C_1^{e_0}\geq 4C$, we have $|R(i)|\leq 1\slash 4$. Hence by (\ref{estimate3}), 
\begin{equation}\label{estimate5}
    \gamma_i\leq \Big(\frac{3\pi i}{2}\Big)^{2\slash 3}\leq (2-\gamma_i^{(n)})n^{2\slash 3}.
\end{equation}
Similarly, we can deduce that
\begin{equation}\label{QEA5}
    \gamma_{i+1}\geq \Big(\frac{3\pi (i+1\slash 2)}{2}\Big)^{2\slash 3}.
\end{equation}
As $n\geq C_1\geq 10$ and $i\leq n^{1\slash 2}$, we have
\begin{equation*}
    \frac{1}{4}\Big( \frac{3\pi i}{\sqrt{2} n}\Big)^{2\slash 3}\leq \frac{1}{4}\Big( \frac{3\pi }{\sqrt{2 n} }\Big)^{2\slash 3}\leq \frac{1}{2}.
\end{equation*}
By (\ref{estimate3}) and the fact that $(1-x)^{-1\slash 3}\leq (1-x)^{-1}\leq 1+2x$ for any $x\in [0,1\slash 2]$,
\begin{eqnarray}\label{estimate6}
    (2-\gamma_i^{(n)}) n^{2\slash 3} &\leq & \Big(\frac{3\pi i}{2}\Big)^{2\slash 3}\Big(1-\frac{1}{4}\Big(\frac{3\pi i}{\sqrt{2} n}\Big)^{2\slash 3}\Big)^{-1\slash 3}\nonumber\\
    &\leq& \Big(\frac{3\pi i}{2}\Big)^{2\slash 3}\Big( 1+\frac{1}{2} \Big(\frac{3\pi i}{\sqrt{2} n}\Big)^{2\slash 3}\Big).
\end{eqnarray}
Note that
\begin{equation}\label{estimate7}
    \gamma_i\geq \Big(\frac{3\pi}{2}\Big(i-\frac{1}{2}\Big)\Big)^{2\slash 3}\geq \Big(\frac{3\pi i}{2}\Big)^{2\slash 3}\Big(1-\frac{1}{2 i}\Big).
\end{equation}
Combining (\ref{estimate6}) and (\ref{estimate7}), we obtain that
\begin{eqnarray}\label{estimate8}
    (2-\gamma_i^{(n)})n^{2\slash 3} &\leq& \gamma_i+\frac{1}{2}\Big(\frac{3\pi i}{2}\Big)^{2\slash 3} \Big(\frac{1}{i}+\Big(\frac{3\pi i}{\sqrt{2} n}\Big)^{2\slash 3}\Big)\nonumber\\
    &<& \gamma_i+C_2 i^{2\slash 3}\big((i\slash n)^{2\slash 3}+i^{-1}\big). 
\end{eqnarray}
Note that $i\geq n^{e_0}\geq C_1^{e_0}\geq 100$. Combining this with (\ref{QEA5})-(\ref{estimate6}) and the facts that $i\leq n^{1\slash 4}$ and $1+x\leq (1+2x)^{2\slash 3}$ for any $x\in [0,1\slash 4]$, we obtain that
\begin{eqnarray}\label{QEA8}
    (2-\gamma_i^{(n)})n^{2\slash 3}&\leq& \Big(\frac{3\pi i}{2}\Big)^{2\slash 3}(1+2i^{-2})\leq \Big(\frac{3\pi i}{2}\Big)^{2\slash 3}(1+4 i^{-2})^{2\slash 3} \nonumber\\
    &<& \Big (\frac{3\pi (i+1\slash 2)}{2}\Big)^{2\slash 3}\leq \gamma_{i+1}.
\end{eqnarray}
Combining (\ref{estimate5}), (\ref{estimate8}), and (\ref{QEA8}), we obtain (\ref{estimate4}).

\end{proof}

\begin{lemma}\label{A}
Assume that $n^{e_1}\leq k\leq n^{e_2}$, where $e_1,e_2$ are fixed positive constants satisfying $e_1\leq e_2<1\slash 10000$. For any $R\geq 1$, there exist positive constants $C,c$ (that only depend on $e_1,e_2$) and a positive constant $K_R$ (that only depends on $R,e_1,e_2$), such that for any $n\in\mathbb{N}_{+}$ with $n\geq K_R$, we have 
\begin{equation}\label{EQR1.1}
\mathbb{P}(d_R(\mu_{n,k}|_{[-R,R]},0)> C R k^{-1\slash 4})\leq C\exp(-ck^{1\slash 8}),
\end{equation}
where $\mu_{n,k}$ is defined in Definition \ref{mu}.
\end{lemma}
\begin{proof}

We decompose $\mu_{n,k}$ as $\mu_{n,k}=\mu_1+\mu_2$, where
\begin{equation*}
\mu_1:=k^{-1}\sum_{i=1}^n\big(\delta_{b_i}-\delta_{(n\slash k)^{2\slash 3}(2-\gamma_i^{(n)})}\big), \quad  \mu_2:=k^{-1}\sum_{i=1}^n\delta_{(n\slash k)^{2\slash 3}(2-\gamma_i^{(n)})}-\mu_0.
\end{equation*}
Throughout the rest of the proof, when we say ``for $n$ sufficiently large'', we mean that there exists a positive constant $K'_R$ that only depends on $R,e_1,e_2$, such that when $n\geq K'_R$, the corresponding property holds. In this proof, we use $C$ and $c$ to denote positive constants that only depend on $e_1,e_2$; the values of these constants may change from line to line.

We take $\eta=15$ and denote by $C_0$ the constant $C_2$ in the statement of Theorem \ref{Main3}. We couple the stochastic Airy operator $H_{\beta}$ with $\beta=2$ with the GUE of size $n$ (i.e., $H_{\beta,n}$ with $\beta=2$) as in Theorem \ref{Main3}. We also define the following events (recall (\ref{Eq2.14})):
\begin{itemize}
    \item $\mathcal{A}_1$: the event that $N(k^{2\slash 3}(R+1))\leq \eta (R+1)^{3\slash 2}k$;
    \item $\mathcal{A}_2$: the event that $|(2-n^{-1\slash 2}\lambda_i^{(n)})n^{2\slash 3}-\lambda_i|\leq C_0 n^{-1\slash 24}$ for every\\ $i\in [1,\min\{\eta(R+1)^{3\slash 2}k+1,n\}]\cap\mathbb{Z}$;
    \item $\mathcal{A}_3$: the event that $|n^{-1\slash 2}\lambda_i^{(n)}-\gamma_i^{(n)}|\leq n^{-2\slash 3}k^{1\slash 2}(\min\{i,n+1-i\})^{-1\slash 3}$ for every $i\in [n]$.
\end{itemize}
By Proposition \ref{PropAiry1} and Theorem \ref{Main3}, for $n$ sufficiently large,
\begin{equation}\label{EQN1.1}
    \mathbb{P}(\mathcal{A}_1^c)\leq C\exp(-c(R+1)^3 k^2), \quad \mathbb{P}(\mathcal{A}_2^c)\leq C\exp(-ck^3).
\end{equation}
In Proposition \ref{Edge}, for $n$ sufficiently large, we take $\epsilon=1\slash 2$, $a=\log{k}\slash (2\log{n})$, and obtain that for any $i\in [n]$,
\begin{equation*}
\mathbb{P}(|n^{-1\slash 2}\lambda_i^{(n)}-\gamma_i^{(n)}|\geq n^{-2\slash 3}k^{1\slash 2}(\min\{i,n+1-i\})^{-1\slash 3})\leq C \exp(-ck^{1\slash 8}).
\end{equation*}
By the union bound, we have $\mathbb{P}(\mathcal{A}_3^c)\leq Cn\exp(-c k^{1\slash 8})\leq C\exp(-c k^{1\slash 8})$, which combined with (\ref{EQN1.1}) gives
\begin{equation}\label{EQW1.1}
    \mathbb{P}((\mathcal{A}_1\cap\mathcal{A}_2\cap\mathcal{A}_3)^c)\leq C\exp(-ck^{1\slash 8}).
\end{equation}

In the rest of the proof, we assume that the event $\mathcal{A}_1\cap\mathcal{A}_2\cap\mathcal{A}_3$ holds. We also denote $\bar{b}_i:=(n\slash k)^{2\slash 3}(2-\gamma_i^{(n)})$ for each $i\in [n]$.

Let $i_0:=\lfloor \eta(R+1)^{3\slash 2}k \rfloor+1$. Then, for $n$ sufficiently large, we have
\begin{equation*}
    (2-n^{-1\slash 2}\lambda_{i_0}^{(n)})n^{2\slash 3}\geq \lambda_{i_0}-Cn^{-1\slash 24}\geq k^{2\slash 3}(R+1)-Cn^{-1\slash 24}> k^{2\slash 3}R,
\end{equation*}
which leads to $\#\{i\in [n]:b_i\leq R\}\leq i_0\leq  2\eta(R+1)^{3\slash 2}k$. As the event $\mathcal{A}_3$ holds, for $n$ sufficiently large,
\begin{eqnarray}\label{EQN10.1}
  \sum_{\substack{i\in [n]:b_i\leq R,\\\bar{b}_i\leq R}}|b_i-\bar{b}_i|\leq \sum_{\substack{i\in [n]:\\i\leq 2\eta(R+1)^{3\slash 2}k}}(n\slash k)^{2\slash 3}|n^{-1\slash 2}\lambda_i^{(n)}-\gamma_i^{(n)}|\leq CR k^{1\slash 2}.
\end{eqnarray}

We let $I_1:=\{i\in [n]:(2-n^{-1\slash 2}\lambda_i^{(n)})(n\slash k)^{2\slash 3}\leq R, (2-\gamma_i^{(n)})(n\slash k)^{2\slash 3}>R\}$ and $I_2:=\{i\in[n]:(2-\gamma_i^{(n)})(n\slash k)^{2\slash 3}\leq R, (2-n^{-1\slash 2}\lambda_i^{(n)})(n\slash k)^{2\slash 3}>R\}$. We also let $t_1:=\# I_1$ and $t_2:=\# I_2$. 

Suppose that $t_1\geq 1$. Let $i_1$ be the smallest $i\in [n]$ such that \begin{equation}\label{EQN1.3}
    (2-\gamma_i^{(n)})(n\slash k)^{2\slash 3}>R.
\end{equation}
By the definition of $I_1$, $i_1+t_1-1\leq n$ and $(2-n^{-1\slash 2}\lambda_{i_1+t_1-1}^{(n)})(n\slash k)^{2\slash 3}\leq R$. This combined with (\ref{EQN1.3}) leads to $\gamma_{i_1}^{(n)}\leq n^{-1\slash 2}\lambda_{i_1+t_1-1}^{(n)}$. As the event $\mathcal{A}_3$ holds,
\begin{eqnarray}\label{EQN1.4}
&&\gamma_{i_1}^{(n)}-\gamma_{i_1+t_1-1}^{(n)}\leq n^{-1\slash 2}\lambda_{i_1+t_1-1}^{(n)}-\gamma_{i_1+t_1-1}^{(n)}\nonumber\\
&\leq& n^{-2\slash 3}k^{1\slash 2}(\min\{i_1+t_1-1,n-i_1-t_1+2\})^{-1\slash 3}.
\end{eqnarray}
For $n$ sufficiently large, by the definition of $I_1$ and Lemma \ref{gamma}, we have either $i_1=1$ or $(2-\gamma_{i_1-1}^{(n)})(n\slash k)^{2\slash 3}\leq R$, hence $i_1\leq n\slash 8$ and $\gamma_{i_1}^{(n)}\geq 1$; by (\ref{EQN1.4}), $\gamma_{i_1}^{(n)}-\gamma_{i_1+t_1-1}^{(n)}\leq n^{-2\slash 3} k^{1\slash 2}\leq 1\slash 2$, hence $\gamma_{i_1+t_1-1}^{(n)}\geq 1\slash 2$ and $i_1+t_1-1\leq n \slash 2$. Plugging this into (\ref{EQN1.4}), we obtain
\begin{equation}\label{EQN2.1}
    \gamma_{i_1}^{(n)}-\gamma_{i_1+t_1-1}^{(n)}\leq n^{-2\slash 3}k^{1\slash 2}(i_1+t_1-1)^{-1\slash 3}. 
\end{equation}
By (\ref{classical}), (\ref{EQN2.1}), and Lemma \ref{gamma}, for $n$ sufficiently large,
\begin{eqnarray*}
  &&  \frac{t_1-1}{n}=\frac{1}{2\pi}\int_{\gamma_{i_1+t_1-1}^{(n)}}^{\gamma_{i_1}^{(n)}}\sqrt{4-x^2}dx\leq \pi^{-1}(2-\gamma_{i_1+t_1-1}^{(n)})^{1\slash 2}(\gamma_{i_1}^{(n)}-\gamma_{i_1+t_1-1}^{(n)})\nonumber\\
  &&\leq C k^{1\slash 2}n^{-1}, \text{ hence } \# I_1=t_1\leq C k^{1\slash 2}.
\end{eqnarray*}

Now suppose that $t_2\geq 1$. Suppose that $i_2$ is the smallest $i\in [n]$ such that $(2-n^{-1\slash 2}\lambda_i^{(n)})(n\slash k)^{2\slash 3}>R$. From the definition of $I_2$, we can deduce that $i_2+t_2-1\leq n$ and
\begin{equation}\label{EQN2.3}
    (2-\gamma^{(n)}_{i_2+t_2-1})(n\slash k)^{2\slash 3}\leq R.
\end{equation}
Hence $n^{-1\slash 2}\lambda_{i_2}^{(n)}\leq \gamma_{i_2+t_2-1}^{(n)}$. For $n$ sufficiently large, by (\ref{EQN2.3}), $\gamma_{i_2+t_2-1}^{(n)}\geq 1$, hence $i_2\leq i_2+t_2-1\leq n\slash 2$; as the event $\mathcal{A}_3$ holds,
\begin{eqnarray}\label{EQN2.5}
 \gamma_{i_2}^{(n)}-\gamma_{i_2+t_2-1}^{(n)}\leq \gamma_{i_2}^{(n)}-n^{-1\slash 2}\lambda_{i_2}^{(n)}\leq n^{-2\slash 3}k^{1\slash 2}i_2^{-1\slash 3}.
\end{eqnarray}
By (\ref{classical}), (\ref{EQN2.5}), and Lemma \ref{gamma}, for $n$ sufficiently large,
\begin{eqnarray*}
  &&  \frac{t_2-1}{n}=\frac{1}{2\pi}\int_{\gamma_{i_2+t_2-1}^{(n)}}^{\gamma_{i_2}^{(n)}}\sqrt{4-x^2} dx \leq \pi^{-1}(2-\gamma_{i_2+t_2-1}^{(n)})^{1\slash 2}(\gamma_{i_2}^{(n)}-\gamma_{i_2+t_2-1}^{(n)})\nonumber\\
  &&\leq C k^{1\slash 2}n^{-1}\Big(1+\frac{t_2-1}{i_2}\Big)^{1\slash 3}\leq C k^{1\slash 2}n^{-1}t_2^{1\slash 3},\text{ hence } \# I_2=t_2\leq C k^{3\slash 4}.
\end{eqnarray*}

Throughout the rest of the proof, we assume that $n$ is sufficiently large.

We first bound $d_R(\mu_1|_{[-R,R]},0)$. Consider any function $\phi(x)$ on $[-R,R]$ with $\|\phi\|_{\infty},\|\phi\|_{Lip}\leq 1$. For any $i\in [n]$, $\bar{b}_i\geq 0$; if $b_i\leq -R$, then $|b_i-\bar{b}_i|\geq R\geq 1$. Hence by (\ref{EQN10.1}) and the above bounds on $\# I_1,\# I_2$,
\begin{eqnarray*}
  &&  \Big|\int_{[-R,R]}\phi d\mu_1\Big|=\Big| k^{-1}\sum_{\substack{i\in [n]: b_i\in [-R,R]}}\phi(b_i)-k^{-1}\sum_{\substack{i\in [n]:\bar{b}_i\in [-R,R]}}\phi(\bar{b}_i)\Big|\nonumber\\
  &\leq& k^{-1}\Big(\# I_1+\# I_2+\#\{i\in [n]:b_i<-R,\bar{b}_i\leq R\}+\sum_{\substack{i\in [n]:b_i\in [-R,R],\\\bar{b}_i\in [-R,R]}}|b_i-\bar{b}_i|\Big)\nonumber\\
  &\leq& k^{-1}\Big(\# I_1+\# I_2+\sum_{i\in [n]: b_i\leq R,\bar{b}_i\leq R}|b_i-\bar{b}_i|\Big)\leq C R k^{-1\slash 4}.
\end{eqnarray*}
Hence $d_R(\mu_1|_{[-R,R]},0)\leq CRk^{-1\slash 4}$.

Now we bound $d_R(\mu_2|_{[-R,R]},0)$. Consider any function $\phi(x)$ on $[-R,R]$ such that $\|\phi\|_{\infty}, \|\phi\|_{Lip}\leq 1 $. Let $\gamma_0:=0$ and $i_0'$ be the largest $i\in [n]$ such that $\bar{b}_i\leq R$. By Lemma \ref{gamma}, we have $\bar{b}_{i_0}\geq (3\pi i_0\slash(2 k))^{2\slash 3} >R$, hence $i_0'< i_0$. Hence by (\ref{classical}) and Lemma \ref{gamma},
\begin{eqnarray*}
&&\Big|\int_{[-R,R]} \phi d\mu_2\Big|=\Big|\frac{1}{k}\sum_{i=1}^{i_0'}\phi\big((n\slash k)^{2\slash 3}(2-\gamma_i^{(n)})\big)\\
&&  \quad\quad\quad\quad\quad\quad\quad\quad-\frac{n}{k}\int_{2-(k\slash n)^{2\slash 3}R}^2  \frac{\sqrt{4-t^2}}{2\pi}
\phi\big((n\slash k)^{2\slash 3}(2-t)\big)dt\Big|\\
&\leq & \frac{n}{k}\sum_{i=1}^{i_0'}\int_{\gamma_{i}^{(n)}}^{\gamma_{i-1}^{(n)}}\big|\phi\big((n\slash k)^{2\slash 3}(2-\gamma_i^{(n)})\big)-\phi\big((n\slash k)^{2\slash 3}(2-t)\big)\big|\frac{\sqrt{4-t^2}}{2\pi} dt\\
&& + \frac{n}{k}\int_{2-(k\slash n)^{2\slash 3}R}^{\gamma_{i_0'}^{(n)}} \frac{\sqrt{4-t^2}}{2\pi} \big|\phi\big((n\slash k)^{2\slash 3}(2-t)\big)\big|dt\\
&\leq& \frac{1}{k}\Big(\frac{n}{k}\Big)^{2\slash 3}\sum_{i=1}^{i_0}(\gamma_{i-1}^{(n)}-\gamma_i^{(n)})+\frac{n}{k}\int_{\gamma_{i_0'+1}^{(n)}}^{\gamma_{i_0'}^{(n)}}\frac{\sqrt{4-t^2}}{2\pi} dt\\
&=& \frac{1}{k}\Big(\frac{n}{k}\Big)^{2\slash 3}(2-\gamma_{i_0}^{(n)})+\frac{1}{k}\leq \frac{1}{k}\Big(\frac{n}{k}\Big)^{2\slash 3}\Big(\frac{3\pi i_0}{\sqrt{2} n}\Big)^{2\slash 3}+\frac{1}{k}\leq CR k^{-1}.
\end{eqnarray*}
Hence $d_R(\mu_2|_{[-R,R]},0)\leq CR k^{-1}$.

Therefore, $d_R(\mu_{n,k}|_{[-R,R]},0)\leq d_R(\mu_1|_{[-R,R]},0)+d_R(\mu_2|_{[-R,R]},0)\leq CR k^{-1\slash 4}$. We obtain (\ref{EQR1.1}) by noting (\ref{EQW1.1}).

\end{proof}

Throughout the rest of this section, we fix $R_1,R_2\in\mathbb{R}_{+}$ such that
\begin{equation}\label{cond}
     R_1\geq 10, \quad  R_2\geq 2R_1.
\end{equation}

\begin{definition}
Recall the definition of $\mu_{n,k}$ from Definition \ref{mu}. We let $\mu_1,\mu_2,\mu_2',\mu_3,\mu_3'$ be finite signed Borel measures on $\mathbb{R}$ such that for any $A\in\mathcal{B}_{\mathbb{R}}$,
\begin{eqnarray*}
  &  \mu_1(A)=\mu_{n,k}(A\cap [-R_1,R_1]),\quad \mu_2(A)=\mu_{n,k}(A\cap (R_1,R_2]),\nonumber\\
  & \mu'_{2}(A)=\mu_{n,k}(A\cap[-R_2,-R_1)),\quad \mu_3(A)=\mu_{n,k}(A\cap (R_2,\infty)),\nonumber\\
  & \mu'_3(A)=\mu_{n,k}(A\cap (-\infty,-R_2)).
\end{eqnarray*}
\end{definition}

\begin{lemma}\label{field}
Assume that $k,n\in\mathbb{N}_{+}$ and $n^{5\cdot 10^{-7}}\leq k\leq  n^{10^{-6}}$. Recall that we have fixed $R_1, R_2$ such that (\ref{cond}) holds. For any $x\in [-R_1,R_1]$, let
\begin{equation*}
    E(x):=\int_{(R_2,\infty)} \frac{1}{x-y}d\mu_3(y), \quad E'(x):=\int_{(-\infty,-R_2)} \frac{1}{x-y}d\mu'_3(y).
\end{equation*}
For any $M\geq 1$ and $C>0$, let 
\begin{equation*}
\mathcal{C}_{M,C}:=\big\{|E(x)|\leq CM R_2^{-1\slash 2}\text{ for all }x\in [-R_1, R_1]\big\},
\end{equation*}
\begin{equation*}
\mathcal{C}'_{M,C}:=\big\{|E'(x)|\leq CMR_2^{-5\slash 2}\text{ for all }x\in [-R_1,R_1]\big\},
\end{equation*}
\begin{equation*}
\mathcal{C}''_{M,C}:=\big\{\mu_{n,k}\left((-\infty,-R_2]\right)\leq CMR_2^{-3\slash 2}\big\}.
\end{equation*}
Then there exist positive absolute constants $C_1,C_2$, such that the following holds. For any $M\geq 1$, there exists a positive constant $K_{R_1,R_2,M}$ that only depends on $R_1,R_2,M$, such that when $k\geq K_{R_1,R_2,M}$, we have
\begin{equation}
\mathbb{P}((\mathcal{C}_{M,C_1})^c\cup(\mathcal{C}'_{M,C_1})^c\cup (\mathcal{C}''_{M,C_1})^c)\leq C_2\exp(-Mk^2).
\end{equation}
\end{lemma}
\begin{proof}

Throughout the proof, we assume that $k$ (hence $n$) is sufficiently large (depending on $R_1,R_2,M$). Using Theorem \ref{Main3}, we couple $\{b_i\}_{i=1}^n$ with $\{\lambda_i\}_{i=1}^{\infty}$, such that with probability at least $1-C\exp(-c k^3)$, for any $i\in [k^{50}]$,
\begin{equation}\label{EQA1}
    |k^{2\slash 3}b_i-\lambda_i|\leq C n^{-1\slash 24}.
\end{equation}

For any $x\in [-R_1,R_1]$ and $y\leq -R_2$, we have $|x-y|\geq R_2-R_1\geq R_2\slash 2$. Hence for any $x\in [-R_1,R_1]$,
\begin{equation}\label{EQA2}
  |E'(x)|\leq  \int_{(-\infty,-R_2)} \frac{1}{|x-y|}d\mu_3'(y)\leq 2R_2^{-1}\mu_{n,k}\left((-\infty,-R_2]\right).
\end{equation}
Arguing as in the proof of Proposition \ref{PropAiry1} and recalling (\ref{Eq2.14}), we can deduce that there exists a positive absolute constant $C'$, such that
\begin{eqnarray}
\mathbb{P}(N(-R_2 k^{2\slash 3}\slash 2)>C'M R_2^{-3\slash 2}k\slash 2) \leq C \exp(-M k^2). 
\end{eqnarray}
Now assume that $N(-R_2 k^{2\slash 3}\slash 2)\leq C'M R_2^{-3\slash 2}k\slash 2$ and take $l_0:=\lfloor C'M R_2^{-3\slash 2} k \rfloor$. As we have taken $k$ to be sufficiently large (depending on $R_1,R_2,M$), we have $l_0\in [k^{50}]$ and $\lambda_{l_0}\geq -R_2 k^{2\slash 3}\slash 2$. Assuming that (\ref{EQA1}) holds, we have
\begin{equation}\label{NN00}
    b_{l_0}\geq k^{-2\slash 3}(\lambda_{l_0}-C n^{-1\slash 24})\geq -R_2\slash 2-C n^{-1\slash 24}> -R_2,
\end{equation}
which leads to $\mu_{n,k}((-\infty,-R_2])< l_0\slash k\leq C'M R_2^{-3\slash 2}$. Hence by (\ref{EQA2}),
\begin{equation}\label{Q01}
    \mathbb{P}(\mu_{n,k}((-\infty,-R_2])\geq C' M R_2^{-3\slash 2})\leq C\exp(-Mk^2),
\end{equation}
\begin{equation}\label{Q02}
    \mathbb{P}(|E'(x)|\geq 2C' M R_2^{-5\slash 2}\text{ for some }x\in [-R_1,R_1])\leq C\exp(-Mk^2).
\end{equation}

In the following, we consider an arbitrary $x\in [-R_1,R_1]$ and bound $|E(x)|$. We denote $\gamma_i':=(2-\gamma_i^{(n)})(n\slash k)^{2\slash 3}$ for any $i\in \{0\}\cup [n]$. We have
\begin{equation}\label{ESA10.1}
|E(x)|\leq\Big|\int_{(R_2,\gamma'_{k^{48}})}\frac{1}{x-y}d\mu_3(y)\Big|
+\Big|\int_{[\gamma'_{k^{48}},\infty)}\frac{1}{x-y}d\mu_3(y)\Big|=: M_2+M_1.
\end{equation}
By Proposition \ref{Edge}, with probability at least $1-C\exp(-ck^3)$, 
\begin{equation}
    |n^{-1\slash 2}\lambda_{k^{48}}^{(n)}-\gamma_{k^{48}}^{(n)}|\leq n^{-2\slash 3}k^{-4}, \quad |n^{-1\slash 2}\lambda_{k^{48}+1}^{(n)}-\gamma_{k^{48}+1}^{(n)}|\leq n^{-2\slash 3}k^{-4}.
\end{equation}
By Lemma \ref{gamma} and Proposition \ref{Airyop}, we have $\gamma_{k^{48}}\leq \gamma'_{k^{48}}k^{2\slash 3} < \gamma_{k^{48}+1}$. By Proposition \ref{Airyop}, we have $\min\{\gamma_{k^{48}}-\gamma_{k^{48}-1}, \gamma_{k^{48}+2}-\gamma_{k^{48}+1}\}\geq k^{-16}\slash 4\geq  n^{-1\slash 30}$. By Proposition \ref{PropAiry2}, with probability at least $1-C\exp(-c k^3)$,
\begin{equation}\label{EQA15}
    |N(\gamma'_{k^{48}}k^{2\slash 3}+ n^{-1\slash 30})-k^{48}|\leq \lfloor k^{25}\slash 2 \rfloor,\quad |N(\gamma'_{k^{48}}k^{2\slash 3}- n^{-1\slash 30})-k^{48}|\leq \lfloor k^{25}\slash 2 \rfloor.
\end{equation}
Let $i_0:=k^{48}+\lfloor k^{25}\slash 2 \rfloor+1$ and $i_0':=k^{48}-\lfloor k^{25}\slash 2 \rfloor-1$. By (\ref{EQA15}), we have $\lambda_{i_0}\geq \gamma'_{k^{48}}k^{2\slash 3}+n^{-1\slash 30}$ and $\lambda_{i_0'}\leq \gamma'_{k^{48}}k^{2\slash 3}-n^{-1\slash 30}$. Combining this with (\ref{EQA1}), we obtain that $b_{i_0'}< \gamma'_{k^{48}}< b_{i_0}$. Letting 
\begin{equation}\label{DT}
    T_1:=\#\{i\in [n]: b_i<\gamma'_{k^{48}}\}, 
\end{equation}
we obtain that with probability at least $1-C\exp(-c k^3)$, $|T_1-k^{48}|\leq k^{25}$. By Lemma \ref{gamma}, $\gamma'_{k^{48}}\geq k^{30}$. Hence if $k^{48}-k^{25}\leq T_1\leq k^{48}-1$, for any $x\in [-R_1,R_1]$,
\begin{equation}\label{QQ1.1}
    \Big|\frac{1}{k}\sum_{i=T_1+1}^{k^{48}}\frac{1}{b_i-x}\Big|\leq \frac{k^{24}}{\gamma'_{k^{48}}-R_1}\leq Ck^{-6};
\end{equation}
if $k^{48}\leq T_1\leq k^{48}+k^{25}$, for any $x\in [-R_1,R_1]$,
\begin{equation}\label{QQ1.2}
    \Big|\int_{\gamma'_{k^{48}}}^{\gamma'_{T_1}}\frac{1}{y-x}d\mu_0(y)\Big| \leq \frac{k^{24}}{\gamma'_{k^{48}}-R_1}\leq Ck^{-6}.
\end{equation}

Now we bound $M_1$ (defined in (\ref{ESA10.1})). By Proposition \ref{Edge}, with probability at least $1-C\exp(-ck^3)$, for any $i\in [n]$,
\begin{equation}
    |n^{-1\slash 2}\lambda_i^{(n)}-\gamma_i^{(n)}|\leq n^{-2\slash 3}k^{12}(\min\{i,n+1-i\})^{-1\slash 3}.
\end{equation}
By Lemma \ref{gamma}, $\gamma_0^{(n)}-\gamma_1^{(n)}\leq 4 n^{-2\slash 3}$. By (\ref{classical}) and Lemma \ref{gamma}, for any $i\in [n]$ such that $2\leq i\leq n\slash 2$,
\begin{eqnarray}\label{QSS1}
    \frac{1}{n}&=&\frac{1}{2\pi}\int_{\gamma_{i}^{(n)}}^{\gamma_{i-1}^{(n)}}\sqrt{(4-x^2)_{+}}dx\geq  \frac{\sqrt{2}}{2\pi}(2-\gamma_{i-1}^{(n)})^{1\slash 2} (\gamma_{i-1}^{(n)}-\gamma_i^{(n)})\nonumber\\
  &\geq& \frac{\sqrt{2}}{2\pi} \Big(\frac{3\pi(i-1)}{2n}\Big)^{1\slash 3}(\gamma_{i-1}^{(n)}-\gamma_i^{(n)})\geq \frac{1}{4}(i\slash n)^{1\slash 3} (\gamma_{i-1}^{(n)}-\gamma_i^{(n)}),
\end{eqnarray}
hence $\gamma_{i-1}^{(n)}-\gamma_i^{(n)}\leq 4i^{-1\slash 3}n^{-2\slash 3}$. If $n$ is odd and $i=(n+1)\slash 2$, by (\ref{classical}), arguing as in (\ref{QSS1}), we have $\gamma_{i}^{(n)}=-\gamma_{i-1}^{(n)}$ and \begin{equation}
    \frac{1}{2n}=\frac{1}{2\pi}\int_0^{\gamma^{(n)}_{i-1}}\sqrt{(4-x^2)_{+}}dx\geq \frac{1}{4}(i\slash n)^{1\slash 3}\gamma_{i-1}^{(n)},
\end{equation}
hence $\gamma_{i-1}^{(n)}-\gamma_i^{(n)}=2\gamma_{i-1}^{(n)}\leq 4i^{-1\slash 3} n^{-2\slash 3}$. Now consider any $i\in [n]$ such that $n\slash 2+1\leq i\leq n-1$. By (\ref{classical}),
\begin{equation}\label{QSS2}
    \frac{1}{n}=\frac{1}{2\pi}\int_{\gamma_i^{(n)}}^{\gamma_{i-1}^{(n)}}\sqrt{(4-x^2)_{+}}dx\geq \frac{\sqrt{2}}{2\pi}(2+\gamma_i^{(n)})^{1\slash 2}(\gamma_{i-1}^{(n)}-\gamma_i^{(n)}),
\end{equation}
\begin{equation}\label{QS3}
    \frac{n-i}{n}=\frac{1}{2\pi}\int_{-2}^{\gamma_i^{(n)}}\sqrt{(4-x^2)_{+}}dx\leq \frac{2}{3\pi}(2+\gamma_i^{(n)})^{3\slash 2}.
\end{equation}
By (\ref{QSS2}) and (\ref{QS3}), we have $\gamma_{i-1}^{(n)}-\gamma_i^{(n)}\leq 4 (n+1-i)^{-1\slash 3}n^{-2\slash 3}$. Finally, by (\ref{classical}), we have $\gamma_{n}^{(n)}=-2$ and  
\begin{equation*}
    \frac{1}{n}=\frac{1}{2\pi}\int_{-2}^{\gamma_{n-1}^{(n)}}\sqrt{(4-x^2)_{+}}dx \geq \frac{\sqrt{2}}{2\pi}\int_{-2}^{\gamma_{n-1}^{(n)}}\sqrt{2+x}dx= \frac{\sqrt{2}}{3\pi}(\gamma_{n-1}^{(n)}+2)^{3\slash 2},
\end{equation*}
hence $\gamma_{n-1}^{(n)}-\gamma_n^{(n)}=\gamma_{n-1}^{(n)}+2\leq 4 n^{-2\slash 3}$. Therefore, for any $i  \in [n]$, 
\begin{equation}
    \gamma_{i-1}^{(n)}-\gamma_i^{(n)} \leq  4 n^{-2\slash  3} (\min\{i,n+1-i\})^{-1\slash 3}.
\end{equation}

Therefore, with probability at least $1-C\exp(-c k^3)$, for any $i  \in [n]$ such that $i\geq k^{48}$, the following properties hold:
\begin{itemize}
    \item[(a)] $|b_i-\gamma'_i|=(n\slash k)^{2\slash 3}|n^{-1\slash 2}\lambda_i^{(n)}-\gamma_i^{(n)}|\leq k^{12} (\min\{i,n+1-i\})^{-1\slash 3}$;
    \item[(b)]
    \begin{eqnarray*}
        |b_i-\gamma'_{i-1}|&\leq& (n\slash k)^{2\slash 3}(|n^{-1\slash 2}\lambda_i^{(n)}-\gamma_i^{(n)}|+|\gamma_i^{(n)}-\gamma_{i-1}^{(n)}|)\nonumber\\
        &\leq& 2 k^{12} (\min\{i,n+1-i\})^{-1\slash 3};
    \end{eqnarray*}
    \item[(c)] By Lemma \ref{gamma}, $\gamma_i'\geq 2(\min\{i, \lfloor n\slash 2\rfloor\})^{2\slash 3} k^{-2\slash 3}\geq 2k^{30}$,\\ $\gamma'_{i-1}\geq (\min\{i,\lfloor n\slash 2\rfloor\})^{2\slash 3}k^{-2\slash 3}\geq k^{30}$;
    \item[(d)] By (a) and (c), $b_i\geq (\min\{i,\lfloor n\slash 2\rfloor\})^{2\slash 3}k^{-2\slash 3}\geq k^{30}$.
\end{itemize}
From the above properties, we can deduce that for any $x\in [-R_1,R_1]$, if $i\in [n]$ and $k^{48}\leq i\leq n\slash 2$, then 
\begin{equation}\label{QQ1}
    \Big|\frac{1}{b_i-x}-\frac{1}{\gamma'_i-x}\Big|= \frac{|b_i-\gamma'_i|}{(b_i-x)(\gamma'_i-x)}  \leq C k^{14} i^{-5\slash 3},
\end{equation}
\begin{equation}\label{QQ2}
    \Big|\frac{1}{b_i-x}-\frac{1}{\gamma'_{i-1}-x}\Big|=\frac{|b_i-\gamma'_{i-1}|}{(b_i-x)(\gamma'_{i-1}-x)}\leq Ck^{14} i^{-5\slash 3},
\end{equation}
if $i\in [n]$ and $i>n\slash 2$,
\begin{equation}\label{QQ3}
    \Big|\frac{1}{b_i-x}-\frac{1}{\gamma'_i-x}\Big|\leq C k^{14}n^{-4\slash 3}, \quad \Big|\frac{1}{b_i-x}-\frac{1}{\gamma'_{i-1}-x}\Big|\leq C k^{14} n^{-4\slash 3}.
\end{equation}
Note that for any $x\in [-R_1,R_1]$ and any $i\in [n]$ such that $i\geq k^{48}$, 
\begin{eqnarray}\label{QQ4}
  &&  \Big|\frac{1}{k}\frac{1}{b_i-x}-\int_{\gamma'_{i-1}}^{\gamma'_i}\frac{1}{y-x}d\mu_0(y)\Big|\leq \int_{\gamma'_{i-1}}^{\gamma'_i}\Big|\frac{1}{b_i-x}-\frac{1}{y-x}\Big|d\mu_0(y)\nonumber\\
  &\leq& \frac{1}{k}\max\Big\{\Big|\frac{1}{b_i-x}-\frac{1}{\gamma'_i-x}\Big|,\Big|\frac{1}{b_i-x}-\frac{1}{\gamma'_{i-1}-x}\Big|\Big\}.
\end{eqnarray}
By (\ref{QQ1})-(\ref{QQ4}), with probability at least $1-C\exp(-c k^3)$, the following holds: For any $x\in [-R_1,R_1]$, if $i\in [n]$ and $k^{48}\leq i\leq n\slash 2$,
\begin{equation}\label{QQ1.3}
    \Big|\frac{1}{k}\frac{1}{b_i-x}-\int_{\gamma'_{i-1}}^{\gamma'_i}\frac{1}{y-x}d\mu_0(y)\Big|\leq C k^{13}i^{-5\slash 3};
\end{equation}
if $i\in [n]$ and $i>n\slash 2$,
\begin{equation}\label{QQ1.4}
    \Big|\frac{1}{k}\frac{1}{b_i-x}-\int_{\gamma'_{i-1}}^{\gamma'_i}\frac{1}{y-x}d\mu_0(y)\Big|\leq C k^{13} n^{-4\slash 3}.
\end{equation}

With probability at least $1-C\exp(-c k^3)$, the following statements in this paragraph hold. For any $x\in [-R_1,R_1]$, if $T_1$ as defined in (\ref{DT}) satisfies $T_1\leq k^{48}-1$, then by (\ref{QQ1.1}), (\ref{QQ1.3}), and (\ref{QQ1.4}),
\begin{eqnarray}
 M_1&=&\Big|\sum_{i=k^{48}+1}^n \Big(\frac{1}{k}\frac{1}{b_i-x}-\int_{\gamma'_{i-1}}^{\gamma'_i}\frac{1}{y-x}d\mu_0(y)\Big)+\frac{1}{k}\sum_{i=T_1+1}^{k^{48}}\frac{1}{b_i-x}\Big|\nonumber\\
 &\leq& \sum_{i=k^{48}+1}^n \Big|\frac{1}{k}\frac{1}{b_i-x}-\int_{\gamma'_{i-1}}^{\gamma'_i}\frac{1}{y-x}d\mu_0(y)\Big|+\Big|\frac{1}{k}\sum_{i=T_1+1}^{k^{48}}\frac{1}{b_i-x}\Big|\nonumber\\
 &\leq& C\sum_{i=k^{48}+1}^{\lfloor n\slash 2\rfloor}k^{13} i^{-5\slash 3}+C k^{13} n^{-1\slash 3}+C k^{-6}\leq C k^{-6};
\end{eqnarray}
if $T_1\geq k^{48}$, then by (\ref{QQ1.2}), (\ref{QQ1.3}), and (\ref{QQ1.4}), 
\begin{eqnarray}
 M_1&=&\Big|\sum_{i=T_1+1}^n \Big(\frac{1}{k}\frac{1}{b_i-x}-\int_{\gamma'_{i-1}}^{\gamma'_i}\frac{1}{y-x}d\mu_0(y)\Big)-\int_{\gamma'_{k^{48}}}^{\gamma'_{T_1}}\frac{1}{y-x}d\mu_0(y)\Big|\nonumber\\
 &\leq& \sum_{i=T_1+1}^n \Big|\frac{1}{k}\frac{1}{b_i-x}-\int_{\gamma'_{i-1}}^{\gamma'_i}\frac{1}{y-x}d\mu_0(y)\Big|+\Big|\int_{\gamma'_{k^{48}}}^{\gamma'_{T_1}}\frac{1}{y-x}d\mu_0(y)\Big|\nonumber\\
 &\leq& C\sum_{i=k^{48}+1}^{\lfloor n\slash 2\rfloor}k^{13} i^{-5\slash 3}+C k^{13} n^{-1\slash 3}+C k^{-6}\leq C k^{-6}.
\end{eqnarray}
Hence we conclude that with probability at least $1-C\exp(-c k^3)$, for any $x\in [-R_1,R_1]$, $M_1\leq C k^{-6}$.

Next we bound $M_2$ (defined in (\ref{ESA10.1})). Suppose that $\gamma'_{l-1}\leq R_2<\gamma'_{l}$, where $l\in [n]$. We first show that $l\leq n\slash 2$: If $l\geq \lfloor n\slash 2\rfloor+1$, by Lemma \ref{gamma}, 
\begin{equation*}
    \gamma_{l-1}^{(n)}\leq \gamma_{\lfloor n\slash 2\rfloor}^{(n)}\leq 2-(3\pi\lfloor n\slash 2\rfloor)^{2\slash 3}\slash (2n)^{2\slash 3}\leq 1, \text{ hence } \gamma'_{l-1}\geq (n\slash k)^{2\slash 3}>R_2,
\end{equation*}
which leads to a contradiction. By Lemma \ref{gamma}, 
\begin{equation}\label{QQ1.5}
    R_2\geq \gamma'_{l-1}\geq \Big(\frac{3\pi(l-1)}{2k}\Big)^{2\slash 3}, \text{ hence } l\leq \frac{2}{3\pi}R_2^{3\slash 2} k+1,
\end{equation}
\begin{eqnarray}\label{QQ1.6}
  &&  R_2<\gamma'_l\leq \Big(\frac{3\pi l}{2k}\Big)^{2\slash 3}\Big(1-\frac{1}{4} \Big(\frac{3\pi l}{\sqrt{2} n}\Big)^{2\slash 3}\Big)^{-1\slash 3}\leq \Big(\frac{3\pi l}{2k}\Big)^{2\slash 3}(1-n^{-1\slash 3})^{-1\slash 3},\nonumber\\
  && \text{ hence } l\geq \frac{2}{3\pi} R_2^{3\slash 2}k-1,
\end{eqnarray}
where we use (\ref{QQ1.5}) in the third inequality in (\ref{QQ1.6}).

For any $i\in [k^{47}]$, we let
\begin{equation*}
    S_i:=\mu_{n,k}((-\infty,\gamma'_{ki})),\quad S_i^{+}:=k^{-1}\#\{j\in [n]: b_j<\gamma'_{ki}\}, \quad S_i^{-}:=\mu_0([0,\gamma'_{ki})),
\end{equation*}
\begin{eqnarray*}
 && \tilde{S}_i^{++}:=k^{-1}\#\{j\in\mathbb{N}_{+}: \lambda_j\leq\gamma'_{ki}k^{2\slash 3}+n^{-1\slash 30}\},\\
 && \tilde{S}_i^{+-}:=k^{-1}\#\{j\in\mathbb{N}_{+}: \lambda_j\leq\gamma'_{ki}k^{2\slash 3}-n^{-1\slash 30}\},
\end{eqnarray*}
and note that $S_i=S_i^{+}-S_i^{-}$ and $S_i^{-}=i$. By Lemma \ref{gamma} and Proposition \ref{Airyop}, $\gamma_{ki}\leq \gamma'_{ki}k^{2\slash 3}<\gamma_{ki+1}$. By Proposition \ref{Airyop},
\begin{equation*}
    \min\{\gamma_{ki}-\gamma_{ki-1}, \gamma_{ki+2}-\gamma_{ki+1}\}\geq k^{-16}\slash 4\geq  n^{-1\slash 30}.
\end{equation*}
Hence $\gamma_{ki-1}\leq \gamma'_{ki}k^{2\slash 3}-n^{-1\slash 30}<\gamma'_{ki}k^{2\slash 3}+n^{-1\slash 30}< \gamma_{ki+2}$. By Propositions \ref{PropAiry1} and \ref{PropAiry2}, with probability at least $1-C\exp(-M k^2)$, for any $i\in\mathbb{N}_{+}$ such that $\lceil l\slash k\rceil\leq i\leq k^{1\slash 2}$,
\begin{equation}\label{QQQ1.1}
    \max\{|\tilde{S}_i^{++}-i|,|\tilde{S}_i^{+-}-i|\}\leq CM\log(i);
\end{equation}
for any $i\in\mathbb{N}_{+}$ such that $k^{1\slash 2}\leq i\leq k^{47}$, 
\begin{equation}\label{QQQ1.3}
    \max\{|\tilde{S}_i^{++}-i|,|\tilde{S}_i^{+-}-i|\}\leq C\log(i)(\log(k))^{1\slash 2}(\max\{i\slash k,1\})^{1\slash 2}.
\end{equation}
Now consider any $i\in \mathbb{N}_{+}$ such that $\lceil l\slash k\rceil\leq i\leq k^{47}$, and let $j_0:=k \tilde{S}_i^{++}$ and $j_0':=k\tilde{S}_i^{+-}$. From (\ref{QQQ1.1}) and (\ref{QQQ1.3}), we can deduce that $j_0+1,j_0'\leq k^{50}$. By (\ref{EQA1}), we can deduce that with probability at least $1-C\exp(-Mk^2)$, $\tilde{S}_i^{+-}\leq S_i^{+}\leq \tilde{S}_i^{++}$ as follows: 
\begin{itemize}
    \item[(a)] As $\lambda_{j_0+1}>\gamma'_{ki} k^{2\slash 3}+n^{-1\slash 30}$, we have 
    \begin{equation*}
        b_{j_0+1}\geq k^{-2\slash 3}(\lambda_{j_0+1}-C n^{-1\slash 24})>\gamma'_{ki}, \text{ hence } S_i^{+}\leq j_0\slash k= \tilde{S}_i^{++}.
    \end{equation*}
    \item[(b)] If $j_0'=0$, then $\tilde{S}_i^{+-}=0$ and $S_i^{+}\geq \tilde{S}_i^{+-}$; otherwise, $\lambda_{j_0'}\leq \gamma'_{ki} k^{2\slash 3}-n^{-1\slash 30}$, and
    \begin{equation*}
        b_{j_0'}\leq k^{-2\slash 3}(\lambda_{j_0'}+C n^{-1\slash 24})< \gamma'_{ki}, \text{ hence }S_i^{+}\geq j_0'\slash k=\tilde{S}_i^{+-}.
    \end{equation*}
 \end{itemize}
Therefore, with probability at least $1-C\exp(-M k^2)$, for any $i\in\mathbb{N}_{+}$ such that $\lceil l\slash k\rceil\leq i\leq k^{1\slash 2}$, 
\begin{equation}\label{QQQ2.2}
   |S_i|= |S_i^{+}-i|\leq CM\log(i);
\end{equation}
for any $i\in\mathbb{N}_{+}$ such that $k^{1\slash 2}\leq i\leq k^{47}$,
\begin{equation}\label{QQQ2.3}
    |S_i|=|S_i^{+}-i|\leq C\log(i)(\log(k))^{1\slash 2}(\max\{i\slash k,1\})^{1\slash 2}.
\end{equation}

Consider an arbitrary $i\in\mathbb{N}_{+}$ such that $\lceil l\slash k \rceil \leq i\leq k^{47}-1$. By Lemma \ref{gamma}, $|\gamma'_{(i+1)k}-\gamma'_{ik}|\leq C i^{-1\slash 3}$ and $\gamma'_{ik}\geq i^{2\slash 3}$. Hence for any $x\in [-R_1,R_1]$,
\begin{eqnarray}\label{QQQ2.11}
&&\Big|\int_{[\gamma'_{ik},\gamma'_{(i+1)k})}\frac{1}{y-x}d\mu_{n,k}(y)-\frac{S_{i+1}-S_i}{\gamma'_{ik}-x}\Big|\nonumber\\
&=&\Big|\int_{[\gamma'_{ik},\gamma'_{(i+1)k})}\Big(\frac{1}{y-x}-\frac{1}{\gamma'_{ik}-x}\Big)d\mu_{n,k}(y)\Big|\nonumber\\
&\leq& \int_{[\gamma'_{ik},\gamma'_{(i+1)k})}\frac{(\gamma'_{(i+1)k}-\gamma'_{ik})}{(y-x)(\gamma'_{ik}-x)}d(\mu_{n,k}+2\mu_0)(y)\nonumber\\
& \leq  & Ci^{-5\slash 3}(|S_{i+1}-S_i|+1 ).
\end{eqnarray}

By (\ref{QQQ2.2}) and (\ref{QQQ2.3}), with probability at least $1-C\exp(-Mk^2)$, for any $x\in [-R_1,R_1]$,
\begin{eqnarray}\label{QQQ2.12}
&&\Big|\sum_{i=\lceil l\slash k \rceil}^{k^{47}-1}\frac{S_{i+1}-S_i}{\gamma'_{ik}-x}\Big|\nonumber\\
&\leq& \Big|\sum_{i=\lceil l\slash k \rceil}^{k^{47}-1} S_{i+1}\Big(\frac{1}{\gamma'_{ik}-x}-\frac{1}{\gamma'_{(i+1)k}-x}\Big)\Big|+\frac{|S_{k^{47}}|}{\gamma'_{k^{48}}-x}+\frac{|S_{\lceil l\slash  k \rceil}|}{\gamma'_{\lceil l\slash k \rceil k}-x}\nonumber\\
&\leq& \sum_{i=\lceil l\slash k\rceil}^{k^{47}-1}\frac{|S_{i+1}|(\gamma'_{(i+1)k}-\gamma'_i)}{(\gamma'_{ik}-x)(\gamma'_{(i+1)k}-x)}+\frac{|S_{k^{47}}|}{\gamma'_{k^{48}}-x}+\frac{|S_{\lceil l\slash  k \rceil}|}{\gamma'_{\lceil l\slash k \rceil k}-x} \nonumber \\
&\leq& C\sum_{i=\lceil l\slash k \rceil+1}^{k^{47}}i^{-5\slash 3}|S_i|+\frac{C M \log(R_2)}{R_2}.
\end{eqnarray}

Now let $I_0:=(R_2, \gamma'_{\lceil l\slash k\rceil k})$. As $\gamma'_{l-1}\leq R_2<\gamma'_l$ and $|l-(2 R_2^{3\slash 2}k)\slash (3\pi)|\leq 1$, by Lemma \ref{gamma}, $\gamma_{l-1}\leq R_2 k^{2\slash 3}<\gamma_{l+1}$. By Propositions \ref{PropAiry1} and \ref{PropAiry2} and (\ref{EQA1}), with probability at least $1-C\exp(-Mk^2)$,
\begin{equation*}
    |\#\{j\in [n]:b_j\leq R_2\}-l|\leq CM\log(R_2)k.
\end{equation*}
Noting (\ref{QQQ2.2}) and that $(l-1)\slash k\leq \mu_0([0,R_2])\leq l\slash k$, $\mu_0([0,\gamma'_{\lceil l\slash k \rceil k }))=\lceil l\slash k\rceil$, we have
\begin{eqnarray*}
  && | \mu_{n,k}(I_0)|\leq |S_{\lceil l\slash k\rceil }^{+}-\lceil l\slash k\rceil|+ |k^{-1}\#\{j\in [n]: b_j\leq R_2\}-\mu_0([0,R_2])|\\
  &\leq& CM\log(\lceil l\slash k\rceil)+CM\log(R_2)+k^{-1}\leq CM\log(R_2),
\end{eqnarray*}
\begin{equation*}
    \mu_0(I_0)=\mu_0([0,\gamma'_{\lceil l\slash k \rceil k }))-\mu_0([0,R_2])\leq 2.
\end{equation*}
Hence $\mu_{n,k}(I_0)+2\mu_0(I_0)\leq CM\log(R_2)$. Therefore, with probability at least $1-C\exp(-M k^2)$, for any $x\in [-R_1,R_1]$,
\begin{eqnarray}\label{QQQ2.16}
&&\Big|\int_{I_0}\frac{1}{y-x}d\mu_{n,k}(y)\Big|\leq \frac{2}{R_2}\int_{I_0}d(\mu_{n,k}+2\mu_0)(y)\nonumber\\
&=&\frac{2(\mu_{n,k}(I_0)+2\mu_0(I_0))}{R_2}\leq \frac{CM\log(R_2)}{R_2}.
\end{eqnarray}

With probability at least $1-C\exp(-M k^2)$, the following statements in this paragraph hold. By (\ref{QQQ2.11})-(\ref{QQQ2.16}), for any $x\in [-R_1,R_1]$,
\begin{eqnarray}\label{QQ11}
M_2 &\leq& \Big|\int_{I_0}\frac{1}{y-x}d\mu_{n,k}(y)\Big|+\Big|\sum_{i=\lceil l\slash k \rceil}^{k^{47}-1}\frac{S_{i+1}-S_i}{\gamma'_{ik}-x}\Big|\nonumber\\
&& +\sum_{i=\lceil l\slash k\rceil}^{k^{47}-1}\Big|\int_{[\gamma'_{ik},\gamma'_{(i+1)k})}\frac{1}{y-x}d\mu_{n,k}(y)-\frac{S_{i+1}-S_i}{\gamma'_{ik}-x}\Big|\nonumber\\
&\leq& \frac{CM\log(R_2)}{R_2}+C\sum_{i=\lceil l\slash k\rceil}^{k^{47}}\frac{|S_i|+1}{i^{5\slash 3}}.
\end{eqnarray}
By (\ref{QQQ2.2}), (\ref{QQQ2.3}), and the fact that $|l-(2 R_2^{3\slash 2}k)\slash (3\pi)|\leq 1$,
\begin{equation*}
\sum_{i=\lceil l\slash k \rceil}^{\lfloor k^{1\slash 2}\rfloor }\frac{|S_i|+1}{i^{5\slash 3}}\leq \sum_{i=\lceil l\slash k\rceil}^{\lfloor k^{1\slash 2}\rfloor}\frac{CM \log(i)}{i^{5\slash 3}}
\leq CM R_2^{-1\slash 2},
\end{equation*}
\begin{equation*}
\sum_{i=\lfloor k^{1\slash 2}\rfloor+1}^k \frac{|S_i|+1}{i^{5\slash 3}}\leq \sum_{i=\lfloor k^{1\slash 2} \rfloor+1}^k \frac{C\log(i)(\log(n))^{1\slash 2}}{i^{5\slash 3}}\leq Ck^{-1\slash 6},
\end{equation*}
\begin{equation*}
\sum_{i=k+1}^{k^{47}} \frac{|S_i|+1}{i^{5\slash 3}}\leq \sum_{i=k+1}^{k^{47}}\frac{C\log(i)(\log(n))^{1\slash 2}(i\slash k)^{1\slash 2}}{i^{5\slash 3}}\leq Ck^{-1\slash 6}.
\end{equation*}
Plugging these estimates into (\ref{QQ11}), we obtain that for any $x\in [-R_1,R_1]$, $M_2\leq CMR_2^{-1\slash 2}$.

Combining the above estimates on $M_1$ and $M_2$, we conclude that with probability at least $1-C\exp(-M k^2)$, for any $x\in [-R_1,R_1]$, $|E(x)|\leq C MR_2^{-1\slash 2}$. Noting (\ref{Q01}) and (\ref{Q02}), we obtain the conclusion of the lemma.

\end{proof}

The following lemma provides lower bounds on $J(\mu_1)$ and $J(\mu_1,\mu_2+\mu_2')$.

\begin{lemma}\label{J12}
We assume that $R_1\geq 10$, $R_1^5\leq R_2 \leq R_1^6$, $\delta\in (0,1\slash 2)$, and $\epsilon\in (0,1]$. Let $\mu$ be a finite signed Borel measure on $\mathbb{R}$ such that $\mu(\mathbb{R})=0$ and $\log(r)|\mu|((-r,r)^c)<\epsilon$ for all $r\geq R_1\slash 2$. We also assume that $\tilde{d}_R(\mu_{n,k},\mu)\leq \delta$ for some $R\geq R_2+10$. Then there exist positive constants $C_{\mu}$ (which only depends on $\mu$) and $C_{\mu,R_1}$ (which only depends on $\mu$ and $R_1$), such that
\begin{equation}\label{FM0}
    |\mu_{n,k}([-R_1,R_1])|\leq \frac{2\epsilon}{\log(R_1)}+2\sqrt{\delta} R_1^{1\slash 4},
\end{equation}
\begin{eqnarray}\label{FM1}
J(\mu_1)&\geq& -\int_{[-R_1,R_1]^2} \log(\max\{|x-y|,R_1^{-3}\})d\mu(x)d\mu(y)-C_{\mu}(\epsilon+R_1^{-1})\nonumber\\
&&-C_{\mu,R_1}\sqrt{\delta}-3 k^{-2}\log(R_1)\#\{i\in [n]:b_i\in [-R_1,R_1]\},
\end{eqnarray}
\begin{equation}\label{FM2}
J(\mu_1,\mu_2+\mu_2')\geq -C_{\mu}(\epsilon+R_1^{-1\slash 2})-C_{\mu,R_1}\sqrt{\delta}.
\end{equation}
\end{lemma}
\begin{proof}

For any $\epsilon'\in (0,1]$ and $x\in\mathbb{R}$, we define
\begin{equation*}
\psi_{\epsilon'}(x):=\mathbbm{1}_{[-R_1-\epsilon',R_1+\epsilon']}(x)-\frac{x-R_1}{\epsilon'}\mathbbm{1}_{[R_1, R_1+\epsilon']}(x)+\frac{R_1+x}{\epsilon'}\mathbbm{1}_{[-R_1-\epsilon',-R_1]}(x),
\end{equation*}
\begin{equation*}
    \phi_{\epsilon'}(x):=\mathbbm{1}_{[-R_1,R_1]}(x)-\frac{x-(R_1-\epsilon')}{\epsilon'}\mathbbm{1}_{[R_1-\epsilon',R_1]}(x)+\frac{x+R_1-\epsilon'}{\epsilon'}\mathbbm{1}_{[-R_1,-R_1+\epsilon']}(x). 
\end{equation*}
Note that $\|\psi_{\epsilon'}\|_{BL},\|\phi_{\epsilon'}\|_{BL}\leq (\epsilon')^{-1}$. As $\tilde{d}_R(\mu_{n,k},\mu)\leq\delta$, we have
\begin{eqnarray*}
    & (\mu_{n,k}+\mu_0)([-R_1,R_1])\leq\int\psi_{\epsilon'}d(\mu_{n,k}+\mu_0)\leq \int  \psi_{\epsilon'}d(\mu+\mu_0)+\frac{\delta}{\epsilon'},\\
    & (\mu_{n,k}+\mu_0)([-R_1,R_1])\geq\int\phi_{\epsilon'}d(\mu_{n,k}+\mu_0)\geq \int  \phi_{\epsilon'}d(\mu+\mu_0)-\frac{\delta}{\epsilon'}.
\end{eqnarray*}
Hence noting that $\mu(\mathbb{R})=0$, we have 
\begin{eqnarray*}
 & \mu_{n,k}([-R_1,R_1])\leq\int (\psi_{\epsilon'}(x)-1)d\mu(x)+\int(\psi_{\epsilon'}(x)-\mathbbm{1}_{[-R_1,R_1]}(x))d\mu_0(x)+\frac{\delta}{\epsilon'}\nonumber\\
   &\leq|\mu|((-R_1,R_1)^c)+\frac{2}{\pi}\sqrt{R_1+\epsilon'}\epsilon'+\frac{\delta}{\epsilon'}\leq \frac{\epsilon}{\log(R_1)}+\sqrt{R}_1\epsilon'+\frac{\delta}{\epsilon'},
\end{eqnarray*}
\begin{eqnarray*}
   & \mu_{n,k}([-R_1,R_1])\geq\int(\phi_{\epsilon'}(x)-1)d\mu(x)+\int(\phi_{\epsilon'}(x)-\mathbbm{1}_{[-R_1,R_1]}(x))d\mu_0(x)-\frac{\delta}{\epsilon'} \nonumber\\
   &\geq -|\mu|((-R_1\slash 2,R_1\slash 2)^c)-\frac{2}{\pi}\sqrt{R_1}\epsilon'-\frac{\delta}{\epsilon'}\geq -\frac{2\epsilon}{\log(R_1)}-\sqrt{R_1}\epsilon'-\frac{\delta}{\epsilon'}.
\end{eqnarray*}
Taking $\epsilon'=\sqrt{\delta} R_1^{-1\slash 4}\in (0,1]$, we obtain (\ref{FM0}).   

Below we denote $\gamma:=R_1^{-3}$, and let 
\begin{equation*}
    K_1:=-\int_{[-R_1,R_1]^2} \log(\max\{|x-y|,\gamma\})d\mu_1(x)d\mu_1(y),
\end{equation*}
\begin{equation*}
    K_2:=-\int_{[-R_1,R_1]^2\backslash\Delta} \log\big(|x-y|\slash\gamma\big)\mathbbm{1}_{|x-y|\leq \gamma}d\mu_1(x)d\mu_1(y).
\end{equation*}
Note that
\begin{equation}\label{Eqnn3}
    J(\mu_1) \geq K_1+K_2+k^{-2}\#\{i\in [n]: b_i\in [-R_1,R_1]\}\log(\gamma). 
\end{equation}
As $\mu_1+\mu_0$ is a positive measure, by (\ref{FM0}), we have
\begin{eqnarray}\label{Eqnn4}
K_2&\geq& 2\sqrt{R_1}\int_{[-R_1,R_1]}d(\mu_1+\mu_0)(x)\int_{|y-x|\leq \gamma}\log\big(|x-y|\slash\gamma\big)dy\nonumber\\
&=&-4\sqrt{R_1}\gamma(\mu_{n,k}([-R_1,R_1])+\mu_0([-R_1,R_1]))\geq -C R_1^{-1}.
\end{eqnarray}
For any $x,y\in\mathbb{R}$, we let $h(x,y):=-\log(\max\{|x-y|,\gamma\})$. We also let
\begin{eqnarray*}
     & T_1:=\int_{[-R_1,R_1]^2}h(x,y)d\mu(x)d(\mu_1-\mu)(y),\nonumber\\
     & T_2:=\int_{[-R_1,R_1]^2}h(x,y)d(\mu_1-\mu)(x)d(\mu_1-\mu)(y).
\end{eqnarray*}
Note that $K_1=2T_1+T_2+\int_{[-R_1,R_1]^2}h(x,y)d\mu(x)d\mu(y)$. For any $\epsilon'\in (0,1]$ and $x\in\mathbb{R}$, we define
\begin{eqnarray*}
    \Upsilon_{\epsilon'}(x)&:=& \mathbbm{1}_{[R_1-2\epsilon',R_1+2\epsilon']}(x)-\frac{x-(R_1+\epsilon')}{\epsilon'}\mathbbm{1}_{[R_1+\epsilon', R_1+2\epsilon']}(x)\\
&& +\frac{x-(R_1-\epsilon')}{\epsilon'}\mathbbm{1}_{[R_1-2\epsilon',R_1-\epsilon']}(x).
\end{eqnarray*}
As $\|\Upsilon_{\epsilon'}\|_{BL}\leq (\epsilon')^{-1}$, we have
\begin{equation*}
    (\mu_{n,k}+\mu_0)([R_1-\epsilon',R_1+\epsilon'])\leq \int\Upsilon_{\epsilon'}d(\mu_{n,k}+\mu_0)\leq \int\Upsilon_{\epsilon'}d(\mu+\mu_0)+\frac{\delta}{\epsilon'}.
\end{equation*}
Hence
\begin{eqnarray}\label{EEEE1}
    && |\mu_{n,k}|([R_1-\epsilon',R_1+\epsilon'])\leq (\mu_{n,k}+2\mu_0)([R_1-\epsilon',R_1+\epsilon'])\nonumber\\
    &\leq& 2\mu_0([R_1-2\epsilon',R_1+2\epsilon'])+|\mu|((-R_1\slash 2,R_1\slash 2)^c)+\frac{\delta}{\epsilon'}\nonumber\\
    &\leq& C\Big(\sqrt{R_1}\epsilon'+\frac{\epsilon}{\log(R_1)}+\frac{\delta}{\epsilon'}\Big).
\end{eqnarray}
Similarly,
\begin{equation}\label{EEEE2}
    |\mu_{n,k}|([-R_1-\epsilon',-R_1+\epsilon'])\leq C\Big(\sqrt{R_1}\epsilon'+\frac{\epsilon}{\log(R_1)}+\frac{\delta}{\epsilon'}\Big).
\end{equation}
Hence
\begin{eqnarray}\label{Eqnn1}
  &&  \Big|\int h(x,y)\mathbbm{1}_{[-R_1,R_1]}(x)(\mathbbm{1}_{[-R_1,R_1]}(y)-\phi_{\epsilon'}(y))d\mu(x)d(\mu_1-\mu)(y)\Big| \nonumber\\
  &\leq& C\big(|\mu_{n,k}|([R_1-\epsilon',R_1+\epsilon']\cup [-R_1-\epsilon',-R_1+\epsilon'])+|\mu|((-R_1\slash 2,R_1\slash 2)^c)\big)\nonumber\\
  && \times \log(R_1)|\mu|(\mathbb{R})\nonumber\\
  &\leq& C\log(R_1)|\mu|(\mathbb{R})\Big(\sqrt{R_1}\epsilon'+\frac{\epsilon}{\log(R_1)}+\frac{\delta}{\epsilon'}\Big).
\end{eqnarray}
Note that $\int_{[-R_1,R_1]}h(x,y)\phi_{\epsilon'}(y)d\mu(x)$ as a function of $y\in\mathbb{R}$ has absolute value upper bounded by $C\log(R_1)|\mu|(\mathbb{R})$ and is $(C\log(R_1)\slash \epsilon'+\gamma^{-1})|\mu|(\mathbb{R})$-Lipschitz. Hence
\begin{equation}\label{Eqnn2}
     \Big|\int h(x,y)\mathbbm{1}_{[-R_1,R_1]}(x)\phi_{\epsilon'}(y)d\mu(x)d(\mu_1-\mu)(y)\Big| \leq (C\log(R_1)\slash \epsilon'+\gamma^{-1})|\mu|(\mathbb{R})\delta. 
\end{equation}
Combining (\ref{Eqnn1}) and (\ref{Eqnn2}), taking $\epsilon'=\sqrt{\delta} R_1^{-1\slash 4}$, we obtain that
\begin{equation}\label{Eqnn8}
    |T_1|\leq C|\mu|(\mathbb{R})(\epsilon+R_1^3\sqrt{\delta})\leq C_{\mu}\epsilon +C_{\mu,R_1}\sqrt{\delta}.
\end{equation}
Similarly, we can deduce that 
\begin{equation}\label{Eqnn7}
    |T_2|\leq C_{\mu}\epsilon +C_{\mu,R_1}\sqrt{\delta}.
\end{equation}
The bound (\ref{FM1}) follows from (\ref{Eqnn3}), (\ref{Eqnn4}), (\ref{Eqnn8}), and (\ref{Eqnn7}). The bound (\ref{FM2}) can be similarly established. 

\end{proof}

Now we give the proof of Theorem \ref{Main2} when $\mu(\mathbb{R})=0$.

\begin{proof}[Proof of Theorem \ref{Main2}, $\mu(\mathbb{R}) = 0$ case]
Assume that $\mu\in\mathcal{W}$ and $\mu(\mathbb{R})=0$. We fix any $\epsilon\in (0,1]$, and take $R_1\geq 20$ sufficiently large such that for all $r\geq R_1\slash 2$, $\log(r)|\mu|((-r,r)^c)<\epsilon$. We take $R_2=R_1^6$, $R\geq R_2+10$, $\delta\in (0,1\slash 2)$, and assume that $n^{5\cdot 10^{-7}}\leq k\leq  n^{10^{-6}}$. We also assume that $k$ is sufficiently large throughout the proof. From the joint density of the GUE eigenvalues (see (\ref{Eqn10}) with $\beta=2$; see also \cite{AGZ,SS}), we obtain the joint density of $(b_1,\cdots,b_n)$ as 
\begin{equation}\label{rhod}
\rho(b_1,\cdots,b_n)=\tilde{Z}_n^{-1}\exp(-k^2J_0(\mu_{n,k})),
\end{equation}
where $\tilde{Z}_n$ is a normalizing constant and $J_0(\cdot)$ is defined in Definition \ref{xi}.

Let $C_1$ be the constant appearing in Lemma \ref{field}. For any $\nu,\nu'\in\mathcal{Y}$, we define
\begin{equation*}
    \tilde{d}_{[-R,R]\backslash[-R_1,R_1]}(\nu,\nu'):=\sup\limits_{\substack{f:\mathbb{R}\rightarrow\mathbb{R}:\|f\|_{BL}\leq 1,\\\supp(f)\subseteq [-R,R]\backslash[-R_1,R_1]}}\bigg|\int fd\nu-\int fd\nu'\bigg|.
\end{equation*}
Below we fix an arbitrary $M\geq 1$, and define the following events:
\begin{eqnarray*}
    \mathcal{D}_1(M)&:=& \bigg\{\bigg|\int_{(R_2,\infty)} \frac{1}{y-x}d\mu_3 (y)\bigg|\leq C_1 M R_2^{-1\slash 2}\text{ and}\nonumber\\ && \quad 
 \bigg|\int_{(-\infty,-R_2)}  \frac{1}{y-x}d\mu'_3(y)\bigg|\leq C_1 M R_2^{-5\slash 2} \text{ for all }x\in [-R_1,R_1]\bigg\},
\end{eqnarray*}
\begin{eqnarray*}
    &&\mathcal{D}_2:=\bigg\{\tilde{d}_{[-R,R]\backslash[-R_1,R_1]}(\mu_{n,k},\mu)\leq \delta,\nonumber\\
    &&\quad\quad\quad |\mu_{n,k}|\big([-R_1-\sqrt{\delta},-R_1)\cup(R_1,R_1+\sqrt{\delta}]\big) \leq C_2\Big(\sqrt{\delta R_1}+\frac{\epsilon}{\log(R_1)}\Big)\bigg\},
\end{eqnarray*}
\begin{equation*}
    \mathcal{D}_3(M):=\big\{\#\{i\in [n]:|b_i|\leq R_1\}\leq C_2 M R_1^{3\slash 2}k\big\},
\end{equation*}
where $C_2$ is a sufficiently large positive absolute constant. We also define $\mathcal{D}(M):=\mathcal{D}_1(M)\cap \mathcal{D}_2 \cap \mathcal{D}_3(M)$. By (\ref{EEEE1}) and (\ref{EEEE2}) (with $\epsilon'=\sqrt{\delta}$), we have $\{\tilde{d}_R(\mu_{n,k},\mu)\leq\delta\}\subseteq\mathcal{D}_2$. By Lemma \ref{field}, Theorem \ref{Main3}, and Proposition \ref{PropAiry1}, we have $\mathbb{P}(\mathcal{D}_1(M)^c\cup\mathcal{D}_3(M)^c)\leq C\exp(-Mk^2)$. Hence 
\begin{equation}\label{EEEE3}
\mathbb{P}(\tilde{d}_R(\mu_{n,k},\mu)\leq\delta)
\leq C\exp(-Mk^2)+\mathbb{P}(\{\tilde{d}_R(\mu_{n,k},\mu)\leq \delta\}\cap\mathcal{D}(M)).
\end{equation}

In the following, we condition on the $b_i$'s (where $i\in [n]$) that lie in
$[-R_1,R_1]^c$. Namely, we fix any choice of $K\in \{0\}\cup [n]$ and $\alpha_1,\alpha_2,\cdots,\alpha_{n-K}\in [-R_1,R_1]^c$ such that $\alpha_1<\cdots<\alpha_{n-K}$, and denote 
\begin{equation}
\mathcal{L}:=\{\{b_1,b_2,\cdots,b_n\}\cap [-R_1,R_1]^c=\{\alpha_1,\cdots,\alpha_{n-K}\}\}.
\end{equation}
Below, we consider $K$ and $\alpha_1,\alpha_2,\cdots,\alpha_{n-K}$ such that $\mathcal{L}\subseteq \mathcal{D}(M)$, and condition on $\mathcal{L}$. By the definition of $\mathcal{D}_3(M)$, we have $K\leq C_2 M R_1^{3\slash 2}k$. Let $\vec{x}=(x_1,x_2,\cdots,x_K)$ be such that $\{b_1,\cdots,b_n\}\cap [-R_1,R_1]=\{x_1,\cdots,x_K\}$ and $x_1<\cdots<x_K$. Note that $\mu_{n,k}=k^{-1}\sum_{i=1}^K\delta_{x_i}+k^{-1}\sum_{i=1}^{n-K}\delta_{\alpha_i}-\mu_0$. If $\tilde{d}_R(\mu_{n,k},\mu)\leq \delta$, as $\mu(\mathbb{R})=0$, by Lemma \ref{J12}, we have 
\begin{eqnarray}\label{bdd1}
   |K\slash k-\mu_0([0,R_1])|=|\mu_{n,k}([-R_1,R_1])|\leq \frac{2\epsilon}{\log(R_1)}+2\sqrt{\delta} R_1^{1\slash 4}.
\end{eqnarray}
Hence if (\ref{bdd1}) does not hold, we have $\mathbb{P}(\{\tilde{d}_R(\mu_{n,k},\mu)\leq \delta\}|\mathcal{L})=0$. For the remainder of the proof, we assume that (\ref{bdd1}) holds. By (\ref{rhod}), there exists a constant $C(\mathcal{L})$ that only depends on $\mathcal{L}$, such that for any measurable event $\mathcal{A}$,
\begin{equation*}
    \mathbb{P}(\mathcal{A}|\mathcal{L})=C(\mathcal{L})\int_{\{\vec{x}:-R_1\leq x_1<\cdots<x_K\leq R_1\text{ and } \mathcal{A}\text{ holds}\}}\exp(-k^2J_0(\mu_{n,k}))d\vec{x}.
\end{equation*}
Thus denoting by $\mathcal{E}$ the set of $\vec{x}$ such that $-R_1\leq x_1<\cdots<x_K\leq R_1$ and $\tilde{d}_R(\mu_{n,k},\mu)\leq \delta$, we have   
\begin{eqnarray}\label{EEEE11}
&& \mathbb{P}(\{\tilde{d}_R(\mu_{n,k},\mu)\leq \delta\}|\mathcal{L})
= C(\mathcal{L})\int_{\mathcal{E}}\exp(-k^2J_0(\mu_{n,k}))d\vec{x}\nonumber\\
&\leq& C(\mathcal{L})\exp\Big(-k^2\inf_{\vec{x}\in\mathcal{E}} \{J_0(\mu_{n,k})\}\Big)(2R_1)^K.
\end{eqnarray}
Hereafter, we denote by $C_{\mu}$, $C_{R_1}$ (or $c_{R_1}$), and $C_{\mu,R_1}$ positive constants that depend only on $\mu$, on $R_1$, and on $\mu$ and $R_1$, respectively; the values of such constants may change from line to line. By Lemma \ref{J12}, recalling Definition \ref{mu}, we obtain that 
\begin{eqnarray}\label{EEEE4}
  && \inf_{\vec{x}\in \mathcal{E}} J_0(\mu_{n,k})=\inf_{\vec{x}\in \mathcal{E}}\Big\{J(\mu_1)+2J(\mu_1,\mu_2+\mu_2')+2J(\mu_1,\mu_3+\mu_3')\nonumber\\
    &&\quad\quad\quad\quad\quad\quad\quad\quad\quad+J(\mu_2+\mu_2'+\mu_3+\mu_3')+2\int_{\mathbb{R}}\tilde{\xi}(x)d\mu_{n,k}(x)\Big\} \nonumber\\
    &\geq& -\int_{[-R_1,R_1]^2} \log(\max\{|x-y|,R_1^{-3}\})d\mu(x)d\mu(y)-C_{\mu}(\epsilon+R_1^{-1\slash 2})\nonumber\\
    && -C_{\mu,R_1}\sqrt{\delta} - CMR_1^{3\slash 2}\log(R_1)k^{-1}+J(\mu_2+\mu_2'+\mu_3+\mu_3')\nonumber\\
    && +2\inf_{\vec{x}\in\mathcal{E}} J(\mu_1,\mu_3+\mu_3')+2\inf_{\vec{x}\in\mathcal{E}}\int_{\mathbb{R}}\tilde{\xi}(x)d\mu_{n,k}(x).
\end{eqnarray}

For the remainder of the proof, we assume that $\max\{\epsilon,\sqrt{\delta}R_1^{1\slash 4}\}<1\slash 100$ (which implies $2\epsilon\slash \log(R_1)+2\sqrt{\delta}R_1^{1\slash 4}<1\slash 10$). Now we fix $R_0'\in [1,10]$ such that $L:=K-k\mu_0([R_0',R_1])\in\mathbb{Z}$. Note that by (\ref{bdd1}), as $k$ is taken to be sufficiently large, we have $L\in\mathbb{N}_{+}$. We let $c_0:=0$. For each $i\in [L]$, we let $c_i\in [0,R_0']$ be such that $\mu_0([0,c_i])=i\mu_0([0,R_0'])\slash(L+1)$; for each $i\in [L+1,K]\cap\mathbb{N}$, we let $c_i\in[R_0',R_0]$ be such that $\mu_0([R_0',c_i])=(i-L-1)\mu_0([R_0',R_1])\slash(K-L)$. We define $\mathcal{C}$ to be the set of $\vec{x}$ such that $|x_i-c_i|\leq (2n)^{-1}$ for all $i\in [K]$. Note that
\begin{eqnarray}\label{EEEE12}
 && 1=C(\mathcal{L})\int_{-R_1\leq x_1<\cdots<x_K\leq R_1}\exp(-k^2 J_0(\mu_{n,k}))d\vec{x}\nonumber\\
 &\geq& C(\mathcal{L})\int_{\mathcal{C} }\exp(-k^2 J_0(\mu_{n,k}))d\vec{x}
\geq C(\mathcal{L}) n^{-K} \exp\Big(-k^2\sup_{\vec{x}\in\mathcal{C}}J_0(\mu_{n,k})\Big).\nonumber\\
&&
\end{eqnarray}

Now we bound $\sup_{\vec{x}\in\mathcal{C}}J_0(\mu_{n,k})$. Below we only present the details for bounding $J_0(\mu_{n,k})$ when $\vec{x}=(c_1,\cdots,c_K)$, and the same bound holds for general $\vec{x}\in\mathcal{C}$. Let $\mu_{11}$ and $\mu_{12}$ be finite signed Borel measures on $\mathbb{R}$ such that for any $A\in\mathcal{B}_{\mathbb{R}}$, $\mu_{11}(A)=\mu_1(A\cap [0,R_0'))$ and $\mu_{12}(A)=\mu_1(A\cap [R_0',R_1])$. Note that $\mu_1=\mu_{11}+\mu_{12}$. For any $i,j\in [L]$ with $i\neq j$, we define
\begin{equation*}
L_{ij}:=\frac{1}{k}\log{|c_i-c_j|}-\int_{c_{j-1}}^{c_j}\log(|x-c_i|)d\mu_0(x).
\end{equation*}
For any $i\in [L]$, we define
\begin{eqnarray*}
 && L_{i}:=\frac{1}{k}\sum_{j\in [L]: j\neq i}\log(|c_i-c_j|)-\int_0^{R_0'}\log(|x-c_i|)d\mu_0(x)\nonumber\\
&=& \sum_{j\in [L]: j\neq i}L_{ij}-\int_{c_{i-1}}^{c_i}\log(|x-c_i|)d\mu_0(x)-\int_{c_L}^{R_0'}\log(|x-c_i|)d\mu_0(x).
\end{eqnarray*}
By the definition of $\{c_i\}_{i=1}^K$ and (\ref{bdd1}), for each $i\in [L+1]$, we have 
\begin{equation*}
    \Big|\mu_0([c_{i-1},c_i])-k^{-1}\Big|\leq \frac{C}{k}\Big(\frac{\epsilon}{\log(R_1)}+\sqrt{\delta} R_1^{1\slash 4}\Big),\quad ck^{-1}\leq c_i-c_{i-1} \leq Ck^{-2\slash 3};
\end{equation*}
for each $i\in [L+1,K]\cap\mathbb{N}$, we have
\begin{equation*}
  \mu_0([c_i,c_{i+1}])=k^{-1},\quad c_{R_1} k^{-1}\leq c_{i+1}-c_i\leq C k^{-1},
\end{equation*}
where $c_{K+1}:=R_1$. When $j\in [L]\backslash \{i,i+1\}$, by the above estimates, we have
\begin{eqnarray*}
    |L_{ij}| &\leq& \max \big\{|\log(|c_j-c_i|)|,|\log(|c_{j-1}-c_i|)|\big\}\big|k^{-1}
-\mu_0([c_{j-1},c_j])\big|\nonumber\\
    &&+k^{-1}\Big|\log\Big(\frac{|c_j-c_i|}{|c_{j-1}-c_i|}\Big)\Big|\nonumber\\
    &\leq& Ck^{-1}\Big(\frac{\epsilon}{\log(R_1)}+\sqrt{\delta} R_1^{1\slash 4}\Big)\big(|\log(|c_j-c_i|)|+|\log(|c_{j-1}-c_i|)|\big)\nonumber\\
    && + k^{-1}\Big|\log\Big(\frac{|c_j-c_i|}{|c_{j-1}-c_i|}\Big)\Big|.
\end{eqnarray*}
From this, we can deduce that $|L_i|\leq C\big(\epsilon\slash \log(R_1)+\delta^{1\slash 2}R_1^{1\slash 4}\big)$ for each $i\in [L]$ and $|J(\mu_{11})|\leq C\big(\epsilon\slash \log(R_1)+\delta^{1\slash 2}R_1^{1\slash 4}\big)$ (note that we have taken $k$ to be sufficiently large). Similarly, we can deduce that
\begin{equation*}
    |J(\mu_{12})|\leq C_{R_1}\log(k)\slash k, \quad |J(\mu_{11},\mu_{12})|\leq C_{R_1}\log(k)\slash k,
\end{equation*}
\begin{equation*}
    |J(\mu_{11},\mu_2+\mu_2')|\leq C_{\mu}\epsilon+C_{\mu,R_1}\sqrt{\delta}, \quad |J(\mu_{12},\mu_2+\mu'_{2})|\leq C_{\mu,R_1}\log(k) \slash k.
\end{equation*}
Hence we have 
\begin{eqnarray}\label{EEEE5}
\sup_{\vec{x}\in\mathcal{C}}J_0(\mu_{n,k})&\leq& C_{\mu}\epsilon+C_{\mu,R_1}\sqrt{\delta}+J(\mu_2+\mu_2'+\mu_3+\mu_3')\nonumber\\
&&+2\sup_{\vec{x}\in\mathcal{C}}J(\mu_1,\mu_3+\mu_3')+2\sup_{\vec{x}\in\mathcal{C}}\int_{\mathbb{R}}\tilde{\xi}(x)d\mu_{n,k}(x).
\end{eqnarray}

By Lemma \ref{L7.1}, as $k$ is sufficiently large, for any $x\in [-R_1,0]$, we have $|\tilde{\xi}(x)|\leq R_1^{3\slash 2}$ and $|\tilde{\xi}'(x)|\leq 2\sqrt{R_1}$; for any $x\in [0,R_1]$, we have $\tilde{\xi}(x)=0$. By the definition of $\mathcal{C}$, for any $\vec{x}\in \mathcal{C}$, we have $\int_{[-R_1,0]}\tilde{\xi}(x) d\mu_{n,k}(x)=0$. Let $\psi(x):=\mathbbm{1}_{[-R_1,1]}(x)+(x+R_1-1)\mathbbm{1}_{[-R_1,-R_1+1]}(x)-x\mathbbm{1}_{[0,1]}(x)$ for any $x\in\mathbb{R}$. Note that when $\vec{x}\in\mathcal{E}$, we have 
\begin{eqnarray*}
   && \int_{[-R_1,0]}\tilde{\xi}(x)d\mu_{n,k}(x) \geq \int_{\mathbb{R}}\psi(x)\tilde{\xi}(x)d\mu_{n,k}(x)\nonumber\\
   &\geq&\int_{\mathbb{R}}\psi(x)\tilde{\xi}(x)d\mu(x)-2R_1^{3\slash 2}\delta\geq\int_{[-R_1+1,0]}\tilde{\xi}(x)d\mu(x)-2R_1^{3\slash 2}\delta.
\end{eqnarray*}
Hence
\begin{eqnarray}\label{EEEE6}
 && 2\inf_{\vec{x}\in\mathcal{E}}\int_{\mathbb{R}}\tilde{\xi}(x)d\mu_{n,k}(x)
-2\sup_{\vec{x}\in\mathcal{C}}\int_{\mathbb{R}}\tilde{\xi}(x)d\mu_{n,k}(x)\nonumber\\
&=& 2\inf_{\vec{x}\in\mathcal{E}}\int_{[-R_1,0]}\tilde{\xi}(x)d\mu_{n,k}(x)
-2\sup_{\vec{x}\in\mathcal{C}}\int_{[-R_1,0]}\tilde{\xi}(x)d\mu_{n,k}(x) \nonumber\\
&\geq& 2\int_{[-R_1+1,0]}\tilde{\xi}(x)d\mu(x)-4R_1^{3\slash 2}\delta.
\end{eqnarray}
By Lemma \ref{L7.1} and the dominated convergence theorem (as $|\mu|(\mathbb{R})<\infty$),
\begin{equation}\label{EEEE7}
\lim_{k\rightarrow \infty}2\int_{[-R_1+1,0]}\tilde{\xi}(x)d\mu(x)=\int_{[-R_1+1,0]}\frac{4}{3}|x|^{3\slash 2}d\mu(x).
\end{equation}

Now consider any $\vec{x}=(x_1,\cdots,x_K)\in\mathcal{C}$ and $\vec{y}=(y_1,\cdots,y_K)\in\mathcal{E}$. We denote $\tilde{\mu}_{n,k}:=k^{-1}\sum_{i=1}^K\delta_{y_i}+k^{-1}\sum_{i=1}^{n-K}\delta_{\alpha_i}-\mu_0$, and let $\tilde{\mu}_1$ be the finite signed Borel measure on $\mathbb{R}$ such that $\tilde{\mu}_1(A)=\tilde{\mu}_{n,k}(A\cap [-R_1,R_1])$ for any $A\in\mathcal{B}_{\mathbb{R}}$. For any $s,s'\in [-R_1,R_1]$ such that $s\leq s'$, as $\mathcal{L}\subseteq\mathcal{D}(M)$, we have
\begin{eqnarray*}
&&\bigg|\int (\log(|s-y|)-\log(|s'-y|))d(\mu_3+\mu'_3)(y)\bigg|\\
&\leq& \int_{s}^{s'}\bigg(\bigg|\int\frac{1}{|x-y|}d\mu_3(y)\bigg|+\bigg|\int\frac{1}{|x-y|}d\mu_3'(y)\bigg|\bigg)dx\leq CMR_1 R_2^{-1\slash 2}.
\end{eqnarray*}
Hence we have
\begin{eqnarray*}
&&|J(\mu_1,\mu_3+\mu'_3)-J(\tilde{\mu}_1,\mu_3+\mu_3')|\\
&=&\bigg|\frac{1}{k}\sum_{i=1}^K \int (\log(|x_i-y|)-\log(|y_i-y|))d(\mu_3+\mu'_3)(y)\bigg|\\
&\leq& CKMR_1R_2^{-1\slash 2}k^{-1}\leq CM^2 R_1^{-1\slash 2}.
\end{eqnarray*}
Therefore, we have
\begin{equation}\label{EEEE8}
    \Big|\sup_{\vec{x}\in \mathcal{C}}J(\mu_1,\mu_3+\mu'_3)-\inf_{\vec{x}\in\mathcal{E}}J(\mu_1,\mu_3+\mu'_3)\Big|  \leq CM^2 R_1^{-1\slash 2}.
\end{equation}

Combining (\ref{EEEE4}), (\ref{EEEE5}), (\ref{EEEE6}), and (\ref{EEEE8}), we obtain that
\begin{eqnarray*}\label{EEEE9}
&&\inf_{\vec{x}\in\mathcal{E}}J_0(\mu_{n,k})-\sup_{\vec{x}\in\mathcal{C}}J_0(\mu_{n,k}) \nonumber\\
&\geq& -\int_{[-R_1,R_1]^2}\log(\max\{|x-y|,R_1^{-3}\})d\mu(x)d\mu(y)-CMR_1^{3\slash 2}\log(R_1)k^{-1}\nonumber\\
&& +2\int_{[-R_1+1,0]}\tilde{\xi}(x)d\mu(x)-C_{\mu}(\epsilon+M^2 R_1^{-1\slash 2}
)-C_{\mu,R_1}\sqrt{\delta}=:\Gamma.
\end{eqnarray*}
Hence by (\ref{EEEE11}) and (\ref{EEEE12}), 
\begin{eqnarray}\label{EEEE10}
\mathbb{P}(\{\tilde{d}_R(\mu_{n,k},\mu)\leq \delta\}|\mathcal{L}) &\leq& \exp\Big(-k^2\Big(\inf_{\vec{x}\in\mathcal{E}}J_0(\mu_{n,k})-\sup_{\vec{x}\in\mathcal{C}}J_0(\mu_{n,k})\Big)\Big)(2nR_1)^{K}\nonumber\\
&\leq& \exp(-k^2\Gamma)(2nR_1)^{CMR_1^{3\slash 2}k}.
\end{eqnarray}
We denote 
\begin{eqnarray*}
\tilde{\Gamma}&:=&-\int_{[-R_1,R_1]^2}\log(\max\{|x-y|,R_1^{-3}\})d\mu(x)d\mu(y)\\
&&+\int_{[-R_1+1,0]}\frac{4}{3}|x|^{3\slash 2}d\mu(x)-C_{\mu}(\epsilon+M^2 R_1^{-1\slash 2})-C_{\mu,R_1}\sqrt{\delta}.
\end{eqnarray*}
By (\ref{EEEE3}), (\ref{EEEE7}), and (\ref{EEEE10}), we have
\begin{equation*}
\mathbb{P}(\tilde{d}_{R}(\mu_{n,k},\mu)\leq \delta)\leq C\exp(-Mk^2)+\exp(-k^2\Gamma)(2nR_1)^{CMR_1^{3\slash 2}k},
\end{equation*}
\begin{equation*}
\limsup_{k\rightarrow\infty}\frac{1}{k^2}\log{\mathbb{P}(\tilde{d}_{R}(\mu_{n,k},\mu)\leq \delta)}\leq \max\{-M,-\tilde{\Gamma}\}.
\end{equation*}
By first sending $\delta\rightarrow 0^{+}$, then sending $R\rightarrow\infty$ (hence we can send $R_1\rightarrow\infty$ and $\epsilon\rightarrow 0^{+}$), and finally sending $M\rightarrow\infty$, we obtain that (by the monotone convergence theorem)
\begin{eqnarray}\label{neqq7}
&&\limsup_{R\rightarrow \infty}\limsup_{\delta\rightarrow 0^{+}}\limsup_{k\rightarrow\infty}\frac{1}{k^2}\log{\mathbb{P}(\tilde{d}_{R}(\mu_{n,k},\mu)\leq \delta)}\nonumber\\
&\leq& \liminf_{R\rightarrow \infty}\int_{[-R,R]^2}\log(\max\{|x-y|,R^{-3}\})d\mu(x)d\mu(y)-\int_{(-\infty,0]} \frac{4}{3}|x|^{3\slash 2}d\mu(x). \nonumber\\
&&
\end{eqnarray}

For any function $\phi:\mathbb{R}\rightarrow\mathbb{R}$ with $\supp(\phi)\subseteq[-R,R]$ and $||\phi||_{BL}\leq 1$, as $k\rightarrow\infty$, we have
\begin{equation}\label{EEEE1.2}
\bigg|\int_{0}^R \phi(x) d\mu_0(x)-\int_{0}^R \phi(x) d\nu_0(x)\bigg|\leq \left(1-\sqrt{1-\frac{1}{4}\Big(\frac{k}{n}\Big)^{2\slash 3}R}\right)\frac{2R^{3\slash 2}}{3\pi}\rightarrow 0.
\end{equation}
Below we fix an arbitrary $\eta\geq 15$. By Theorem \ref{Main3}, we can couple $\{b_i\}_{i=1}^n$ with $\{a_i\}_{i=1}^{\infty}$ (the ordered points of the Airy point process), such that if we denote by $\mathcal{A}_1$ the event that $|k^{2\slash 3}b_i+a_i|\leq Cn^{-1\slash 24}$ for all $i\in [\lceil C_3\eta R^{3\slash 2} k\rceil]$ (where $C_3\geq 1$ is a sufficiently large absolute constant), then $\mathbb{P}(\mathcal{A}_1^c)\leq C\exp(-c k^3)$. Let $I_1:=\{i\in [n]:b_i\in [-R,R]\}$, $I_2:=\{i\in\mathbb{N}_{+}:k^{-2\slash 3}a_i\in [-R,R]\}$, and $I:=I_1\cap I_2$. We denote by $\mathcal{A}_2$ the event that
\begin{equation*}
    \max\{\#\{i\in [n]: b_i\leq R\},\#\{i\in\mathbb{N}_{+}:k^{-2\slash 3}a_i\geq -R\}\}\leq C_3\eta R^{3\slash 2}k.
\end{equation*}
By Proposition \ref{PropAiry1}, $\mathbb{P}(\mathcal{A}_2^c)\leq C\exp(-\eta k^2)$. Below we assume that the event $\mathcal{A}_1\cap\mathcal{A}_2$ holds. Note that this implies $\max\{\# I_1,\# I_2\}\leq C_3\eta R^{3\slash 2}k$. We have 
\begin{eqnarray}
&& \bigg|\frac{1}{k}\sum_{i\in I_1}\phi(b_i)-\frac{1}{k}\sum_{i\in I_2}\phi(-k^{-2\slash 3}a_i)\bigg| \nonumber\\
&\leq& \frac{1}{k}\sum_{i\in I}\Big|\phi(b_i)-\phi(-k^{-2\slash 3}a_i)\Big|+\frac{1}{k}\bigg|\sum_{i\in I_1\backslash I}\phi(b_i)\bigg|+\frac{1}{k}\bigg|\sum_{i\in I_2\backslash I}\phi(-k^{-2\slash 3}a_i)\bigg|\nonumber\\
&=:& P_1+P_2+P_3.
\end{eqnarray}
As $k\rightarrow\infty$, we have 
\begin{equation}
P_1\leq \frac{1}{k}\sum_{i\in I}|b_i+k^{-2\slash 3}a_i|\leq C n^{-1\slash 24} k^{-2\slash 3} \eta R^{3\slash 2}\rightarrow 0.
\end{equation}
As $\supp(\phi)\subseteq [-R,R]$ and $\|\phi\|_{Lip}\leq 1$, we have $\phi(R)=\phi(-R)=0$. For any $i\in I_1\backslash I$, we have
\begin{equation*}
    \min\{|b_i-R|,|b_i+R|\}\leq Cn^{-1\slash 24}k^{-2\slash 3},\quad\text{hence }|\phi(b_i)|\leq Cn^{-1\slash 24}k^{-2\slash 3}.
\end{equation*}
For any $i\in I_2\backslash I$, we have
\begin{eqnarray*}
   & \min\{|-k^{-2\slash 3}a_i-R|,|-k^{-2\slash 3}a_i+R|\}\leq Cn^{-1\slash 24}k^{-2\slash 3},\nonumber\\
  &\text{hence }|\phi(-k^{-2\slash 3}a_i)|\leq Cn^{-1\slash 24}k^{-2\slash 3}.
\end{eqnarray*}
Hence as $k\rightarrow\infty$, we have 
\begin{equation}\label{EEEE1.3}
    P_2+P_3\leq  C n^{-1\slash 24}k^{-2\slash 3} \eta R^{3\slash 2}\rightarrow 0.
\end{equation}

Now we let $\tilde{\nu}_{k;R}$ be the finite signed Borel measure on $\mathbb{R}$ such that for any $A\in\mathcal{B}_{\mathbb{R}}$, $\tilde{\nu}_{k;R}(A)=\nu_{k;R}(A\cap [-R,R])$. By (\ref{EEEE1.2})-(\ref{EEEE1.3}), for any $\delta>0$, when $k$ is sufficiently large, 
\begin{equation*}
    \mathbb{P}(\tilde{d}_R(\mu_{n,k},\tilde{\nu}_{k;R})\geq \delta)\leq C\exp(-\eta k^2).
\end{equation*}
As $\tilde{d}_R(\tilde{\nu}_{k;R},\mu)\leq d_R(\nu_{k;R},\mu|_{[-R,R]})$, we have
\begin{equation*}
    \mathbb{P}(d_R(\nu_{k;R},\mu|_{[-R,R]})\leq \delta)\leq \mathbb{P}(\tilde{d}_R(\mu_{n,k},\mu)\leq 2\delta)+\mathbb{P}(\tilde{d}_R(\mu_{n,k},\tilde{\nu}_{k;R})\geq \delta).
\end{equation*}
Noting (\ref{neqq7}) and sending $\eta\rightarrow\infty$, we conclude that
\begin{eqnarray*}
&&\limsup_{R\rightarrow \infty}\limsup_{\delta\rightarrow 0^{+}}\limsup_{k\rightarrow\infty}\frac{1}{k^2}\log{\mathbb{P}(d_{R}(\nu_{k;R},\mu|_{[-R,R]})\leq \delta)}\\
&\leq& \liminf_{R\rightarrow \infty}\int_{[-R,R]^2}\log(\max\{|x-y|,R^{-3}\})d\mu(x)d\mu(y)-\int_{(-\infty,0]}
\frac{4}{3}|x|^{3\slash 2}d\mu(x).
\end{eqnarray*}

\end{proof}

Next we give the proof of Theorem \ref{Main2} when $\mu(\mathbb{R})\neq 0$.

\begin{proof}[Proof of Theorem \ref{Main2}, $\mu(\mathbb{R})\neq 0$ case]

We denote $\kappa:=\mu(\mathbb{R})$. Without loss of generality, we assume that $\kappa>0$. We take $R_1\geq 10$ sufficiently large (depending on $\mu$) such that $|\mu|((-R_1\slash 2,R_1\slash 2)^c)\leq \kappa\slash 100$, and take $R\geq 4R_1$. We assume that $d_R(\nu_{k;R},\mu|_{[-R,R]})\leq\delta$ with $\delta\in (0,\kappa\slash 12)$ and $k$ is sufficiently large, and take $m:=\lceil 2\sqrt{R}\slash \delta\rceil$. For any $\lambda\in [2R_1,R-1]$ and $x\in\mathbb{R}$, we define
\begin{equation*}
    \psi_{\lambda}(x):=\mathbbm{1}_{[-R_1,\lambda+1]}(x)-(x-\lambda)\mathbbm{1}_{[\lambda,\lambda+1]}(x), \quad \tilde{\psi}_{\lambda}(x):=\frac{1}{m}\sum_{j=0}^{m-1}\mathbbm{1}_{[-R_1,\lambda+j\slash m]}(x).
\end{equation*}
It can be checked that
\begin{equation}\label{Rel}
    \int_{[-R,R]} (\psi_{\lambda}-\tilde{\psi}_{\lambda})d\nu_{k;R}\leq \kappa\slash 100+2\delta, \quad \int_{[-R,R]}\psi_{\lambda}d\nu_{k;R}\geq 3\kappa\slash 4-\delta.
\end{equation}
We recall the notations in Section \ref{Sect.n1}. By (\ref{Rel}) and Propositions \ref{Diffu}-\ref{Airyop},  
\begin{eqnarray}\label{EEEE1.5}
&&\frac{1}{m}\sum_{j=0}^{m-1}\Big(N\Big(\Big(\lambda+\frac{j}{m}\Big)k^{2\slash 3}\Big)-N_0\Big(\Big(\lambda+\frac{j}{m}\Big)k^{2\slash 3}\Big)\Big)-N(-R_1k^{2\slash 3})\nonumber\\
&\geq& k\int_{[-R,R]}\tilde{\psi}_{\lambda}d\nu_{k;R}-2\geq (\kappa\slash 2-3\delta)k\geq \kappa k\slash 4.
\end{eqnarray}

We take $K:=\log\log{R}$ and $I_0:=\lfloor \log{R}\slash (2\log{K})\rfloor$. In the following, we assume that $R$ is sufficiently large, so that $K>2R_1$ and $I_0\geq 10$. For any $j_1,\cdots,j_{I_0}\in \{0,1,\cdots, m-1\}$, we let $\mathcal{B}(j_1,j_2,\cdots,j_{I_0})$ be the event that 
\begin{equation*}
    N\big((K^i+j_i\slash m)k^{2\slash 3}\big)-N_0\big((K^i+j_i\slash m)k^{2\slash 3}\big)\geq \kappa k\slash 8
\end{equation*}
for all $i\in [I_0]$. We further define $\mathcal{C}:=\bigcup_{(j_1,\cdots,j_{I_0})\in \{0,1,\cdots,m-1\}^{I_0}}\mathcal{B}(j_1,\cdots,j_{I_0})$. By the arguments above, for any $\delta\in (0,\kappa\slash 12)$, $\{d_R(\nu_{k;R},\mu|_{[-R,R]})\leq \delta\}\subseteq \mathcal{C}$.

Below we bound $\mathbb{P}(\mathcal{B}(j_1,\cdots,j_{I_0}))$ for any $j_1,\cdots,j_{I_0}\in \{0,1,\cdots,m-1\}$. For any $\lambda\in\mathbb{R}$, we define $q_{\lambda}(x),x\geq 0$ by $q_{\lambda}'(x)=x-\lambda-q_{\lambda}^2(x),q_{\lambda}(0)=\infty$. Let $s_0:=0$ and $s_{I_0+1}:=\infty$. For each $i\in [I_0]$, we let $s_i:=(K^i+j_i\slash m)k^{2\slash 3}$, and let $\mathcal{A}_i$ be the event that the absolute difference between the number of blow-ups of $p_{s_i}(x)$ and $q_{s_i}(x)$ on $[0,s_{i-1}]\cup (2s_{i},\infty)$ is at least $K^{-1\slash 3}\kappa k\slash 8$. From the proof of Proposition \ref{PropAiry2}, we have $\mathbb{P}(\mathcal{A}_i)\leq \exp(-cK^{1\slash 3}k^2)$.

Let $\mathcal{A}:=\bigcap_{i\in [I_0]} \mathcal{A}_i^c$. By the union bound, $\mathbb{P}(\mathcal{A}^c)\leq \exp(-cK^{1\slash 3}k^2)$. When the event $\mathcal{B}(j_1,\cdots,j_{I_0})\cap\mathcal{A}$ holds, for each $i\in [I_0]$, the difference between the number of blow-ups of $p_{s_i}(x)$ and $q_{s_i}(x)$ on $(s_{i-1},2s_{i}]$ is greater than or equal to $(1-K^{-1\slash 3})\kappa k\slash 8\geq \kappa k\slash 16$. Hence from the proof of Proposition \ref{PropAiry2} and the strong Markov property, we have
\begin{equation*}
    \mathbb{P}(\mathcal{B}(j_1,\cdots,j_{I_0})\cap\mathcal{A}) \leq \prod_{i=1}^{\lfloor\frac{\log{R}}{2\log{K}}\rfloor} C\exp\left(-c\frac{k^2}{i\log{K}}\right)\leq C\exp\left(-ck^2\frac{\log\log{R}}{\log\log\log{R}}\right).
\end{equation*}
By taking a union bound, we obtain that 
\begin{equation*}
    \mathbb{P}(\mathcal{B}(j_1,\cdots,j_{I_0}))\leq C\exp(-ck^2(\log\log{R})^{1\slash 3}).
\end{equation*}
Hence
\begin{equation*}
   \mathbb{P}(d_R(\nu_{k;R},\mu|_{[-R,R]})\leq \delta)\leq \mathbb{P}(\mathcal{C})\leq C(\lceil 2\sqrt{R}\slash \delta\rceil)^{I_0}\exp(-ck^2(\log\log{R})^{1\slash 3}).
\end{equation*}
We arrive at the conclusion by first taking $k\rightarrow \infty$ and then taking $R\rightarrow \infty$.
\end{proof}

\subsection{Proof of Theorem \ref{Main1}, part (a)}
\label{sec:5.3}

In this subsection, we finish the proof of Theorem \ref{Main1}.

\begin{proof}[Proof of Theorem \ref{Main1}, part (a)]

We use a similar strategy as in the proof of Theorem \ref{Main2}. Recall the definitions of $\mathcal{Z}$, $\{b_i\}_{i=1}^n$, and $\mu_{n,k}$ from (\ref{DefiZ}) and Definition \ref{mu}. 

We fix an arbitrary $\mu\in\mathcal{Z}$ such that $\mu(\mathbb{R})=0$, and take $\delta\in (0,1\slash 2)$. We let $\mathcal{B}$ be the event that $|b_i-R_0|\geq n^{-1\slash 30}$ and $|b_i+R_0|\geq n^{-1\slash 30}$ for all $i\in [n]$, and let $\mathcal{A}:=\{d_{R_0}(\mu_{n,k},\mu)\leq \delta\slash 2\}\cap \mathcal{B}$. Throughout the proof, we assume that $n^{5\cdot 10^{-7}}\leq k\leq  n^{10^{-6}}$ and that $k$ is sufficiently large. We take $R_0'\geq R_0+10$ such that $k\nu_0([0,R_0'])\in\mathbb{Z}$, $\supp(\mu)\subseteq [-R_0'\slash 2,R_0'\slash 2]$, and $R_0'$ is uniformly bounded (for fixed $\mu$ and $R_0$). We take $R_0''=R_0'+10$, $R_1\geq 2R_0''$, and $R_2=R_1^6$ such that $k\mu_0([R_0'',R_1])\in\mathbb{Z}$. We also take $R\geq R_2+10$. We fix an arbitrary $M\geq 1$, and define the following events: 
\begin{eqnarray*}
    \mathcal{D}_1(M)&:=& \bigg\{\bigg|\int_{(R_2,\infty)} \frac{1}{y-x}d\mu_3 (y)\bigg|\leq C_1 M   R_2^{-1\slash 2}\text{ and} \nonumber\\ && \quad 
 \bigg|\int_{(-\infty,-R_2)}\frac{1}{y-x}d\mu'_3(y)\bigg|\leq C_1 M R_2^{-5\slash 2} \text{ for all }x\in [-R_1,R_1]\bigg\},
\end{eqnarray*}
\begin{equation*}
    \mathcal{D}_2(M):=\big\{\#\{i\in [n]:|b_i|\leq R_1\}\leq C_2 M R_1^{3\slash 2}k\big\},
\end{equation*}
where $C_1$ is the constant appearing in Lemma \ref{field} and $C_2$ is a sufficiently large positive absolute constant. We also define $\mathcal{D}(M):=\mathcal{D}_1(M)\cap\mathcal{D}_2(M)$. By Lemma \ref{field}, Theorem \ref{Main3}, and Proposition \ref{PropAiry1},
\begin{equation}\label{EEEE1.7}
    \mathbb{P}(\mathcal{D}(M)^c)\leq C\exp(-Mk^2).
\end{equation}

Similar to the proof of Theorem \ref{Main2}, we fix any choice of $K\in\{0\}\cup [n]$ and $\alpha_1,\cdots,\alpha_{n-K}\in [-R_1,R_1]^c$ such that $\alpha_1<\cdots<\alpha_{n-K}$, and denote
\begin{equation*}
    \mathcal{L}:=\{\{b_1,b_2,\cdots,b_n\}\cap [-R_1,R_1]^c=\{\alpha_1,\cdots,\alpha_{n-K}\}\}.
\end{equation*}
In the following, we consider $K$ and $\alpha_1,\cdots,\alpha_{n-K}$ such that $\mathcal{L}\subseteq\mathcal{D}(M)$, and condition on $\mathcal{L}$. Let $\vec{x}=(x_1,x_2,\cdots,x_K)$ be such that $x_1<\cdots<x_K$ and $\{b_1,\cdots,b_n\}\cap [-R_1,R_1]=\{x_1,\cdots,x_K\}$. Denote by $C_0$ the constant $C$ appearing in Lemma \ref{A} (with $e_1=5\cdot 10^{-7}$ and $e_2=10^{-6}$). We aim to provide a uniform lower bound $T_0\geq 0$ such that
\begin{equation}\label{B}
\mathbb{P}(\mathcal{A}|\mathcal{L})\geq T_0\mathbb{P}(d_R(\mu_{n,k},0)\leq C_0 R k^{-1\slash 4}|\mathcal{L}).
\end{equation}
Assuming this result, by Lemma \ref{A} and (\ref{EEEE1.7}), we have (note that $k$ is sufficiently large)
\begin{eqnarray*}
\mathbb{P}(\mathcal{A})&\geq& \mathbb{P}(\mathcal{A}\cap\mathcal{D}(M))
= \mathbb{E}[\mathbb{P}(\mathcal{A}|\mathcal{L})\mathbbm{1}_{\mathcal{L}\subseteq\mathcal{D}(M)}]\\
&\geq& T_0\mathbb{E}[\mathbb{P}(d_R(\mu_{n,k},0)\leq C_0 R k^{-1\slash 4}|\mathcal{L})\mathbbm{1}_{\mathcal{L}\subseteq\mathcal{D}(M)}]\\
&=&T_0\mathbb{P}(\{d_R(\mu_{n,k},0)\leq C_0 R k^{-1\slash 4}\}\cap\mathcal{D}(M))
\geq \frac{1}{2}T_0.
\end{eqnarray*}

In the following, we will verify (\ref{B}). Without loss of generality, we assume that $\mathbb{P}(d_R(\mu_{n,k},0)\leq C_0 R k^{-1\slash 4}|\mathcal{L})>0$. We denote by $\tilde{\mathcal{A}}$ ($\tilde{\mathcal{B}}$, respectively) the set of $\vec{x}$ for which $-R_1 \leq x_1<\cdots<x_K \leq R_1$ and the event $\mathcal{A}$ ($\mathcal{B}$, respectively) holds, and denote by $\mathcal{E}$ the set of $\vec{x}$ for which $-R_1 \leq x_1<\cdots<x_K \leq R_1$ and $d_R(\mu_{n,k},0)\leq C_0 R k^{-1\slash 4}$. As in the proof of Theorem \ref{Main2}, we have
\begin{equation}\label{EEEE21}
\frac{\mathbb{P}(\mathcal{A}|\mathcal{L})}{\mathbb{P}(d_R(\mu_{n,k},0)\leq C_0 R k^{-1\slash 4}|\mathcal{L})}=\frac{\int_{\tilde{\mathcal{A}}}\exp(-k^2 J_0(\mu_{n,k}))d\vec{x}}
{\int_{\mathcal{E}}\exp(-k^2 J_0(\mu_{n,k}))d\vec{x}}.
\end{equation}
Hereafter, we denote by $C_{R}$ and $c_R$ positive constants that depend only on $\mu$ and $R$; the values of such constants may change from line to line. By Lemma \ref{J12}, for any $\vec{x}\in\mathcal{E}$,  
\begin{equation*}
    \min\{J(\mu_1),J(\mu_1,\mu_2+\mu'_2)\}\geq -C_R k^{-1\slash 10}-C R_1^{-1\slash 2},
\end{equation*}
\begin{eqnarray}
\text{ hence} \quad J_0(\mu_{n,k})&\geq& -C_R k^{-1\slash 10}- C R_1^{-1\slash 2}+J(\mu_2+\mu_2'+\mu_3+\mu_3')\nonumber\\
    &+&2\inf_{\vec{x}\in\mathcal{E}}J(\mu_1,\mu_3+\mu'_3)
+2\inf_{\vec{x}\in\mathcal{E}}\int
_{\mathbb{R}}\tilde{\xi}(x)d\mu_{n,k}(x).
\end{eqnarray}

Without loss of generality, we assume that $\mu(\{x\})=0$ for all $x\in\mathbb{R}$. Let $L:=k\nu_0([0,R_0'])\in\mathbb{N}_{+}$, and let $c_0:=-R_0'$ and $c_L:=R_0'$. For each $i\in [L-1]$, we take $c_i\in [-R_0',R_0']$ such that $(\mu+\nu_0)([-R_0',c_i])=i\slash k$. Now we define $\{\tilde{c}_i\}_{i=0}^L$ as follows. For each $i\in [L]$ with $c_i<-R_0$, we let $\tilde{c}_i:=c_i-2n^{-1\slash 30}$; for each $i\in [L]$ with $c_i>R_0$, we let $\tilde{c}_i:=c_i+2n^{-1\slash 30}$. We also let $\tilde{c}_0:=c_0-2n^{-1\slash 30}$. Below we consider the case where $\{c_1,\cdots,c_L\}\cap [-R_0,R_0]\neq\emptyset$. Suppose that
\begin{equation*}
    \{c_1,\cdots,c_L\}\cap [-R_0,R_0]=\{c_{i_1},c_{i_1+1},\cdots,c_{i_2}\},\quad\text{where }1\leq i_1\leq i_2\leq L.
\end{equation*}
We denote $\bar{c}_{i_1-1}:=-R_0$, $\bar{c}_{i_2+1}:=R_0$, and $\bar{c}_{i}:=c_i$ for each $i\in [i_1,i_2]\cap\mathbb{N}$. Note that $\max\{\bar{c}_{i}-\bar{c}_{i-1}:i\in [i_1,i_2+1]\cap\mathbb{N}\}\geq 2R_0\slash (L+1) \geq c_R\slash k$. Let $i_0:=\min\{i\in [i_1,i_2+1]\cap\mathbb{N}:\bar{c}_i-\bar{c}_{i-1}\geq 2R_0\slash (L+1)\}$. For each $i\in [i_1,i_2]\cap\mathbb{N}$ with $i< i_0$, we let $\tilde{c}_i:=c_i+2n^{-1\slash 30}$; for each $i\in [i_1,i_2]\cap\mathbb{N}$ with $i  \geq  i_0$, we let $\tilde{c}_i:=c_i-2n^{-1\slash 30}$. As $k$ is sufficiently large, we have $\tilde{c}_0<\tilde{c}_1<\tilde{c}_2<\cdots<\tilde{c}_L$.

Note that we have assumed that $\mathbb{P}(d_R(\mu_{n,k},0)\leq C_0 R k^{-1\slash 4}|\mathcal{L})>0$. By Lemma \ref{J12}, we have $|K\slash k-\mu_0([0,R_1])|=|\mu_{n,k}([-R_1,R_1])|\leq C_R k^{-1\slash 10}$. We take $L':=K-L-k\mu_0([R_0'',R_1])$ and $K':=k\mu_0([R_0'',R_1])\in\mathbb{N}_{+}$. As $k$ is sufficiently large, we can deduce that $L'\in\mathbb{N}_{+}$. For each $i\in [L']$, we take $d_i\in (R_0',R_0'')$ such that $\mu_0([R_0',d_i])=i\mu_0([R_0',R_0''])\slash (L'+1)$. For each $i\in [K']$, we take $e_i\in [R_0'',R_1]$ such that $\mu_0([R_0'',e_i])=(i-1)\slash k$.

Now we take $\mathcal{C}$ to be the set of $\vec{x}=(x_1,\cdots,x_K)$, such that for each $i\in[L]$, $x_i=\tilde{c}_i+t_i$ with $0< t_1< t_2<\cdots<t_L\leq (2n)^{-1}$; for each $i\in [L+1,L+L']\cap\mathbb{N}$, $|x_i-d_{i-L}|\leq(2n)^{-1}$; for each $i\in [L+L'+1,K]\cap\mathbb{N}$, $|x_i-e_{i-(L+L')}|\leq(2n)^{-1}$. It can be verified that for any $\vec{x}\in\mathcal{C}$, $d_{R_0}(\mu_{n,k},\mu)\leq \delta\slash 2$ (note that $k$ is taken to be sufficiently large). Moreover, by the construction of $\mathcal{C}$, we have $\mathcal{C}\subseteq\tilde{\mathcal{B}}$, hence $\mathcal{C}\subseteq\tilde{\mathcal{A}}$.

Now we bound $J_0(\mu_{n,k})$ for any $\vec{x}\in\mathcal{C}$. We have
\begin{eqnarray*}
&&-\int_{[-R_0',R_0']^2} \log(|x-y|)d(\mu+\nu_0)(x)d(\mu+\nu_0)(y)\\
&\geq& -\frac{1}{k^2}\sum_{i=1}^L\log(|c_i-c_{i-1}|)-\frac{2}{k^2}\sum_{1\leq j\leq L,0\leq i\leq j-2}\log(|c_i-c_j|)\\
&\geq& -\frac{1}{k^2}\sum_{i=1}^L\log(|\tilde{c}_i-\tilde{c}_{i-1}|)-\frac{2}{k^2}\sum_{1\leq j\leq L,0\leq i\leq j-2}\log(|\tilde{c}_i-\tilde{c}_j|)- C_R k^{-1}\nonumber\\
&\geq& -\frac{1}{k^2}\sum_{i=2}^L\log(|\tilde{c}_i-\tilde{c}_{i-1}|)-\frac{2}{k^2}\sum_{1\leq j\leq L,1\leq i\leq j-2}\log(|\tilde{c}_i-\tilde{c}_j|)- C_R k^{-1}.
\end{eqnarray*}
Hence for any $\vec{x}\in\mathcal{C}$,
\begin{eqnarray*}
&&-\frac{2}{k^2}\sum_{1\leq i<j\leq L}\log(|x_i-x_j|)\nonumber\\
&\leq&  -\frac{1}{k^2}\sum_{i=2}^L\log(|t_i-t_{i-1}|) -\frac{1}{k^2}\sum_{i=2}^L\log(|\tilde{c}_i-\tilde{c}_{i-1}|) -\frac{2}{k^2}\sum_{\substack{1\leq j\leq L,\\1\leq i\leq j-2}}\log(|\tilde{c}_i-\tilde{c}_j|)\nonumber\\
&\leq& -\frac{1}{k^2}\sum_{i=2}^L\log(|t_i-t_{i-1}|)+C_R k^{-1}\nonumber\\
&&-\int_{[-R_0',R_0']^2} \log(|x-y|)d(\mu+\nu_0)(x)d(\mu+\nu_0)(y).
\end{eqnarray*}
Using the above bound, similar to the proof of Theorem \ref{Main2}, we can deduce that for any $\vec{x}\in\mathcal{C}$,
\begin{equation*}
    J(\mu_1)\leq J(\mu)-\frac{1}{k^2}\sum_{i=2}^L\log(|t_i-t_{i-1}|)+C_R k^{-1\slash 25}, \quad J(\mu_1,\mu_2+\mu_2')\leq C_R k^{-1\slash 25}.
\end{equation*}
Therefore, for any $\vec{x}\in\mathcal{C}$, we have
\begin{eqnarray}
    J_0(\mu_{n,k})&\leq& J(\mu)+C_R k^{-1\slash 25}+J(\mu_2+\mu'_2+\mu_3+\mu'_3)+2\sup_{\vec{x}\in\mathcal{C}}J(\mu_1,\mu_3+\mu'_3)\nonumber\\
    &&+2\sup_{\vec{x}\in\mathcal{C}}\int_{\mathbb{R}} \tilde{\xi}(x)d\mu_{n,k}(x)-\frac{1}{k^2}\sum_{i=2}^L\log(|t_i-t_{i-1}|).
\end{eqnarray}

Now similar to the proof of Theorem \ref{Main2}, as $k$ is sufficiently large, we have 
\begin{equation}
\Big|\sup_{\vec{x}\in\mathcal{C}}J(\mu_1,\mu_3+\mu'_3)-\inf_{\vec{x}\in\mathcal{E}}J(\mu_1,\mu_3+\mu'_3)\Big|\leq CM^2 R_1^{-1\slash 2},
\end{equation}
\begin{eqnarray}\label{EEEE22}
&& 2\sup_{\vec{x}\in\mathcal{C}}\int_{\mathbb{R}}\tilde{\xi}(x) d\mu_{n,k}(x)
-2\inf_{\vec{x}\in\mathcal{E}}\int_{\mathbb{R}}\tilde{\xi}(x) d\mu_{n,k}(x) \nonumber\\
&=& 2\sup_{\vec{x}\in\mathcal{C}}\int_{[-R_1,0]}\tilde{\xi}(x) d\mu_{n,k}(x)
-2\inf_{\vec{x}\in\mathcal{E}}\int_{[-R_1,0]}\tilde{\xi}(x) d\mu_{n,k}(x) \nonumber\\
&\leq& 2\sup_{\vec{x}\in\mathcal{C}}\int_{[-R_1,0]} \tilde{\xi}(x)d\mu_{n,k}(x)\leq 2\int_{[-R_1,0]} \tilde{\xi}(x)d\mu(x)+C_R k^{-1},
\end{eqnarray}
where in the last inequality we note that for any $\vec{x}\in\mathcal{C}$ (note that $\tilde{\xi}(x)=0$ for any $x\in [0,R_1]$, and $\tilde{\xi}(x)\in [0,R_1^{3\slash 2}],\tilde{\xi}'(x)\in[-2\sqrt{R}_1,0]$ for any $x\in [-R_1,R_1]$),
\begin{eqnarray*}
  && \int_{[-R_1,0]}\tilde{\xi}(x)d\mu_{n,k}(x)=\frac{1}{k}\sum_{i=1}^L\tilde{\xi}(x_i)\leq  \frac{1}{k}\sum_{i=1}^L\tilde{\xi}(c_i)+C_R n^{-1\slash 30}\nonumber\\
  &\leq& \sum_{i\in [L]:c_i\leq 0}\int_{(c_{i-1},c_i]}\tilde{\xi}(x)d(\mu+\nu_0)(x)+C_R k^{-1}\nonumber\\
  &\leq& \int_{[-R_1,0]}\tilde{\xi}(x)d(\mu+\nu_0)(x)+C_R k^{-1}=\int_{[-R_1,0]}\tilde{\xi}(x)d\mu(x)+C_R k^{-1}.
\end{eqnarray*}

We denote 
\begin{equation*}
\Gamma:=J(\mu)+2\int_{[-R_1,0]}\tilde{\xi}(x)d\mu(x)+CM^2 R_1^{-1\slash 2} + C_R k^{-1\slash 25},
\end{equation*}
\begin{equation*}
\tilde{\Gamma}:=J(\mu)+\int_{[-R_1,0]} \frac{4}{3}|x|^{3\slash 2}d\mu(x)+CM^2 R_1^{-1\slash 2}.
\end{equation*}
As $\mathcal{C}\subseteq\tilde{\mathcal{A}}$, combining (\ref{EEEE21})-(\ref{EEEE22}), we obtain that
\begin{eqnarray*}
&&\frac{\mathbb{P}(\mathcal{A}|\mathcal{L})}{\mathbb{P}(d_R(\mu_{n,k},0)\leq C_0 R k^{-1\slash 4}|\mathcal{L})}\nonumber\\
&\geq& \frac{\exp(-k^2\Gamma)}{(2nR_1)^K}\int_{0<t_1<\cdots<t_L\leq (2n)^{-1}}\prod_{i=2}^L|t_i-t_{i-1}|dt_1\cdots dt_L\\
&\geq& \exp(-k^2\Gamma)(2nR_1)^{-C_Rk}(C_R nk)^{-C_R k}.
\end{eqnarray*}
Thus we can take $T_0=\exp(-k^2\Gamma)(2nR_1)^{-C_Rk}(C_R nk)^{-C_R k}$ in (\ref{B}), and obtain that $\mathbb{P}(\mathcal{A})\geq\exp(-k^2\Gamma) (2n R_1)^{-C_R k} (C_R nk)^{-C_R k}\slash 2$. Hence we have 
\begin{equation*}
\liminf_{k\rightarrow\infty}\frac{1}{k^2}\log \mathbb{P}(\mathcal{A})\geq -\tilde{\Gamma}.
\end{equation*}
Sending $R\rightarrow\infty$, we obtain that
\begin{equation*}
    \liminf_{k\rightarrow\infty}\frac{1}{k^2}\log \mathbb{P}(\mathcal{A})
\geq -I_{\infty,\infty}(\mu).
\end{equation*}
Following a similar argument as in the proof of Theorem \ref{Main2}, we obtain that for any $\eta\geq 15$ and sufficiently large $k$,
\begin{equation*}
    \mathbb{P}(d_{R_0}(\nu_{k;R_0},\mu|_{[-R_0,R_0]})\leq \delta)\geq \mathbb{P}(\mathcal{A})-\mathbb{P}(\mathcal{B}\cap\{d_{R_0}(\mu_{n,k}|_{[-R_0,R_0]},\nu_{k;R_0})\geq\delta\slash 2\}),
\end{equation*}
\begin{equation*}
    \mathbb{P}(\mathcal{B}\cap\{d_{R_0}(\mu_{n,k}|_{[-R_0,R_0]},\nu_{k;R_0})\geq\delta\slash 2\})\leq C\exp(-\eta k^2).
\end{equation*}
Thus by sending $\eta\rightarrow \infty$, we obtain that
\begin{equation*}
\liminf_{\delta\rightarrow 0^{+}}\liminf_{k\rightarrow\infty}\frac{1}{k^2}\log \mathbb{P}(d_{R_0}(\nu_{k;R_0},\mu|_{[-R_0,R_0]})\leq \delta)\\
\geq -I_{\infty,\infty}(\mu).
\end{equation*}

Now for any $\mu\in\mathcal{X}$ and any $\tilde{\mu}\in\mathcal{Z}$ such that $\tilde{\mu}(\mathbb{R})=0$, $\tilde{\mu}|_{[-R_0,R_0]}=\mu$, we have $\mathbb{P}(d_{R_0}(\nu_{k;R_0},\mu)\leq\delta)=\mathbb{P}(d_{R_0}(\nu_{k;R_0},\tilde{\mu}|_{[-R_0,R_0]})\leq \delta)$. Hence by taking a supremum over $\tilde{\mu}$, we have
\begin{equation*}
\liminf_{\delta\rightarrow 0^{+}}\liminf_{k\rightarrow\infty}\frac{1}{k^2}\log{\mathbb{P}(d_{R_0}(\nu_{k;R_0},\mu)\leq \delta)}\geq -I_1(\mu),
\end{equation*}
which implies the conclusion of Theorem \ref{Main1}, part (a).  
\end{proof}

\section*{Acknowledgments}
The author is grateful to Amir Dembo for suggesting the problem, encouragement, and many helpful conversations. The author thanks Ivan Corwin for proposing the problem and helpful conversations, and thanks Li-Cheng Tsai for helpful conversations. 

\bibliographystyle{acm}
\bibliography{Airy.bib}

\end{document}